\documentclass{amsart}
\usepackage{amsthm,amsmath,amssymb,color,graphics,enumerate,enumitem}
\usepackage[dvipsnames]{xcolor}
\usepackage{tikz-cd}
\usepackage{mathrsfs,mathtools}
\usepackage{adjustbox}

\usepackage[T1]{fontenc}
\usepackage{lmodern}
\usepackage{eurosym}
\usepackage{mathrsfs}
\usepackage{xspace}
\usepackage{booktabs}
\usepackage{graphicx}
\usepackage{multirow}
\usepackage{array}
\usepackage{xcolor}
\usepackage{csquotes}
\usepackage{bibentry}
\usepackage{atbegshi}
\usepackage{tabularx}
\usepackage{tikz}

\usepackage[colorlinks=true, citecolor=black, linkcolor=black, urlcolor=blue]{hyperref}
\usepackage{lscape}
\usepackage{amsfonts}

\AtBeginDocument{%
   \def\MR#1{}
}

\newtheorem{theorem}{Theorem}[section]
\newtheorem{lemma}[theorem]{Lemma}
\newtheorem{proposition}[theorem]{Proposition}
\newtheorem{corollary}[theorem]{Corollary}

\newtheorem*{theorem*}{Theorem}
\newtheorem*{theoremA}{Theorem A}
\newtheorem*{theoremB}{Theorem B}
\newtheorem*{theoremC}{Theorem C}
\newtheorem*{theoremD}{Theorem D}
\newtheorem*{theoremE}{Theorem E}
\newtheorem*{theoremF}{Theorem F}
\newtheorem*{theoremG}{Theorem G}
\newtheorem*{theoremH}{Theorem H}
\newtheorem*{theoremI}{Theorem I}
\newtheorem*{theoremJ}{Theorem J}
\newtheorem*{theoremK}{Theorem K}
\newtheorem*{theoremL}{Theorem L}

\newtheorem*{theoremELT}{The Ending Lamination Theorem}
\newtheorem*{theoremAFT}{Ahlfors's Finiteness Theorem}

\theoremstyle{definition}
\newtheorem{definition}[theorem]{Definition}
\newtheorem{remark}[theorem]{Remark}
\newtheorem{example}[theorem]{Example}
\newtheorem{question}{Question}

\newtheorem*{conjecture*}{Conjecture}

\newcounter{proofcount}
\AtBeginEnvironment{proof}{\stepcounter{proofcount}}
\newtheorem{claim}{Claim}
\makeatletter                  
\@addtoreset{claim}{proofcount}
\makeatother                   

\renewcommand{\phi}{\varphi} 

\begin{document}

\title[Manifold classification from the descriptive viewpoint]{Manifold classification \\ from the descriptive viewpoint}

\author[J. Bergfalk]{Jeffrey Bergfalk}
\address{Departament de Matem\`{a}tiques i Inform\`{a}tica \\
Universitat de Barcelona \\
Gran Via de les Corts Catalanes 585 \\ 08007 Barcelona, Catalonia}
\email{bergfalk@ub.edu}
\urladdr{https://www.jeffreybergfalk.com/}

\author[I. B. Smythe]{Iian B. Smythe}
\address{Department of Mathematics and Statistics\\ University of Winnipeg\\ 515 Portage Avenue\\ Winnipeg, Manitoba, Canada\\ R3B 2E9}
\email{i.smythe@uwinnipeg.ca}
\urladdr{https://www.iiansmythe.ca/}

\subjclass[2020]{Primary 03E15, 58H05, 57K20, 57K32, 22E40; Secondary 22E46, 20H10, 30F40, 57M99.}
\thanks{The first author was partially supported by Mar\'{i}a Zambrano and Marie Sk\l odowska Curie (CatT 101110452) and Ram\'{o}n y Cajal Fellowships; the second author was partially supported by an NSERC Discovery Grant (RGPIN-2023-03343). Additional support was provided by the 2021 SGR grant 00348.}

\begin{abstract}
We consider classification problems for manifolds and discrete subgroups of Lie groups from a descriptive set-theoretic point of view.
This work is largely foundational in conception and character, recording both a framework for general study and Borel complexity computations for some of the most fundamental classes of manifolds. We show, for example, that for all $n\geq 0$, the homeomorphism problem for compact topological $n$-manifolds is Borel equivalent to the relation $=_{\mathbb{N}}$ of equality on the natural numbers, while the homeomorphism problem for noncompact topological $2$-manifolds is of maximal complexity among equivalence relations classifiable by countable structures. 
A nontrivial step in the latter consists of proving Borel measurable formulations of the Jordan--Schoenflies and surface triangulation theorems.
Turning our attention to groups and geometric structures, we show, strengthening results of Stuck--Zimmer and Andretta--Camerlo--Hjorth, that the conjugacy relation on discrete subgroups of any noncompact semisimple Lie group is essentially countable universal.
So too, as a corollary, is the isometry relation for complete hyperbolic $n$-manifolds for any $n\geq 2$, generalizing a result of Hjorth--Kechris.
We then show that the isometry relation for complete hyperbolic $n$-manifolds with finitely generated fundamental group is, in contrast, Borel equivalent to the equality relation $=_{\mathbb{R}}$ on the real numbers when $n=2$, but that it is not concretely classifiable when $n=3$; thus there exists no Borel assignment of numerical complete invariants to finitely generated Kleinian groups up to conjugacy.
We close with a survey of the most immediate open questions.
\end{abstract}

\maketitle

\setcounter{tocdepth}{1}
\tableofcontents

\section{Introduction}
A curiosity of the enormous advances in the fields of manifold classification and invariant descriptive set theory over the past forty years has been their comparative lack of interaction.
We can hardly attribute this, for example, to any inherent antipathy between manifold theory and logic; against the background of the extensive literature interrelating manifolds and computability theory \cite{MR0263090, MR1997069, MR2275866, MR4744736}, in fact, this curiosity grows all the more striking.
Add to this not only invariant descriptive set theory's programmatic interest in classification problems throughout mathematics (particularly in manifold-adjacent fields like dynamics 
\cite{MR2800720, MR4416587} and geometric group theory \cite{MR2393184,MR4252215,MR4822247}), but Hjorth's framing, in one of the field's foundational texts, of the classification of compact surfaces as among the most exemplary of solutions to such problems \cite[Exam.\ 0.1]{MR1725642}, and we have the makings of a positive mystery.

The purpose of this paper is to begin to remedy this lack (this involves, naturally, identifying one of the mystery's likely culprits; see below).
A few sample theorems will help to suggest the benefits of doing so.
First, note the empirical fact that classes of mathematical objects, and their classification up to standard notions of equivalence, routinely admit natural parametrizations as, respectively, Polish (or at least standard Borel) spaces and analytic equivalence relations thereon (see Section \ref{subsection:invariant_DST} below for fuller discussion of this paragraph's terms and claims).
The existence of Borel measurable maps between two such spaces respecting these relations determines the partial order $\leq_B$ of \emph{Borel reducibility}; modulo bireducibility, or \emph{Borel equivalence}, we term the elements of this partial order \emph{Borel complexity degrees}, and it is the study both of the rich structure of these degrees and of the locations of particular classification problems throughout mathematics among them by which we mean the phrase \emph{invariant descriptive set theory}. 
Within this framework, the relation $E\leq_B F$ carries the heuristic that the classification problem represented by $F$ is \emph{at least as hard} as that represented by $E$, since a map witnessing $E\leq_B F$ will convert any solution for $F$ into a solution for $E$.
The strict relation $E<_B F$ then carries the heuristic that the classification problem represented by $F$ is \emph{significantly harder} than that represented by $E$, in the sense that \emph{merely countable resources} (instantiated by a Borel map) will never suffice to convert a solution for $E$ into one for $F$. See \cite{MR3837073} for a non-technical overview of this area.

Against this background, consider the following theorem of Markov \cite{MR0097793}.
\begin{theorem*}[Markov, 1958]
There exists an effective map $M$ from the class of finite presentations $P$ of groups $G(P)$ to that of compact $4$-manifolds $M(P)$ such that:
\begin{enumerate}[label=\textup{\roman*.}]
\item $\pi_1(M(P))\cong G(P)$ for all $P$, and
\item $M(P)\cong M(Q)$ if and only if $G(P)\cong G(Q)$.
\end{enumerate}
\end{theorem*}
\noindent This $M$ reduces the recursively unsolvable group isomorphism problem to the classification of compact $4$-manifolds up to homeomorphism and carries the folklore corollary that the latter is ``impossible'', but this inference should be handled with care.
For we have, in contrast, the following (our Corollaries \ref{cor:compact_manifolds_countable} and \ref{cor:Borel_compact_with_boundary} below):

\begin{theoremA}
For all $n\geq 0$, both the classification
\begin{itemize}
\item of compact topological $n$-manifolds up to homeomorphism, and
\item of compact smooth $n$-manifolds up to diffeomorphism
\end{itemize}
are Borel equivalent to the relation $=_{\mathbb{N}}$ of equality on the natural numbers.
\end{theoremA}

\noindent In particular, from an invariant descriptive set-theoretic perspective, the classification of compact $4$-manifolds up to homeomorphism amounts to the \emph{simplest} of all infinite classification problems (that is, of equivalence relations possessing infinitely many distinct equivalence classes).
This is, of course, compatible with a view, deriving from Markov's theorem, of the classification of compact $4$-manifolds as \emph{hard}, but only alongside a view of classification problems generally accredited as solved as being even harder.
The degrees appearing in the following theorems (corresponding to our Corollary \ref{cor:2isometrysmooth}, Theorem \ref{thm:complexity_3-mans}, Corollary \ref{cor:isometry_universal}, and Theorem \ref{thm:surfaces} and Corollary \ref{cor:smooth_surfaces}, respectively), for example, form a strictly increasing $<_B$-chain above $=_{\mathbb{N}}$.

\begin{theoremB} The classification of complete hyperbolic $2$-manifolds with finitely generated fundamental group up to isometry is Borel equivalent to the relation $=_{\mathbb{R}}$ of equality on the real numbers.
\end{theoremB}

\begin{theoremC} The classification of complete hyperbolic $3$-manifolds with finitely generated fundamental group up to isometry is of Borel complexity at least that of the relation $E_0$ of eventual equality of infinite binary sequences.
\end{theoremC}

\begin{theoremD} For any $n\geq 2$, the classification of all complete hyperbolic $n$-manifolds up to isometry is essentially countable universal, or, in other words, Borel equivalent to the degree $E_\infty$.
\end{theoremD}

\begin{theoremE} The classification of connected topological $2$-manifolds (i.e., surfaces) up to homeomorphism is complete for countable structures, as is the classification of connected smooth $2$-manifolds up to diffeomorphism.
Put differently, each is complete for orbit equivalence relations induced by Borel actions of the infinite permutation group $S_\infty$.
\end{theoremE}

Especially alert readers will have observed that the classification problems associated to Theorems B and E each have well-known classical solutions; our proof of the latter, for example, amounts to an analysis of the Borel content of the Ker\'{e}kj\'{a}rt\'{o}--Richards \cite{zbMATH02598767, MR143186} classification of noncompact surfaces. A preliminary component of this analysis consists in establishing Borel measurable (with respect to the appropriate spaces) forms of the classical Jordan--Schoenflies and triangulation theorems for surfaces (our Theorems \ref{thm:Borel_JS} and \ref{thm:Borel_triangulation}, respectively):

\begin{theoremF}
	There is a Borel measurable function which takes any embedding $\gamma$ of $S^1$ into $\mathbb{R}^2$ to a self-homeomorphism of $\mathbb{R}^2$ which extends $\gamma$.
\end{theoremF}

\begin{theoremG}
	There is a Borel measurable function which takes any surface $S$ to a simplicial complex whose geometric realization is homeomorphic to $S$.
\end{theoremG}

Theorem C, like Theorem E, is largely a story of manifolds' ends and draws heavily on a celebrated classification result, namely the Ending Lamination Theorem \cite{MR2630036, MR2925381} and its aftermath.
Theorem D is somewhat different in nature, as we will see.
In any case, the apparent tension of Theorems A through E with the implications of Markov's is, of course, ultimately illusory, for they involve different registers: while the latter derives its force from the question of whether the word problem admits resolution with \emph{finite} resources, the underlying concern in the former is, again, for what is or is not achievable with \emph{countable} resources.\footnote{Note that if it were merely a matter of assigning arbitrary complete invariants to objects, then any classification problem would be trivially solvable by the Axiom of Choice. What finitary or countable constraints as above help to formalize is the notion of a \emph{reasonable} solution to a classification problem: only solutions making good, economical use of properties of the objects in question will tend to satisfy them. See Section \ref{subsection:invariant_DST} for further discussion, one emphasizing the complementary heuristic of \emph{definable} solutions to classification problems.}
To summarize our discussion so far: \emph{One value of the Borel complexity framework is its description of a finer and more definite structure of ``degrees of difficulty'' within the study of manifolds than has tended to be perceived}.

Extending this point is one outstanding exception to the comparative lack of interactions between manifold theory and Borel complexity theory cited above, namely Hjorth and Kechris's \emph{The complexity of the classification of Riemann surfaces and complex manifolds} \cite{MR1731384}. 
Therein, motivated by questions regarding the complexity of conformal invariants from \cite{MR0584077}, the authors show the following:
\begin{theorem*}[Hjorth--Kechris, 2000]\label{thm:Hjorth-Kechris}
\quad
\begin{enumerate}[label=\textup{\arabic*.}]
	\item The classification of Riemann surfaces up to biholomorphism is essentially countable universal.
	\item For any $n\geq 2$, the biholomorphism relation on complex $n$-manifolds does not admit classification by countable structures.
\end{enumerate}
\end{theorem*}
\noindent Since by the uniformization theorem all but an exceptional few Riemann surfaces are hyperbolic, our Theorem D may be viewed as a generalization of (1).
Clause (2), on the other hand, is an application of Hjorth's theory of turbulence \cite{MR1725642}; more colloquially, it says that there exists no reasonable assignment of countable groups or algebras, let alone numerical invariants, to distinguish between, for example, complex $2$-manifolds up to biholomorphism.
We highly recommend this paper to the reader; it forms the present work's principal inspiration.
But we should also note that roughly 9 of its 34 pages are taken up with what its authors term the ``technically cumbersome, although mathematically rather shallow'' business of parametrizing its objects of study, and we do so for what it suggests of at least one of the culprits behind the ``mystery'' framing our work: it seems in large part to have been a deficiency of compelling parametrizations of manifolds that has hindered their descriptive set-theoretic study thus far.

For what is ultimately wanted is an approach to parametrization which is sufficiently versatile to accommodate the multiplicity of manifold structures arising in practice, but sufficiently precise not to confuse them; a transparency facilitating such analyses as underlie the theorems above is, of course, a value as well.
To underscore the challenge, note that these are not criteria that Gromov--Hausdorff parametrizations, for example, ultimately satisfy (invoking, as they do, metrics which topological manifolds don't intrinsically possess; see Remark \ref{rmk:parametrization} for further discussion). 
Perhaps this paper's most fundamental contribution is a parametrization framework which does meet these criteria.
In Section \ref{section:parametrization} below, after briefly reviewing the pseudogroup approach to manifold theory, we introduce the notion of a \emph{Borel pseudogroup} and prove the following result (as Theorems \ref{thm:M(G,X)_Borel} and \ref{thm:equivalence_analytic}).

\begin{theoremH}
For any Borel pseudogroup $\mathcal{G}$ on a locally compact Polish space $X$, there is a standard Borel parameter space $\mathfrak{M}(\mathcal{G},X)$ of all $(\mathcal{G},X)$-manifolds on which the relation of $(\mathcal{G},X)$-equivalence is an analytic equivalence relation.
\end{theoremH}

\noindent Prominent examples of Borel pseudogroups include $\mathsf{Top}$ and $\mathsf{C}^\infty$ on $\mathbb{R}^n$, and $\mathsf{Isom}$ on $\mathbb{H}^n$, determining standard Borel parameter spaces of all topological, smooth, and hyperbolic $n$-manifolds, respectively (see Theorem \ref{thm:Top_is_Borel} and Examples \ref{ex:smooth_pseudogroup} and \ref{ex:isometry_pseudogroups} below).

In particular cases, of course, manifold theory boasts some of the most celebrated parameter spaces in all of mathematics, namely moduli spaces whose geometric structures encode deep properties of the families of objects under study; examples will figure in our Sections \ref{section:isometry_for_2} and \ref{section:isometry_for_3} below.
Theorem H may be read as a more uniform provision of moduli spaces of manifolds, but at the cost of coarsening their structure to that of standard Borel spaces. Borel complexity theory can then help clarify which of these parameter spaces meaningfully support further topological or geometric structures (see, for example, Corollary \ref{cor:not_T0}).

In contrast to the situation just described for manifolds, the discrete subgroups of a locally compact group form a well-known parameter space, via the Chabauty (also known as the \emph{Fell} or \emph{geometric}) topology, on which the ambient group acts by conjugation. The resulting conjugacy relations have been extensively studied in both descriptive set theory and ergodic theory.
To briefly summarize the most relevant results for our present work: Stuck and Zimmer showed in \cite{MR1283875} that the conjugacy relation on subgroups of the free group on two generators $F_2$ is not \emph{concretely classifiable}, meaning it does not Borel reduce to $=_\mathbb{R}$, from which they derived the same conclusion for the conjugacy relation on the discrete subgroups of any noncompact semisimple Lie group. Thomas and Velickovic \cite{MR1700491} showed that the conjugacy relation on subgroups of $F_2$ is, in fact, countable universal, which was generalized by Andretta, Camerlo, and Hjorth \cite{MR1815088} to all countable groups which contain $F_2$. We extend these results to a large class of Lie groups and prove a best-possible sharpening of the Stuck and Zimmer result (our Theorem \ref{cor:semisimple_universal}):

\begin{theoremI}
The conjugacy relation on discrete subgroups of any noncompact semisimple Lie group is essentially countable universal.
\end{theoremI}

This detour through groups and conjugacy relations connects back to our study of manifolds via the classical correspondence between complete hyperbolic $n$-manifolds up to isometry and discrete subgroups of the isometry group $\mathrm{Isom}(\mathbb{H}^n)$ up to conjugacy. We show in Section \ref{section:bireducibility} that this correspondence is witnessed by Borel reductions in both directions, by way of which we deduce Theorem D as a corollary of Theorem I.

We then turn our attention to the groups of orientation-preserving isometries $\mathrm{Isom}^+(\mathbb{H}^2)$ and $\mathrm{Isom}^+(\mathbb{H}^3)$; these naturally identify with the Lie groups $\mathrm{PSL}(2,\mathbb{R})$ and $\mathrm{PSL}(2,\mathbb{C})$, and their discrete subgroups are exactly what Poincar\'{e} dubbed the \emph{Fuchsian} and \emph{Kleinian} groups, respectively.
Thus, by way of the correspondence cited just above, the families of finitely generated torsion-free Fuchsian and Kleinian groups, respectively, parametrize the classes of complete orientable hyperbolic $2$- and $3$-manifolds possessing a finitely generated fundamental group.
For brevity, as well as for contrast with other related manifold finiteness conditions, we frequently term such manifolds \emph{algebraically finite} below.

Observe that within a countable group, there are only countably many finitely generated subgroups, hence their conjugacy relation trivially reduces to $=_{\mathbb{N}}$.
Less trivially, Stuck and Zimmer proved in \cite{MR1283875} that, in any Lie group, the conjugacy problem for lattices, i.e., for discrete subgroups with finite co-volume (which are necessarily finitely generated), is always concretely classifiable.
In the special case of $\mathrm{PSL}(2,\mathbb{R})$, our Theorem B translates via the aforementioned correspondence to the following extension of these results.

\begin{theoremJ}
	The conjugacy relation on finitely generated torsion-free Fuchsian groups is Borel equivalent to $=_\mathbb{R}$.
\end{theoremJ}
The algebraically finite hyperbolic $2$-manifolds are exactly the \emph{geometrically finite} ones, meaning that their convex core is always of finite volume, but in higher dimensions, these conditions no longer coincide.
In three dimensions, the algebraically finite class is, in senses both of inclusion and of complexity, the strictly bigger one, by the following reformulation of our Theorem \ref{thm:complexity_3-mans}.

\begin{theoremK}
The Borel complexity degrees of the conjugacy relation on the major finiteness classes of torsion-free Kleinian groups are as follows:
\begin{itemize}
\item for lattices, it is $=_{\mathbb{N}}$;
\item for geometrically finite groups, it is $=_{\mathbb{R}}$;
\item for finitely generated groups, it is at least $E_0$.
\end{itemize}
Corresponding results hold for the isometry relation on the classes of finite volume, geometrically finite, and algebraically finite hyperbolic $3$-manifolds, respectively.
\end{theoremK}

Much as for Theorem D and Theorem I, one might regard Theorems B and C as corollaries of Theorems J and K, respectively.
In actuality, however, our arguments of the latter run in the opposite direction: both essentially proceed from the manifold side to the group side, to the degree that these two sides can even be meaningfully disentangled.
The third item in Theorem K, for example, largely derives from the following (our Theorem \ref{thm:line_bundles}); here, by $\mathcal{D}[\Gamma]$ we mean the set of discrete subgroups of $\mathrm{PSL}(2,\mathbb{C})$ which are group-isomorphic to $\Gamma$, endowed with the Chabauty topology.
\begin{theoremL}
Let $S$ be a closed orientable hyperbolizable surface.
The classification up to isometry of doubly degenerate hyperbolic manifolds homeomorphic to $S\times\mathbb{R}$ is Borel equivalent to $E_0$.
In particular, there exists no Borel assignment of complete numerical invariants, up to $\mathrm{PSL}(2,\mathbb{C})$-conjugacy, to the elements of $\mathcal{D}[\pi_1(S)]$.
\end{theoremL}
The theorem allows us to strengthen a result of \cite{MR3008919}, with the aid of which we show the isometry relation on a much wider range of algebraically finite hyperbolic $3$-manifolds to be essentially hyperfinite as well; see Section \ref{subsection:classifying_gi_ends} for discussion.
This brings us to a sequence of concluding points.

First, our proof of Theorem J interacts closely with the recent work \cite{BLL26+}, while those of Theorems K and L draw heavily on works \cite{MR2582104, MR2876139, MR3008919, MR4095420, MR4252215} falling broadly, as noted, in the orbit of the Ending Lamination Theorem.
They are, in other words, applications of deep results in manifold theory to a problem originating in descriptive set theory, underscoring that the gain from closer dialogue between these fields is bilateral.

Second, the scarcity, underscored above, of interactions between the fields of manifold classification and invariant descriptive set theory has very recently begun to abate; the present work figures even as part of an emergent countertrend.
In addition to \cite{MR1731384}, cited above, let us briefly survey the main exceptions to that scarcity which we're aware of.
In the (withdrawn) preprint \cite{Kulikov_homeos}, Weinstein considered the complexity of the homeomorphism relation on different classes of manifolds and conjectured the complexity result for surfaces that we verify in our Theorem E.
More recently, and independently from our own work, Hoganson and Zomback showed in \cite{2024arXiv241000901H} that the homeomorphism relation for orientable surfaces endowed with a pants decomposition is complete for countable structures.\footnote{In \cite{2024arXiv241000901H}, the pants decomposition of a surface is used to identify it with a graph encoding the relevant combinatorial information. Contrast this with our Theorem G, which extracts that information from a purely topological presentation of the surface.}
In their forthcoming \cite{IW26+}, Iannella and Weinstein show that the piecewise-linear homeomorphism relation on simplicial complexes, in \emph{any} dimension, is classifiable by countable structures, and prove an analogue of our Theorem G for $3$-manifolds which shows that the corresponding homeomorphism relation is also complete for countable structures. And in work currently in progress, Gompf and Panagiotopoulos have shown that neither contractible $3$-manifolds up to homeomorphism, nor smooth structures on $\mathbb{R}^4$ up to diffeomorphism, are concretely classifiable \cite{GP25}; see our conclusion for a bit more on these matters. 

Third, it is largely for their status as the ``generic'' geometric manifolds (particularly in dimensions 2 and 3; see \cite{MR0648524}) that we've taken hyperbolic manifolds as a focus, and if in the present context Theorem D is the first thing one might wish to know about them, Theorems B and C are main steps towards the next.
Note, though, that our analyses of algebraically finite manifolds reach no higher than 3 hyperbolic dimensions, for the reason that their mainspring, a deep and cohesive theory of finitely generated discrete subgroups of $\mathrm{Isom}^+(\mathbb{H}^n)$, at present doesn't seem to either.
As Kapovich, in the major survey of dimensions above 3, writes: ``There is a vast variety of Kleinian groups in higher dimensions: It appears that there is no hope for a comprehensive structure theory similar to the theory of discrete groups of isometries of $\mathbb{H}^3$'' \cite[p.\ 485]{MR2402415}.
This is the sort of perception which invariant descriptive set theory excels at making precise; proofs, for example, for some $n>3$, that the conjugacy relation on the class of geometrically finite discrete subgroups of $\mathrm{Isom}^+(\mathbb{H}^n)$ is \emph{not} concretely classifiable --- or that it is \emph{not} hyperfinite on the class of all finitely generated discrete subgroups of $\mathrm{Isom}^+(\mathbb{H}^n)$ --- would be significant steps in this direction (see our conclusion's Question \ref{ques:geom_alg_fin}), and this prospect typifies the kinds of applications of Borel complexity theory to manifold theory which we have ultimately in mind.

Lastly, considerations of readability have tempered a striving for results of maximum generality; the qualifier \emph{torsion-free} in Theorems J and K, for example, is likely unnecessary, but removing it would, in translation, entail a study of orbifolds, taking us too far afield (of course, there may be cleverer ways of removing it).
Similarly for our restrictions to finite-dimensional manifolds and, in another direction, for our negligence of the neighboring field of the classification of algebraic varieties (see \cite[Ch.\ I.8]{MR0463157}).
It is, after all, our most overarching point that contact between the fields of invariant descriptive set theory and manifold theory gives rise to many more excellent questions than any one of us alone might answer. We show that these admit satisfying answers by supplying some of them; a sample of the many that remain is collected in our conclusion.
Let us turn now to a more orderly survey of this paper's contents.

\subsection{Organization of the paper} We hope readers will feel free, at least at a first pass, to periodically skip ahead in their reading. This is for several reasons. One is the text's sheer length, another is its diversity of intended audience, and another is 
the unfortunate reality that verifications that a given map or construction is Borel, though necessary, are not as uniformly edifying as we might wish them to be.

Our paper's overall structure is as follows: Section \ref{section:preliminaries} organizes into the subsections \emph{Classical descriptive set theory}, \emph{Invariant descriptive set theory}, and \emph{Spaces of subsets}, and records the core of what we'll need under these headings for the remainder of the paper.
This material is all standard, but that of the third subsection plays a sufficiently critical role in what follows to merit special emphasis.

In Section \ref{section:parametrization}, we record our overarching framework as well as the basic tools for its manipulation and analysis.
Section \ref{subsection:pseudogroups} recalls the definition of a $(\mathcal{G},X)$-manifold for any pseudogroup $\mathcal{G}$ on a model space $X$; a parameter space $\mathfrak{M}(\mathcal{G},X)$ for such manifolds is defined and shown to be standard Borel whenever $\mathcal{G}$ is a Borel pseudogroup on a locally compact Polish space $X$. The representative examples of topological, smooth, and hyperbolic $n$-manifolds are each shown to fit into this framework.
In Section \ref{subsection:equivalences}, we show that the natural notion of equivalence for $(\mathcal{G},X)$-manifolds manifests as an analytic equivalence relation on $\mathfrak{M}(\mathcal{G},X)$, completing the proof of Theorem H, and derive Theorem A as a corollary.
In Section \ref{subsection:subclasses}, we collect a number of lemmas for later application; we show, for example, that exhaustions of manifolds by compact sets can be computed, and that manifolds can be reparametrized as equivalent ones possessing charts with preferred properties (e.g.,~local finiteness or geodesic convexity), in Borel ways.
We also show that the subspaces of $\mathfrak{M}(\mathcal{G},X)$ corresponding to complete or connected $(\mathcal{G},X)$-manifolds are Borel, and
isolate in metric contexts a Borel subspace $\mathfrak{M}^*(\mathcal{G},X)$ of $\mathfrak{M}(\mathcal{G},X)$ in which the rapport between manifolds' induced metrics and chart-metrics is optimal.
As their descriptions should suggest, Sections \ref{subsection:pseudogroups} and \ref{subsection:equivalences} are fundamental to all that follows, while the more technical \ref{subsection:subclasses} admits more superficial initial readings, to be augmented as the need arises.

Section \ref{section:surfaces} concerns the classification of surfaces. We first formulate and prove our Borel forms of the Jordan--Schoenflies and triangulation theorems, Theorems F and G above, by adapting graph-theoretic proofs of these results, due to Thomassen \cite{MR1144352}, to our framework. Then we show that the space of ends of a surface, as well as the relevant topological invariants (genus, orientability, planarity), can all be computed in a Borel way. The Ker\'{e}kj\'{a}rt\'{o}--Richards result reduces the classification of surfaces to that of their spaces of ends; as the latter are nested triples of compact totally disconnected Polish spaces, this is then reduced to the classification of countable Boolean algebras equipped with a pair of nested filters, using the techniques of \cite{MR1804507}, proving Theorem E. This section ends with a discussion of disconnected $2$-manifolds and, relatedly, records descriptive set-theoretic reasons to prioritize, within various manifold classification programs, the connected ones.

The organizing concern of Section \ref{section:bireducibility} is translation mechanisms between spaces of manifolds of a given type and spaces $\mathcal{D}_{\mathrm{tf}}(G)$ of torsion-free subgroups of an associated group $G$.
As noted, the latter are equipped with the \emph{Chabauty--Fell} or \emph{geometric} topology, and we begin by recalling that if $G$ is locally compact and Polish, then $\mathcal{D}_{\mathrm{tf}}(G)$ is Polish as well.
We then focus, with an eye to subsequent sections, on the cases of $G=\mathrm{Isom}(\mathbb{H}^n)$ for $n\geq 2$, describing Borel maps in both directions between $\mathcal{D}_{\mathrm{tf}}(G)$ and the parameter space $\mathfrak{C}^{*}_c(\mathsf{Isom},\mathbb{H}^n)$ of complete connected $n$-hyperbolic manifolds.
These maps are, in essence, the quotient map in one direction and the holonomy representation in the other, and it is this section's main result that they implement a classwise Borel equivalence between the isometry and conjugacy relations, respectively, on $\mathfrak{C}^{*}_c(\mathsf{Isom},\mathbb{H}^n)$ and $\mathcal{D}_{\mathrm{tf}}(\mathrm{Isom}(\mathbb{H}^n))$.

In Section \ref{section:conjugacy}, we consider the conjugacy problem for discrete subgroups of Lie groups. We show that whenever $G$ is a matrix group which contains a discrete nonabelian free subgroup, the associated conjugacy problem $E(G,\mathcal{D}(G))$ is essentially countable universal, extending the cited result of Andretta--Camerlo--Hjorth.
We then apply this result to prove our Theorem I, and, via the results of Section \ref{section:bireducibility}, derive Theorem D as a corollary in the manner described above.
As part of this analysis, and for use in subsequent sections, we also prove several lemmas showing how the complexity of the conjugacy problem passes between a group, its subgroups, and its quotients.
In particular, these lemmas free us to study those relations on whichever of $\mathrm{Isom}^{+}(\mathbb{H}^n)$ and $\mathrm{Isom}(\mathbb{H}^n)$ is more convenient thereafter.

After Section \ref{section:conjugacy}, we turn to a study of conjugacy relations on \emph{finitely generated} discrete subgroups of Lie groups $G$; when $G=\mathrm{Isom}(\mathbb{H}^n)$, the torsion-free such subgroups correspond (via the translations of Section \ref{section:bireducibility}) to those hyperbolic $n$-manifolds satisfying the most general of the finiteness conditions typically considered (e.g., \emph{closed}, or \emph{finite volume}, or \emph{geometrically finite}, etc.).
We term such manifolds \emph{algebraically finite}, as noted above, and they are the setting, both when $n=2$ and $n=3$ for some of the subject's most fundamental classification results.
We focus on the $n=2$ case in Section \ref{section:isometry_for_2}, wherein we record one proof via moduli space theory, along with a sketch of a second proof due to Ian Biringer, of Theorems B and J listed above.

We then turn to the $n=3$ case in Section \ref{section:isometry_for_3}.
This material draws heavily, both technically and conceptually, on the Teichm\"{u}ller and moduli spaces machinery of Section \ref{section:isometry_for_2}; here, though, the distinction between marked and unmarked manifolds is, from a complexity theoretic perspective, a more consequential one.
After noting that the isometry classes of the former admit natural Polish parametrizations, we work through (the manifold renderings of) the cases of Theorem K to show that the isometry classes of \emph{unmarked} algebraically finite hyperbolic $3$-manifolds in general do not.
Along the way, we show that obstructions to the latter concentrate among trivial line bundles like those considered in Theorem L.
We draw on \cite{MR4095420} and \cite{MR4252215} to prove Theorem L and apply the second work to bound the complexity of the isometry relation for a broader class of algebraically finite manifolds; in the process, again via the translations of Section \ref{section:bireducibility}, we complete the proof of Theorem C.

In Section \ref{section:questions}, we conclude with a brief survey of open questions.
Since no survey of any modest length could be exhaustive, we content ourselves with a discussion of what seem to us the most conspicuous open questions, as well as of their interrelations.
As we underscore therein, what we have aimed for is a conclusion which opens onto further work and dialogue more than it closes anything at all.

\subsection*{Acknowledgements} 
Multiple visits to the Fields Institute, most recently for its 2023 \emph{Thematic Program on Set Theoretic Methods in Algebra, Dynamics and Geometry}, significantly contributed to this work; we wish to thank both the institute and the program's organizers for their hospitality.
We thank the University of Barcelona for its hospitality to the second author in the spring of 2025 as well.
This work has been shaped by discussions too numerous to sensibly detail here.
Our thanks to all who partook in them; particular thanks to Ian Biringer, Brian Bowditch, Richard Canary, Martina Iannella, Michael Kapovich, Ken'ichi Ohshika, Aristotelis Panagiotopoulos, Joan Porti, Marcin Sabok, Simon Thomas, and Ferr\'{a}n Valdez for especially helpful ones.

\section{Preliminaries}
\label{section:preliminaries}

In this section, we review for a general audience the necessary background material in descriptive set theory, or the study of definable subsets of Polish spaces. Experts are welcome to skip over this material, although we recommend at least perusing Section \ref{subsection:spaces_of_subsets}, concerning topologies and Borel structures on spaces of subsets of a Polish space, as it will be crucial for our parametrizations of spaces of manifolds in the sequel and may be less familiar. Among our conventions, we note without delay that we regard $0$ as a natural number.

\subsection{Classical descriptive set theory}\label{subsection:classical_DST}

Recall that a topological space $X$ is \emph{Polish} if it is separable and completely metrizable. Examples include the Euclidean spaces $\mathbb{R}^n$, Cantor space $\{0,1\}^\mathbb{N}$, Baire space $\mathbb{N}^\mathbb{N}$, any second countable Hausdorff manifold $M$, and its group of self-homeomorphisms $\mathrm{Homeo}(M)$ when endowed with the compact-open topology. We highlight in this section, as bullet points, basic facts which will be used frequently throughout the rest of the paper. All of their proofs may be found in the standard reference \cite{MR1321597}.

\begin{itemize}
	\item The class of Polish spaces is closed under the taking of countable products, countable disjoint unions, and countable intersections of open subsets.
\end{itemize}

A subset $B$ of a Polish space $X$ is \emph{Borel} if it is contained in the $\sigma$-algebra generated by the open subsets of $X$. That is, $B$ can be obtained from basic open sets by taking countable unions, intersections, and complements, countably many times. A set $X$ together with a $\sigma$-algebra $\mathcal{B}$ on $X$ is a \emph{standard Borel space} if $\mathcal{B}$ is the collection of Borel sets corresponding to \emph{some} Polish topology on $X$. A map $f:X\to Y$ between Polish or standard Borel spaces $X$ and $Y$ is \emph{Borel measurable} if inverse images of Borel sets are Borel; we omit the word ``measurable'' where it is understood. Notice that we may view the closure of the class of Borel sets under countable unions and intersections as saying that any set defined by a formula consisting of finitely many quantifiers over fixed countable sets in front of a ``Borel predicate'' is also Borel.

\begin{itemize}
	\item The class of standard Borel spaces is closed under the taking of countable products (with Borel structure generated by products of Borel sets), countable disjoint unions, countable increasing unions, and Borel subsets.
	\item All uncountable standard Borel spaces are isomorphic (via an invertible Borel map) to $\mathbb{R}$ with its usual Borel structure.
	\item A map $f:X\to Y$ between standard Borel spaces is Borel if and only if its graph is Borel when viewed as a subset of the product $X\times Y$.
	\item (\textbf{Luzin--Suslin Theorem}) An injective image of a Borel set under a Borel function is Borel.
\end{itemize}

In contrast to the previous bullet, the image of a Borel set under an arbitrary Borel (or even continuous) function may fail to be Borel. Such sets are called \emph{analytic}, and their complements \emph{coanalytic}.

\begin{itemize}
	\item A subset $A$ of a standard Borel space $X$ is analytic if and only if there is a standard Borel space $Y$ and a Borel set $B\subseteq X\times Y$ such that $A$ is the projection $\mathrm{proj}_X[B]$ of $B$ onto $X$.
	\item A useful way of rephrasing the previous point is that analytic sets can be viewed as those defined by a formula with finitely many consecutive existential quantifiers over objects from standard Borel spaces, in front of a Borel predicate. E.g., if $B\subseteq X\times Y$ is Borel, then the set 
	\[
		\{x\in X:\exists y\in Y((x,y)\in B)\}
	\]
	is analytic. Formulas of this type, and the sets they define, are called $\mathbf{\Sigma}^1_1$.
	\item Coanalytic sets, then, correspond to finitely many consecutive universal quantifiers in front of a Borel predicate. These formulas, and the resulting sets, are called $\mathbf{\Pi}^1_1$.
	\item The collection of all analytic subsets of a standard Borel space is closed under both countable unions and countable intersections. The same holds for the collection of coanalytic subsets.
	\item (\textbf{Suslin's Theorem}) A subset $B$ of a standard Borel space is Borel if and only if it is both analytic and coanalytic. In particular, if a set has both a $\mathbf{\Sigma}^1_1$ and a $\mathbf{\Pi}^1_1$ description, then it is Borel.
\end{itemize}

Lastly, we will need the following ``uniformization theorem'':

\begin{itemize}
	\item (\textbf{Luzin--Novikov Uniformization}) If $X$ and $Y$ are standard Borel spaces and $B\subseteq X\times Y$ is a Borel set such that for every $x\in X$, the section $B_x=\{y\in Y:(x,y)\in B\}$ is countable, then the projection $\mathrm{proj}_X[B]$ is Borel and $B$ has a \emph{Borel uniformization}, i.e.,~a Borel function $f:\mathrm{proj}_X[B]\to Y$ such that for all $x\in\mathrm{proj}_X[B]$, $(x,f(x))\in B$. In particular, countable-to-one images of Borel sets under Borel functions are also Borel.
\end{itemize}

\subsection{Invariant descriptive set theory}\label{subsection:invariant_DST}

An equivalence relation $E$ on a standard Borel space is \emph{Borel} (or \emph{analytic}, respectively) if it is Borel (analytic) when viewed as the set of pairs $\{(x,y):xEy\}$ in $X^2$. Given equivalence relations $E$ and $F$ on standard Borel spaces $X$ and $Y$, respectively, a \emph{Borel homomorphism} from $E$ to $F$ is a Borel function $f:X\to Y$ such that for all $x,y\in X$,
\[
	xEy \quad\text{implies}\quad f(x)Ff(y).
\]
That is, $f$ is a definable way of assigning objects in $Y$ up to $F$-equivalence as invariants for objects in $X$ up to $E$-equivalence. Such an $f$ is a \emph{Borel reduction} from $E$ to $F$ if, moreover, for all $x,y\in X$,
\[
	xEy \quad\text{if and only if}\quad f(x)Ff(y).
\]
In other words, the objects in $Y$ modulo $F$ are \emph{complete} invariants for those in $X$ modulo $E$, computed via $f$. If such a reduction exists, we say that $E$ is \emph{Borel reducible} to $F$ and write $E\leq_B F$. If $E\leq_B F$ and $F\leq_B E$, we say that $E$ and $F$ are \emph{Borel bireducible} or \emph{Borel equivalent}, written $E\sim_B F$. Invariant descriptive set theory is the study of such definable equivalence relations and the reducibilities between them; for a general reference, we direct readers to \cite{MR2455198}. 

Borel reducibility provides a way of comparing the complexity of different classification problems, realized as analytic equivalence relations on suitable spaces, across mathematics. The simplest such equivalence relations, which we will call \emph{concretely classifiable},\footnote{Such equivalence relations are usually called ``smooth'' in the descriptive set theory literature; we will forego this terminology to avoid the obvious potential confusion. Within this class are, of course, benchmarks of an even greater simplicity; the complete picture is $=_1\,<_B\, =_2\,<_B\,\cdots\,<_B\,=_{\mathbb{N}}\,<_B\,=_{\mathbb{R}}$.} are those which are Borel reducible to the equality relation $=_{\mathbb{R}}$ on $\mathbb{R}$ (or, equivalently, any uncountable standard Borel space); in other words, they admit numerical complete invariants in a definable way. 

A paradigmatic example of a Borel equivalence relation which is not concretely classifiable is the relation $E_0$ of eventual equality between infinite binary sequences:
\[
	(x_n)_{n\in\mathbb{N}}\;E_0\; (y_n)_{n\in\mathbb{N}} \quad\text{if and only if}\quad \exists m\in\mathbb{N}\;\forall n\geq m(x_n=y_n),
\]
for $(x_n)_{n\in\mathbb{N}},(y_n)_{n\in\mathbb{N}}\in\{0,1\}^\mathbb{N}$.
In fact, $E_0$ is (up to bireducibility) the least such equivalence relation, by the following fundamental result \cite{MR1057041}.
\begin{itemize}
\item (\textbf{Glimm--Effros Dichotomy}) If $E$ is a Borel equivalence relation on a Polish space then either $E$ is concretely classifiable or $E_0\leq_B E$.
\end{itemize}

\noindent For a second example, consider the orbit equivalence relation $E_\mathbb{Z}$ induced by the \emph{shift action} of $\mathbb{Z}$ on the space $\{0,1\}^\mathbb{Z}$:
\[
	m\cdot(x_n)_{n\in\mathbb{Z}}=(x_{n-m})_{n\in\mathbb{Z}}
\]
for $m\in\mathbb{Z}$ and $(x_n)_{n\in\mathbb{Z}}\in\{0,1\}^\mathbb{Z}$. 
$E_0$ and $E_{\mathbb{Z}}$ are both \emph{hyperfinite}, meaning that they can be expressed as countable increasing unions of Borel equivalence relations with finite classes, and they are Borel bireducible.

In general, if $G$ is a Polish topological group acting in a Borel way on a standard Borel space $X$, meaning that the induced map $G\times X\to X$ is Borel, then the resulting orbit equivalence relation on $X$ is analytic. A particularly salient family of examples are the isomorphism relations on classes of countable first-order (e.g.,~relational or algebraic) structures. For instance, we can view any countably infinite group as being the set $\mathbb{N}$ endowed with a ternary relation $m\subseteq\mathbb{N}^3$ which codes the group operation: $m(p,q,r)$ if and only if $pq=r$. In this way, there is a Borel subset $\mathsf{Grp}$ of the space $\{0,1\}^{\mathbb{N}^3}$ of all ternary relations on $\mathbb{N}$, identified via characteristic functions, which parametrizes the class of all countably infinite groups. The infinite symmetric group $S_\infty$ of all permutations of $\mathbb{N}$ then acts on $\{0,1\}^{\mathbb{N}^3}$ in a natural way, called the \emph{logic action}:
\[
	(g\cdot m)(p,q,r)=m(g^{-1}(p),g^{-1}(q),g^{-1}(r))
\]
for $g\in S_\infty$, $m\in\{0,1\}^{\mathbb{N}^3}$, and $p,q,r\in\mathbb{N}$. The resulting orbit equivalence relation, when restricted to $\mathsf{Grp}$, is exactly the isomorphism relation $\cong_{\mathsf{Grp}}$ on all countably-infinite groups. Similar constructions parametrize all classes of countably-infinite first-order structures $\mathsf{C}$, with their isomorphism relations $\cong_\mathsf{C}$ induced by Borel actions of $S_\infty$.

An analytic equivalence relation $E$ is \emph{classifiable by countable structures} if there is some class $\mathsf{C}$ of countable first-order structures such that $E\,\leq_B\,\cong_\mathsf{C}$. This precisely describes those classification problems which admit countable algebraic complete invariants in a definable way. By the previous paragraph, any such equivalence relation is reducible to the orbit equivalence relation of a Borel action of $S_\infty$; the converse holds as well (see \cite[Thm.\ 3.6.1]{MR2455198}). Such an $E$ is, moreover, \emph{complete for countable structures}\footnote{Beginning with \cite{MR1011177}, such equivalence relations have typically been called ``Borel complete''. However, we feel that this terminology is misleading as these equivalence relations are \emph{never} themselves Borel (there is \emph{no} Borel equivalence relation of maximal complexity), and thus conflicts with the established usage of ``\underbar{\quad} complete'' in the context of other reducibility notions. We also forgo the terminology ``$S_\infty$-complete'', which while technically correct and carrying the benefit of brevity, obscures the relationship to countable structures.} if $\cong_\mathsf{C}\,\leq_B\,E$ for \emph{all} classes of countable first-order structures $\mathsf{C}$. The isomorphism relations on the classes of countable groups, graphs, rings, linear orders \cite{MR1011177}, and, as will be relevant for our Section \ref{section:surfaces}, Boolean algebras \cite{MR1804507}, are all complete for countable structures.

A Borel equivalence relation $E$ on $X$ is \emph{countable} if each equivalence class is countable; these are exactly the orbit equivalence relations of Borel actions of countable groups \cite{MR0578656}. A countable Borel equivalence relation is called \emph{universal (countable)} if every countable Borel equivalence relation reduces to it. Such equivalence relations often arise through the actions of nonabelian free groups. For example, the shift action of the free group on two generators, $F_2$, on $\{0,1\}^{F_2}$ produces a universal countable Borel equivalence relation \cite{MR1149121}, denoted (possibly up to bireducibility) by $E_\infty$.

A Borel equivalence relation is \emph{essentially countable} if it is Borel bireducible\footnote{Being Borel \emph{reducible} to a countable Borel equivalence relation is not, in general, equivalent to being essentially countable \cite{MR2155276}; however, the notions coincide for equivalence relations induced by Polish group actions \cite{MR3549382}.} with a countable Borel equivalence relation. Especially important for the present work is the fact that orbit equivalence relations induced by Borel actions of locally compact Polish groups are essentially countable \cite{MR1176624}. Every (essentially) countable Borel equivalence relation is classifiable by countable structures, but those which are complete for countable structures are never essentially countable. Such a Borel equivalence relation is \emph{essentially hyperfinite} (or \emph{essentially universal}) if it is bireducible with a hyperfinite (universal, respectively) countable Borel equivalence relation. Note that being essentially hyperfinite is equivalent to being Borel reducible to $E_0$ (see \cite{MR1149121}).

In closing out this section, we wish to highlight some of the reasons why Borel reducibility is so propitious a notion for studying classification problems in (separable) mathematics. Firstly, it is generous enough to capture most, if not all, examples of complete invariants which can be computed in a ``reasonably definable'' way. One may argue that it is \emph{too} generous a notion, but as much of the power of this theory comes from its ability to show that one equivalence relation does \emph{not} reduce to another, the weaker notion of definability (e.g.,\ via Borel, as opposed to continuous or computable, functions) translates to stronger negative results. Furthermore, if we are given two different ways to parametrize a class of mathematical objects as elements of standard Borel spaces, on which the natural notion of equivalence is analytic, then the translations between the parametrizations will, invariably, be given by Borel reductions.\footnote{This is a kind of ``Church--Turing Thesis'' for Borel reducibility; see \cite[p.\ 328]{MR2455198}.}

On the other hand, Borel reducibility is strict enough a notion to draw much finer distinctions than we could otherwise make if we allowed arbitrary reductions; the latter amounts to invoking the Axiom of Choice to compare the cardinalities of the resulting quotients. But while the cardinalities of quotients $X/E$, where $E$ is an analytic equivalence relation on a standard Borel space $X$, are limited to being countable, $\aleph_1$, or $2^{\aleph_0}$ (see \cite{MR0476524}), there are continuum many different degrees up to Borel bireducibility and they are far from being even linearly ordered \cite{MR1775739}. Lastly, the existence or non-existence of a Borel reduction is highly robust; it is not particularly sensitive to the ambient set-theoretic universe and thus immune to the sorts of independence phenomena that are so otherwise common in set theory. In the terminology of the field, it is \emph{absolute}.\footnote{More concretely, whether there is a Borel reduction between two fixed Borel equivalence relations is a $\mathbf{\Sigma}^1_2$ statement and thus absolute between transitive models of set theory by Shoenfield's Absoluteness Theorem. For analytic equivalence relations, the existence of a Borel reduction is $\mathbf{\Sigma}^1_3$, and hence is, under sufficient large cardinal assumptions, absolute between generic extensions. See \cite{MR1940513} for details on absoluteness.}

\subsection{Spaces of subsets}\label{subsection:spaces_of_subsets}

Throughout this section, we consider a fixed locally compact Polish space $X$. We denote by $\mathcal{K}(X)$, $\mathcal{F}(X)$, and $\mathcal{O}(X)$ the collections of all compact, closed, and open subsets of $X$, respectively. We view $\mathcal{K}(X)$ as a Polish space with the familiar Vietoris topology, generated by a subbasis consisting of open sets of the form $\{K\in\mathcal{K}(X):K\subseteq U\}$ and $\{K\in\mathcal{K}(X):K\cap U\neq\varnothing\}$, where $U$ is open in $X$. When $X$ is endowed with a compatible metric, the Vietoris topology is that induced by the Hausdorff metric on $\mathcal{K}(X)$. We wish to describe suitable topological and Borel structures on both $\mathcal{F}(X)$ and $\mathcal{O}(X)$.
	
The \emph{Fell topology} on $\mathcal{F}(X)$ is generated by open sets of the form $\{F\in\mathcal{F}(X):F\cap K=\varnothing\}$ and $\{F\in\mathcal{F}(X):F\cap U\neq\varnothing\}$, where $K$ and $U$ range over the compact and open subsets of $X$, respectively. When endowed with this topology, $\mathcal{F}(X)$ is a compact Polish space \cite[Exer.\ 12.7]{MR1321597}. We endow $\mathcal{O}(X)$ with the topology inherited from $\mathcal{F}(X)$ by declaring complementation to be a homeomorphism $\mathcal{F}(X)\to\mathcal{O}(X)$. The Fell and Vietoris topologies coincide on $\mathcal{F}(X)$ when $X$ is compact, and the former may be viewed as induced by the latter via the one-point compactification of $X$; see \cite{MR3116379}.
	
The above topology was originally isolated by Chabauty \cite{MR0038983} for the space of closed subgroups of a locally compact group. Later, and apparently independently, Fell \cite{MR0139135} defined his topology on the closed subsets of an arbitrary locally compact space. For this reason, in the geometric literature, this topology is often referred to as the \emph{Chabauty} or \emph{Chabauty--Fell topology}; it is also (following Thurston \cite{MR1435975}) sometimes termed the \emph{geometric topology}, in part due to the following lemma; see p.~161 of \cite{MR1219310} for a proof and the discussions in Sections \ref{section:bireducibility}, \ref{section:isometry_for_2}, and \ref{section:isometry_for_3} below for its manifestations in spaces of quotient manifolds. 
\begin{lemma}\label{lem:convergence_in_Fell}
A sequence $(F_n)_{n\in\mathbb{N}}$ converges to a closed set $F$ in the Fell topology on $\mathcal{F}(X)$ if and only if:
\begin{enumerate}[label=\textup{\roman*.}]
\item every $x\in F$ is the limit in $X$ of some sequence $(x_i)_{i\in\mathbb{N}}$, where $x_i\in F_i$ for all $i\in\mathbb{N}$, and
\item every limit in $X$ of a convergent sequence $(x_i)_{i\in\mathbb{N}}$, where $(F_{n_i})_{i\in\mathbb{N}}$ is an infinite subsequence of $(F_n)_{n\in\mathbb{N}}$ and  $x_i\in F_{n_i}$ for all $i\in\mathbb{N}$, is in $F$.
\end{enumerate}
\end{lemma}

We will adopt the convention of using ``Fell topology'' when referring to the spaces $\mathcal{F}(X)$ and $\mathcal{O}(X)$ for a general locally compact Polish space $X$, and ``Chabauty--Fell topology'' for the space of closed subgroups of a Lie group.

Being Polish, the Fell topology induces a standard Borel structure on both $\mathcal{F}(X)$ and $\mathcal{O}(X)$. This structure coincides with the \emph{Effros Borel structure} \cite{MR0181983} on $\mathcal{F}(X)$; note that the latter is in fact generated by the collection of sets $\{F\in\mathcal{F}(X):F\cap K=\varnothing\}$, for $K$ compact in $X$, alone. It will be important to know that various standard relations and operations on $\mathcal{F}(X)$ and $\mathcal{O}(X)$ are Borel measurable, so we list the most fundamental of these here (see \cite[Exer.\ 12.11 and 27.7]{MR1321597}):

\begin{itemize}
	\item The subset relation $\{(F_1,F_2):F_1\subseteq F_2)\}$ is Borel in $\mathcal{F}(X)^2$.
	\item The union operation $(F_1,F_2)\mapsto F_1\cup F_2$ is a Borel map $\mathcal{F}(X)^2\to\mathcal{F}(X)$.
	\item The intersection operation $(F_1,F_2)\mapsto F_1\cap F_2$ is a Borel map $\mathcal{F}(X)^2\to\mathcal{F}(X)$. Since taking limits is always a Borel operation on Polish spaces and $$\lim_n\left(\bigcap_{i=0}^nF_i\right)=\bigcap_{n\in\mathbb{N}}F_n$$ in $\mathcal{F}(X)$, it follows that $(F_n)_{n\in\mathbb{N}}\mapsto \bigcap_{n\in\mathbb{N}}F_n$ is a Borel map $\mathcal{F}(X)^\mathbb{N}\to\mathcal{F}(X)$.
	\item If $f:X\to Y$ is continuous and $Y$ is locally compact Polish, then $E\mapsto f^{-1}[E]$ is a Borel map $\mathcal{F}(Y)\to\mathcal{F}(X)$, and $F\mapsto \overline{f(F)}$ is a Borel map $\mathcal{F}(X)\to\mathcal{F}(Y)$.
\end{itemize}

By taking complements, each of the above induces a corresponding Borel map or relation on $\mathcal{O}(X)$. Notably, while the Effros Borel structure can be defined on any Polish space, it is the local compactness of $X$ which ensures that intersection is a Borel map on $\mathcal{F}(X)$, and thus also that union is a Borel map on $\mathcal{O}(X)$. We also have the following fact (see \cite[p.\ 74]{MR0348724}), for which local compactness is again crucial:

\begin{itemize}
	\item The operation $F\mapsto\overline{X\setminus F}$ is a Borel map $\mathcal{F}(X)\to\mathcal{F}(X)$. Consequently, so are the maps $U\mapsto\overline{U}$ from $\mathcal{O}(X)\to\mathcal{F}(X)$ and $U\mapsto X\setminus\overline{U}$ from $\mathcal{O}(X)\to\mathcal{O}(X)$.
\end{itemize}

\noindent For verifications that other fundamental relations between elements of $\mathcal{O}(X)$ are Borel, see Section \ref{subsection:subclasses} below.

\section{Parameter spaces of manifolds}
\label{section:parametrization}

\subsection{Pseudogroups $\mathcal{G}$, model spaces $X$, and spaces of $(\mathcal{G},X)$-manifolds}
\label{subsection:pseudogroups}

In this section, we describe a unified framework for parametrizing many familiar classes of manifolds. This framework applies the \emph{pseudogroup} approach, perhaps most prominently associated to Thurston and his study of geometric structures on manifolds, as expounded in \cite{MR1435975}.\footnote{The pseudogroup concept itself, though, traces back at least to \cite{zbMATH03003091}; see also \cite{MR3965773} for a history.} We begin with a brief overview of this approach.

\begin{definition}\label{defn:pseudogroup}
	A \emph{pseudogroup} $\mathcal{G}$ on a topological space $X$
	is a set of homeomorphisms between open subsets of $X$ satisfying the following conditions:
	\begin{enumerate}[label=\textup{\roman*.}]
		\item the domains of the elements of $\mathcal{G}$ cover $X$,
		\item the restriction of any element of $\mathcal{G}$ to an open subset of $X$ is also in $\mathcal{G}$,
		\item the composition of any two elements of $\mathcal{G}$ is in $\mathcal{G}$, 
		\item the inverse of any element of $\mathcal{G}$ is in $\mathcal{G}$, and
		\item if $\phi:U\to V$ is a homeomorphism between open subsets $U$ and $V$ of $X$, and $U$ is covered by a collection of open sets $U_\alpha$ with the property that the restriction $\phi|_{U_\alpha}$ is in $\mathcal{G}$ for every $\alpha$, then $\phi$ is in $\mathcal{G}$.
	\end{enumerate}
\end{definition}

We note that in condition (iii), we permit compositions of \emph{arbitrary} pairs of elements of $\mathcal{G}$, that is, if $\phi:U_1\to V_1$ and $\psi:U_2\to V_2$ are in $\mathcal{G}$, then $\psi\circ \phi$ is the resulting homeomorphism from $U_1\cap \phi^{-1}[U_2]$ onto $V_2\cap \psi[V_1]$; included here is the possibility that these sets may each be empty. Condition (v), which says that $\mathcal{G}$ is ``local'', can also be viewed as saying that unions of elements of $\mathcal{G}$, when they produce a homeomorphism, lie in $\mathcal{G}$.

The pseudogroups on any topological space $X$ are naturally ordered by inclusion.
This ordering possesses both a maximum and a minimum element --- namely, the collection of all homeomorphisms between pairs of open subsets of $X$ and the collection of all identity maps on open subsets of $X$, respectively.
Any collection of homeomorphisms between open subsets of $X$ falls in some unique minimal pseudogroup, which it may thus be said to generate \cite[Exer.\ 3.1.8.a]{MR1435975} (see also Example \ref{ex:isometry_pseudogroups} below).
In particular, any group $G$ of homeomorphisms of $X$ may be said to generate a pseudogroup, which we denote by a sans-serif $\mathsf{G}$; these form a prominent class of examples.

\begin{definition}
	Let $X$ denote $\mathbb{R}^n$, $\mathbb{C}^n$, hyperbolic space $\mathbb{H}^n=\{(x_1,\ldots,x_n)\in\mathbb{R}^n\mid x_n>0\}$, or the $n$-dimensional unit sphere $\mathbb{S}^n$, the latter two endowed with their standard hyperbolic and spherical metrics, respectively.
	\begin{enumerate}[label=\textup{\arabic*.}]
		\item $\mathsf{Top}$ is the \emph{topological pseudogroup} on $X$, the collection of all homeomorphisms between open subsets of $X$.
		This is the maximal pseudogroup on $X$ referred to above.
		\item $\mathsf{C}^k$ ($1\leq k\leq \infty$) is the pseudogroup of all $C^k$-diffeomorphisms between pairs of open subsets of $X$.
		\item In the case of $X=\mathbb{C}^n$, $\mathsf{Hol}$ is the pseudogroup of all biholomorphic maps between open subsets of $\mathbb{C}^n$.
		\item $\mathsf{Isom}$ is the pseudogroup generated by the group of isometries of $X$, and 	$\mathsf{Isom}^+$ is that generated by the orientation-preserving isometries of $X$.
	\end{enumerate}
	As in the above examples, we will often omit mention of the \emph{model space} $X$ when it is clear from context. 
\end{definition}

The standard definition of a topological $n$-manifold may now be phrased as: \emph{a second countable Hausdorff space possessing an atlas of charts whose transition maps all lie in the topological pseudogroup on $\mathbb{R}^n$}. To define other sorts of manifolds, we simply vary the pseudogroup appearing therein.

\begin{definition}\label{defn:(G,X)-manifold}
	Let $M$ and $X$ be topological spaces and $\mathcal{G}$ a pseudogroup on $X$.
	\begin{enumerate}[label=\textup{\arabic*.}]
		\item A \emph{$(\mathcal{G},X)$-atlas} on $M$ is a collection of pairs $(U,\phi)$, called \emph{charts}, where $U$ is an open subset of $M$ and $\phi:U\to X$ is a homeomorphism onto its image, such that the sets $U$ cover $M$, and whenever $(U,\phi)$ and $(V,\psi)$ are charts with $U\cap V\neq\varnothing$, the \emph{transition map}
		\[
			\psi\circ\phi^{-1}|_{\phi[U\cap V]}:\phi[U\cap V]\to\psi[U\cap V]
		\]
		is in $\mathcal{G}$.
		\item $M$ is a \emph{locally $(\mathcal{G},X)$-space} if it is endowed with a $(\mathcal{G},X)$-atlas; if $M$ is, in addition, Hausdorff and second countable, it is a \emph{$(\mathcal{G},X)$-manifold}.
	\end{enumerate}
\end{definition}

Our language is deliberately neutral with respect to the question of whether a choice of atlas forms part of a manifold's intrinsic structure. Philosophically, it does not \cite[p.\ 111]{MR1435975}. In the generality in which we wish to work, however, the most natural parametrizations of various classes of manifolds will be by way of just such choices of atlases; hence in the course of these parametrizations, we will \emph{to some degree} regard such a choice as part of the data of a manifold. The essential point for our purposes is that the relation of compatibility among such choices will always be subsumed by the relation of manifold equivalence (Definition \ref{defn:(G,X)-isomorphism}) under consideration. Readers in the meantime should feel free to adopt whichever of these perspectives suits them best.

\begin{example}
\label{ex:first_manifolds_example}
The \emph{topological}, \emph{smooth}, and \emph{complex} $n$-manifolds are precisely the $(\mathsf{Top},\mathbb{R}^n)$, $(\mathsf{C}^\infty,\mathbb{R}^n)$, and $(\mathsf{Hol},\mathbb{C}^n)$-manifolds, respectively. The \emph{Euclidean}, \emph{spherical}, and \emph{hyperbolic} $n$-manifolds are precisely the $(\mathsf{Isom},\mathbb{R}^n)$, $(\mathsf{Isom},\mathbb{S}^n)$, and $(\mathsf{Isom},\mathbb{H}^n)$-manifolds, respectively.\footnote{The subtleties associated to their induced global metrics are addressed in Section \ref{subsection:subclasses}.} \emph{Manifolds with boundary} fit into this framework by letting $\mathbb{R}^n_{+}=\{(x_1,\dots,x_{n})\in\mathbb{R}^n\mid x_n\geq 0\}$; the topological and smooth $n$-manifolds with boundary are then precisely the $(\mathsf{Top},\mathbb{R}^n_+)$-manifolds and $(\mathsf{C}^\infty,\mathbb{R}^n_+)$-manifolds, respectively.
\end{example}

As seen in the preceding examples (see also \cite[pp.\ 111--118]{MR1435975}), the most prominent classes of finite-dimensional $(\mathcal{G},X)$-manifolds are all modelled on spaces $X$ which are both Polish and locally compact. Since the latter condition carries the additional benefit of entailing a more pleasant Borel structure on its space of open subsets $\mathcal{O}(X)$ (see Section \ref{subsection:spaces_of_subsets}), in what follows we will always take $X$ to be a locally compact Polish space. We will comment, briefly, on infinite-dimensional manifolds in Section \ref{section:questions}.

We turn now to the parametrization, for some fixed pseudogroup $\mathcal{G}$ on $X$, of the class of $(\mathcal{G},X)$-manifolds.
One might begin by identifying the $M$ of Definition \ref{defn:(G,X)-manifold} with a triple $(m,\mathcal{U},c)$, where
\begin{itemize}
\item $m$ is the underlying set of $M$,
\item $\mathcal{U}=\{U_i\mid i\in I\}$ is a family of open subsets of $X$, and
\item $c$ is the collection of chart-maps patching the structures in $\mathcal{U}$ together ``along $m$''; in other words, for each $U_i$ in $\mathcal{U}$, there exists a $V_i\subseteq m$ and bijection $\varphi_i: V_i\to U_i$ in $c$.
\end{itemize}
For notational convenience, we will write $U_{i,j}$ for $\varphi_i[V_i\cap V_j]$ below and $U_i$ for $U_{i,i}$. Provided it is second countable and Hausdorff, $M$ is then a $(\mathcal{G},X)$-manifold if and only if \begin{itemize} 
\item $\bigcup_{i\in I} V_i=m$, and
\item the functions in $c$ satisfy:
\begin{align*} \varphi_{i,j}:=\varphi_j\circ\varphi_i^{-1}: U_{i,j}\to U_{j,i}\text{ is an element of }\mathcal{G}\text{ for all }i,j\in I.
\end{align*}
\end{itemize}
We will arrive to the \emph{space of $(\mathcal{G},X)$-manifolds} simply by discarding the redundancies in this description. Since, for example, $(\mathcal{G},X)$-manifolds $M$ are by definition second countable, every such $M$ is represented by a triple $(m,\mathcal{U},c)$ where the indexing set $I$ is $\mathbb{N}$, and we may henceforth restrict our attention to triples of this sort. Observe furthermore that for the study of the \emph{manifold structure} of $M$, the set $m$ is essentially irrelevant. More precisely, $M$ is, up to the relevant notion of isomorphism (see Definition \ref{defn:(G,X)-isomorphism} below), readily recovered from the sets $U_i$ and the functions $\varphi_{i,j}$ alone, for the charts in $c$ witness that 
\begin{align}\label{eq:quotient_of_sum}
M\cong\coprod_{i\in \mathbb{N}} U_i\,/ \sim,
\end{align}
where $x\sim y$ for any $x\in U_i$ and $y\in U_j$ if and only if $\varphi_{i,j}(x)=y$. Note that $\sim$ is a closed equivalence relation on the disjoint union $\coprod_{i\in\mathbb{N}} U_i$ and the quotient map $\pi_{(\mathcal{U},c)}:\coprod_{i\in\mathbb{N}}U_i\to\coprod_{i\in\mathbb{N}}U_i/\sim$ is a local homeomorphism.
(This is in fact the perspective taken by Mac Lane and Moerdijk in \cite[pp.\ 74--75]{MR1300636}, where a manifold is defined as the coequalizer of the diagram
$$\coprod_{(i,j)\in\mathbb{N}^2}U_{i,j}\rightrightarrows\coprod_{i\in\mathbb{N}} U_i,$$
in which the higher arrow is the coproduct of the inclusions $\iota_{i,j}:U_{i,j}\to U_i$ and the lower arrow is the coproduct of the functions $\iota_{j,i}\circ\varphi_{i,j}$.)

Our procedure will therefore be as follows: we will parametrize the collection of second countable locally $(\mathcal{G},X)$-spaces via pairs encoding countable collections of open sets and transition maps. We will then show that if $\mathcal{G}$ is, in a suitable sense Borel (Definition \ref{defn:Borel_pseudogroup} below), then this parametrization $\mathfrak{P}(\mathcal{G},X)$ carries a natural standard Borel structure. The restriction of $\mathfrak{P}(\mathcal{G},X)$ to Hausdorff pairs will define the \emph{standard Borel space $\mathfrak{M}(\mathcal{G},X)$ of $(\mathcal{G},X)$-manifolds}; this will form the basic setting for our subsequent work.
We will conclude this subsection by verifying that the most prominent pseudogroups $\mathcal{G}$ are indeed Borel.

\begin{definition}\label{defn:parametrization_locally_(G,X)_spaces}
Let $\mathcal{G}$ be a pseudogroup on a locally compact Polish space $X$.
The \emph{parameter space $\mathfrak{P}(\mathcal{G},X)$ of second countable locally $(\mathcal{G},X)$-spaces} consists of all pairs $(\mathcal{U},c)$ where \begin{itemize}
	\item $\mathcal{U}=\langle U_{i,j}\mid (i,j)\in \mathbb{N}^2\rangle$, a family of open subsets of $X$, and
	\item $c=\langle \varphi_{i,j}:U_{i,j}\to U_{j,i}\mid (i,j)\in \mathbb{N}^2\rangle\subseteq\mathcal{G}$
\end{itemize}
satisfy
\begin{enumerate}[label=\textup{\roman*.}]
\item $U_{i,j}\subseteq U_i:=U_{i,i}$ for all $(i,j)\in \mathbb{N}^2$,
\end{enumerate}
together with conditions ensuring that the associated $\sim$ is an equivalence relation:
\begin{enumerate}[label=\textup{\roman*.}]
\setcounter{enumi}{1}
\item $\varphi_{i,i}=\mathrm{id}_{U_i}:U_i\to U_i$ for all $i\in \mathbb{N}$,
\item $\varphi_{i,j}=\varphi_{j,i}^{-1}$ for all $(i,j)\in\mathbb{N}^2$, and
\item $\varphi^{-1}_{i,j}[U_{j,i}\cap U_{j,k}]\subseteq U_{i,k}$ and $\varphi_{j,k}\circ\varphi_{i,j}\big|_{\varphi^{-1}_{i,j}[U_{j,i}\,\cap \,U_{j,k}]}=\varphi_{i,k}\big|_{\varphi^{-1}_{i,j}[U_{j,i}\,\cap\, U_{j,k}]}$ for all $i,j,k\in \mathbb{N}$.
\end{enumerate}
Within this framework, it will be useful to term any $U_i$ a \emph{chart}, any $U_{i,j}$ an \emph{overlap}, and any $\varphi_{i,j}$ a \emph{transition map}. 
\end{definition}

Any pair $(\mathcal{U},c)$ as above will determine a locally $(\mathcal{G},X)$-space $M_{(\mathcal{U},c)}$ in the manner of equation (\ref{eq:quotient_of_sum}), and, up to the appropriate notion of isomorphism, every such space admits a representation of this form.

We turn now to the Borel structure on $\mathfrak{P}(\mathcal{G},X)$. As in Section \ref{subsection:spaces_of_subsets}, we endow the set $\mathcal{O}(X)$ of open subsets of $X$ with the Effros Borel structure, or equivalently, the Borel structure induced by the Fell topology. We define a pseudogroup on $X$ to be Borel when its Borel structure is compatible with that of $\mathcal{O}(X)$, as follows: 

\begin{definition}\label{defn:Borel_pseudogroup}
	A pseudogroup $\mathcal{G}$ on $X$ is \emph{Borel} if it is endowed with a standard Borel structure such that each of the following maps are Borel on the relevant spaces:
	\begin{enumerate}[label=\textup{\roman*.}]
		\item domain projection $\mathcal{G}\to\mathcal{O}(X):\phi\mapsto\mathrm{dom}(\phi)$,
		\item range projection $\mathcal{G}\to\mathcal{O}(X):\phi\mapsto\mathrm{ran}(\phi)$,
		\item domain inclusion $\mathcal{O}(X)\to\mathcal{G}:U\mapsto\mathrm{id}_U$, the identity map on $U$,
		\item direct image $\mathcal{O}(X)\times\mathcal{G}\to\mathcal{O}(X):(U,\phi)\mapsto \phi[U]$,
		\item inverse image $\mathcal{O}(X)\times\mathcal{G}\to\mathcal{O}(X):(U,\phi)\mapsto \phi^{-1}[U]$,
	 	\item restriction $\mathcal{O}(X)\times\mathcal{G}\to\mathcal{G}:(U,\phi)\mapsto \phi|_U$,
		\item composition $\mathcal{G}\times\mathcal{G}\to\mathcal{G}:(\phi,\psi)\mapsto \psi\circ \phi$, and
		\item inversion $\mathcal{G}\to\mathcal{G}:\phi\mapsto \phi^{-1}$.
	\end{enumerate}
\end{definition}

There are some obvious redundancies in the above definition, but we provide the full list for ease of reference. We note that a pseudogroup $\mathcal{G}$ on $X$ forms an example of a \emph{groupoid}, a (small) category in which all arrows are invertible: the set of objects is $\mathcal{O}(X)$ and the set of arrows is $\mathcal{G}$. By properties (iii), (vii), and (viii) above, every Borel pseudogroup is a \emph{Borel groupoid}, in the terminology of \cite{MR3660238}. The following lemma is now immediate from Definitions \ref{defn:parametrization_locally_(G,X)_spaces} and \ref{defn:Borel_pseudogroup}, together with the fact that the subset relation and intersection operation on $\mathcal{O}(X)$ are Borel.

\begin{lemma} 
If $\mathcal{G}$ is a Borel pseudogroup on a locally compact Polish space $X$, then $\mathfrak{P}(\mathcal{G},X)$ inherits from $\mathcal{O}(X)^{\mathbb{N}\times\mathbb{N}}\times\mathcal{G}^{\mathbb{N}\times\mathbb{N}}$ a standard Borel structure.\qed
\end{lemma}

\begin{definition}
	Let $\mathcal{G}$ be a pseudogroup on a locally compact Polish space $X$. The \emph{parameter space $\mathfrak{M}(\mathcal{G},X)$ of $(\mathcal{G},X)$-manifolds} is the set of all pairs $(\mathcal{U},c)$ in $\mathfrak{P}(\mathcal{G},X)$ such that $M_{(\mathcal{U},c)}$ is a $(\mathcal{G},X)$-manifold, or, equivalently, is Hausdorff.
\end{definition}

In order to show that $\mathfrak{M}(\mathcal{G},X)$ has the structure of a standard Borel space, it suffices to show that it is a Borel subset of $\mathfrak{P}(\mathcal{G},X)$, whenever $\mathcal{G}$ is a Borel pseudogroup. The essential (and somewhat surprising) difficulty, then, is to show that whether a pair $(\mathcal{U},c)$ describes a Hausdorff manifold can be expressed in a Borel way in the parameter space $\mathfrak{P}(\mathcal{G},X)$.

\begin{theorem}\label{thm:M(G,X)_Borel}
If $\mathcal{G}$ is a Borel pseudogroup on a locally compact Polish space $X$, then $\mathfrak{M}(\mathcal{G},X)$ is a Borel subset of $\mathfrak{P}(\mathcal{G},X)$. In particular, $\mathfrak{M}(\mathcal{G},X)$ forms a standard Borel space.
\end{theorem}

\begin{proof} 

Our proof proceeds in two parts: first we show that $\mathfrak{M}(\mathcal{G},X)$ is a coanalytic subset of $\mathfrak{P}(\mathcal{G},X)$, then we show that it is an analytic subset of $\mathfrak{P}(\mathcal{G},X)$. It thus follows by Suslin's Theorem that $\mathfrak{M}(\mathcal{G},X)$ is Borel. These two characterizations of $\mathfrak{M}(\mathcal{G},X)$ may be loosely regarded as targeting the equivalent conditions, for a second countable topological manifold, of Hausdorffness and normality, respectively.

More precisely, we will first show that $$\neg\mathrm{Haus}:=\{(\mathcal{U},c)\in\mathfrak{P}(\mathcal{G},X)\mid M_{(\mathcal{U},c)}\textnormal{ is not Hausdorff}\}$$ is an analytic subset of $\mathfrak{P}(\mathcal{G},X)$. For convenience, fix a countable basis $\mathcal{B}$ for the topology on $X$. $\neg\mathrm{Haus}$ is the collection of those $(\mathcal{U},c)\in \mathfrak{P}(\mathcal{G},X)$ for which there exists an $(x,y)\in X^2$ such that the following hold:
\begin{enumerate}[label=\textup{\roman*.}]
\item $x\neq y$, and
\item $\exists\,i,j\in\mathbb{N}$ such that
\begin{enumerate}[label=\textup{\alph*.}]
\item $x\in U_{i}$ and $y\in U_{j}$, and
\item $\forall\,V,W\in\mathcal{B}\;[(x\in V\wedge y\in W)\rightarrow\,\varphi_{i,j}[V]\cap W\neq\varnothing]$.
\end{enumerate}
\end{enumerate}
Item (i) is an open condition in $X^2$ and the conditions in (ii.a) are each open in $X\times\mathcal{O}(X)^{\mathbb{N}\times\mathbb{N}}$, where $\mathcal{O}(X)$ has the Fell topology. All enumerated quantifications are countable, and the implication in (ii.b) may be rewritten as 
\[
	\neg(x\in V\wedge y\in W)\vee \varphi_{i,j}[V]\cap W\neq\varnothing,
\]
the left half of which is closed in $X^2$.
This leaves only the right half: \emph{for any two open sets $V,W\subseteq X$, we would like $\{\varphi\in\mathcal{G}\mid \varphi[V]\cap W\neq\varnothing\}$ to be a Borel subset of $\mathcal{G}$}. This is indeed the case since the direct image is a Borel operation $\mathcal{O}(X)\times\mathcal{G}\to\mathcal{O}(X)$, intersection is a Borel operation in $\mathcal{O}(X)$, and $\phi[V]\cap W\neq\varnothing$ can be verified by checking whether $q_i\in\phi[V]\cap W$, where $q_i$ $(i\in\mathbb{N})$ is a fixed enumeration of a countable dense subset of $X$. Thus, the collection of $((x,y),(\mathcal{U},c))\in X^2\times \mathfrak{P}(\mathcal{G},X)$ satisfying items (i) and (ii) above is Borel. The projection of this collection onto $\mathfrak{P}(\mathcal{G},X)$ is exactly $\neg\mathrm{Haus}$, concluding the argument.

Next, we show $\mathfrak{M}(\mathcal{G},X)$ is an analytic subset of $\mathfrak{P}(\mathcal{G},X)$. Fix $i,j\in\mathbb{N}$ and let
\begin{align*} H_{i,j}=\{((V,W),(\mathcal{U},c))\in\mathcal{O}(X)^2\times \mathfrak{P}(\mathcal{G},X)\mid V\supseteq U_{i}\backslash U_{i,j} & \textnormal{ and } W\supseteq U_{j}\backslash U_{j,i}\\ &\textnormal{ and }\varphi_{i,j}[V]\cap W=\varnothing\}.\end{align*}
Using the local compactness of $X$ to write $U_{i}\backslash U_{i,j}$ as a union of compact sets $K_\ell$ (for $\ell\in\mathbb{N}$), we have
\[
\{V\in\mathcal{O}(X)\mid V\supseteq U_{i}\backslash U_{i,j}\}=\bigcap_{\ell\in\mathbb{N}}\{V\in\mathcal{O}(X)\mid V\supseteq K_\ell\},
\]
so the first two defining conditions of $H_{i,j}$ are Borel in $\mathcal{O}(X)$. The third condition is Borel for much the same reason as in the previous paragraph.

We claim that 
\begin{align}
\label{eq:T2}
\mathfrak{M}(\mathcal{G},X)=\bigcap_{(i,j)\in\mathbb{N}^2}\mathrm{proj}_{\mathfrak{P}(\mathcal{G},X)}(H_{i,j}).
\end{align}
As each $H_{i,j}$ is Borel, the right-hand side is a countable intersection of analytic sets, and therefore analytic, as desired.
What remains is only the verification that this formulation does indeed capture the Hausdorffness of the associated manifolds $M_{(\mathcal{U},c)}$.
To this end, temporarily denote by $M_{i,j}$ the restriction to the sum $U_{i}\sqcup U_{j}$ of the quotient defining $M_{(\mathcal{U},c)}$ (as in equation (\ref{eq:quotient_of_sum})), and observe that $M_{(\mathcal{U},c)}$ is Hausdorff if and only if each $M_{i,j}$ is.
Fixing $i$ and $j$, note that the restricted quotient map $\pi:U_{i}\sqcup U_{j}\to M_{i,j}$ is open and injective when restricted to $U_{i}$ or $U_{j}$, hence if the $\pi$-preimages of distinct $x$ and $y$ in $M_{i,j}$ both intersect one of $U_{i}$ or $U_{j}$ then the $\pi$-images of disjoint open neighborhoods therein will separate $x$ and $y$ in $M_{i,j}$. If, on the other hand, they do not, then the $\pi$-images of any $V$ and $W$ witnessing that $(\mathcal{U},c)\in\mathrm{proj}_{\mathfrak{P}(\mathcal{G},X)}H_{i,j}$ will separate $x$ and $y$ in $M_{i,j}$; this shows that $M_{i,j}$ is Hausdorff, as desired.
Strictly speaking, what we've just shown is the $\supseteq$ half of the equality asserted in equation (\ref{eq:T2}).
For the $\subseteq$ half, fix a Hausdorff manifold $M_{(\mathcal{U},c)}$ and $i,j\in\mathbb{N}$ and (continuing with the notation just above) exhaustions of the $\pi$-images of $U_i\backslash U_{i,j}$ and $U_j\backslash U_{j,i}$, respectively, by increasing families $(K^i_\ell)_{\ell\in\mathbb{N}}$ and $(K^j_\ell)_{\ell\in\mathbb{N}}$ of compact sets with $K^i_\ell\cap K^j_\ell=\varnothing$ for all $\ell\in\mathbb{N}$.
By the normality of $M_{(\mathcal{U},c)}$, for each $\ell$ there exist disjoint open $V^i_\ell$ and $V^j_\ell$ in $M_{(\mathcal{U},c)}$ such that $K^i_\ell\subseteq V^i_\ell
\subseteq\pi[U_i]$ and $K^j_\ell\subseteq V^j_\ell\subseteq \pi[U_j]$, with
$$V^i_\ell\cap(V^j_\ell\cup \pi[U_j\backslash U_{j,i}])=(V^i_\ell\cup \pi[U_i\backslash U_{i,j}])\cap V^j_\ell=\varnothing,$$
whence
$$V:=\bigcup_{\ell\in\mathbb{N}}\pi^{-1}[V^i_\ell]\text{ and }W:=\bigcup_{\ell\in\mathbb{N}}\pi^{-1}[V^j_\ell]$$
witness that $M_{(\mathcal{U},c)}\in\mathrm{proj}_{\mathfrak{P}(\mathcal{G},X)}(H_{i,j})$, and as $i$ and $j$ were arbitrary, this completes the proof.
\end{proof}

\begin{remark} The pattern of the above argument will recur throughout this work, in particular in the proofs of Theorem \ref{thm:compact_Borel} and Lemma \ref{lem:complete} that the classes of compact manifolds and complete manifolds, respectively, are each Borel.
Note that, as for Hausdorff-ness, the naive framing of each --- \emph{for all ``reals'' ($x,y\in X$, basic open covers $(U_i)_{i\in I}$, Cauchy sequences $(x_n)_{n\in\mathbb{N}}$, etc.),~there exists a Borel phenomenon (disjoint open sets, a finite subcover, a limit, etc.)}~--- is $\mathbf{\Pi}^1_1$ in form.
Fortunately for our purposes, and perhaps for reasons not unconnected to their centrality, each of these notions, just as above, admits a $\mathbf{\Sigma}^1_1$ formulation as well.
It is this dynamic which accounts for the prominence of Suslin's Theorem in our arguments.
\end{remark}

It remains to show that the relevant pseudogroups are, in fact, Borel. Since $\mathsf{Top}$ is the largest pseudogroup, we begin there; other pseudogroups will be endowed with the Borel structure they inherit as Borel subsets of $\mathsf{Top}$. Our description of this Borel structure and the proof that it is Borel are based on \cite[Lem.\ 7.2]{MR1731384}.

We fix an injective enumeration $q_i$ (for $i\in\mathbb{N}$) of a countable dense subset of $X$ and $u$ (for ``undefined'') a default value outside of $X$. Consider $X\cup\{u\}$ with the standard Borel structure induced by adding $u$ to $X$ as an isolated point. The standard Borel structure which we place on $\mathsf{Top}$ is that resulting from identifying each element $\phi\in\mathsf{Top}$ with the triple consisting of its domain, its range, and its sequence of values on the $q_i$'s, using $u$ to indicate where it is not defined. Explicitly:

\begin{definition}\label{defn:Borel_structure_Top}
	Let $\mathcal{T}\subseteq\mathcal{O}(X)\times\mathcal{O}(X)\times(X\cup\{u\})^\mathbb{N}$ be the set of all triples $(U,V,(r_i)_ {i\in\mathbb{N}})$ in $\mathcal{O}(X)\times\mathcal{O}(X)\times (X\cup\{u\})^\mathbb{N}$ such that:
	\begin{enumerate}[label=\textup{\roman*.}]
		\item if $q_i\in U$, then $r_i\in V$, while if $q_i\notin U$, then $r_i=u$, and
		\item there exists a homeomorphism $\phi:U\to V$ such that for each $q_i\in U$, $\phi(q_i)=r_i$.
	\end{enumerate}
\end{definition}

It will also be convenient in what follows to fix a compatible complete metric $d$ on $X$ and an enumeration $K_\ell$ (for $\ell\in\mathbb{N}$) of compact subsets of $X$ which are closures of precompact basic open sets, so that every open subset of $X$ can be expressed as a union of some $K_\ell$'s; such a family exists by the local compactness and second countability of $X$.
We note that $\mathcal{O}(X)$ can also be given a compatible metric and so $\mathcal{T}$ will inherit a metric structure from $\mathcal{O}(X)\times\mathcal{O}(X)\times(X\cup\{u\})^\mathbb{N}$, but in general we have no reason to suppose that the resulting topology is Polish (or equivalently, that $\mathcal{T}$ is a $G_\delta$ subspace of the latter).

\begin{lemma}\label{lem:T_is_Borel}
	$\mathcal{T}$ is a Borel subset of $\mathcal{O}(X)\times\mathcal{O}(X)\times(X\cup\{u\})^\mathbb{N}$.
\end{lemma}

\begin{proof}
	We leave it to the reader to convince themselves that condition (i) in Definition \ref{defn:Borel_structure_Top} is Borel and instead focus on (ii). Let $\mathbb{Q}^{+}$ denote the set of positive rational numbers. Observe that $(U,V,(r_i)_ {i\in\mathbb{N}})\in\mathcal{O}(X)\times\mathcal{O}(X)\times(X\cup\{u\})^\mathbb{N}$ satisfies (ii) if and only if the following conditions hold:
	\begin{enumerate}[label=\textup{\alph*.}]
		\item For all $K_\ell\subseteq U$,
		\[
			\forall \epsilon\in\mathbb{Q}^+\exists\delta\in\mathbb{Q}^+\forall q_i,q_j\in K_\ell(d(q_i,q_j)<\delta \rightarrow d(r_i,r_j)<\epsilon).
		\]
		\item
		\[
			\forall i\neq j(q_i,q_j\in U\rightarrow r_i\neq r_j)
		\]
		\item 
		\[
			\forall \epsilon\in\mathbb{Q}^+\forall q_i\in V\exists r_j\in V(d(q_i,r_j)<\epsilon)
		\]
		\item For all $K_\ell\subseteq V$,
		\[
			\forall \epsilon\in\mathbb{Q}^+\exists\delta\in\mathbb{Q}^+\forall r_i,r_j\in K_\ell(d(r_i,r_j)<\delta \rightarrow d(q_i,q_j)<\epsilon).
		\]
	\end{enumerate}
	Condition (a) says that the association $q_i\mapsto r_i$ is uniformly continuous on a dense subset of each of the compact sets $K_\ell$ in $U$, which is equivalent, since $U$ is covered by these $K_\ell$'s, to inducing a continuous function $U\to V$. Condition (b) says that this function, at least on the $q_i$'s, is injective, and (c) says that it has dense image in $V$. Lastly, (d) says that the association $r_i\mapsto q_i$ likewise induces a continuous function $V\to U$. We are using the default value $u\notin X$ to avoid accidentally associating an $r_i\in V$ to some $q_i\notin U$.  The uniqueness of continuous extensions of functions from dense domains will ensure that the two resulting maps are inverse to each other. As all quantifiers range over countable sets, these are all Borel conditions.
\end{proof}

We endow $\mathsf{Top}$ with the standard Borel structure it inherits from its natural identification with $\mathcal{T}$.

\begin{lemma}\label{lem:evaluation_is_Borel}
	The evaluation map
	\[
		e:X\times\mathsf{Top}\to X\cup\{u\}:(x,\phi)\mapsto\begin{cases}\phi(x)&\text{if $x\in\mathrm{dom}(\phi)$},\\u&\text{else}.\end{cases}
	\]
	is Borel on the relevant spaces.
\end{lemma}

\begin{proof}
	Observe that
	\[
		e^{-1}[\{u\}]=\{(x,\phi)\in X\times\mathsf{Top}\mid x\notin\mathrm{dom}(\phi)\},
	\]
	which is clearly Borel in $X\times\mathsf{Top}$. If $B\subseteq X$ is Borel, then
	\[
		e^{-1}[B]=\{(x,\phi)\in X\times\mathsf{Top}\mid x\in\mathrm{dom}(\phi)\land \phi(x)\in B\},
	\]
	and note that $\phi(x)\in B$ is a Borel condition in $(x,\phi)$ since we can choose, uniformly and in a Borel way, 
	a subsequence $(q_{i_n})_{n\in\mathbb{N}}$ of the $q_i$'s which converges to $x$ and ask whether $\lim_{n} r_{i_n}\in B$, where $\phi$ is represented by $(U,V,(r_i)_{i\in\mathbb{N}})$ in $\mathcal{T}$.
\end{proof}

\begin{theorem}\label{thm:Top_is_Borel}
	$\mathsf{Top}$ is a Borel pseudogroup.
\end{theorem}

\begin{proof}
	We must verify Borel-ness of the maps (i)--(viii) in Definition \ref{defn:Borel_pseudogroup} of Borel pseudogroup. First, note that the (i) domain and (ii) range projections $\mathsf{Top}\to\mathcal{O}(X)$ are Borel, as they are just the coordinate projections $\mathcal{O}(X)\times\mathcal{O}(X)\times(X\cup\{u\})^\mathbb{N}\to\mathcal{O}(X)$. To see that the (iii) domain inclusion $\mathrm{id}:\mathcal{O}(X)\to\mathsf{Top}$ is Borel, note that $\mathrm{id}(U)=\phi$ in $\mathsf{Top}$, where $\phi$ is represented by $(V,W,(r_i)_{i\in\mathbb{N}})$ in $\mathcal{T}$, if and only if $U=V=W$ and for all $i\in\mathbb{N}$,
	\[
		r_i=\begin{cases}q_i &\text{if $q_i\in U$},\\u&\text{else}.\end{cases}
	\]
	These are clearly Borel conditions.
	
	It is convenient to check (viii) inversion next. We do this by verifying that the graph $\{(\phi,\psi)\in\mathsf{Top}\times\mathsf{Top}:\psi=\phi^{-1}\}$ of this operation is Borel. Writing $(U_1,V_1,(r_i)_{i\in\mathbb{N}})$ and $(U_2,V_2,(s_i)_{i\in\mathbb{N}})$ for the representatives of $\phi$ in $\psi$ respectively, in $\mathcal{T}$, observe that $\psi=\phi^{-1}$ if and only if $U_2=V_1$, $V_2=U_1$, $\psi(r_i)=q_i$ for all $q_i\in U_1$ and $\phi(s_i)=q_i$ for all $q_i\in U_2$. The last two clauses are Borel conditions by Lemma \ref{lem:evaluation_is_Borel}.
	
	For the (iv) direct image, we again check that the graph of this operation is Borel in $\mathcal{O}(X)\times\mathsf{Top}\times\mathcal{O}(X)$. Given $(W_1,\phi,W_2)\in\mathcal{O}(X)\times\mathsf{Top}\times\mathcal{O}(X)$, say with $(U,V,(r_i)_{i\in\mathbb{N}})$ representing $\phi$ in $\mathcal{T}$, $\phi[W_1]\subseteq W_2$ if and only if $\phi[W_1\cap U]\subseteq W_2\cap V$, which in turn is equivalent to saying that for all $K_\ell\subseteq W_1\cap U$, $\phi[K_\ell]\subseteq  W_2\cap V$. Since intersection is a Borel operation in $\mathcal{O}(X)$ and $\phi[K_\ell]\subseteq  W_2\cap V$ if and only if for all $q_i\in K_\ell$, $\phi(q_i)\in W_2\cap V$, this is a Borel condition. We can then use inversion to get that $W_2\subseteq \phi[W_1]$ is Borel as well. The (v) inverse image then follows from another application of inversion.

	To see that the (vi) restriction map is Borel, observe that if $(U_1,V_1,(r_i)_{i\in\mathbb{N}})$ and $(U_2,V_2,(s_i)_{i\in\mathbb{N}})$ in $\mathcal{T}$ represent $\phi$ and $\psi$ in $\mathsf{Top}$, respectively, and $W\in\mathcal{O}(X)$, then $\psi=\phi|_W$ if and only if $U_2=U_1\cap W$, $V_2=V_1\cap \phi[W]$, and for all $i\in\mathbb{N}$,
	\[
		s_i=\begin{cases}r_i &\text{if $q_i\in U_2$},\\u&\text{else}.\end{cases}
	\]
	These are Borel conditions, using that the direct image and intersection of open sets in $\mathcal{O}(X)$ are Borel operations.
	
	Lastly, for (vii) composition, suppose that $(U_1,V_1,(r_i)_{i\in\mathbb{N}})$, $(U_2,V_2,(s_i)_{i\in\mathbb{N}})$, and $(U_3,V_3,(t_i)_{i\in\mathbb{N}})$ in $\mathcal{T}$ represent elements $\phi$, $\psi$, and $\theta$ of $\mathsf{Top}$, respectively. Then, $\psi\circ \phi=\theta$ if and only if $U_3=U_1\cap \phi^{-1}[U_2]$, $V_3=\psi[V_1]$, and for all $i\in\mathbb{N}$,
	\[
		t_i=\begin{cases} \lim_n \psi(q_{i_n}) &\text{if $q_i\in U_3$},\\u&\text{else}.\end{cases}
	\]
	where, as in the previous lemma, $(q_{j_n})_{n\in\mathbb{N}}$ is a subsequence of the $q_j$'s which converges to $r_i=\phi(q_i)$, chosen in a uniform Borel way. These conditions are again Borel, which completes the proof.
\end{proof}

\begin{example}
\label{ex:smooth_pseudogroup}
	Let $X$ be $\mathbb{R}^n$, $\mathbb{C}^n$, $\mathbb{H}^n$, or $\mathbb{S}^n$ and consider the pseudogroup $\mathsf{C}^k$ ($1\leq k<\infty$) of all $C^k$-diffeomorphisms between open subsets	of $X$. The proof of \cite[Lem.\ 7.2]{MR1731384} essentially shows, in the case $X=\mathbb{C}^n$, that $\mathsf{C}^1=\mathsf{Hol}$ is a Borel subset of $\mathcal{T}$ by arguing that continuity of the matrix of partial derivatives can be determined on a fixed countable dense set. This can be easily adapted for higher-order derivatives to show that $\mathsf{C}^k$ is a Borel subset of $\mathsf{Top}$ and, thus, a Borel pseudogroup. Since $\mathsf{C}^\infty=\bigcap_{k=1}^\infty\mathsf{C}^k$, it is also a Borel pseudogroup.
\end{example}

Next, we turn to pseudogroups generated by groups of homeomorphisms.

\begin{lemma}\label{lem:compact-open_Borel}
	Let $X$ be a locally compact Polish space and $\mathrm{Homeo}(X)$ its group of self-homeomorphisms. Then, $\mathrm{Homeo}(X)$ is a Borel subset of $\mathsf{Top}$ and the Borel structures induced by $\mathsf{Top}$ and the compact-open topology on $\mathrm{Homeo}(X)$ coincide.	
\end{lemma}

\begin{proof}
	That $\mathrm{Homeo}(X)$ is a Borel subset of $\mathsf{Top}$ is immediate from Theorem \ref{thm:Top_is_Borel} and the fact that, for $\phi\in\mathsf{Top}$, $\phi\in\mathrm{Homeo}(X)$ if and only if $\mathrm{dom}(\phi)=\mathrm{ran}(\phi)=X$. 
	
	Let $K\subseteq X$ be compact and $U\subseteq X$ open. We claim that the set
	\[
		V(K,U)=\{\phi\in\mathrm{Homeo}(X)\mid\phi[K]\subseteq U\},
	\]
	a subbasic open set for the compact-open topology, is Borel in $\mathsf{Top}$. To see this, fix an enumeration of a countable dense subset $k_m$ $(m\in\mathbb{N})$ of $K$ and a sequence $U_n$ $(n\in\mathbb{N})$ of open subsets of $U$ with compact closures such that $\overline{U_n}\subseteq U_{n+1}$ for all $n\in\mathbb{N}$ and $U=\bigcup_{n\in\mathbb{N}} U_n$. Then, 
	\[
		V(K,U)=\bigcup_{n\in\mathbb{N}}\bigcap_{m\in\mathbb{N}}\{\phi\in\mathrm{Homeo}(X)\mid\phi(k_m)\in U_n\},
	\]
	which is a Borel subset of $\mathsf{Top}$ by Lemma \ref{lem:evaluation_is_Borel}.
	
	It follows that every Borel set in $\mathrm{Homeo}(X)$ with respect to the compact-open topology is Borel with respect to the structure inherited from $\mathsf{Top}$. But then, the identity map $\mathrm{Homeo}(X)\to\mathrm{Homeo}(X)$ is an isomorphism between these two standard Borel spaces and so they must coincide.
\end{proof}

\begin{example}
\label{ex:isometry_pseudogroups}
	Let $X$ be a locally compact Polish space and $G$ a group of homeomorphisms on $X$ which is Borel with respect to the compact-open topology (or equivalently, by Lemma \ref{lem:compact-open_Borel}, as a subset of $\mathsf{Top}$) and has the following property: \emph{any $g\in G$ is uniquely determined by its values on any nonempty open set.} We note that this holds for the groups of isometries of $\mathbb{R}^n$, $\mathbb{S}^n$, and $\mathbb{H}^n$, with the Euclidean, spherical, and hyperbolic metrics, respectively (see \cite[Prop.\ I.2.18]{MR1744486} or \cite[\S6.6]{MR4221225}, where such metric spaces are termed \emph{rigid}). We claim the pseudogroup generated by $G$ is Borel.
	
	Let $\mathsf{G}\subseteq\mathsf{Top}$ be the set of all those $\phi:U\to V$ for which there is some open covering $U_\alpha$ (for $\alpha\in A$) of $U$ such that for each $\alpha\in A$, there is some $g\in G$ such that $\phi|_{U_\alpha}=g|_{U_\alpha\cap U}$. Clearly $\mathsf{G}$ contains $G$, and it is easy to check that it is a pseudogroup on $X$. Moreover, any pseudogroup containing $G$ must contain $\mathsf{G}$, so $\mathsf{G}$ is the pseudogroup generated by $G$. To see that $\mathsf{G}$ is Borel, fix a sequence $U_n$ (for $n\in\mathbb{N}$) of nonempty basic open subsets of $X$. A homeomorphism $\phi:U\to V$ in $\mathsf{Top}$ is then in $\mathsf{G}$ if and only if there is a subsequence $U_{n_k}$ (for $k\in\mathbb{N}$) of the $U_n$'s such that $\bigcup_{k\in\mathbb{N}}U_{n_k}\supseteq U$ and for every $k\in\mathbb{N}$, there is some $g\in G$ such that $\phi|_{U_{n_k}}=g|_{U_{n_k}\cap U}$. This is a $\mathbf{\Sigma}^1_1$ description of $\mathsf{G}$. Moreover, we can also give a $\mathbf{\Pi}^1_1$ description: $\phi\in\mathsf{G}$ if and only if for every $x\in U$, there is an $n\in\mathbb{N}$ such that $x\in U_n\subseteq U$ and $\phi|_{U_n}\in\{g|_{U_n\cap U}\mid g\in G\}$.	Our rigidity assumption on $G$ ensures that, for each $n\in\mathbb{N}$ such that $U_n\cap U\neq\varnothing$, the set $\{g|_{U_n\cap U}\mid g\in G\}$ is the image of the Borel set $G$ under an injective Borel mapping and is thus Borel. Hence, $\mathsf{G}$ is Borel by an application of Suslin's Theorem.
	
	In particular, the pseudogroup $\mathsf{Isom}(\mathbb{H}^n)$ generated by the group $\mathrm{Isom}(\mathbb{H}^n)$ of all isometries of $\mathbb{H}^n$ is Borel, as is the pseudogroup $\mathsf{Isom}^+(\mathbb{H}^n)$ generated by the group of all orientation-preserving isometries of $\mathbb{H}^n$. This will be important in Section \ref{section:bireducibility} and its sequels below.
\end{example}

\begin{example}
\label{ex:pseudogroup_comparison}
Among the virtues of the pseudogroup approach are the comparisons it facilitates between the multiple sorts of structures --- smooth, Riemannian, etc. --- which a given topological manifold may admit.
Systematic relations between such structures will tend to manifest in our framework as Borel maps $f:\mathfrak{M}(\mathcal{G},X)\to \mathfrak{M}(\mathcal{H},Y)$; among the simplest such examples is the inclusion of the class of manifolds without boundary into the class of manifolds with boundary (under the standard convention that the latter means ``with possibly trivial boundary''): Here $\mathcal{G}=\mathcal{H}=\mathsf{Top}$ or $\mathsf{C}^\infty$, $X=\mathbb{R}^n$, $Y=\mathbb{R}^n_+$, and any embedding of $X$ into $Y$ will induce a Borel map $f$ as above; the details are left to the reader. Two remarks on this front are in order: (i) We will adhere below to the somewhat standard convention that the absence of the phrase ``with boundary'' connotes ``without boundary''. (ii) The $f$ just defined is in fact a Borel reduction of $\cong_{\mathcal{G}}$ to $\cong_{\mathcal{H}}$ in the sense of Definition \ref{defn:(G,X)-isomorphism} just below.

For another example, consider the inclusion of the class of smooth manifolds into the class of topological manifolds; here $\mathcal{G}=\mathsf{C}^\infty$, $\mathcal{H}=\mathsf{Top}$, and $X=Y=\mathbb{R}^n$.
This map is plainly Borel, and thus a Borel reduction of $\cong_{\mathsf{C}^\infty}$ to $\cong_{\mathsf{Top}}$ precisely when topological $n$-manifolds admit a unique smooth structure up to diffeomorphism.
By \cite{MR0488059, MR121804}, this in turn holds in exactly the dimensions $n\leq 3$; see Question \ref{ques:R4} of our conclusion and the discussion which follows it for related matters.
\end{example}

\begin{remark}
\label{rmk:parametrization}
To better motivate our approach to manifold parametrization, let us briefly survey a few of its more obvious alternatives.
Most prominent among these, perhaps, is that of this paper's main forerunner \cite{MR1731384}, and indeed, aspects of our constructions are only minor modifications of those appearing therein.
Its approach differs most decisively from ours, though, in encoding manifolds as metric structures endowed with supplementary (complex) structure --- a natural enough choice in that work's contexts, but less so in the setting, for example, of topological manifolds.
The infelicity we have in mind here is, firstly, a conceptual one: an analysis of the homeomorphism relation among topological manifolds $M$ by way of metric structures amounts to selecting compatible metrics for such $M$ only to forget them.
Similar remarks apply to parametrizations via Gromov--Hausdorff metrics, or to parametrizations via manifolds' embedded images in some fixed $\mathbb{R}^n$, for example.
But there are more mathematical downsides to each of these alternatives as well: it is, in general, far from clear that the embedded manifolds of a given type define a Borel subcollection of $\mathcal{F}(\mathbb{R}^n)$; Gromov--Hausdorff metrics on noncompact manifolds call for basepoints; mixed registers (metric and complex, for example) can grow, in practice, unwieldy to manipulate.
To be clear, though, we don't wish to suggest that these alternatives aren't well-suited to any number of particular contexts, but only that the modularity and transparency of the pseudogroup approach accords it several distinctive advantages in the generality in which we wish to work.
\end{remark}

\subsection{Equivalences of manifolds and the case of compact manifolds}
\label{subsection:equivalences}

We have shown that a Borel pseudogroup $\mathcal{G}$ on a locally compact Polish space $X$ determines a standard Borel parametrization $\mathfrak{M}(\mathcal{G},X)$ of the class of $(\mathcal{G},X)$-manifolds. 
The natural notion of equivalence among such manifolds determines in turn an equivalence relation on $\mathfrak{M}(\mathcal{G},X)$.
\begin{definition}
\label{defn:(G,X)-isomorphism}
A homeomorphism $f:M\to N$ between $(\mathcal{G},X)$-manifolds $M$ and $N$ is a $(\mathcal{G},X)$-\emph{isomorphism} if the induced maps between the charts on $M$ and $N$ are locally elements of $\mathcal{G}$ \cite[p.\ 111]{MR1435975}.
The $(\mathcal{G},X)$- prefix will be omitted when it is clear from context.
By $\cong_{\mathcal{G}}$ we denote the equivalence relation on $\mathfrak{M}(\mathcal{G},X)$ given by letting $(\mathcal{U},c)\cong_{\mathcal{G}}(\mathcal{V},d)$ if and only if $M_{(\mathcal{U},c)}$ and $M_{(\mathcal{V},d)}$ are isomorphic.
\end{definition}

Our main aim in the present subsection is to show that, for any Borel pseudogroup $\mathcal{G}$ on a locally compact Polish space $X$, this relation $\cong_{\mathcal{G}}$ is analytic.
From this result it promptly follows that any Borel subspace of $\mathfrak{M}(\mathcal{G},X)$ containing countably many $\cong_{\mathcal{G}}$-equivalence classes is Borel reducible to the relation $=_{\mathbb{N}}$ of equality on the natural numbers.
Several of the most fundamental manifold classification problems manifest in precisely this way, namely the classification of $0$-manifolds, of $1$-manifolds, of compact topological manifolds (with or without boundary), and of compact smooth manifolds (with or without boundary); by showing that the associated parameter spaces are Borel, we will establish that each of these problems is Borel equivalent to $=_\mathbb{N}$.
As a corollary, the computation of any $\cong_{\mathcal{G}}$-invariant of compact topological manifolds is a Borel operation.

We precede this section's main theorem with the following observation, whose straightforward verification is left to the reader.
\begin{lemma}
\label{lem:saturation}
The map which takes an open subset of any $U_i$ in $\mathcal{U}$ to its $(\mathcal{U},c)$-saturation is Borel. More precisely, the map $\mathcal{O}(X)\times\mathbb{N}\times\mathfrak{M}(\mathcal{G},X)\to\mathcal{O}(X)^\mathbb{N}$ given by
$$(U,i,\mathcal{U},c)\mapsto \langle\varphi_{i,j}[U\cap U_i]\mid j\in\mathbb{N}\rangle$$
is Borel.\qed
\end{lemma}

\begin{theorem}
\label{thm:equivalence_analytic}
For any Borel pseudogroup $\mathcal{G}$ on a locally compact Polish space $X$, the equivalence relation $\cong_\mathcal{G}$ on $\mathfrak{M}(\mathcal{G},X)$ is analytic.
\end{theorem}
This assertion and its proof are essentially a translation of those of \cite[Prop.\ 3.3]{MR1731384} to our more generalized setting. In this proof and what follows, we adopt the convention that, for parameters $(\mathcal{U},c)$ and $(\mathcal{V},d)$ in $\mathfrak{M}(\mathcal{G},X)$, elements of $\mathcal{U}$ and $\mathcal{V}$ are denoted by $U_{i,j}$ (or $U_i$) and $V_{i,j}$ (or $V_i$), respectively, while elements of $c$ and $d$ are denoted by $\phi_{i,j}$ and $\psi_{i,j}$, respectively.

\begin{proof}
A $\mathcal{G}$-\emph{good map} $M_{(\mathcal{U},c)}\to M_{(\mathcal{V},d)}$ between two $(\mathcal{G},X)$-manifolds is a local homeomorphism which respects their $(\mathcal{G},X)$ structures; we parametrize such maps via quadruples $(\mathcal{X},\mathcal{Y},k,e)\in\mathcal{O}(X)^{\mathbb{N}\times\mathbb{N}}\times\mathcal{O}(X)^{\mathbb{N}\times\mathbb{N}}\times\mathbb{N}^{\mathbb{N}\times\mathbb{N}}\times\mathcal{G}^{\mathbb{N}\times\mathbb{N}}$ of the following sort:
\begin{enumerate}[label=\textup{\roman*.}]
\item the family $\mathcal{X}=\langle X_{i,j}\in\mathcal{O}(X)\mid (i,j)\in\mathbb{N}^2\rangle$ satisfies $\bigcup_{j\in\mathbb{N}}X_{i,j}=U_{i}$ for all $i\in\mathbb{N}$;
\item the family $\mathcal{Y}=\langle Y_{i,j}\in\mathcal{O}(X)\mid (i,j)\in\mathbb{N}^2\rangle$ satisfies $Y_{i,j}\subseteq V_{k(i,j)}$ for all $(i,j)\in\mathbb{N}^2$;
\item the family $e=\langle\rho_{i,j}:X_{i,j}\to Y_{i,j}\mid (i,j)\in\mathbb{N}^2\rangle$ of elements of $\mathcal{G}$ satisfies
$$\rho_{i',j'}\circ\varphi_{i,i'}=\psi_{k(i,j), k(i',j')}\circ\rho_{i,j}$$
for all $(i,i',j,j')\in\mathbb{N}^4$. 
We reiterate the possibility that the composition of pseudogroup elements may be the empty set.
\end{enumerate}
In words, $(\mathcal{X},\mathcal{Y},k,e)$ is a family of open sets and pseudogroup elements subordinate to the open sets in $\mathcal{U}$ and $\mathcal{V}$ which factor through the quotients $M_{(\mathcal{U},c)}$ and $M_{(\mathcal{V},d)}$, or \emph{cohere} in the sense of item (iii) above.
In particular, just as the domain of any such element $\rho_{i,j}$ should fall within a chart $U_i$ of $M_{(\mathcal{U},c)}$, its range should fall within a chart $V_{k(i,j)}$ of $M_{(\mathcal{V},d)}$, and it is this latter relation which the function $k:\mathbb{N}^2\to\mathbb{N}$ is recording.
Write $g_{(\mathcal{X},\mathcal{Y},k,e)}$ for the induced mapping of the quotients $M_{(\mathcal{U},c)}\to M_{(\mathcal{V},d)}$.

Let $S=\mathcal{O}(X)^{\mathbb{N}\times\mathbb{N}}\times\mathcal{O}(X)^{\mathbb{N}\times\mathbb{N}}\times\mathbb{N}^{\mathbb{N}\times\mathbb{N}}\times\mathcal{G}^{\mathbb{N}\times\mathbb{N}}$, and for any $(\mathcal{U},c)$ and $(\mathcal{V},d)$ in $\mathfrak{M}(\mathcal{G},X)$ write $\mathrm{Map}((\mathcal{U},c),(\mathcal{V},d))$ for the set of quadruples in $S$ satisfying conditions (i) through (iii) above.
Note that, just as in \cite[p.\ 133]{MR1731384}, these are Borel conditions and hence the set
$$\{((\mathcal{U},c),(\mathcal{V},d),(\mathcal{X},\mathcal{Y},k,e))\mid (\mathcal{X},\mathcal{Y},k,e)\in\mathrm{Map}((\mathcal{U},c),(\mathcal{V},d))\}$$
is a Borel subset of $\mathfrak{M}(\mathcal{G},X)\times\mathfrak{M}(\mathcal{G},X)\times S$.
Similarly, the set $\mathrm{Iso}(\mathfrak{M}(\mathcal{G},X))$ of all
$$((\mathcal{U},c),(\mathcal{V},d),(\mathcal{X},\mathcal{Y},k,e), (\mathcal{Z},\mathcal{W},\ell,f))\in \mathfrak{M}(\mathcal{G},X)\times\mathfrak{M}(\mathcal{G},X)\times S\times S$$
satisfying
\begin{itemize}
\item $(\mathcal{X},\mathcal{Y},k,e)\in\mathrm{Map}((\mathcal{U},c),(\mathcal{V},d))$,
\item $(\mathcal{Z},\mathcal{W},\ell,f)\in\mathrm{Map}((\mathcal{V},d),(\mathcal{U},c))$, and
\item $g_{(\mathcal{X},\mathcal{Y},k,e)}={g_{(\mathcal{Z},\mathcal{W},\ell,f)}}^{-1}$
\end{itemize}
is Borel.
To elaborate on the third part of this assertion, writing $\varphi_{i,j}$, $\psi_{i,j}$, $\rho_{i,j}$, and $\theta_{i,j}$ for the functions in $c$, $d$, $e$, and $f$, respectively, we have that $g_{(\mathcal{X},\mathcal{Y},k,e)}={g_{(\mathcal{Z},\mathcal{W},\ell,f)}}^{-1}$ if and only if
\begin{itemize}
\item $\psi_{k(i',j'),i}\circ\rho_{i',j'}=\theta^{-1}_{i,j}\circ\varphi_{i',\ell(i,j)}$ on the intersection of these composite functions' domains for all $i,j,i',j'\in\mathbb{N}$, and
\item the saturations of the images of the functions $\theta_{i,j}$ cover $\coprod_{i\in\mathbb{N}}U_i$,
\end{itemize}
the former of which is Borel by Definition \ref{defn:Borel_pseudogroup}, and the latter by Lemma \ref{lem:saturation}.
As the equivalence relation $\cong_{\mathcal{G}}$ on $\mathfrak{M}(\mathcal{G},X)$ is simply the projection of the Borel set $\mathrm{Iso}(\mathfrak{M}(\mathcal{G},X))$ to $\mathfrak{M}(\mathcal{G},X)\times\mathfrak{M}(\mathcal{G},X)$, this concludes the proof.
\end{proof}

Note for later use that our argument shows that for any $\mathcal{G}$ as in Theorem \ref{thm:equivalence_analytic} and Borel pseudogroup $\mathcal{H}\supseteq\mathcal{G}$, the relation $\cong_\mathcal{H}$ on $\mathfrak{M}(\mathcal{G},X)$ is analytic as well.

We turn now to the problem of classifying compact topological manifolds. Given a Borel pseudogroup $\mathcal{G}$ on a locally compact Polish space $X$, we denote by $\mathfrak{K}(\mathcal{G},X)$ the set of all $(\mathcal{U},c)\in\mathfrak{M}(\mathcal{G},X)$ such that $M_{(\mathcal{U},c)}$ is compact.

\begin{theorem}
\label{thm:compact_Borel}
If $\mathcal{G}$ is a Borel pseudogroup on a locally compact Polish space $X$, then $\mathfrak{K}(\mathcal{G},X)$ forms a Borel subset of $\mathfrak{M}(\mathcal{G},X)$.
\end{theorem}

\begin{proof}
Much as in our argument that Hausdorff manifolds form a Borel subset of $\mathfrak{P}(\mathcal{G},X)$, we will show that $\mathfrak{K}(\mathcal{G},X)$ is both a coanalytic and analytic subset of $\mathfrak{M}(\mathcal{G},X)$; an application of Suslin's Theorem will complete the proof.

The $\mathbf{\Pi}^1_1$ characterization is, in essence, the obvious one: $\mathfrak{K}(\mathcal{G},X)$ consists of those $(\mathcal{U},c)$ such that every open cover of $M_{(\mathcal{U},c)}$ admits a finite subcover.
More formally, write $p$ for the map taking $\mathcal{U}\in\mathcal{O}(X)^{\mathbb{N}\times\mathbb{N}}$ to $(U_k)_{k\in\mathbb{N}}\in\mathcal{O}(X)^\mathbb{N}$; fix also a basis $\{B_i\mid i\in\mathbb{N}\}$ for the topology of $X$. This determines in turn a basis for $M_{(\mathcal{U},c)}$, and hence the pair $(\mathcal{U},c)\in\mathfrak{M}(\mathcal{G},X)$ parametrizes a compact topological manifold if and only if for every $f:\mathbb{N}^2\to\mathbb{N}$, if
$$\left\langle\bigcup_{(i,j)\in\mathbb{N}^2}\varphi_{i,k}[B_{f(i,j)}\cap U_i]\mid k\in\mathbb{N}\right\rangle=p(\mathcal{U})$$
then there exists an $N\in\mathbb{N}$ such that 
$$\left\langle\bigcup_{i,j< N}\varphi_{i,k}[B_{f(i,j)}\cap U_i]\mid k\in\mathbb{N}\right\rangle=p(\mathcal{U}).$$

For the $\mathbf{\Sigma}^1_1$ characterization, observe that $(\mathcal{U},c)$ parametrizes a compact manifold if and only if for some $N\in\mathbb{N}$ there exist compact sets $K_i\subseteq U_i$ $(i<N)$ whose natural images cover $M_{(\mathcal{U},c)}$.
At a more formal level, there are two things to check: first, that the set $A=\{(K,U)\mid K\subseteq U\}$ is a Borel subset of $\mathcal{K}(X)\times\mathcal{O}(X)$, and second, that the covering in question is a Borel condition.
Beginning with the second, note that if the natural images of $U_i$ $(i<N)$ cover $M_{(\mathcal{U},c)}$, then so too will those of $K_i$ $(i<N)$ if and only if $\bigcap_{i<N}\varphi_{i,j}[U_i\backslash K_i]=\varnothing$ for all $j<N$, and that the latter is evidently a Borel condition.
Turning now to the first, invoke the local compactness and second countability of $X$ to write it as an increasing union of compact sets $\{X_i\mid i\in\mathbb{N}\}$.
Note next that $A_i=\{(K,U)\mid K\subseteq U\}$ is, in the Fell-induced topology, an open subset of $\mathcal{K}(X)\times \mathcal{O}(X_i)$.
Now let $f_i:\mathcal{O}(X)\to\mathcal{O}(X_i)$ denote the Borel map $U\mapsto U\cap X_i$, and observe that $A=\bigcup_{i\in\mathbb{N}}(\mathrm{id}\times f_i)^{-1}(A_i)$. This concludes the proof.
\end{proof}

Our first Borel complexity computation follows almost immediately from the preceding theorem in combination with Theorem \ref{thm:equivalence_analytic} and Cheeger and Kister's 1970 result that there are, up to homeomorphism, only countably many compact topological manifolds.
Within our introduction, this result figured as Theorem A.

\begin{corollary}\label{cor:compact_manifolds_countable}
For all $n\geq 0$, the homeomorphism relation $\cong_{\mathsf{Top}}$ on $\mathfrak{K}(\mathsf{Top},\mathbb{R}^n)$ is Borel equivalent to $=_{\mathbb{N}}$; hence the homeomorphism relation $\cong_{\mathsf{Top}}$ on the parameter space of all compact topological manifolds $\coprod_{n\in\mathbb{N}}\mathfrak{K}(\mathsf{Top},\mathbb{R}^n)$ is as well.
\end{corollary}

\begin{proof}
By \cite{MR256399}, we may identify the set of homeomorphism classes of compact topological $n$-manifolds with the natural numbers.
By Theorems \ref{thm:equivalence_analytic} and \ref{thm:compact_Borel}, then, the map $f:\mathfrak{K}(\mathsf{Top},\mathbb{R}^n)\to\mathbb{N}$ taking $(\mathcal{U},c)$ to the homeomorphism class of $M_{(\mathcal{U},c)}$ is Borel, for the reason that any $f^{-1}(\{n\})$ is both analytic and the complement of a countable union of analytic sets, and consequently Borel, again by Suslin's Theorem.
Any reverse map $\mathbb{N}\to\mathfrak{K}(\mathsf{Top},\mathbb{R}^n)$ selecting a representative for each homeomorphism class is trivially Borel as well, and will be a reduction of $=_{\mathbb{N}}$ to $\cong_\mathsf{Top}$.
\end{proof}

The next corollary tells us that we can compute \emph{any} topological invariant for compact manifolds, e.g.,~genus, Euler characteristic, homotopy groups, etc, in a Borel way. Contrast this with a theorem of Becker \cite{MR1233807} which says that, among arbitrary compact subsets of $\mathbb{R}^2$, being simply connected is a $\mathbf{\Pi}^1_1$-complete property.
\begin{corollary}\label{cor:compact_invariants_Borel}
Any assignment of invariants to compact topological manifolds is a Borel operation.
\end{corollary}
\begin{proof}
Any such assignment will amount to postcomposing the Borel reduction to $=_{\mathbb{N}}$ of the previous corollary with a map from $\mathbb{N}$ to a countable set.
\end{proof}

The above line of reasoning applies more generally. 
\begin{corollary}\label{cor:Borel_compact_with_boundary}
Corollaries \ref{cor:compact_manifolds_countable} and \ref{cor:compact_invariants_Borel} continue to hold when the class of compact topological manifolds is replaced with any of the following:
\begin{itemize}
\item compact topological manifolds with boundary,
\item compact smooth manifolds (up to diffeomorphism), or
\item compact smooth manifolds with boundary (up to diffeomorphism).
\end{itemize}
\end{corollary}

\begin{proof}
By Theorem \ref{thm:compact_Borel}, each of these classes of manifolds may be parametrized by a standard Borel space; the assertion therefore reduces to the claim that each of these classification problems involves only countably many equivalence classes.
In the case of the first item, this fact is again due to \cite{MR256399}; for the second and third, apply the fact that any compact topological $n$-manifold $M$ (with or without boundary) admits, up to equivalence, only countably many smooth structures --- in fact only finitely many if $n\neq 4$ (the countability bound follows from, e.g., \cite{MR149491}; see \cite{MR0263092} for the finiteness assertion).
\end{proof}

\begin{corollary}
The classification of topological (or smooth) $0$-manifolds and $1$-manifolds, possibly with boundary, up to homeomorphism (or diffeomorphism, respectively) are each Borel equivalent to $=_{\mathbb{N}}$.
\end{corollary}

\begin{proof}
By a $0$-manifold we mean a second countable Hausdorff space modelled on $\mathbb{R}^0=\{\mathrm{pt}\}$; these are the countable discrete spaces, and they are completely classified by their cardinality.
Each component of a $1$-manifold is either homeomorphic to a circle or to the real line \cite[Thm.\ 5.27]{MR2766102}; the numbers of components of each type provide a complete invariant for such manifolds.
In each case the set of invariants is countable, and the argument of Corollary \ref{cor:compact_manifolds_countable} applies just as before to complete the proof. Similar comments apply to $1$-manifolds with boundary, and in the smooth case.
\end{proof}

\subsection{Exhaustions, connectedness, and metrics on manifolds}
\label{subsection:subclasses}

This final subsection represents something of a toolbox for the remainder of the paper and, we hope, for further work beyond it. Among the results we establish here, for appropriate choices of Borel pseudogroups $\mathcal{G}$ and model spaces $X$, are the existence of Borel functions which compute:
\begin{itemize}
	\item an exhaustion of any $(\mathcal{G},X)$-manifold by compact subsets  (Lemma \ref{lem:exhaustion});
	\item a reparametrization of any $(\mathcal{G},X)$-manifold by one with locally finite charts (Lemma \ref{lem:paracompact});
	\item a reparametrization of any $(\mathcal{G},X)$-manifold by one with simply connected charts (Lemma \ref{lem:reparametrization}); and
	\item the connected components of any $(\mathcal{G},X)$-manifold (Lemma \ref{lem:Borel_computation_of_components}).
\end{itemize}
and Borel parameter spaces for:
\begin{itemize}
	\item all connected $(\mathcal{G},X)$-manifolds (Lemma \ref{lem:connected_Borel}); and
	\item all (metrically) complete $(\mathcal{G},X)$-manifolds (Lemma \ref{lem:complete}).
\end{itemize}
We will also note a few irregularities involving quotient metrics, alongside ways, within our framework, to address them. While these facts will be crucial to several of our main theorems, they are of a technical nature, often describing how to perform a standard construction in a Borel way. For this reason, we suggest that first-time readers skim this section, returning to it as needed later on.

It will be useful throughout to have a uniform way of discussing open subsets of manifolds within our parametrization. Let us fix a Borel pseudogroup $\mathcal{G}$ on a locally compact Polish space $X$ for the duration (we will place stronger demands on $\mathcal{G}$ and $X$ later on).  
Given $(\mathcal{U},c)$ and $(\mathcal{V},d)$ in $\mathfrak{M}(\mathcal{G},X)$, we call $M_{(\mathcal{V},d)}$ an \emph{open submanifold} of $M_{(\mathcal{U},c)}$ if
\begin{equation}
\label{eq:open_submanifold}
V_{i,j}\subseteq U_{i,j}\text{ and }\psi_{i,j}=\varphi_{i,j}\big|_{V_{i,j}}\text{ for all }(i,j)\in\mathbb{N}^2.
\end{equation}
Here, as before, $\varphi_{i,j}$ and $\psi_{i,j}$ denote elements of $c$ and $d$, respectively. If (\ref{eq:open_submanifold}) holds, we will write $d=c|_{\mathcal{V}}$. So defined, the open submanifolds of $M_{(\mathcal{U},c)}$ correspond exactly to its open subsets. The corresponding relation on parameters, for which we abuse terminology and say that $(\mathcal{V},d)$ is an \emph{open submanifold} of $(\mathcal{U},c)$, is clearly Borel on $\mathfrak{M}(\mathcal{G},X)$. Among other alternatives, we can also view an open subset of $M_{(\mathcal{U},c)}$ as given by a sequence $(V_i)_{i\in\mathbb{N}}$ of open subsets of $X$ such that $V_i\subseteq U_i$ for all $i\in\mathbb{N}$. We can then build an open submanifold $M_{(\mathcal{V},d)}$, in the sense just defined, which corresponds to the same underlying open set in $M_{(\mathcal{U},c)}$, using Lemma \ref{lem:saturation}.

\subsubsection{Exhaustions}
\label{subsubsect:exhaustions} 
A standard fact we've now tacitly used more than once is that every locally compact Polish space, and in particular, every manifold $M$, admits a \emph{compact exhaustion}, that is, a sequence $K_n$ (for $n\in\mathbb{N}$) of compact subsets of $M$ such that for all $n\in\mathbb{N}$, $K_n\subseteq\mathrm{int} (K_{n+1})$, and $M=\bigcup_{n\in\mathbb{N}}K_n$ \cite[Prop.\ 4.76]{MR2766102}. As a consequence, every manifold $M$ is \emph{paracompact}, meaning every open cover $\mathcal{U}$ of $M$ admits \emph{locally finite refinement} $\mathcal{U}'$, that is, a collection of open subsets of elements of $\mathcal{U}$ which still covers $M$, but for which each point of $M$ lies in at most finitely many elements of $\mathcal{U}'$ \cite[Thm.\ 4.77]{MR2766102}. We wish to realize both of these facts in a Borel way.
To this end, we first need a lemma which says that we can tell whether an open subset of a manifold contains the \emph{closure} of another.

\begin{lemma}\label{lem:contain_closure_Borel} 
    The relation ``$M_{(\mathcal{V},d)}$ and $M_{(\mathcal{W},e)}$ are open submanifolds of $M_{(\mathcal{U},c)}$ and the closure of $M_{(\mathcal{V},d)}$ is contained in $M_{(\mathcal{W},e)}$'', for $(\mathcal{U},c)$, $(\mathcal{V},d)$, and $(\mathcal{W},e)$ in $\mathfrak{M}(\mathcal{G},X)$, is Borel on $\mathfrak{M}(\mathcal{G},X)$.
\end{lemma}

To prove Lemma \ref{lem:contain_closure_Borel}, we start by showing the corresponding fact for open subsets of our fixed model space $X$, which we will then transfer to the manifold setting locally via the charts. We will make repeated use the elementary fact that the closure of a set can be ``computed locally'', meaning if $\mathcal{U}$ is any open cover of a space $Y$ and $A\subseteq Y$, then
\[
    \overline{A}=\bigcup_{U\in\mathcal{U}}U\cap\overline{A}=\bigcup_{U\in\mathcal{U}}\mathrm{cl}_{U}(A),
\]
where $\mathrm{cl}_U(A)$ is the relative closure of $A$ in the subspace $U$.

\begin{lemma}\label{lem:contain_closure_Borel0}
    The relation $\overline{V}\subseteq W$, for $V,W\in\mathcal{O}(X)$, is Borel on $\mathcal{O}(X)$.
\end{lemma}

\begin{proof}
    Fix a compatible complete metric $d$ on $X$ and a countable dense subset $D$ of $X$. Suppose first that $V$ is precompact in $X$. Then,
    \begin{align*}
        \overline{V}\subseteq W &\Leftrightarrow \exists\delta>0 (d(V,X\setminus W)\geq\delta)\\
        &\Leftrightarrow \exists\delta\in\mathbb{Q}^+\forall x\in V\cap D(B_\delta(x)\subseteq W),
    \end{align*}
    which is clearly Borel in $V$ and $W$.

    For the general case, fix a compact exhaustion $K_n$ (for $n\in\mathbb{N}$) of $X$. By paracompactness, and possibly reindexing, we may let $U_n$ (for $n\in\mathbb{N}$) be a locally finite refinement of the $\mathrm{int}(K_n)$'s. Then, by local finiteness,
    \[
        \overline{V}=\overline{\bigcup_{n\in\mathbb{N}}V\cap U_n}=\bigcup_{n\in\mathbb{N}}\overline{V\cap U_n},
    \]
    and so,
    \begin{align*}
        \overline{V}\subseteq W &\Leftrightarrow \forall n\in\mathbb{N}(\overline{V\cap U_n}\subseteq W),
    \end{align*}
    which we can check in a Borel way by the first case.
\end{proof}

\begin{lemma}\label{lem:contain_closure_Borel1}
	The relation $\overline{V}\cap U\subseteq W$, for $U,V,W\in\mathcal{O}(X)$, is Borel on $\mathcal{O}(X)$.
\end{lemma}

\begin{proof}
   First, it will help to locally compute a compact exhaustion for any nonempty open set $U$, in a Borel way. That is, we claim that there is a Borel map $E:\mathcal{O}(X)\setminus\{\emptyset\}\to\mathcal{O}(X)^\mathbb{N}$ such that for each $U\in X$, $E(U)=(U_n)_{n\in\mathbb{N}}$ is an exhaustion of $U$ by a sequence of open subsets of $U$ whose closures are compact and contained in $U$. Defining this map is straightforward; simply take unions of precompact basic open subsets of $X$, from an enumeration we fix at the outset, whose closures are contained in $U$ and which exhaust $U$.

    Next, given $U,V,W\in\mathcal{O}(X)$, and $E(U)=(U_n)_{n\in\mathbb{N}}$ as above, note that for all $n\in\mathbb{N}$,
    \[
        \overline{V}\cap U_n\subseteq\overline{V\cap U_n}\subseteq \overline{V}\cap\overline{U_n},
    \]
    and so
    \[
        \overline{V}\cap U =\bigcup_{n\in\mathbb{N}}\overline{V}\cap U_n\subseteq\bigcup_{n\in\mathbb{N}}\overline{V\cap U_n}\subseteq\bigcup_{n\in\mathbb{N}}\overline{V}\cap\overline{U_n}=\overline{V}\cap U.
    \]
    Thus,
    \[
        \overline{V}\cap U\subseteq W \Leftrightarrow \forall n\in\mathbb{N}(\overline{V\cap U_n}\subseteq W), 
    \]
    which we may detect in a Borel way by Lemma \ref{lem:contain_closure_Borel0}.
\end{proof}

We can now prove Lemma \ref{lem:contain_closure_Borel}.

\begin{proof}[Proof of Lemma \ref{lem:contain_closure_Borel}]
    Suppose $(\mathcal{U},c)$, $(\mathcal{V},d)$, and $(\mathcal{W},e)$ are given so that the latter two represent open submanifolds of the former; we know that this is a Borel condition. In particular, $d=c|_{\mathcal{V}}$ and $e=c|_{\mathcal{W}}$. Then, the closure of the submanifold represented by $(\mathcal{V},d)$ is contained in that represented by $(\mathcal{W},e)$ if and only if for all $i\in\mathbb{N}$, $\mathrm{cl}_{U_i}(V_i)=\overline{V_i}\cap U_i\subseteq W_i$, which is a Borel condition by Lemma \ref{lem:contain_closure_Borel1}.
\end{proof}

The next lemma will show how to obtain the interiors and complements of a compact exhaustion of $M_{(\mathcal{U},c)}$ in a way which is Borel in the parameter $(\mathcal{U},c)\in\mathfrak{M}(\mathcal{G},X)$. 
As we lack a natural parameter space encoding all compact subsets of $(\mathcal{G},X)$-manifolds, we do not obtain the compact sets themselves; however, Lemma \ref{lem:exhaustion}, together with Lemma \ref{lem:contain_closure_Borel}, will suffice for our present needs. See Lemma \ref{lem:canonical_exhaustion} and the discussion which follows it for possible strengthenings. 

\begin{lemma}\label{lem:exhaustion} 
	There are Borel maps $\mathfrak{M}(\mathcal{G},X)\to\mathfrak{M}(\mathcal{G},X)^\mathbb{N}$, denoted by:
    \begin{align*}
        (\mathcal{U},c)&\mapsto\langle(\mathcal{V}_k,c|_{\mathcal{V}_k})\mid k\in\mathbb{N}\rangle\\ 
        (\mathcal{U},c)&\mapsto\langle(\mathcal{W}_k,c|_{\mathcal{W}_k})\mid k\in\mathbb{N}\rangle,
    \end{align*}
    such that for all $(\mathcal{U},c)\in\mathfrak{M}(\mathcal{G},X)$ and $k\in\mathbb{N}$:
    \begin{enumerate}[label=\textup{\roman*.}]
        \item $(\mathcal{V}_k,c|_{\mathcal{V}_k})$ and $(\mathcal{W}_k,c|_{\mathcal{W}_k})$ are open submanifolds of $(\mathcal{U},c)$,
        \item $M_{(\mathcal{W}_k,c|_{\mathcal{W}_k})}$ is the complement of the closure of $M_{(\mathcal{V}_k,c|_{\mathcal{V}_k})}$ in $M_{(\mathcal{U},c)}$,
    \end{enumerate}
    and
    \begin{enumerate}[label=\textup{\roman*.}]
        \setcounter{enumi}{2}
        \item the closures of the $M_{(\mathcal{V}_k,c|_{\mathcal{V}_k})}$ (for $k\in\mathbb{N}$) in $M_{(\mathcal{U},c)}$ form a compact exhaustion of $M_{(\mathcal{U},c)}$.
    \end{enumerate}
\end{lemma}

\begin{proof}
	Let $B_m$ (for $m\in\mathbb{N}$) be an enumeration of precompact basic open sets in $X$. Let $(\mathcal{U},c)$ be given and write $M=M_{(\mathcal{U},c)}$. We may thin down, in a Borel way, the sequence of $B_m$'s so that for every $m\in\mathbb{N}$, there is some $i\in\mathbb{N}$ for which $\overline{B_m}\subseteq U_i$. The images of such $B_m$'s will still cover $M$. For each $k\in\mathbb{N}$, we will describe $(\mathcal{V}_k,c|_{\mathcal{V}_k})$ and $(\mathcal{W}_k,c|_{\mathcal{W}_k})$, where $\mathcal{V}_k=\langle V_{i,j}^k\mid i,j\in\mathbb{N}\rangle$ and $\mathcal{W}_k=\langle W_{i,j}^k \mid i,j\in\mathbb{N}\rangle$, in a way that will be evidently Borel in $(\mathcal{U},c)$, given prior results.
	
	Let $(\mathcal{V}_0,c|_{\mathcal{V}_0})$ be the saturation of $B_0$ (in the sense of Lemma \ref{lem:saturation}), and $(\mathcal{W}_0,c|_{\mathcal{W}_0})$ the complement of its closure in $M$, i.e., $W_{i,j}^0=U_{i,j}\setminus\overline{V_{i,j}^0}$ for each $i,j\in\mathbb{N}$.

	Given $(\mathcal{V}_k,c|_{\mathcal{V}_k})$ and $(\mathcal{W}_k,c|_{\mathcal{W}_k})$, let $m$ be the least such that the saturation of $B_0\cup\cdots\cup B_m$ covers the closure of $(\mathcal{V}_k,c|_{\mathcal{V}_k})$ in $M$; such an $m$ exists since the latter has compact closure in $M$, and we can detect this in a Borel way by Lemma \ref{lem:contain_closure_Borel}. Let $(\mathcal{V}_{k+1},c|_{\mathcal{V}_{k+1}})$ be the saturation of $B_0\cup\cdots\cup B_m$, and $(\mathcal{W}_{k+1},c|_{\mathcal{W}_{k+1}})$ the complement of its closure in $M$, i.e., $W_{i,k}^{k+1}=U_{i,j}\setminus\overline{V_{i,j}^{k+1}}$ for all $i,j\in\mathbb{N}$. 
	
	Properties (i), (i), and (iii) have all been ensured by construction.
\end{proof}

We note the following useful consequence of Lemma \ref{lem:exhaustion}:

\begin{lemma}\label{lem:bounded_Borel} 
	The relation ``$M_{(\mathcal{V},d)}$ is a bounded open submanifold of $M_{(\mathcal{U},c)}$'', for $(\mathcal{U},c),(\mathcal{V},d)\in\mathfrak{M}(\mathcal{G},X)$, is Borel on $\mathfrak{M}(\mathcal{G},X)$.\qed	
\end{lemma}

We next give a general definition for what it means to ``reparametrize'' $(\mathcal{G},X)$-manifolds in a Borel way.

\begin{definition}
    A \emph{Borel reparametrization} is a Borel function $r:\mathfrak{M}(\mathcal{G},X)\to\mathfrak{M}(\mathcal{G},X)$ such that $r(\mathcal{U},c)\cong_\mathcal{G}(\mathcal{U},c)$ for all $(\mathcal{U},c)\in\mathfrak{M}(\mathcal{G},X)$.
\end{definition}

\begin{definition}
    Given $(\mathcal{U},c),(\mathcal{V},d)\in\mathfrak{M}(\mathcal{G},X)$, where $c=\langle \varphi_{i,j}\mid i,j\in\mathbb{N}\rangle$ and $d=\langle\psi_{i,j}\mid i,j\in\mathbb{N}\rangle$, we say that $(\mathcal{V},d)$ is a \emph{refinement} of $(\mathcal{U},c)$ if:
    \begin{enumerate}[label=\textup{\roman*.}]
        \item  for all $i,j\in\mathbb{N}$, there are $i',j'\in\mathbb{N}$ such that $V_{i,j}\subseteq U_{i',j'}$ and $\psi_{i,j}=\varphi_{i',j'}|_{V_{i,j}}$, and
        \item $\bigcup_{i\in\mathbb{N}}V_i=\bigcup_{i\in\mathbb{N}} U_i$.
    \end{enumerate}
    Note that if $(\mathcal{V},d)$ is a refinement of $(\mathcal{U},c)$, then $(\mathcal{V},d)\cong_{\mathcal{G}}(\mathcal{U},c)$.
\end{definition}

\begin{definition}
    We say that $(\mathcal{U},c)\in\mathfrak{M}(\mathcal{G},X)$ is \emph{locally finite} if for all $i\in\mathbb{N}$, there are at most finitely many $j\in\mathbb{N}$ for which $U_{i,j}\neq\varnothing$.
\end{definition}

The next lemma is a straightforward exercise in point-set topology.

\begin{lemma}\label{lem:closure_difference}
    If $U$ and $V$ are open subsets of a topological space $Y$ and $\overline{V}\subseteq U$, then $\overline{U}\setminus V = \overline{U\setminus\overline{V}}$.\qed
\end{lemma}

Lemma \ref{lem:paracompact} below and its proof are based on a standard argument showing that every topological manifold is paracompact \cite[Thm.\ 1.15]{MR2954043}.\footnote{In fact, the cited argument shows that manifolds are \emph{totally paracompact}, meaning every basis admits a locally finite subcover, in the sense of \cite{MR205226}.} The second basis $\mathcal{C}$ in the statement of the lemma will play a technical role in the proof of Theorem \ref{thm:Borel_triangulation} below.

\begin{lemma}\label{lem:paracompact}
    Let $\mathcal{B}$ be a countable basis for the topology on $X$. There is a Borel reparameterization $r:\mathfrak{M}(\mathcal{G},X)\to \mathfrak{M}(\mathcal{G},X)$ such that for each $(\mathcal{U},c)\in\mathfrak{M}(\mathcal{G},X)$, $r(\mathcal{U},c)=(\mathcal{Z},d)$ is a locally finite refinement of $(\mathcal{U},c)$, with each $Z_i\in\mathcal{B}$. Moreover, if $\mathcal{C}$ is another countable basis for the topology on $X$, it can be arranged that each $Z_i$ contains $\overline{C_j}$, for some $C_j\in\mathcal{C}$, and their interior images also cover $M_{(\mathcal{Z},d)}$.
\end{lemma}

\begin{proof}
    Suppose that $\mathcal{B}=\{ B_m\mid m\in\mathbb{N}\}$ and $\mathcal{C}=\{ C_m\mid m\in\mathbb{N}\}$. Let $(\mathcal{U},c)$ in $\mathfrak{M}(\mathcal{G},X)$ be given. Let $(\mathcal{V}_k,c|_{\mathcal{V}_k})$ and $(\mathcal{W}_k,c|_{\mathcal{W}_k})$, for all $k\in\mathbb{N}$, be as in Lemma \ref{lem:exhaustion}. Write $M=M_{(\mathcal{U},c)}$, $M_k=M_{(\mathcal{V}_k,c|_{\mathcal{V}_k})}$, and $M_k'=M_{(\mathcal{W}_k,c|_{\mathcal{W}_k})}$, so that the manifolds $M_k$ are open subsets of $M$ whose closures form a compact exhaustion of $M$ and having the $M_k'$ as their complements. Let 
    \[
        A_k=\overline{M_{k+1}}\setminus M_k=\overline{M_{k+1}\setminus\overline{M_k}}=\overline{M_{k+1}\cap M_k'},
    \]
    where the second equality uses Lemma \ref{lem:closure_difference}, and
    \[
        Y_k=M_{k+2}\setminus\overline{M_{k-1}}=M_{k+2}\cap M_{k-1}'.
    \]
    Then, each $A_k$ is a compact subset of the open set $Y_k$. In terms of the parameters, $Y_k=M_{(\mathcal{Y}_k,c|_{\mathcal{Y}_k})}$, where $\mathcal{Y}_k=\langle Y_{i,j}^k\mid i,j\in\mathbb{N}\rangle$ and $Y_{i,j}^k=V_{i,j}^{k+2}\cap W_{i,j}^{k-1}$.

    We define a refinement $(\mathcal{Z},d)$ of $(\mathcal{U},c)$ as follows: Let $s_0$ be a finite subset of $\mathbb{N}$ such that for each $m\in s_0$, $\overline{C_m}\subseteq B_m\subseteq U_{i_m}\cap Y^0_{i_m}$, for some $i_m\in\mathbb{N}$, and the images of the interiors of the $C_m$'s (under the restrictions of the quotient map to $U_{i_m}$) cover $A_0$; such a finite set exists because $A_0$ is compact in $M$, and we can find it in a Borel way by Lemma \ref{lem:contain_closure_Borel}. Let $Z_i$, for $0\leq i<|s_0|$, be those $B_m\subseteq U_{i_m}$ for $m\in s_0$. More generally, having defined $s_{\ell}$ for $\ell<k$ and $Z_i$ for $0\leq i<\sum_{\ell<k}|s_\ell|$, let $s_k$ be a finite subset of $\mathbb{N}$ such that for each $m\in s_k$, $\overline{C_m}\subseteq B_m\subseteq U_{i_m}\cap Y^k_{i_m}$, for some $i_m\in\mathbb{N}$, and the images of the interiors of the $\overline{C_m}$'s cover $A_k$  (under the restrictions of the quotient map to $U_{i_m}$); again, we are using that $A_k$ is compact and Lemma \ref{lem:contain_closure_Borel} to find $s_k$. Let $Z_i$, for $\sum_{\ell<k}|s_\ell|\leq i<\sum_{\ell\leq k}|s_\ell|$, be those $B_m\subseteq U_{i_m}\cap Y^k_{i_m}$ for $m\in s_k$. For $i\neq j$, if $Z_i=B_m\subseteq U_{i_m}\cap Y^k_{i_m}$ and $Z_j=B_n\subseteq U_{i_n}\cap Y^\ell_{i_n}$, then let $Z_{i,j}=B_m\cap B_n\subseteq U_{i_m,i_n}\cap Y^k_{i_m}$ and $\psi_{i,j}=\varphi_{i_m,i_n}|_{Z_{i,j}}$. Note that $Z_{i,j}\neq\varnothing$ only if $Y_k\cap Y_\ell=\varnothing$, and $Y_k$ intersects at most four other $Y_\ell$'s, namely $Y_{k-2}$, $Y_{k-1}$, $Y_{k+1}$, and $Y_{k+2}$, ensuring local finiteness. Let $d=\langle \psi_{i,j}\mid i,j\in\mathbb{N}\rangle$.

    Finally, each $A_k\subseteq Y_k$ and the $A_k$'s cover $M$, so the $Z_i$'s, and in fact, the $C_m$'s as chosen above, must cover $M$ as $i$ ranges over $\mathbb{N}$. In particular, $(\mathcal{Z},d)$ is indeed a refinement of $(\mathcal{U},c)$.
\end{proof}

\subsubsection{Connectedness}

In what follows it will be convenient, and to some degree necessary, to confine our attention
to model spaces $X$ which are not only locally compact, but locally path connected, simply connected, or geodesically convex as well; model space properties along these lines must, after all, make an eventual appearance in any theory of classical manifolds \emph{per se}. 
In the present context, these are mild constraints: any smooth Riemannian manifold, and in fact any second countable locally compact metric space of bounded curvature (i.e., having a complete CAT($\kappa$) metric), and, in particular, any $\mathbb{S}^n$, $\mathbb{R}^n$, or $\mathbb{H}^n$ $(n\geq 2)$ will satisfy all of them.
Here we follow the usual convention of saying for any property $P$ that a space $X$ is \emph{locally $P$} if every point in $X$ possess a basis of neighborhoods satisfying $P$.

\begin{definition}
Let $(X,d)$ be a metric space. A \emph{geodesic} from $x$ to $y$ in $X$ is a continuous map $\gamma:[0,\ell]\to X$ such that $\gamma(0)=x$, $\gamma(\ell)=y$, and $d(\gamma(s),\gamma(t))=t-s$ for all $s\leq t$ in $[0,\ell]$; we will at times conflate a geodesic with its image. 
$(X,d)$ is a \emph{geodesic space} if any two points in $X$ are joined by a geodesic, and a subset $Y$ of $X$ is \emph{geodesically convex} if any two points in $Y$ are joined by a geodesic whose image falls in $Y$.
\end{definition}

\begin{lemma}
\label{lem:reparametrization}
For any Borel pseudogroup $\mathcal{G}$ on a locally simply connected, locally compact Polish space $X$, there exists a Borel reparametrization $r:\mathfrak{M}(\mathcal{G},X)\to \mathfrak{M}(\mathcal{G},X)$ such that for all $(\mathcal{U},c)$ in $\mathfrak{M}(\mathcal{G},X)$, 
\begin{itemize}
	\item each $V_i$ $(i\in \mathbb{N})$ is simply connected, where $r(\mathcal{U},c)=(\mathcal{V},d)$.
\end{itemize}

If in addition $X$ is locally geodesically convex and $\mathcal{G}$ is a subpseudogroup of $\mathsf{Isom}$, then we may further stipulate that each $V_{i,j}$ is simply connected, and in fact geodesically convex, for all $(i,j)\in\mathbb{N}^2$.
\end{lemma}

\begin{proof}
Fix a countable basis $\mathcal{W}=(W_i)_{i\in\mathbb{N}}$ of open, simply connected subsets of $X$, as well as a canonical enumeration $e$ of $\mathbb{N}^2$.
For any $(\mathcal{U},c)$ in $\mathfrak{M}(\mathcal{G},X)$,
this basis induces maximal $\mathcal{W}$-decompositions $U_{i,i}=\bigcup_{j\in J_i} W_j$ for each $i\in\mathbb{N}$; add a superscript $i$ to those $W_j$ which appear in this way, for bookkeeping.
By way of $e$, we then have a canonical enumeration $(V_k)_{k\in\mathbb{N}}$ of
$\{W^i_j\mid i\in\mathbb{N}, j\in J_i\}$
determining the ``left side'' coordinates of $r(\mathcal{U},c)$ as follows: let $V_{k,k}=V_k$, and for $k,\ell$ enumerating $W^i_j$ and $W^{i'}_{j'}$, respectively, let $V_{k,\ell}=V_k\cap\varphi^{-1}_{i,i'}[U_{i',i}\cap V_{\ell}]$.
One then lets $\psi_{k,\ell}:V_{k,\ell}\to V_{\ell,k}$ equal $\varphi_{i,i'}\big|_{V_{k,\ell}}$.
This determines an $r(\mathcal{U},c)=(\mathcal{V},d)$, where $d=\langle \psi_{k,\ell}:k,\ell\in\mathbb{N}\rangle$,
whose associated $M_{(\mathcal{V},d)}$ is naturally isomorphic to $M_{(\mathcal{U},c)}$ and may, accordingly, be viewed as a reparametrization of the latter.
For the lemma's first assertion, only the following remains to be checked.
Like its counterpart in Lemma \ref{lem:paracompact}, the claim follows from general considerations, but for thoroughness we record a fuller verification.
\begin{claim} 
The map $r:\mathfrak{M}(\mathcal{G},X)\to\mathfrak{M}(\mathcal{G},X):(\mathcal{U},c)\mapsto (\mathcal{V},d)$ defined above is Borel.
\end{claim}
\begin{proof}[Proof of Claim] 
Regard $r$ as a function $\mathfrak{M}(\mathcal{G},X)\to (\mathcal{O}(X)\times\mathcal{G})^{\mathbb{N}\times\mathbb{N}}$, and for any $B\subseteq \mathcal{O}(X)\times\mathcal{G}$, write $[B(k,\ell)]$ for the subset of $(\mathcal{O}(X)\times\mathcal{G})^{\mathbb{N}\times\mathbb{N}}$ whose elements fall within $B$ on their $(k,\ell)^{\mathrm{th}}$ coordinate.
We will argue, for arbitrary $(k,\ell)\in\mathbb{N}^2$, that $r^{-1}([B(k,\ell)])$ is a Borel subset of $\mathfrak{M}(\mathcal{G},X)$ for $B$ of the form:
\begin{enumerate}[label=\textup{\arabic*.}]
\item $\{U\in\mathcal{O}(X)\mid K\subseteq U\}\times\mathcal{G}$ for some compact $K\subseteq X$;
\item $\mathcal{O}(X)\times E$ for some Borel $E\subseteq\mathcal{G}$.
\end{enumerate}
Since such sets (as $(k,\ell)$ ranges over $\mathbb{N}^2$) generate the Borel $\sigma$-algebra of $(\mathcal{O}(X)\times\mathcal{G})^{\mathbb{N}\times\mathbb{N}}$ and $\mathfrak{M}(\mathcal{G},X)$ is a Borel subset of the latter, this will complete the proof of the claim and, thereby, the first portion of the lemma.

For the first item, fix a compact $K\subseteq X$; in the foregoing notation, $r^{-1}([B(k,\ell)])$ equals the set of $(\mathcal{U},c)\in\mathfrak{M}(\mathcal{G},X)$ for which $V_{k,\ell}=V_k\cap\varphi^{-1}_{i,i'}[U_{i',i}\cap V_{\ell}]\supseteq K$, where $V_k$ and $V_\ell$ correspond to some $W^i_j$ and $W^{i'}_{j'}$ in $\mathcal{W}$, respectively.
Depending as they do on a finite sequence of inclusions and non-inclusions of various $W_j$ into various $U_i$, these correspondences are clearly Borel, and we may therefore henceforth assume them fixed.
The task then reduces to showing that the sets $\{(\mathcal{U},c)\mid\varphi^{-1}_{i,i'}[U_{i',i}]\supseteq K\}$ and $\{(\mathcal{U},c)\mid\varphi^{-1}_{i,i'}[V_\ell]\supseteq K\}$ are each Borel in $\mathfrak{M}(\mathcal{G},X)$.
Each of these claims, though, is essentially immediate from item (v) of Definition \ref{defn:Borel_pseudogroup}.

For the second item, the question is, in the foregoing notation, whether for some fixed Borel $E\subseteq\mathcal{G}$, the set $r^{-1}([B(k,\ell)])=\{(\mathcal{U},c)\mid\varphi_{i,i'}\big|_{V_{k,\ell}}\in E\}$ is Borel; again it suffices to consider the case in which $i$ and $i'$ are fixed.
The restriction in question is written more fully as $$\varphi_{i,i'}\big|_{V_k\cap\varphi^{-1}_{i,i'}[U_{i',i}\cap V_{\ell}]}$$
whereupon it is clear from Definition \ref{defn:Borel_pseudogroup}(vi) (restriction is Borel) and (v) (inverse image is Borel) and the fact that the intersection operation is Borel that $r^{-1}(B[(k,\ell)])$ is indeed Borel. This completes the proof of the claim.
\end{proof}

The point in the lemma's concluding claim is that, in the context of a geodesic-preserving $\mathcal{G}\subseteq\mathsf{Isom}$, one may with a little more care ensure that each $V_{k,\ell}$ of the foregoing construction is geodesically convex, and simply connected in particular.
One works in this case from a basis $\mathcal{W}$ of open, geodesically convex balls, and the care consists in more judiciously ``covering'' $(\mathcal{U},c)$ by $\{W^i_j\mid i\in\mathbb{N}, j\in J_i\}$ so that whenever $k$ and $\ell$ enumerate $W_j^i$ and $W_{j'}^{i'}$, respectively, either $V_k\subseteq U_{i,i'}$ or $V_{\ell}\subseteq U_{i',i}$ or $V_{k,\ell}=\varnothing$.
Much as in the arguments of Section \ref{subsection:pseudogroups}, a dense subset $(q_i)_{i\in\mathbb{N}}$ of $X$ guides the construction: at stage $n$, if $e(n)=(a,b)$ then $i(n)$ indexes the $b^{\mathrm{th}}$ element of $(U_i)_{i\in\mathbb{N}}$ which contains $q_a$.
One then adds to $J_{i(n)}$ the least index of a $W(n)\in\mathcal{W}$ containing $q_a$ such that
\begin{enumerate}
\item $\overline{W(n)}\subseteq U_{i(n)}$
\end{enumerate}
and for each $m<n$ either
\begin{enumerate}
\setcounter{enumi}{1}
\item $W(n)\subseteq U_{i(n),i(m)}$ or
\item $W(n)\cap\varphi_{i(m),i(n)}[W(m)]=\varnothing$.
\end{enumerate}
The point is that condition (1) at previous stages $m$ ensures that one of the associated conditions (2) or (3) is satisfiable at stage $n$.
The fact that any finite combination of satisfiable requirements is satisfiable then ensures that the construction can always continue; the fact that only finitely many ``bits'' of data are needed to meet them ensures that the construction is Borel (this time we leave the details to the reader).
Our bookkeeping ensures that $\bigcup_{j\in J_i}W^i_j = U_i$ for all $i$.
The reward for our efforts is that for all $k$ and $\ell$, either
\begin{itemize}
\item $V_{k,\ell}=\varnothing$, or
\item $V_{k,\ell}=V_k\cap\varphi^{-1}_{i,i'}[U_{i',i}\cap V_{\ell}]=V_k\cap\varphi^{-1}_{i,i'}[V_{\ell}]$, or
\item $V_{\ell,k}=V_\ell\cap\varphi^{-1}_{i',i}[U_{i,i'}\cap V_k]=V_\ell\cap\varphi^{-1}_{i',i}[V_k]$.
\end{itemize}
In other words, if nonempty, then either $V_{k,\ell}$ or $V_{\ell,k}$ is an intersection of geodesically convex sets, and thus convex, and thus its counterpart is as well.
\end{proof}

Given a Borel pseudogroup $\mathcal{G}$ on $X$, we denote by $\mathfrak{C}(\mathcal{G},X)$ the collection of all $(\mathcal{U},c)\in\mathfrak{M}(\mathcal{G},X)$ such that $M_{(\mathcal{U},c)}$ is connected.
The next lemma then figures essentially as a corollary of the last one.

\begin{lemma}\label{lem:connected_Borel}
For any Borel pseudogroup $\mathcal{G}$ on a locally path connected, locally compact Polish space $X$, the parameter space $\mathfrak{C}(\mathcal{G},X)\subseteq \mathfrak{M}(\mathcal{G},X)$ of all connected $(\mathcal{G},X)$-manifolds is a standard Borel space.
\end{lemma}
\begin{proof}
Define a reparametrization map $r$ just as in Lemma \ref{lem:reparametrization}, but with respect to a basis of path-connected open sets; by the same argument as before, $r$ is Borel. A $(\mathcal{U},c)$ in the image of the map $r:\mathfrak{M}(\mathcal{G},X)\to \mathfrak{M}(\mathcal{G},X)$ parametrizes a connected manifold if and only if for any $i,j\in\mathbb{N}$ there exists a finite sequence $i=i_0,\dots,i_\ell,\dots,i_n=j$ such that $U_{i_\ell,i_{\ell+1}}\neq\varnothing$ for all $\ell<n$  (this observation plays an identical role in \cite[p.\ 132]{MR1731384}). This condition is evidently Borel, hence $\mathfrak{C}(\mathcal{G},X)\subseteq\mathfrak{M}(\mathcal{G},X)$ is the preimage of a Borel set by the Borel function $r$.
\end{proof}

\begin{lemma}
\label{lem:Borel_computation_of_components}
Let $X$ and $\mathcal{G}$ be as in Lemma \ref{lem:connected_Borel}. There is a Borel function which computes the connected components of any $(\mathcal{G},X)$-manifold.
\end{lemma}

\begin{proof}
Fix a dense subset $Q=\{q_i\mid i\in\mathbb{N}\}$ of $X$ and a standard enumeration $e$ of $\mathbb{N}^3$ and a basis $\mathcal{W}=\{W_i\mid i\in\mathbb{N}\}$, consisting of path-connected open sets, for the topology of $X$.
For each $i,j\in\mathbb{N}$, we describe a Borel function $f_{i,j}:\mathfrak{M}(\mathcal{G},X)\to\mathfrak{M}(\mathcal{G},X)$ recording the component, as an open submanifold, of $M_{(\mathcal{U},c)}$ associated to the image of $q_i$ in the chart $U_j$.
Taken together, these determine an $f:\mathfrak{M}(\mathcal{G},X)\to\mathfrak{M}(\mathcal{G},X)^{\mathbb{N}\times\mathbb{N}}$ which is as desired.

$f_{i,j}(\mathcal{U},c)$ is defined as follows: if $q_i\not\in U_j$, then $f_{i,j}(\mathcal{U},c)$ is the element of $\mathfrak{M}(\mathcal{G},X)$ consisting of empty charts.
If on the other hand $q_i\in U_j$, then fix $W_{k_0}$ least in the $\mathcal{W}$-enumeration with $q_i\in W_{k_0}\subseteq U_j$ and let $i_0=i$ and $j_0=j$.
Select then the $e$-least yet-unselected $(i',j',k')$ with
\begin{enumerate}
\item $q_{i'}\in W_{k'}\subseteq U_{j'}$, and
\item $\varphi_{j_s,j'}[W_{k_s}]\cap W_{k'}\neq\varnothing$ for some $s<t$,
\end{enumerate}
if such exist (if they do not, the process terminates). Let $i_t=i'$, $j_t=j'$, and $k_t=k'$, and repeat. For each $m\in\mathbb{N}$, let $V_m$ by the $(\mathcal{U},c)$-saturation (in the sense of Lemma \ref{lem:saturation}) of $\cup\{W_{k_s}\mid j_s=m\}$. Put $f_{i,j}(\mathcal{U},c)=(\mathcal{V},c|_{\mathcal{V}})$, the open submanifold corresponding to the $V_m$'s; this mapping is evidently Borel. It remains to show that $N=M_{(\mathcal{V},c|_{\mathcal{V}})}$ is a connected component of $M_{(\mathcal{U},c)}$.

Since $M_{(\mathcal{U},c)}$ is locally path-connected, it will suffice to show that $N$ contains the path-component of any of its elements; therefore let $x$ denote the image of $q_{i_0}$ in $N$, and let $\gamma$ denote a path in $M_{(\mathcal{U},c)}$ from $x$ to any other point $y$. We wish to show that $y\in N$ as well --- but this is almost immediate: recalling that $\pi_{(\mathcal{U},c)}$ is the quotient map onto $M_{(\mathcal{U},c)}$ as in (\ref{eq:quotient_of_sum}), $\gamma$ will be covered by the $\pi_{(\mathcal{U},c)}$-images of finitely many $W\in\mathcal{W}$ which are each contained in some $U_j\in\mathcal{U}$, and each of these $W$ will arise in the formation of $N$ described above, by the connectedness of $\gamma$.
\end{proof}

\subsubsection{Metrics and completeness}

Lastly, we turn our attention to metric structures.
Here we append to our standing assumptions that all model spaces $X$ are all equipped with a compatible complete metric $d_X$; any $(\mathsf{Isom},X)$-manifold $M_{(\mathcal{U},c)}$ is then naturally endowed with the \emph{quotient pseudometric} $d_{(\mathcal{U},c)}$ defined as follows. For any $U_i\in\mathcal{U}$ and $x\in U_i$, write $\lfloor x,i\rfloor$ for the $\pi_{(\mathcal{U},c)}$-image in $M_{(\mathcal{U},c)}$ of $x$, and let
\begin{equation}
\label{eq:pseudometric}
d_{(\mathcal{U},c)}(\lfloor x,i\rfloor,\lfloor y,j\rfloor)=\mathrm{inf}\{d_X(x_0,y_0)+\cdots+d_X(x_m,y_m)\mid P((x_\ell),(y_\ell))\},
\end{equation}
where $P(( x_\ell),( y_\ell))$ is the conjunction of the following properties:
\begin{enumerate}[label=\textup{\roman*.}]
\item the sequences $( x_\ell)$ and $(y_\ell)$ are of the same finite length $m+1$;
\item $x_\ell$ and $y_\ell$ both fall in some common $U_{k(\ell)}$ for each $\ell\leq m$;
\item $\lfloor y_\ell,k(\ell)\rfloor=\lfloor x_{\ell+1},k(\ell+1)\rfloor$ for all $\ell<m$; and
\item $\lfloor x_0,k(0)\rfloor=\lfloor x,i\rfloor$ and $\lfloor y_m,k(m)\rfloor=\lfloor y,j\rfloor$.
\end{enumerate}
Here we should beware of subtleties: we cannot even take for granted that $d_{(\mathcal{U},c)}$ defines a metric, for example, nor, if it does, that it defines one agreeing with the charts' metrics.
Under suitable assumptions, on the other hand, $\pi_{(\mathcal{U},c)}$ is a local isometry, as the following lemma will show. 

\begin{lemma}
\label{lem:metric_reparametrization}
Let $X$ be locally geodesically convex and suppose for some $(\mathcal{U},c)\in\mathfrak{M}(\mathsf{Isom},X)$ that
\begin{equation*}
\label{eq:condition}
\tag{*}
\textnormal{ each }U_i\textnormal{ in }\mathcal{U}\textnormal{ is geodesically convex.}
\end{equation*}
Then
\begin{enumerate}[label=\textup{\roman*.}]
\item $d_{(\mathcal{U},c)}$ is a metric, and
\item for each $i\in\mathbb{N}$ and $x\in U_i$ there exists a neighborhood $B_{\varepsilon}(x)\subseteq U_i$ such that the natural map $(B_{\varepsilon}(x),d_X)\to(\pi_{(\mathcal{U},c)}[B_{\varepsilon}(x)],d_{(\mathcal{U},c)})$ is an isometry.
\end{enumerate}
Moreover, for any $(\mathcal{U},c)$ and $(\mathcal{U}',c')$ satisfying (\ref{eq:condition}), any $(\mathsf{Isom},X)$-equivalence of the associated manifolds is an isometry of quotient metrics.
\end{lemma}
\begin{proof}
For the first item, it suffices to show that $d_{(\mathcal{U},c)}(\lfloor x,i\rfloor,\lfloor y,j\rfloor)\neq 0$ whenever $\lfloor x,i\rfloor\neq\lfloor y,j\rfloor$.
To this end, fix open neighborhoods $B_{\varepsilon}(x)$ and $B_{\varepsilon}(y)$ of such $x$ and $y$ whose closures are contained in $U_i$ and $U_j$, respectively, and whose $\pi_{(\mathcal{U},c)}$-images are disjoint.
Next, observe that for any sequences $(x_\ell)$, $(y_\ell)$ defining $d_{M(\mathcal{U},c)}(\lfloor x,i\rfloor,\lfloor y,i\rfloor)$ as in equation (\ref{eq:pseudometric}), there exists a least $n\leq m$ such that $y_n$ falls outside the $(\mathcal{U},c)$-saturation of $B_{\varepsilon}(x)$. Note that we may then take a $z_n\in U_i$ on the boundary of $B_{\varepsilon}(x)$ so that
$d_X(x_n,y_n)=d_X(\varphi_{k(n),i}(x_n),z_n)+d_X(\varphi_{i,k(n)}(z_n),y_n)$.
Substituting the right-hand side in for the left one
then gives us $$d_X(x_0,y_0)+\cdots+d_X(x_m,y_m)\geq\sum_{\ell<n} d_X(x_\ell,y_\ell)+d_X(\varphi_{k(n),i}(x_n),z_n)>\varepsilon,$$
and this establishes the claim.

For any $x$ and $i$ as in the second item, let $\varepsilon=d_X(x,X\backslash U_i)/2$. The reasoning is much as before, so we record it less formally:
note first that for any $y,z\in B_{\varepsilon}(x)$, the distance $d_{(\mathcal{U},c)}(\lfloor y,i\rfloor,\lfloor z,i\rfloor)$ can only be less than $d_X(y,z)< 2\varepsilon$ by way of some sequence of paths in $\coprod_j U_j$ involving a point outside the $(\mathcal{U},c)$-saturation of $U_i$.
Necessarily, though, $d_X(y,X\backslash U_i)>\varepsilon$ and $d_X(z,X\backslash U_i)>\varepsilon$, and this furnishes the desired contradiction.

Our concluding claim then follows from item (ii), the fact that $d_{(\mathcal{U},c)}$ is a length metric, and \cite[Cor.\ 3.1.2]{MR1835418}.
\end{proof}

By the above, whenever $X$ is locally geodesically convex --- as, for example, the standard model spaces $M^n_\kappa$ (denoting the complete simply-connected Riemannian $n$-manifold of constant sectional curvature $\kappa\in\mathbb{R}$ \cite[Ch.\ I.2]{MR1744486}) all are --- and $\mathcal{G}$ is a subpseudogroup of $\mathsf{Isom}$, it will be both natural and uniformizing to focus our attention on $(\mathcal{G},X)$-manifold parametrizations $(\mathcal{U},c)$ satisfying the condition (\ref{eq:condition}).
One way of doing so is via the reparametrizing map of Lemma \ref{lem:reparametrization}; by the reasoning above, we may even reparametrize any $(\mathcal{U},c)\in\mathfrak{M}(\mathcal{G},X)$ as an $r(\mathcal{U},c)$ whose chart and quotient metrics entirely agree.
The more direct route is simply to restrict our attention to
\begin{equation}
\label{eq:metric_manifolds}
\mathfrak{M}^*(\mathcal{G},X):=\{(\mathcal{U},c)\in\mathfrak{M}(\mathcal{G},X)\mid\textnormal{ each }U_i\textnormal{ in }\mathcal{U}\textnormal{ is geodesically convex}\},
\end{equation}
alongside, in some cases, a further restriction $\mathfrak{M}^{*,\varepsilon}(\mathcal{G},X)$ to $U_i$ of uniformly bounded diameter $\varepsilon$.
Under mild assumptions, this space is standard Borel; the condition of the following lemma, for example, is satisfied by all CAT$(\kappa)$ spaces \cite[Prop.\ II.1.4]{MR1744486}, and hence by any model geodesic space $X$ which we will have occasion to consider.
In fact, when $\kappa=0$, and in particular when $X$ is hyperbolic or Euclidean $n$-space, we may let $\varepsilon$ equal infinity, or, equivalently, ignore it altogether.
\begin{lemma}
\label{lem:convex_Borel}
Suppose that the geodesic space $X$ admits an $\varepsilon\in (0,\infty]$ such that any $x$ and $y$ in $X$ of distance less than $\varepsilon$ are connected by a unique geodesic, and that the image of this geodesic varies continuously with the choice of its endpoints in open neighborhoods about $x$ and $y$. Then for all $\delta<\varepsilon$,
$$\{U\in\mathcal{O}(X)\mid\mathrm{diam}(U)\leq\delta\textnormal{ and }\mathcal{U}\textnormal{ is geodesically convex}\}$$
is a closed subset of $\mathcal{O}(X)$.
In particular, if $\mathcal{G}$ is a Borel subpseudogroup of $\mathsf{Isom}$, then $\mathfrak{M}^{*,\varepsilon}(\mathcal{G},X)$ is a Borel subspace of $\mathfrak{M}(\mathcal{G},X)$ intersecting each of its $(\mathcal{G},X)$-equivalence classes, and for which any $(\mathcal{G},X)$-equivalence of parameters is an isometry of their associated manifolds.
\end{lemma}
\begin{proof}
Suppose a $W\in\mathcal{O}(X)$ of diameter at most $\delta$ is not geodesically convex, so that there exists an $x,y\in W$ and $z\in\gamma(x,y)\backslash W$. Fix compact neighborhoods $K_x,K_y\subseteq W$ of $x,y$, respectively (our standing assumption that $X$ is locally compact remains in force), and an open $V\subseteq \bigcup_{(x',y')\in K_x\times K_y}\gamma(x',y')$ containing $z$.
No diameter-$\leq\delta$ element of the open neighborhood
$$\{U\in\mathcal{O}(X)\mid U\supseteq K_x\cup K_y\}\cap\{U\in\mathcal{O}(X)\mid U\not\supseteq V\}$$
of $W$ is convex, and this completes the proof of our first assertion. Our second assertion is then immediate from the preceding lemma and the following remark.
\end{proof}
\begin{remark}
Speaking generally, when constraints on parameters $(\mathcal{U},c)$ are desired, one may impose them in either of two obvious ways: via reparametrization, as in Lemmas \ref{lem:paracompact} and \ref{lem:reparametrization}, or via restriction, as in $\mathfrak{M}^*(\mathcal{G},X)$ above and the Lemma \ref{lem:convex_Borel} underpinning it. While the latter is, arguably, more conceptually natural, the former may be the only Borel approach available. A case in point is the condition \emph{simply connected}: as noted in Section \ref{subsection:equivalences}, by \cite{MR1233807} this condition is, in general, very far from Borel, and it was this fact which motivated the approach of Lemma \ref{lem:reparametrization}.
Note also that it is, in essence, the construction of Lemma \ref{lem:reparametrization} which ensures for us that each $(\mathcal{G},X)$-equivalence class of $\mathfrak{M}(\mathcal{G},X)$ intersects the latter's restriction to $\mathfrak{M}^{*,\varepsilon}(\mathcal{G},X)$.
\end{remark}
Of importance in Section \ref{section:bireducibility} and beyond will be classes of \emph{complete} hyperbolic manifolds. The relevant lemma is the following.
\begin{lemma}
\label{lem:complete}
Let $X$ and $\varepsilon$ be as in Lemma \ref{lem:convex_Borel}, with $\mathcal{G}$ a Borel subpseudogroup of $\mathsf{Isom}$. The set
$$\{(\mathcal{U},c)\in\mathfrak{M}^{*,\varepsilon}(\mathcal{G},X)\mid d_{(\mathcal{U},c)}\textnormal{ is complete}\}$$
is then a Borel subset of $\mathfrak{M}^{*,\varepsilon}(\mathcal{G},X)$.
\end{lemma}

\begin{proof}
The contours of our approach are by now routine: we will show that the class of (parameters for) complete $(\mathcal{G},X)$-manifolds is Borel by showing that it is both coanalytic and analytic.
Most intuitive of these two characterizations is the first: $M_{(\mathcal{U},c)}$ is complete if and only if for every $d_{(\mathcal{U},c)}$-Cauchy sequence $\vec{x}=(x_i)_{i\in\mathbb{N}}$ in $M_{(\mathcal{U},c)}$, there exists an $x\in M_{(\mathcal{U},c)}$ such that $\vec{x}$ converges to $x$.
Since the convergence condition may be replaced by the existence of an $n\in\mathbb{N}$ and precompact $V\subseteq M_{(\mathcal{U},c)}$ such that $x_i\in\overline{V}$ for all $i>n$, this statement is $\mathbf{\Pi}^1_1$ in form.
Somewhat more formally, fix a countable basis $\mathcal{B}$ for $X$ consisting of precompact open sets, as well as a countable dense $Q\subseteq X$; the complete manifolds then correspond to the complement of the projection onto the first coordinate of the collection $$\{((\mathcal{U},c),\vec{x},k)\in\mathfrak{M}^{*,\varepsilon}(\mathcal{G},X)\times X^\mathbb{N}\times\mathbb{N}^\mathbb{N}\mid\vec{x}\text{ is Cauchy, but non-convergent, in }M_{(\mathcal{U},c)}\},$$
where the text unpacks as follows:
\begin{enumerate}[label=\textup{\roman*.}]
\item for all $i\in\mathbb{N}$, $x_i\in U_{k(i)}$ for some $k(i)\in\mathbb{N}$;
\item for all $\delta\in\mathbb{Q}^+$ there exists an $n(\delta)\in\mathbb{N}$ such that $$d_{(\mathcal{U},c)}(\lfloor x_i,k(i)\rfloor,\lfloor x_j,k(j)\rfloor)<\delta$$
for all $i,j\geq n(\delta)$;
\item there does not exist an $m\in\mathbb{N}$ and $V\in\mathcal{B}$ with $\overline{V}\subseteq U_m$ and an $n\in\mathbb{N}$ such that $\varphi_{k(i),m}(x_i)\in\overline{V}$ for all $i>n$.
\end{enumerate}
Here only possibly the condition $d_{(\mathcal{U},c)}(\lfloor x_i,k(i)\rfloor,\lfloor x_j,k(j)\rfloor)<\delta$ isn't evidently Borel.
Note, though, that it holds if and only if this inequality holds for a sum as in equation (\ref{eq:pseudometric}) whose linking elements are all drawn from $Q$. Since the quantification is over the countable collection of finite sequences in $Q$, our claim follows.

For a $\mathbf{\Sigma}^1_1$ characterization, multiple approaches, drawing on the multiple equivalent characterizations of completeness in this context (see \cite[Thm.\ 2.5.28]{MR1835418} or \cite[Prop.\ 3.4.15]{MR1435975}), are possible; the idea of ours is the following: $M_{(\mathcal{U},c)}$ is complete if there exists a sequence of compact $L_j\subseteq M_{(\mathcal{U},c)}$ such that $\bigcup_{j\in\mathbb{N}}L_j=M_{(\mathcal{U},c)}$ and $d_{(\mathcal{U},c)}(x,y)\geq 1$ whenever $x\in L_j$ and $y\in M_{(\mathcal{U},c)}\backslash L_{j+1}$ for some $j\in\mathbb{N}$.
Since we may fix a countable sequence $(K_i)_{i\in\mathbb{N}}$ of compact sets which are closures of basic open sets in $X$ and (much as for Lemma \ref{lem:exhaustion}), whenever sequences as described exist, build up such $L_j$ from finite unions of $\pi_{(\mathcal{U},c)}[K_i]$, this is a $\mathbf{\Sigma}^1_1$ condition:
it essentially asks if there exist functions $\mathbb{N}\to\mathbb{N}$ (selecting the $K_i$) and $g:\mathbb{N}\to\mathbb{N}$ (the ``cutoff'' function, so that $L_j=\bigcup_{i\leq g(j)}\pi_{(\mathcal{U},c)}[K_{f(i)}]$) satisfying the conditions described above.
To see that the distance condition on the $L_j$ is Borel, note that it may be checked on a countable dense set, and since the other conditions are even more obviously Borel, this concludes the proof.
\end{proof}

\section{The homeomorphism and diffeomorphism problems for surfaces}
\label{section:surfaces}

In this section, we will determine the exact Borel complexity of the classification problem for \emph{surfaces}, i.e., connected topological $2$-manifolds without boundary, up to homeomorphism. Recall that $\mathfrak{C}(\mathsf{Top},\mathbb{R}^2)$ denotes the Borel parameter space of all surfaces and $\cong_{\mathsf{Top}}$ the homeomorphism relation on it. We will show:
 
\begin{theorem}\label{thm:surfaces}
	The homeomorphism relation $\cong_{\mathsf{Top}}$ on $\mathfrak{C}(\mathsf{Top},\mathbb{R}^2)$ is complete for countable structures. 
\end{theorem}

\noindent As a consequence, we will derive (as Corollary \ref{cor:smooth_surfaces}) that the diffeomorphism relation $\cong_{\mathsf{C}^\infty}$ on the space of all smooth surfaces $\mathfrak{C}(\mathsf{C}^\infty,\mathbb{R}^2)$ also has the same level of Borel complexity. Taken together, these two results prove Theorem E from our introduction.
 
Proving Theorem \ref{thm:surfaces} entails showing both an upper bound --- \emph{there is some class of countable first-order structures for which the homeomorphism relation on surfaces reduces to the associated isomorphism relation} --- and a lower bound --- \emph{any such isomorphism relation reduces to homeomorphism of surfaces}. Consequently, this classification problem has precisely the same degree of complexity as that of the isomorphism relation on all countable groups, graphs, or linear orders (see \cite{MR1011177}). 
Recall from Corollary \ref{cor:compact_manifolds_countable} that the classification of compact surfaces reduces to $=_{\mathbb{N}}$, so any additional complexity arising in the broader collection of all surfaces must lie among the noncompact ones. 

Our upper bound for the complexity of this classification problem will be based on the following theorem, which says that surfaces are classified by their genus, orientability, and \emph{spaces of ends}, generalizing the well-known classification for compact surfaces. This result traces to Ker\'{e}kj\'{a}rt\'{o} \cite{zbMATH02598767}, but Richards \cite{MR143186} is generally credited with its first complete modern proof.

\begin{theorem}[Ker\'{e}kj\'{a}rt\'{o}--Richards]\label{thm:Richards}
	Surfaces $X$ and $Y$ are homeomorphic if and only if they have the same genus and orientability class, and the triples of spaces $(e(X),e'(X),e''(X))$ and $(e(Y),e'(Y),e''(Y))$ are equivalent.
\end{theorem}

The triples $(e(X),e'(X),e''(X))$ consist of three nested spaces, the \emph{space of ends} $e(X)$ of $X$, and the subspaces consisting of the \emph{nonplanar} and \emph{nonorientable ends}, $e'(X)$ and  $e''(X)$, respectively. Each is a compact totally disconnected Polish space. Two such triples are \emph{equivalent} if there is a homeomorphism between the first coordinate spaces which restricts to a homeomorphism on each of the two subspaces. 

Much like all proofs of the classification theorem for compact surfaces (\cite{MR0488059}, \cite{MR1699257}, \cite{MR2766102}, \cite{MR3026641}, among others), the proof of Theorem \ref{thm:Richards} begins by assuming that the surfaces in question have been \emph{triangulated}, that is, are homeomorphic to polyhedra constructed from simplicial complexes, itself a nontrivial fact \cite{zbMATH02591746}. Proofs of the triangulation theorem, in turn, often rely on the Jordan--Schoenflies theorem \cite{MR1511377}. To give an account of this classification result in our framework, we must therefore show that the outputs of both the Jordan--Schoenflies and triangulation theorems can be computed in a Borel way. We will do so by adapting the proofs of these theorems given by Thomassen in \cite{MR1144352} (see also \cite{MR1844449} and \cite{MR3026641}) to the Borel setting.

The observation that surfaces admit classification by countable algebraic invariants is not itself new. For example, Goldman \cite{MR275436} gave such a classification using cohomological invariants based directly on Richards's proof of Theorem \ref{thm:Richards} (of course, Richards's result may itself be understood in these terms, by Stone duality; cf.\ Lemma \ref{lem:Stone_3} below). What is novel in our account here is that we are casting such a classification in the terminology of invariant descriptive set theory and verifying its Borel content.

The lower bound for Theorem \ref{thm:surfaces} was previously observed by Kulikov in \cite{Kulikov_homeos}, drawing on Clinton Conley's answer \cite{64608} to a question on MathOverflow. It is based on the fact, due to Carmelo and Gao \cite{MR1804507}, that the homeomorphism relation $\cong$ on the space $\mathcal{K}(\{0,1\}^\mathbb{N})$ of closed subsets of the Cantor space $\{0,1\}^\mathbb{N}$ is complete for countable structures. For reference, we sketch the argument here.

\begin{theorem}\label{thm:clinton}
	The relation $\cong$ on $\mathcal{K}(\{0,1\}^\mathbb{N})$ Borel reduces to $\cong_{\mathsf{Top}}$ on $\mathfrak{C}(\mathsf{Top},\mathbb{R}^2)$, and to $\cong_{\mathsf{C}^\infty}$ on $\mathfrak{C}(\mathsf{C}^\infty,\mathbb{R}^2)$.\end{theorem}

\begin{proof}
	Let $C$ be the usual middle-thirds Cantor set lying on the real line, viewed as a subset of the sphere $S=\mathbb{R}^2\cup\{\infty\}$. We identify $\mathcal{K}(\{0,1\}^\mathbb{N})$ with $\mathcal{K}(C)$. Consider the map $f:\mathcal{K}(C)\to\mathcal{O}(S)$ given by complementation: $f(K)=S\setminus K$. This map is easily seen to be Borel. Any set $S\setminus K$ in the image of this map is a planar surface whose space of ends (see below) can be identified homeomorphically with $K$. That is, $S\setminus K$ corresponds to the triple $(K,\varnothing,\varnothing)$ in the notation of Theorem \ref{thm:Richards}, and so we have that $K\cong K'$ if and only $f(K)\cong f(K')$, for any $K,K'\in\mathcal{K}(C)$. We leave it to the reader to convince themselves, using the definitions in Section \ref{section:parametrization}, that $\mathcal{O}(S)$ can be identified in a Borel way with a subset of $\mathfrak{C}(\mathsf{Top},\mathbb{R}^2)$, through which $f$ becomes a Borel reduction. 
The smooth case follows using the same reduction, together with the fact that surfaces possess unique smooth structures up to diffeomorphism as cited in Example \ref{ex:pseudogroup_comparison}.
\end{proof}

Lastly, we note the following corollary of Theorem \ref{thm:clinton} and the fact any equivalence relation which is complete for countable structures is $\mathbf{\Sigma}^1_1$-complete\footnote{An analytic subset $C$ of a standard Borel $Y$ is \emph{$\mathbf{\Sigma}^1_1$-complete} if for any analytic subset $A$ of a standard Borel $X$, there is a Borel $f:X\to Y$ such that $x\in A$ if and only if $f(x)\in C$. We do \emph{not} mean that any analytic equivalence relation Borel reduces to it (i.e., is a \emph{complete analytic equivalence relation}); that this is not the case for $\cong_{\mathsf{Top}}$ and $\cong_{\mathsf{C}^\infty}$ follows from our Theorem \ref{thm:surfaces}.} \cite{MR1011177}:

\begin{corollary}
	The homeomorphism and diffeomorphism relations on surfaces are $\mathbf{\Sigma}^1_1$-complete. In particular, they are not Borel equivalence relations.\qed
\end{corollary}

\subsection{A Borel form of the Jordan--Schoenflies Theorem}\label{sec:JS}

We begin by stating the Jordan--Schoenflies theorem in the form most convenient for our purposes.

\begin{theorem}[Jordan--Schoenflies]\label{thm:JS}
	If $\gamma:S^1\to\mathbb{R}^2$ is a continuous embedding, then there is a homeomorphism $h:\mathbb{R}^2\to\mathbb{R}^2$ such that $h|_{S^1}=\gamma$.
\end{theorem}

This result is a strengthening of the Jordan curve theorem, which says that if $C$ is any simple closed curve in $\mathbb{R}^2$, then $\mathbb{R}^2\setminus C$ has exactly two connected components; one bounded, the \emph{interior} of $C$, and one unbounded, the \emph{exterior} of $C$.

To give a Borel version of Theorem \ref{thm:JS}, we will first need to describe the relevant Polish spaces of maps and some aspects of their Borel structure. In what follows, suppose that $X$ and $Y$ are locally compact Polish spaces. Let $C(X,Y)$ be the space of all continuous functions $X\to Y$ endowed with the compact-open topology, which is generated by subbasic open sets
\[
    V(K,U)=\{f\in C(X,Y)\mid f[K]\subseteq U\},
\]
where $K$ and $U$ range over the compact and open subsets of $X$ and $Y$, respectively. The following facts are standard (see \cite{MR1039321}):
\begin{itemize}
    \item $C(X,Y)$ forms a Polish space.
    \item If $Z$ is a locally compact Polish space, then the composition map $C(Y,Z)\times C(X,Y)\to C(X,Z):(f,g)\mapsto f\circ g$ is continuous.
    \item The evaluation map $e:C(X,Y)\times X\to Y:(f,x)\mapsto f(x)$ is continuous.
\end{itemize}
We will denote by $\mathrm{Emb}(X,Y)$ the subset of $C(X,Y)$ consisting of all injective continuous functions $X\to Y$. In our cases of interest, $X$ will be compact, so these are exactly the topological embeddings of $X$ into $Y$.

\begin{lemma}\label{lem:compactly_many_sections_open}
    If $X$ is compact, $Z$ any metrizable space, and $U\subseteq X\times Z$ is open, then $W=\{z\in Z\mid \forall x\in X((x,z)\in U)\}$ is open in $Z$.
\end{lemma}

\begin{proof}
    Let $(z_n)_{n\in\mathbb{N}}$ be a sequence in $Z\setminus W$ converging to some $z\in Z$. For each $n\in\mathbb{N}$, there is an $x_n\in X$ such that $(x_n,z_n)\notin U$. Passing to a convergent subsequence, we may assume that $x_n$ converges to $x$, so $(x_n,z_n)$ converges to $(x,z)\notin U$, since $U$ is open. Thus, $z\notin W$.
\end{proof}

The following lemma and its proof are adapted from \cite{nlab:polish_space} (see also \cite{MR3646780}).

\begin{lemma}
    If $X$ is compact, then $\mathrm{Emb}(X,Y)$ is a $G_\delta$ subset of $C(X,Y)$, and thus itself a Polish space.
\end{lemma}

\begin{proof}
  Fix compatible complete metrics $d_X$ and $d_Y$ on $X$ and $Y$, respectively. For each $m,n\in\mathbb{N}$, let
    \[
        U_{m,n}=\left\{f\in C(X,Y)\mid\forall x,y\in X\left(d_X(x,y)\geq\frac{1}{m} \Rightarrow d_Y(f(x),f(y))>\frac{1}{n}\right)\right\}.
    \]

    \begin{claim}
        $\mathrm{Emb}(X,Y)=\bigcap_{m=1}^\infty\bigcup_{n=1}^\infty U_{m,n}$.
    \end{claim}

    Clearly, if $f$ is in $\bigcap_{m=1}^\infty\bigcup_{n=1}^\infty U_{m,n}$, then $f$ is injective. Conversely, suppose that $f\in\mathrm{Emb}(X,Y)$. Since $X$ is compact, so is $f[X]$, and thus $f^{-1}:f[X]\to Y$ is uniformly continuous. Hence, for every $m$, there is an $n$ such that for all $x,y\in X$, $d_Y(f(x),f(y))\leq\frac{1}{n}$ implies $d_X(x,y)<\frac{1}{m}$. Taking the contrapositive of this last implication, we see that $f\in\bigcap_{m=1}^\infty\bigcup_{n=1}^\infty U_{m,n}$, proving the claim.

    It  now suffices to show that each $U_{m,n}$ is open. For each $m,n$, let
    \[
        W_{m,n}=\left\{(x,y,f)\in X^2\times C(X,Y)\mid d_X(x,y)\geq\frac{1}{m} \Rightarrow d_Y(f(x),f(y))>\frac{1}{n}\right\}.
    \]
    Then, $(x,y,f)\notin W_{m,n}$ if and only if $d_X(x,y)\leq\frac{1}{m}$ and $d_Y(f(x),f(y))\leq\frac{1}{m}$, a closed condition, showing that $W_{m,n}$ is open. But
    \[
        U_{m,n}=\{f\in C(X,Y)\mid\forall (x,y)\in X^2((x,y,f)\in W_{m,n})\},
    \]
    which is open by Lemma \ref{lem:compactly_many_sections_open}, completing the proof.
\end{proof}

\begin{lemma}\label{lem:cts_extension_Borel}
    Let $C$ be the set of all pairs $((x_n)_{n\in\mathbb{N}},(y_n)_{n\in\mathbb{N}})$ of sequences for which $\{x_n\mid n\in\mathbb{N}\}$ is dense in $X$ and the assignment $x_n\mapsto y_n$ defines a uniformly continuous function. Then:
    \begin{enumerate}[label=\textup{\arabic*.}]
        \item $C$ is a Borel subset of $X^\mathbb{N}\times Y^\mathbb{N}$.
        \item The map $\varphi:C\to C(X,Y)$ taking $((x_n)_{n\in\mathbb{N}},(y_n)_{n\in\mathbb{N}})$ to the unique continuous extension of the assignment $x_n\mapsto y_n$ to all of $X$ is Borel.
    \end{enumerate}
\end{lemma}

\begin{proof}
    This argument is similar to that of Lemma \ref{lem:T_is_Borel}. Fix a compatible complete metric $d$ on $X$ and a countable dense subset $\{q_n\mid n\in\mathbb{N}\}$ of $X$. Then, $((x_n)_{n\in\mathbb{N}},(y_n)_{n\in\mathbb{N}})\in C$ if and only if:
    \begin{enumerate}[label=\textup{\roman*.}]
        \item $\forall n\forall r\in \mathbb{Q}^+\exists m(d(x_m,q_n)<r))$ and
        \item $\forall r\in\mathbb{Q}^+\exists s\in\mathbb{Q}^+\forall n,m(d(x_n,x_m)<s\Rightarrow d(y_n,y_m)<r)$.
    \end{enumerate}
    These are clearly Borel conditions. Next, observe that the graph of $\varphi$ is exactly
    \[
        \{((x_n)_{n\in\mathbb{N}},(y_n)_{n\in\mathbb{N}},f)\in X^\mathbb{N}\times Y^{\mathbb{N}}\times C(X,Y)\mid\forall n\in\mathbb{N}(f(x_n)=y_n)\},
    \]
    which is Borel, since the evaluation map $C(X,Y)\times X\to Y$ is continuous.
\end{proof}

Recall that $\mathcal{K}(X)$ is the set of compact subsets of $X$ with the Vietoris topology and $\mathcal{F}(X)$ is the space of closed subsets of $X$ with the Fell topology.

\begin{lemma}\label{lem:images_Borel_compact_open}
    The following maps are Borel:
    \begin{enumerate}[label=\textup{\arabic*.}]
        \item Direct image: $C(X,Y)\times\mathcal{K}(X)\to\mathcal{K}(Y):(f,K)\mapsto f[K]$.
        \item Inverse image: $C(X,Y)\times\mathcal{F}(Y)\to\mathcal{F}(X):(f,E)\mapsto f^{-1}[E]$.
    \end{enumerate}
\end{lemma}

\begin{proof}
	It is easy to show that the map in (1) is separately continuous in each variable (see \cite[Exer.\ 4.29(vi)]{MR1321597} for one case), and that the map in (2) is continuous in the first variable and Borel in the second. Thus, both maps are Borel measurable on their domains by \cite[Lem.\ 4.51]{MR2378491}.
\end{proof}

As in Definition \ref{defn:Borel_structure_Top}, we fix a $u\notin[0,1]$ which we view as an isolated point.

\begin{lemma}\label{lem:intersection_curves_Borel}
    The map $I:C([0,1],Y)\times C(X,Y)\to[0,1]\cup\{u\}$ defined by
    \[
        I(f,g)=\begin{cases}\min\{t\mid f(t)\in \mathrm{ran}(g)\} &\text{if $\mathrm{ran}(f)\cap\mathrm{ran}(g)\neq\emptyset$}\\ u&\text{if $\mathrm{ran}(f)\cap\mathrm{ran}(g)=\emptyset$},\end{cases}
    \]
    for $f\in C([0,1],Y)$ and $g\in C(X,Y)$, is Borel.
\end{lemma}

\begin{proof}
    Note that, if $\mathrm{ran}(f)\cap\mathrm{ran}(g)\neq\emptyset$, then
    \[
        I(f,g)=\min f^{-1}[g[X]].
    \]
    The assignment $C([0,1],Y)\times C(X,Y)\to\mathcal{K}([0,1]):(f,g)\mapsto  f^{-1}[g[X]]$ is Borel by Lemma \ref{lem:images_Borel_compact_open}, and it is easy to see that $\min:\mathcal{K}([0,1])\setminus\{\emptyset\}\to [0,1]$ is continuous.
\end{proof}

By an \emph{arc}
we mean a continuous embedding $\gamma$ of $[0,1]$ into $\mathbb{R}^2$, that is, an element of $\mathrm{Emb}([0,1],\mathbb{R}^2)$. The \emph{endpoints} of $\gamma$ are $\gamma(0)$ and $\gamma(1)$. A \emph{closed curve} is a continuous embedding of $S^1$ into $\mathbb{R}^2$, an element of $\mathrm{Emb}(S^1,\mathbb{R}^2)$. We will frequently blur the distinction between an arc or closed curve and its image $\mathrm{im}(\gamma)$; by Lemma \ref{lem:images_Borel_compact_open}, the passage from the former to the latter is Borel. Lemma \ref{lem:intersection_curves_Borel} allows us to compute intersection points between arcs and curves in a Borel way.

A \emph{polygonal arc} (or \emph{closed curve}, respectively) is an arc (closed curve) consisting of finitely many straight line segments. By a \emph{rational point} in $\mathbb{R}^2$, we mean a point both of whose coordinates are rational. A \emph{rational arc} is a polygonal arc all of whose endpoints and break points are rational; a \emph{rational closed curve} is defined similarly. Rational polygonal arcs and curves have canonical parametrizations, and there are only countably many, so we may fix an enumeration of them going forwards.

A key definition used in Thomassen's proof of Theorem \ref{thm:JS} is as follows: given a closed subset $C$ of $\mathbb{R}^2$ and $\Omega$ a nonempty open connected region in $\mathbb{R}^2\setminus C$, a point $p\in C$ is \emph{accessible} from $\Omega$ if for some (equivalently, all) $q\in\Omega$, there is a polygonal arc from $q$ to $p$ whose only intersection with $C$ is $p$. We can make the obvious modification to say that $p$ is \emph{rationally accessible} from $\Omega$; note that $p$ itself need not be a rational point. The set of accessible points is dense in $C$ \cite[p.\ 122]{MR1144352}. However, this definition is apparently weaker than what is needed to carry out the rest of Thomassen's argument.\footnote{This was pointed out by Robin Chapman in a comment to an answer \cite{17582} on MathOverflow.} Namely, we need points $p\in C$ which remain accessible to any of the new regions in $\Omega$ created by adding a polygonal arc to $p$. Instead, we say that a point $p\in C$ is \emph{(rationally) accessible*} from $\Omega$ if there is a positive angle's worth of segments of (rational) polygonal arcs connecting $p$ to some (rational) $q\in\Omega$. Note that this is equivalent to having two distinct nonparallel (rational) segments incident at $p$.

\begin{lemma}
    Let $C$ be a closed subset $\mathbb{R}^2$ and $\Omega$ a nonempty region in $\mathbb{R}^2\setminus C$. 
    \begin{enumerate}[label=\textup{\arabic*.}]
        \item If $\Omega$ is a open region in $\mathbb{R}^2$ and $p,q$ are rational points in $\Omega$, then there is a rational polygonal arc from $p$ to $q$.
        \item The set of rationally accessible points in $C$ is dense in $C$.
        \item The set of rationally accessible* points in $C$ is dense in $C$.
    \end{enumerate}
\end{lemma}

\begin{proof}
    The proofs of (1) and (2) are identical to the arguments for (not necessarily rational) polygonal arcs and accessible points in \cite{MR1144352}.
    
    The proof of (3) comes from Jan Kyncl's answer \cite{488793} to a question asked by the second author on MathOverflow. It suffices to show that rationally accessible* points can be found arbitrarily close to rationally accessible points in $C$. Let $p\in C$ be rationally accessible and $P$ the straight line segment of a rational polygonal arc ending at $p$. Let $\epsilon>0$ and choose some rational $q$ on $P$ such that $d(p,q)<\epsilon$. Let $r\in C$ be the nearest point to $q$. Then, the open disk $D$ centered at $q$ having radius $d(q,r)$ must be disjoint from $C$ and lie entirely in $\Omega$, and so $r$ can be reached from a positive angle's worth of segments within $D$, at least two of which can be made as rational segments. Thus, $r$ is rationally accessible*.
\end{proof}

Given $\gamma\in\mathrm{Emb}(S^1,\mathbb{R}^2)$, we denote by $\mathrm{int}(\gamma)$ and $\mathrm{ext}(\gamma)$ the bounded and unbounded components of $\mathbb{R}^2\setminus\mathrm{im}(\gamma)$, respectively, and by $\overline{\mathrm{int}}(\gamma)$ and $\overline{\mathrm{ext}}(\gamma)$ their closures.

\begin{lemma}
    The following maps are Borel:
    \begin{enumerate}[label=\textup{\arabic*.}]
        \item $\mathrm{im}:\mathrm{Emb}(S^1,\mathbb{R}^2)\to\mathcal{F}(\mathbb{R}^2)$.
        \item $\mathrm{int}:\mathrm{Emb}(S^1,\mathbb{R}^2)\to\mathcal{O}(\mathbb{R}^2)$.
        \item $\mathrm{ext}:\mathrm{Emb}(S^1,\mathbb{R}^2)\to\mathcal{O}(\mathbb{R}^2)$.
    \end{enumerate}
\end{lemma}

\begin{proof}
    As in the proof of Lemma \ref{lem:images_Borel_compact_open}, the map $\gamma\mapsto\mathrm{im}(\gamma)=\gamma[S^1]$ is continuous. To then obtain $\mathrm{int}(\gamma)$ and $\mathrm{ext}(\gamma)$, we can take the complement of $\mathrm{im}(\gamma)$, a continuous operation, then compute its connected components by taking unions of overlapping basic open sets (a simplified version of Lemma \ref{lem:Borel_computation_of_components}). We can then detect which is bounded and which is unbounded in a Borel way, using an exhaustion of $\mathbb{R}^2$.
\end{proof}

We can now state our Borel form of Theorem \ref{thm:JS}.

\begin{theorem}\label{thm:Borel_JS}
    There is a Borel map $H:\mathrm{Emb}(S^1,\mathbb{R}^2)\to \mathrm{Homeo}(\mathbb{R}^2)$ such that
    \[
        H(\gamma)|_{S^1}=\gamma
    \]
    for all $\gamma\in \mathrm{Emb}(S^1,\mathbb{R}^2)$.
\end{theorem}

The proof of Theorem \ref{thm:Borel_JS} largely consists of translating Thomassen's proof of Theorem \ref{thm:JS} from \cite{MR1144352} into a Borel construction. The idea of said proof is, given an embedding $\gamma:S\to\mathbb{R}^2$, where we have replaced $S^1$ with the unit square $S$, we construct a ``mesh'' of finer and finer finite graphs $\Gamma_n$ and $\Gamma_n'$, and isomorphisms $g_n:\Gamma_n\to\Gamma_n'$, where $\Gamma_n$ contains $S$ and lies inside $\overline{\mathrm{int}}(S)$, $\Gamma_n'$ contains $\mathrm{im}(\gamma)$ and lies inside $\overline{\mathrm{int}}(\gamma)$, and $g_n$ agrees with $\gamma$ on $S$. In the limit, we get a homeomorphism from $\overline{\mathrm{int}}(\gamma)$ to $\overline{\mathrm{int}}(S)$ which extends $\gamma$. These graphs, besides the portion of $\Gamma_n'$ given by $\mathrm{im}(\gamma)$, consist of rational polygonal arcs. A similar procedure can then be used to construct a homeomorphism of the exteriors which also agrees with $\gamma$ on $S$.

We will use the following graph-theoretic terminology: A finite graph $\Gamma$ is \emph{$2$-connected} if it is connected and for every vertex $v$, $\Gamma\setminus\{v\}$ remains connected. Components of $\mathbb{R}^2\setminus\Gamma$ are called \emph{faces}, their boundaries are \emph{facial cycles} in $\Gamma$, the (unique) unbounded face is the \emph{outer face}, and its boundary in $\Gamma$ is the \emph{outer cycle}. If $\Gamma$ and $\Gamma'$ are $2$-connected graphs consisting of closed curves $C$ and $C'$, and polygonal arcs in $\overline{\mathrm{int}}(C)$ and $\overline{\mathrm{int}}(C')$, respectively, an isomorphism $g:\Gamma\to\Gamma'$ is a \emph{plane isomorphism} if $g$ maps facial cycles in $\Gamma$ to facial cycles in $\Gamma'$ and the outer cycle of $\Gamma$ to the outer cycle of $\Gamma'$. A \emph{subdivision} of a graph is obtained by inserting vertices on edges, in the obvious sense.

\begin{proof}[Proof of Theorem \ref{thm:Borel_JS}]
    By translating via any fixed self-homeomorphism of $\mathbb{R}^2$ which maps $S$ onto $S^1$, we may work with elements of $\mathrm{Emb}(S,\mathbb{R}^2)$, rather than $\mathrm{Emb}(S^1,\mathbb{R}^2)$. We will describe a procedure which will transform a given $\gamma\in\mathrm{Emb}(S,\mathbb{R}^2)$ into a homeomorphism $H(\gamma)\in\mathrm{Homeo}(\mathbb{R}^2)$ which extends $\gamma$, and this procedure will be evidently Borel, given the preceding lemmas.

    Let $\gamma\in\mathrm{Emb}(S,\mathbb{R}^2)$ be given. Write $C=\mathrm{im}(\gamma)$. We will extend $\gamma$ to a homeomorphism from the closed square $\overline{\mathrm{int}}(S)$ onto $\overline{\mathrm{int}}(C)$. Fix $b_n$ (for $n\in\mathbb{N}$), an enumeration of all rational points in $\mathrm{int}(C)$. We can find an enumeration $a_n$ (for $n\in\mathbb{N}$) of a dense set of points on $C$ which are rationally accessible* from $\mathrm{int}(C)$ as follows: fix $q$ a rational point in $\mathrm{int}(C)$ and $r$ a rational point in $\mathrm{ext}(C)$. Scan across all possible rational polygonal arcs from $q$ to $r$, and let $a_n$ be the $n$th point which occurs as the ``first'' intersection of two such arcs with $C$ which are incident at different angles. Let $p_n$ (for $n\in\mathbb{N}$) be an enumeration of all the $a_n$'s and $b_n$'s so that each occurs infinitely often.

    Thomassen's proof in \cite{MR1144352} proceeds as follows:\footnote{We are technically describing the inverse of the map constructed in \cite{MR1144352}.} He define sequences $\Gamma_n$ and $\Gamma_n'$ of finite $2$-connected graphs and plane isomorphisms $g_n:\Gamma_n\to\Gamma_n'$ such that for each $n\geq 1$:
    \begin{itemize}
        \item $\Gamma_n$ ($\Gamma_n'$, respectively) is an extension of a subdivision of $\Gamma_{n-1}$ ($\Gamma_{n-1}'$);
        \item $g_n$ extends $g_{n-1}$ and agrees with $\gamma$ on vertices of $\Gamma_n$ which lie on $S$;
        \item $\Gamma_n$ ($\Gamma_n'$, respectively) consists of $S$ ($C$) together with simple polygonal arcs in $\overline{\mathrm{int}}(S)$ ($\overline{\mathrm{int}}(C)$);
        \item $\Gamma_n\setminus S$ ($\Gamma_n'\setminus C$) is connected;
        \item $p_n$ is a vertex of $\Gamma_n'$; and
        \item if $p_n\in\mathrm{int}(C)$, then in some sufficiently small square containing $p_n$ in its interior all bounded faces of $\Gamma_n'$ have diameter $<1/n$, and all bounded faces in $\Gamma_n$ have diameter $\leq 1/2n$.
    \end{itemize}
    It is clear that, since we have chosen the $p_n$'s to be rational or rationally accessible* points, the graphs $\Gamma_n$ and $\Gamma_n'$ can be chosen to consist of rational polygonal arcs, together with $S$ and $C$, respectively. There are only countably many such graphs, so may fix an enumeration of them and simply search for ones that work at each stage of the construction; in particular, they can be found in a Borel way. As we go, we build enumerations $(v_m)_{m\in\mathbb{N}}$ of the vertices of the $\Gamma_n$'s and $(w_m)_{m\in\mathbb{N}}$, vertices of the $\Gamma_n'$'s, so that if $v_m$ is a vertex of $\Gamma_n$, then $g_n(v_m)=w_m$. By construction the $v_m$'s are dense in $\overline{\mathrm{int}}(S)$, the $w_m$'s are dense in $\overline{\mathrm{int}}(C)$, and the mapping $v_m\mapsto w_m$ agrees with $\gamma$ on those $v_m$'s which lie on $S$. Thomassen then shows that this mapping extends uniquely to a homeomorphism $\overline{\mathrm{int}}(S)\to\overline{\mathrm{int}}(C)$ extending $\gamma$. 
    
    A similar argument can be used to define dense sequences $s_m$ and $t_m$ (for $m\in\mathbb{N}$) in $\overline{\mathrm{ext}}(S)$ and $\overline{\mathrm{ext}}(C)$, respectively, such that the mapping $s_m\mapsto t_m$ extends uniquely to a homeomorphism $\overline{\mathrm{ext}}(S)\to\overline{\mathrm{ext}}(C)$ which also agrees with $\gamma$ on $S$. Putting these sequences together, and applying Lemma \ref{lem:cts_extension_Borel}, we obtain a homeomorphism $\mathbb{R}^2\to\mathbb{R}^2$ which extends $\gamma$ in a Borel way.
\end{proof}

While we will not make use of this here, we note that, by \cite[Thm.\ 8.38]{MR1321597}, Theorem \ref{thm:Borel_JS} implies there is a dense $G_\delta$ subset of $\mathrm{Emb}(S^1,\mathbb{R}^2)$ on which embeddings can be extended to homeomorphisms in $\mathrm{Homeo}(\mathbb{R}^2)$ continuously. It would be interesting to know, from a purely topological standpoint, whether the extension map $H:\mathrm{Emb}(S^1,\mathbb{R}^2)\to\mathrm{Homeo}(\mathbb{R}^2)$ in Theorem \ref{thm:Borel_JS} can be made globally continuous. Doing so would seem to require carrying out Thomassen's construction in a highly uniform way.

\subsection{A Borel form of the triangulation theorem}\label{sec:triangulation}

We first review some basic facts about simplicial complexes, largely following the presentation in \cite{MR0114911}, and consider how to parametrize them by elements of a standard Borel space. For us an \emph{(abstract) simplicial complex} is a set $V$ of \emph{vertices} together with a family $K$ of finite subsets of $V$, called \emph{simplices}, such that:
\begin{itemize}
    \item every vertex is in at least one and at most finitely many simplices, and
    \item every subset of a simplex is a simplex.
\end{itemize}
The \emph{dimension of a simplex} is one less than its size, and the \emph{dimension of a simplicial complex} is the supremum of the dimensions of its simplices.
We term $n$-dimensional simplices \emph{$n$-simplices} and identify vertices with $0$-simplices and tend accordingly to refer to the simplicial complex simply as $K$. Below, we deal exclusively with countable $2$-dimensional simplicial complexes.

Given a simplicial complex $K$, its \emph{geometric realization} is the space $K_g$ consisting of all real-valued functions $\lambda$ on $K$ such that:
\begin{itemize}
    \item $\lambda(x)\geq 0$ for all $x\in K$,
    \item the elements $x$ for which $\lambda(x)>0$ form a simplex, and
    \item $\sum_{x\in K}\lambda(x)=1$.
\end{itemize}
If $s$ is a simplex in $K$, then we let
\[
    s_g=\{\lambda\in K_g\mid \lambda(x)=0 \text{ for all } x\notin s\}.
\]

For example, if $s$ is a $2$-simplex, say $s=\{x_1,x_2,x_3\}$, then elements of $s_g$ can be identified with triples $(\lambda_1,\lambda_2,\lambda_3)$ of real numbers such that $\lambda_i\geq 0$ and $\lambda_1+\lambda_2+\lambda_3=1$; thus $s_g$ can be represented as a closed triangle in $\mathbb{R}^3$, from which it inherits a topology. As $s$ ranges over all simplices in $K$, the $s_g$'s cover $K_g$, and we topologize the latter by declaring a subset to be closed if its intersection with each $s_g$ is closed.

An \emph{elementary subdivision} of a $2$-dimensional simplicial complex $K$ is formed as follows: if $a=\{x_1,x_2\}$ is a $1$-simplex in $K$ which belongs to the $2$-simplex $A=\{x_1,x_2,x_3\}$, we introduce a new vertex $x$ and replace $a$ and $A$ with three $1$-simplices $\{x_1,x\}$, $\{x_2,x\}$, and $\{x_3,x\}$, and two $2$-simplices $\{x_1,x_3,x\}$ and $\{x_2,x_3,x\}$, respectively. If $a$ happened to belong to two $2$-simplices $A$ and $B$, we carry out this construction on both $A$ and $B$, simultaneously. A \emph{simple subdivision} of $K$ results by applying elementary subdivisions to (possibly infinitely) many $1$-simplices $a$ simultaneously, provided no two of them belong to the same $2$-simplex. $K'$ is a \emph{subdivision} of $K$ if it is obtained from $K$ by a finite sequence of simple subdivisions. More generally, $K'$ is \emph{equivalent} to $K$ if there is a finite sequence $K=K_1$, $K_2$, \ldots, $K_n=K'$ where each $K_i$ is a subdivision of $K_{i+1}$ or vice versa. It is clear that if $K$ and $K'$ are equivalent, then $K_g$ and $K_g'$ are homeomorphic.

We define a space $\mathcal{S}$ of countable $2$-dimensional simplicial complexes as follows: $\mathcal{S}$ consists of all triples $S=(S_0,S_1,S_2)\in \{0,1\}^\mathbb{N}\times \{0,1\}^{\mathbb{N}^2}\times\{0,1\}^{\mathbb{N}^3}$ such that:
\begin{itemize}
    \item $S_1(i,j)=S_1(j,i)$ for all $i,j\in\mathbb{N}$;
    \item $S_2(i,j,k)=S_2(i,k,j)=S_2(j,i,k)=S_2(j,k,i)=S_2(k,i,j)=S_2(k,j,i)$ for all $i,j,k\in\mathbb{N}$;
    \item $S_2(i,j,k)=1$ implies $S_1(i,j)=S_1(i,k)=S_1(j,k)=1$ for all $i,j,k\in\mathbb{N}$;
    \item $S_1(i,j)=1$ implies $S_0(i)=S_0(j)=1$ for all $i,j\in\mathbb{N}$; and
    \item for all $i\in\mathbb{N}$, there are at most finitely many $j$ such that $S_2(i,j)=1$.
\end{itemize}
Clearly, $\mathcal{S}$ is a Borel subset of $\{0,1\}^\mathbb{N}\times \{0,1\}^{\mathbb{N}^2}\times\{0,1\}^{\mathbb{N}^3}$. An element $S=(S_0,S_1,S_2)$ of $\mathcal{S}$ can be identified with the abstract simplicial complex given by $K=\mathrm{supp}(S_0)=\{n:S_0(n)=1\}$, where $\{i,j\}$ is a $1$-simplex if and only if $S_1(i,j)=1$, and $\{i,j,k\}$ is a $2$-simplex if and only if $S_2(i,j,k)=1$. We use this representation of the space of simplicial complexes, rather than the more familiar way of treating spaces of countable structures described in Section \ref{subsection:invariant_DST}, since the finite complexes are naturally included in this representation as those for which $\mathrm{supp}(S_0)$ is finite. Given $S\in\mathcal{S}$, we write $S_g$ for its geometric realization, as above.

There is a natural notion of isomorphism for complexes which can be represented by the following equivalence relation on $\mathcal{S}$: $S=(S_0,S_1,S_2)\cong S'=(S_0',S_1',S_2')$ if there is a bijection $f:\mathbb{N}\to\mathbb{N}$ such that for all $i,j,k\in\mathbb{N}$:
\begin{align*}
    &S_0(i)=S'_0(f(i))\\
    &S_1(i,j)=S'_1(f(i),f(j))\\
    &S_2(i,j,k)=S'_2(f(i),f(j),f(k)).
\end{align*}
Clearly, $\cong$ is an analytic equivalence relation, induced by an action of $S_\infty$ on $\mathcal{S}$. Similarly, we can say that $S$ and $S'$ in $\mathcal{S}$ are \emph{combinatorially equivalent}, written $S\equiv S'$, if $S$ and $S'$ have subdivisions which are isomorphic; this refinement of $\cong$ is likewise an analytic equivalence relation on $\mathcal{S}$.

It is fairly straightforward to characterize those $2$-dimensional simplicial complexes $K$ for which $K_g$ is a surface; we follow \cite[Thm.\ 22E]{MR0114911} in terming such $K$ \emph{polyhedra}:

\begin{theorem}
	A $2$-dimensional simplicial complex $K$ is a polyhedron if and only if the following conditions hold:
	\begin{enumerate}[label=\textup{\roman*.}]
		\item every $1$-simplex is in exactly two $2$-simplices,
		\item the $1$-simplices and $2$-simplices which contain a given vertex $x$ can be written as $a_1,\ldots,a_m$ and $A_1,\ldots,A_m$, respectively, for $m\geq 3$, so that $a_1=A_1\cap A_m$ and $a_i=A_i\cap A_{i-1}$ for $2\leq i\leq m$, and
		\item $K$ is connected (i.e.,~it cannot be written as a union of two disjoint subcomplexes). 
	\end{enumerate}
\end{theorem}

When translated to $\mathcal{S}$, conditions (i) and (ii) above are evidently Borel. For (iii), connectivity can be rephrased in terms of having a finite path (of $1$-simplices) between any two vertices, much like for graphs, which is also a Borel condition. Thus, we have a Borel subset $\mathcal{S}_{\mathrm{poly}}$ of $\mathcal{S}$ consisting of all polyhedra.

The relevant form of the triangulation theorem for us is that every surface is homeomorphic to the geometric realization of a polyhedron.

\begin{theorem}[Triangulation Theorem]
	For any surface $M$, there is a simplicial complex whose geometric realization is homeomorphic to $M$.
\end{theorem}

Our Borel\footnote{In \cite{MR4744736}, Harrison-Trainor and Melnikov gave an arithmetic form of the triangulation theorem for compact surfaces, also based on \cite{MR1144352}. However, their argument seems to implicitly rely on an arithmetic form of the Jordan--Schoenflies theorem which they do not
state.} form of this theorem is then:

\begin{theorem}\label{thm:Borel_triangulation}
    There is a Borel map $T:\mathfrak{C}(\mathsf{Top},\mathbb{R}^2)\to\mathcal{S}_{\mathrm{poly}}$ such that for each $(\mathcal{U},c)\in\mathfrak{C}(\mathsf{Top},\mathbb{R}^2)$, $T(\mathcal{U},c)=K_{(\mathcal{U},c)}$ is a simplicial complex whose geometric realization is homeomorphic to $M_{(\mathcal{U},c)}$.
\end{theorem}

Our proof is, again, based on that of Thomassen \cite{MR1144352}, but we more closely follow the presentation in \cite{MR3026641}. Note that the proofs in the aforementioned references are all for compact surfaces, but we will use Lemma \ref{lem:paracompact} to adapt them to the general case. For a self-contained, but less obviously constructive, proof of the triangulation theorem for arbitrary surfaces, see \cite{MR0114911}.

\begin{proof}[Proof of Theorem \ref{thm:Borel_triangulation}]
    Let $(\mathcal{U},c)\in\mathfrak{C}(\mathsf{Top},\mathbb{R}^2)$ be given and write $M$ for $M_{(\mathcal{U},c)}$. By Lemma \ref{lem:paracompact}, we may reparametrize, in a Borel way, so that $(\mathcal{U},c)$ has the following properties:
    \begin{itemize}
    	\item $(\mathcal{U},c)$ is locally finite, i.e.,~for all $i\in\mathbb{N}$, there are at most finitely many $j\in\mathbb{N}$ for which $U_{i,j}\neq\varnothing$;
    	\item for each $i\in\mathbb{N}$, $U_i$ is a nonempty open disk in $\mathbb{R}^2$;
    	\item and each $U_i$ contains a pair of nested rectangles $Q_i^1$ and $Q_i^2$ with rational corners, which we view as rational polygonal curves, which are nested, meaning $Q_i^1\subseteq \mathrm{int}(Q_i^2)$, and such that the images of the interiors of the $Q_i^1$'s cover the surface $M$.
    \end{itemize}
    In the notation of Lemma \ref{lem:paracompact}, $\mathcal{B}$ is a basis consisting of open disks in $\mathbb{R}^2$ and $\mathcal{C}$ a basis consisting of rectangles with rational corners, so that the $U_i$'s are taken from $\mathcal{B}$, the $\mathrm{int}(Q_i^1)$'s from $\mathcal{C}$, and then we can easily inscribe $Q_i^2$ between $Q_i^1$ and $U_i$.

    We will describe a procedure for obtaining a simplicial complex $K\in\mathcal{S}$ which will be Borel in the parameter $(\mathcal{U},c)$. In the course of this construction, we will modify the rectangles $Q_i^1$ and the chart maps $\varphi_{i,j}$, effectively reparametrizing $(\mathcal{U},c)$, in a way which is Borel at each step. The images of the resulting $Q_i^1$'s will form a graph on the new surface which can then be easily converted into a complex. For notational simplicity, we will use the same notation for $Q_i^1$, $\varphi_{i,j}$, and $M$ throughout the construction. At stage $k$, where we deal directly with $Q_k^1$, we may have to modify the $\varphi_{i,k}$ and $Q_i^1$ for $i< k$ if $U_i$ and $U_k$ overlap in $M$, i.e., $U_{i,k}\neq\varnothing$. However, local finiteness will insure that each $Q_k^1$ and $\varphi_{i,k}$ is only modified finitely many times, and thus the construction converges to a fixed parameter.\footnote{Logicians are free to think of this as a kind of ``finite injury'' construction.}

    Suppose that $Q_1^1,\ldots,Q_{k-1}^1$ have been chosen so that any two of their images have \emph{only finitely many points in common} in $M$. We focus on $Q^2_k$. A \emph{bad segment} is a segment $P$ of some $Q^1_i$ for $i<k$ such that $\varphi_{i,k}[P]$ joins two points of $Q^2_k$ and has all other points in $\mathrm{int}(Q^2_k)$. We call $\varphi_{i,k}[P]$ a \emph{bad segment in $Q^2_k$}. Let $Q^3_k$ be a rectangle with rational corners strictly between $Q^1_k$ and $Q^2_k$, meaning $Q^1_k\subseteq\mathrm{int}(Q^3_k)$ and $Q^3_k\subseteq\mathrm{int}(Q^2_k)$. We call a bad segment $P$ \emph{very bad} if it intersects $\mathrm{int}(Q^3_k)$. The key fact about these notions is that, while there may be infinitely many bad segments, there are only finitely many very bad ones \cite[Lem.\ E.2]{MR3026641}.

    The finite set of very bad segments $\varphi_{i,k}[P]$ inside $Q^2_k$, together with $Q^2_k$ itself, defines a $2$-connected planar graph $\Gamma$. By \cite[Prop.\ E.3]{MR3026641}, $\Gamma$ can be redrawn inside $Q^2_k$, keeping $Q^2_k$ fixed, in such a way that the resulting graph $\Gamma'$ is plane isomorphic to $\Gamma$ and all edges are simple rational polygonal arcs. Then, Theorem \ref{thm:Borel_JS} can be applied on each face of $\Gamma'$ to extend the isomorphism of $\Gamma$ to $\Gamma'$ to $\overline{\mathrm{int}}(Q^2_k)$, keeping $Q^2_k$ fixed (this is where the homeomorphisms $\varphi_{i,k}$ are being modified). This transforms $Q^1_k$ and $Q^3_k$ into simple closed curves ${Q^1}'$ and ${Q^3}'$.

	There is a rational simple closed curve ${Q^3}''$ inside $\mathrm{int}(Q^2_k)$ such that ${Q^1}'\subseteq\mathrm{int}({Q^3}'')$ which intersects no bad segment inside $Q^2_k$ except for the very bad ones (which are now rational polygonal arcs): to see this, consider a covering of ${Q^3}'$ by rational squares which do not intersect ${Q^1}'$ nor any bad segment which is not very bad, and let ${Q^3}''$ be the outer cycle of the graph formed by their union. Using \cite[Prop.\ E.3]{MR3026641} again, we can then redraw the graph $\Gamma'\cup {Q^3}''$ so that ${Q^3}''$ becomes a rectangle with rational corners having ${Q^1}'$ in its interior, and then use Theorem \ref{thm:Borel_JS} to extend this isomorphism to a homeomorphism on $\overline{\mathrm{int}}(Q^2_k)$, keeping $Q^2_k$ fixed. We let the resulting rectangle, redrawn from ${Q^3}''$, be the new choice of $Q^1_k$. It then follows that $Q^1_1,\ldots,Q^1_{k-1},Q^1_k$ now have only finitely many points in common in $M$, completing the induction step.

    At the end of the above construction, we have a sequence of rectangles $Q^1_i\subseteq Q^2_i\subseteq U_i$ for each $i\in\mathbb{N}$ such that in each $Q^2_i$ there are only finitely many very bad segments, all of which are simple rational polygonal arcs which form a $2$-connected plane graph. Upon gluing together charts, we may view these arcs, together with their points of intersection, as a locally finite graph $\Gamma$ drawn on the surface $M$ such that each component of $M\setminus\Gamma$ is bounded by a cycle in $\Gamma$. We can then form a simplicial complex $S$ by picking a point in each of these components and connecting it to each vertex of the corresponding cycle, so that each new component is bounded by a $3$-cycle, i.e., a triangle, but we must describe how to do this in a Borel way.

    Start with the least rational $q_0\in \mathrm{int}(Q^2_0)$ (with respect to some fixed enumeration of the rational points in $\mathbb{R}^2$) which is not in any very bad segment in $Q^2_0$. There is a unique cycle consisting of very bad segments in $Q^2_0$ which bounds the component of $q_0$, with intersections at rational points $q_1,\ldots,q_{n_0}$ listed counterclockwise, say. Let $S'_0$ be the simplicial complex formed from points $q_0,q_1,\ldots,q_{n_0}$, with $1$-simplices given by very bad segments connecting $q_1,\ldots,q_{n_0}$, together with $\{q_0,q_i\}$ for each $1\leq i\leq n_0$, and $2$-simplices given by $\{q_0,q_i,q_j\}$ for $1\leq i,j\leq n_0$ whenever $q_i$ and $q_j$ are connected by a very bad segment. Next, we let $q_{n_0+1}\in\mathrm{int}(Q^2_0)$ be the least rational point in a different component than $q_0$. We find the intersection points $q_{n_0+2},\ldots,q_{n_1}$ of the cycle which bounds the component of $q_{n_0+1}$, subdivide and add $1$- and $2$-simplices as above, and add these to $S_0'$ to get $S_1'$. We continue this construction until we have met each of the finitely many components of $\mathrm{int}(Q^2_0)$ bounded by very bad segments. Next, we repeat the same with $Q^2_1$, however, if any of the rational points we wish to add in the construction are already the image, under some $\varphi_{i1}$ of a rational point which we have already added, we do not need to add it; similar comments apply to the $1$- and $2$-simplices. We continue this construction across all of the $Q^2_k$'s, building an increasing sequence of finite simplicial complexes of rational points, $S'_k$. Let $S'$ be their union. Finally, we obtain $S\in\mathcal{S}$ by simply forgetting the particular rational points and just remembering the indices in $\mathbb{N}$ we have assigned them along the way, completing the construction.
\end{proof}

Given a polyhedron $K$, it is straightforward to show that it can be exhausted by a sequence of a finite subcomplexes, which when represented in $\mathcal{S}$ can be determined in a Borel way. However, in our applications, we will need slightly more, namely, to exhaust $K$ by a sequence of finite subcomplexes \emph{which are themselves polyhedra}. More explicitly, we call a sequence $P_n$ (for $n\in\mathbb{N}$) of finite simplicial complexes a \emph{canonical exhaustion}\footnote{This is slightly weaker than the definition in \cite{MR0114911}, but will suffice.} of $K$ if for all $n\in\mathbb{N}$,
\begin{itemize}
	\item each $P_n$ is a polyhedron,
	\item $P_{n+1}$ is a subcomplex of $P_n$,
	\item the border simplices of $P_{n+1}$ are not in $P_n$, and
	\item each simplex in $K$ belongs to some $P_n$.
\end{itemize}

\begin{lemma}\label{lem:canonical_exhaustion}
	There is a Borel function $E:\mathcal{S}_{\mathrm{poly}}\to\mathcal{S}_{\mathrm{poly}}^\mathbb{N}$ such that for each $K\in\mathcal{S}_{\mathrm{poly}}$, $E(K)$ is a canonical exhaustion of a subdivision of $K$.	
\end{lemma}

\begin{proof}
	The construction of such an exhaustion given in \cite[I.4.29]{MR0114911} is clearly Borel, once translated into our setting; readers may consult [ibid.]~for details.
\end{proof}

When combined with Theorem \ref{thm:Borel_triangulation}, Lemma \ref{lem:canonical_exhaustion} can be viewed as a strengthening of Lemma \ref{lem:exhaustion} for surfaces, giving not only an exhaustion by compact sets, but by compact \emph{submanifolds} (with boundary), realized as polyhedra. In fact, this will be the main application of Theorem \ref{thm:Borel_triangulation} in our proof Theorem \ref{thm:surfaces}. We had initially hoped for a simpler approach to proving the existence of such a Borel exhaustion along the lines of the proof of Lemma \ref{lem:exhaustion}, but were not able to overcome technical issues concerning how the boundaries patch together.

We close this section with a few comments on related results and generalizations. Clearly, equivalent complexes have homeomorphic geometric realizations. The converse of this statement is known as the \emph{Hauptvermutung}, and holds for complexes whose geometric realizations are manifolds in dimensions $2$ and $3$ \cite{MR24619} \cite{MR48805}, but fails in higher dimensions (see \cite{MR1434100}). Since we can pass from polyhedra to surfaces by taking geometric realizations (it is routine to show that this construction is Borel), and back by Theorem \ref{thm:Borel_triangulation}, it follows that homeomorphism of surfaces and equivalence $\equiv$ of $2$-dimensional polyhedra have the same Borel complexity. In particular, if we knew that the latter was classifiable by countable structures, our proof of Theorem \ref{thm:surfaces} would be complete. However, as we do not have a direct proof of this fact,\footnote{This has been recently proved independently by Iannella and Weinstein \cite{IW26+}.} we must continue onwards via Theorem \ref{thm:Richards}. We will return to this point in Corollary \ref{cor:equiv_polyedra_classifiable} below.

Finally, we note that, like surfaces, all topological $3$-manifolds can be triangulated \cite{MR0488059}, but again, triangulation fails in higher dimensions \cite{MR679066, MR3402697}. The case of smooth manifolds (in any dimension) is simpler; they are always triangulable, and it seems likely one could produce a Borel exhaustion by compact submanifolds directly using regular sublevel sets of exhaustion functions and Sard's Theorem (see \cite{MR2954043} or \cite{MR226651}).

\subsection{The construction of the space of ends}
\label{subsection:space_of_ends}

The abstract notion of an ``end'' of a manifold is due to Freudenthal \cite{MR1545233} who used it to construct a canonical compactification which generalizes that for the plane, obtained by adjoining a single point at $\infty$ (``one end''), and the line, obtained by adjoining points at $+\infty$ and $-\infty$ (``two ends''). Standard English-language presentations of this construction can be found in \cite{MR0114911} or \cite{MR120637}. In this section, we will verify that the construction of the space of ends of a connected topological manifold, in \emph{any} dimension, can be accomplished in a Borel way. 

The key elementary fact which ensures the compactness of the space of ends is the following lemma \cite[Lem.\ 1.1]{MR120637}; for a proof see \cite[Lem.\ 9]{MR686504}.

\begin{lemma}\label{lem:finitely_many_components}
	If $M$ is a connected manifold and $K\subseteq M$ is compact, then $M\setminus K$ has only finitely many unbounded connected components.
	\qed
\end{lemma}

Recall that a sequence $K_i$ (for $i\in \mathbb{N}$) of compact subsets of a manifold $M$ is a compact exhaustion if $K_i\subseteq \mathrm{int}(K_{i+1})$ for all $i\in\mathbb{N}$ and $M=\bigcup_{i\in\mathbb{N}}K_i$.

\begin{definition}
	Given a connected manifold $M$ and a compact exhaustion $K_i$ (for $i\in\mathbb{N}$), an \emph{end} of $M$ is a sequence $\langle Q_i\mid i\in\mathbb{N}\rangle$ of nonempty open subsets of $M$ such that for each $i\in\mathbb{N}$:
	\begin{enumerate}[label=\textup{\roman*.}]
		\item $Q_i\supseteq Q_{i+1}$,
		\item $Q_i$ is an unbounded connected component of $M\setminus K_i$, and
		\item for any compact subset $K\subseteq M$, there is some $j\in\mathbb{N}$ for which $Q_j\cap K=\varnothing$.
	\end{enumerate}	 
\end{definition}

While the above definition depends on a particular choice of exhaustion, different exhaustions will produce ``equivalent'' ends, in the sense that for each end $\langle Q_i\mid i\in\mathbb{N}\rangle$ with respect to one exhaustion, there is an end $\langle Q_i'\mid i\in\mathbb{N}\rangle$ with respect to the other such that for all $i\in\mathbb{N}$, there is a $j\in\mathbb{N}$ for which $Q_i\subseteq Q_j'$, and vice-versa. It is a consequence of the canonical nature of the end compactification (see Lemma \ref{lem:end_compact} below) that it is independent, up to a homeomorphism which fixes the underlying manifold, of the particular choice of exhaustion.

An end of a manifold is an infinite branch through a certain finitely-branching (by Lemma \ref{lem:finitely_many_components}) tree of finite sequences of open connected sets. We must describe how to define this tree in a Borel way. As this is a tree of elements in an uncountable space, some care must be taken in ensuring a Borel target space of trees. Let $\mathsf{Tr}(\mathbb{N})$ denote the set of all finitely-branching trees on $\mathbb{N}$, which inherits a standard Borel structure when viewed as a subset of $\{0,1\}^{\mathbb{N}^{<\infty}}$. The idea for the following lemma comes from Arno Pauly's answer \cite{434262} to a question asked by the second author on MathOverflow.

\begin{lemma}\label{lem:space_of_trees}
	Let $X$ be a standard Borel space. Then, there is a standard Borel space $\mathsf{Tr}(X)$ of all finitely-branching trees in $X$ such that:
	\begin{enumerate}[label=\textup{\arabic*.}]
		\item There is a Borel map $p:\mathsf{Tr}(X)\to\mathsf{Tr}(\mathbb{N})$ which assigns to each tree in $\mathsf{Tr}(X)$ an isomorphic tree on $\mathbb{N}$. Moreover, this map preserves subtrees.
		\item There is a Borel map $\ell:\mathsf{Tr}(X)\to (X\cup\{u\})^\mathbb{N}$ which assigns to each tree in $\mathsf{Tr}(X)$ an enumeration of its nodes, where $u\notin X$ represents that a node is undefined.
	\end{enumerate}
\end{lemma}

\begin{proof}
	Choose some $u\notin X$ and view $X\cup\{u\}$ as being endowed with a Borel linear ordering $<$ where $u$ is the maximum. As the full tree $\mathbb{N}^{<\infty}$ on $\mathbb{N}$ is countable, it induces a well-ordering on the nodes of any tree in $\mathsf{Tr}(\mathbb{N})$ having the property that the root $\varnothing$ is always the least element and each level of the tree is always completely enumerated before any successive level. We can thus identify a finitely-branching tree $T$ on $X$ with a ``labelled tree'', that is, a pair $(S,\ell)\in\mathsf{Tr}(\mathbb{N})\times (X\cup\{u\})^\mathbb{N}$, where $S$ is the underlying tree structure of $T$ and for each $n\in\mathbb{N}$, $\ell(n)$ is the element of $X$ occurring at the $n$th node of $T$, with $\ell(n)=u$ if there is no such node. We require that $\ell$ labels siblings (i.e., immediate successors of a common parent node) in strictly increasing order with respect to $<$ and the ordering on nodes described above. This ensures that each tree $T$ corresponds to a unique pair $(S,\ell)$. The subset of $\mathsf{Tr}(\mathbb{N})\times (X\cup\{u\})^\mathbb{N}$ consisting all such pairs is clearly Borel and the maps described in (1) and (2) are simply the coordinate projections.
\end{proof}

We can now describe the construction of the space of ends of a connected manifold $M=M_{(\mathcal{U},c)}$ in a way that will be evidently Borel in the parameter $(\mathcal{U},c)\in\mathfrak{C}(\mathsf{Top},\mathbb{R}^n)$, given our preceding lemmas. First, if $M$ happens to be compact, which we may detect in a Borel way by Theorem \ref{thm:compact_Borel}, then we put $T_M=\varnothing$. For $M$ noncompact, let $ K_i$ (for $i\in\mathbb{N}$) be a compact exhaustion of $M$. Define a tree $T_M$ of open subsets of $M$ as follows: The root of the tree is the node given by $M$ itself. If $\langle Q_0,\ldots,Q_n\rangle$ has been placed in $T_M$, then let $\langle Q_0,\ldots,Q_n,Q_{n+1}\rangle$ be in $T_M$ if and only if $Q_{n+1}$ is an unbounded connected component of $Q_n\setminus K_n$ (equivalently, an unbounded connected component of $M\setminus K_n$ which is contained in $Q_n$). The (complements of the) compact exhaustion can be obtained in a Borel way by Lemma \ref{lem:exhaustion}, boundedness detected by Lemma \ref{lem:bounded_Borel}, and the connected components of open subsets computed using Lemma \ref{lem:Borel_computation_of_components}, so we have thus proven:

\begin{lemma}\label{lemma:T_M_Borel}
	The map $\mathfrak{C}(\mathsf{Top},\mathbb{R}^n)\to\mathsf{Tr}(\mathfrak{C}(\mathsf{Top},\mathbb{R}^n)):(\mathcal{U},c)\mapsto T_{M_{(\mathcal{U},c)}}$ is Borel.\qed
\end{lemma}

Next, we identify $[T_M]$, the set of infinite branches through $T_M$, with the set of ends of $M$. Define 
\[
	\mathcal{E}(M)=M\sqcup [T_M],
\]
which we topologize as follows: a basis for the topology consists of all open subsets of $M$ together with, for each $i\in\mathbb{N}$ and each unbounded connected component $Q$ of $M\setminus K_i$, $Q\cup \varepsilon(Q)$, where $\varepsilon (Q)$ is the set of all ends eventually contained in $Q$:
\[
	\varepsilon(Q)=\{\langle Q_i\mid i\in\mathbb{N}\rangle\in [T_M]\mid\exists i\in\mathbb{N}(Q_i\subseteq Q)\}.
\] 
To see that this is, indeed, a basis, it suffices to consider $\varepsilon(Q)\cap\varepsilon(Q')$, where $Q$ and $Q'$ are unbounded open connected sets as above. Given $q\in\varepsilon(Q)\cap\varepsilon(Q')$, say with $q=\langle Q_i\mid i\in\mathbb{N}\rangle$, there are $i,j\in\mathbb{N}$ such that $Q_i\subseteq Q$ and $Q_j\subseteq Q'$. We may assume that $i\leq j$, and so $Q_j\subseteq Q_i\subseteq Q$, from which it follows that $\varepsilon(Q_j)\subseteq\varepsilon(Q)\cap\varepsilon(Q')$. We endow $[T_M]$ with the induced subspace topology generated by the sets $\varepsilon(Q)$.

There is another natural topology on $[T_M]$ which it inherits from the tree structure on $T_M$, namely the one generated by basic open sets of the form
\[
	N_{\langle Q_0,\ldots,Q_n\rangle}=\{q\in[T_M]\mid\langle Q_0,\ldots,Q_n\rangle \sqsubseteq q\}
\]
for $\langle Q_0,\ldots,Q_n\rangle\in T_M$. This is the topology induced on $[T_M]$ as a subspace of the product $\mathcal{Q}^\mathbb{N}$, where $\mathcal{Q}$ is the discrete set of the countably-many $Q_i$'s which occur in $T_M$. Consequently, under this topology, $[T_M]$ is a totally disconnected compact (since $T_M$ is finitely branching) Polish space.

\begin{lemma}
	The topologies on $[T_M]$ generated by the sets $\varepsilon(Q)$ and $N_{\langle Q_0,\ldots,Q_n\rangle}$ above, respectively, coincide.
\end{lemma}

\begin{proof}
	We will show that each basic open set is open in the other topology. Let $\langle Q_0,\ldots,Q_n\rangle\in T_M$, so each $Q_i$ is an unbounded connected component of $M\setminus K_i$ and $M=Q_0\supseteq Q_1\supseteq\cdots\supseteq Q_n$. Clearly, $N_{\langle Q_0,\ldots,Q_n\rangle}\subseteq\varepsilon(Q_n)$. If $q=\langle Q_i'\mid i\in\mathbb{N}\rangle\in\varepsilon(Q_n)$, then since $Q_n$ and $Q_n'$ are both connected components of $M\setminus K_n$ which must overlap, we have that $Q_n'=Q_n$. By a similar argument, $Q_n$ uniquely determines each of its predecessors $Q_{n-1},\ldots,Q_0$ in $T_M$. Hence, $N_{\langle Q_0,\ldots,Q_n\rangle}=\varepsilon(Q_n)$.
	
	Conversely, given an unbounded connected component $Q$ of $M\setminus K_i$ for some $i$, take $q=\langle Q_i\mid i\in\mathbb{N}\rangle\in\varepsilon(Q)$. If we let $n\in\mathbb{N}$ be the least such that $Q_n\subseteq Q$, then $q\in N_{\langle Q_0,\ldots,Q_n\rangle}\subseteq\varepsilon(Q)$. This shows that the topologies coincide.
\end{proof}

The next lemma will verify that $\mathcal{E}(M)$ is, in fact, the end compactification of $M$.

\begin{lemma}\label{lem:end_compact}
	For each connected manifold $M$:
	\begin{enumerate}[label=\textup{\arabic*.}]
		\item $\mathcal{E}(M)$ is Hausdorff, second countable, and $M$ is dense open in $\mathcal{E}(M)$.
		\item $\mathcal{E}(M)$ is connected.
		\item $\mathcal{E}(M)$ is compact.
		\item $[T_M]=\mathcal{E}(M)\setminus M$ is totally disconnected.
		\item For any nonempty connected open set $U\subseteq\mathcal{E}(M)$, $U\setminus[T_M]$ is connected.
	\end{enumerate}
	Consequently, $\mathcal{E}(M)$ can be identified with the end compactification of $M$.
\end{lemma}

\begin{proof}
	For (1), that $\mathcal{E}(M)$ is second countable and $M$ is open and dense in $\mathcal{E}(M)$ follows immediately from the definition of the topology on $\mathcal{E}(M)$. Since both $M$ and $[T_M]$ are Hausdorff, to show that $\mathcal{E}(M)$ is Hausdorff, it suffices to show that if $m\in M$ and $q=\langle Q_i\mid i\in\mathbb{N}\rangle\in [T_M]$, then they can be separated by disjoint open sets. Choose an element of the exhaustion $K_i$ which contains $m$ in its interior. Then, $\mathrm{int}(K_i)$ and $Q_{i+1}\cup\varepsilon(Q_{i+1})$ are disjoint open sets containing $m$ and $q$, respectively.
	
	For (2), since $M$ is dense in $\mathcal{E}(M)$ and connected, $\mathcal{E}(M)$ is connected as well.
	
	In order see (3) that $\mathcal{E}(M)$ is compact, it suffices to show that it is sequentially compact, since it is second countable. Let $(x_n)_{n\in\mathbb{N}}$ be a sequence in $\mathcal{E}(M)$; we claim that it has a convergent subsequence. Since we already know that $[T_M]$ is compact, we may assume the $x_n$'s are entirely contained in $M$, and further, that they have no subsequence which is bounded in $M$. Define a subsequence $(x_{n_i})_{i\in\mathbb{N}}$ of $(x_n)_{n\in\mathbb{N}}$ as follows: Start by taking $x_{n_0}=x_0$ and $Q_0=M$. Let $Q_1$ be an unbounded connected component of $M\setminus K_0$ which contains infinitely many $x_n$'s. Such a component exists because there are only finitely many unbounded components of $M\setminus K_0$ and the $x_n$'s are unbounded. Choose one such $x_n$ with $n>n_0$, call it $x_{n_1}$. Suppose we have chosen $x_{n_0},x_{n_1},\ldots,x_{n_i}$ and $\langle Q_0,Q_1,\ldots,Q_i\rangle\in T_M$ such that $x_{n_i}\in Q_i$ and infinitely many $x_n$'s lie in $Q_i$. Let $Q_{i+1}$ be an unbounded component of $M\setminus K_{i+1}$ which contains infinitely many $x_n$'s. Choose one such $x_{n_{i+1}}\in Q_{i+1}$ with $n_{i+1}>n_i$. It then follows from the definition of the topology on $\mathcal{E}(M)$ that $(x_{n_i})_{i\in\mathbb{N}}$ converges to $q=\langle Q_i\mid i\in\mathbb{N}\rangle\in[T_M]$.
	
	We have already observed (4) that $[T_M]$ is totally disconnected.
	
	For (5), let $U$ be a connected open set in $\mathcal{E}(M)$, which we may assume intersects $[T_M]$. Suppose that $U\setminus [T_M]=U\cap M=O_0\cup O_1$, where $O_0$ and $O_1$ are disjoint open sets in $M$. Let $U_0$ be the union of $O_0$ and all sets $Q\cup\varepsilon(Q)$ where $Q\subseteq O_0$, and $U_1$ the the union of $O_1$ and all sets $Q\cup\varepsilon(Q)$ where $Q\subseteq O_1$. Clearly, $U_0$ and $U_1$ are open and disjoint. Moreover, if $q=\langle Q_i\mid i\in\mathbb{N}\rangle\in U\cap [T_M]$, then $Q_i\subseteq O_0$ or $O_1$, and so $q\in U_0$ or $q\in U_1$. This shows that $U\subseteq U_0\cup U_1$, but $U$ is connected and $U_0\cap U_1=\varnothing$, so one of $U_0$ or $U_1$ is empty. It follows that either $O_0$ or $O_1$ is empty, proving that $U\setminus[T_M]$ is connected.
	
	Lastly, the end compactification of $M$ is  (up to a homeomorphism which is the identity on $M$) the unique compactification satisfying (1) through (5) (\cite[Thm.\ 1.8]{MR120637}).
\end{proof}

\subsection{Genus, orientability, and planarity}\label{sec:genus}

The classification result in Theorem \ref{thm:Richards} refers to the ``genus'' and ``orientability class'' of an arbitrary surface, generalizing the familiar notions for compact surfaces. Intuitively, for an orientable surface, the genus counts the number of ``handles'', while for a nonorientable surface, it counts the number of ``cross caps''\footnote{A ``cross cap'' is simply an embedded M\"obius band.} divided by a factor of $2$. In the case of surfaces homeomorphic to finite simplicial complexes, these invariants can be computed directly from the finite complex. 
For arbitrary surfaces, the \emph{genus} can then be defined as the (possibly infinite) limit of the genuses coming from an exhaustion of the surface by finite polyhedra, as in Lemma \ref{lem:canonical_exhaustion}, which is a Borel operation. 

Following \cite{MR143186}, there are four \emph{orientability-classes} for arbitrary surfaces:

\begin{itemize}
	\item the \emph{orientable}: those for which all compact subsurfaces are orientable;
	\item the \emph{infinitely nonorientable}: those for which no complement of a compact subsurface is orientable;
	\item the \emph{odd nonorientable}: those for which every sufficiently large compact subsurface contains an odd number of cross caps;
	\item the \emph{even nonorientable}: likewise, but with an even number of cross caps.
\end{itemize}
Again, the ability to compute exhaustions by finite polyhedra and their invariants (e.g., Betti numbers), by Lemma \ref{lem:canonical_exhaustion}, implies that we can compute the orientability class of an arbitrary surface in a Borel way, as well (cf.~Corollary \ref{cor:compact_invariants_Borel}). We also need the notion of a \emph{planar} surface, i.e.,~one in which every compact subsurface has genus zero, which can likewise be checked in a Borel way. To summarize:

\begin{lemma}\label{lemma:genus_Borel}
	The following can be computed in a Borel fashion on $\mathfrak{C}(\mathsf{Top},\mathbb{R}^2)$:
	\begin{enumerate}[label=\textup{\arabic*.}]
		\item Genus (in $\mathbb{N}\cup\{\infty\}$).
		\item Planarity (``yes'' or ``no'').
		\item Orientability class (``orientable'', ``infinitely nonorientable'', ``odd nonorientable'', or ``even nonorientable''.)\qed
	\end{enumerate}
\end{lemma}

We can now define (non)planarity and (non)orientability for ends:

\begin{definition}
	Suppose that $q=\langle Q_i\mid i\in\mathbb{N}\rangle$ is an end of a surface $M$.
	\begin{enumerate}[label=\textup{\arabic*.}]
		\item We say that $q$ is \emph{nonplanar} if $Q_i$ is nonplanar for all $i\in\mathbb{N}$.
		\item We say that $q$ is \emph{nonorientable} if $Q_i$ is nonorientable for all $i\in\mathbb{N}$.
	\end{enumerate}
\end{definition}

Given a surface $M$ and the tree $T_M$ defined in Section \ref{subsection:space_of_ends}, we can define subtrees $T_M'$ and $T_M''$ of $T_M$ such that $\langle Q_0,\ldots,Q_n\rangle\in T_M'$ ($T_M''$, respectively) if and only if for all $i\leq n$, $Q_n$ is nonplanar (nonorientable). Then, for $q\in [T_M]$, we have that $q\in[T_M']$ ($[T_M'']$, respectively) if and only if $q$ is a nonplanar (nonorientable) end of $M$. By Lemmas \ref{lem:space_of_trees} and \ref{lemma:genus_Borel}, these trees can be likewise computed in a Borel way. Combining this with Lemma \ref{lemma:T_M_Borel}, we have the following:

\begin{lemma}\label{lemma:trees_M_Borel}
The map $\mathfrak{C}(\mathsf{Top},\mathbb{R}^2)\to\mathsf{Tr}(\mathfrak{C}(\mathsf{Top},\mathbb{R}^2))^3$ given by $$(\mathcal{U},c)\mapsto (T_{M_{(\mathcal{U},c)}},T_{M_{(\mathcal{U},c)}}',T_{M_{(\mathcal{U},c)}}'')$$ is Borel.\qed
\end{lemma}

\subsection{Classification by countable structures}

Recall that $\mathcal{K}(\mathbb{N}^\mathbb{N})$ is the space of all compact subsets of the Baire space $\mathbb{N}^\mathbb{N}$, endowed with the Vietoris topology. Let $\mathcal{K}(\mathbb{N}^\mathbb{N})_3$ be the closed subspace of $\mathcal{K}(\mathbb{N}^\mathbb{N})^3$ consisting of all nested triples $(K,K',K'')$ of compact subsets $K\supseteq K'\supseteq K''$ of $\mathbb{N}^\mathbb{N}$. The space $\mathcal{K}(\{0,1\}^\mathbb{N})_3$ is defined similarly for the Cantor space $\{0,1\}^\mathbb{N}$. On such nested triples, we write $(K,K',K'')\cong_3(L,L',L'')$ if there is a homeomorphism $f:K\to L$ that restricts to homeomorphisms $K'\to L'$ and $K''\to L''$.

By Lemma \ref{lemma:trees_M_Borel}, we have a Borel function on $\mathfrak{C}(\mathsf{Top},\mathbb{R}^2)$ which, given a surface $M=M_{(\mathcal{U},c)}$, for $(\mathcal{U},c)\in\mathfrak{C}(\mathsf{Top},\mathbb{R}^2)$, produces a triple $(T_M,T_M',T_M'')$ of trees of open subsets of $M$ whose branch spaces are exactly the space of ends $[T_M]$, the space of nonplanar ends $[T_M]'$, and the space of nonorientable ends $[T_M]''$. Using Lemma \ref{lem:space_of_trees}, we can then obtain a triple $(S_M,S_M',S_M'')$ of nested trees on $\mathbb{N}$ isomorphic to $(T_M,T_M',T_M'')$, in the obvious sense, and compute the triple $$(e(M),e(M)',e(M)''):=([S_M],[S_M'],[S_M''])\in\mathcal{K}(\mathbb{N}^\mathbb{N})_3.$$ Applying Theorem \ref{thm:Richards}, we have thus proven:

\begin{theorem}\label{thm:surfaces_fixed_genus_orient}
	On each genus and orientability class of $\mathfrak{C}(\mathsf{Top},\mathbb{R}^2)$, there is a Borel reduction of homeomorphism to $\cong_3$ on $\mathcal{K}(\mathbb{N}^\mathbb{N})_3$.\qed
\end{theorem}

The following general lemma lets us fix the genus and orientability class, and focus on equivalence in $\mathcal{K}(\mathbb{N}^\mathbb{N})_3$.

\begin{lemma}\label{lem:partition_ctbl_structures}
If $E$ is an analytic equivalence relation on a standard Borel space $X$ and there is a partition of $X$ into countably many $E$-invariant Borel pieces $X_i$ (for $i\in\mathbb{N}$) such that $E|_{X_i}$ is classifiable by countable structures for each $i\in\mathbb{N}$, then so is $E$.
\end{lemma}

\begin{proof}
	By assumption, for each $i\in\mathbb{N}$, $E|_{X_i}\leq_B F_i$, where $F_i$ is the orbit equivalence relation induced by a continuous action of $S_\infty$ on a Polish space $Y_i$. Put $Y=\coprod_{i\in\mathbb{N}}Y_i$ and 	let $S_\infty$ act on $Y$ via its action on each of the disjoint pieces $Y_i$, with orbit equivalence relation $F$. Then, $E\leq_B F$ and $F$ is classifiable by countable structures, being an orbit equivalence relation of an $S_\infty$-action.
\end{proof}

To complete the proof of Theorem \ref{thm:surfaces}, it thus suffices to produce a Borel reduction from the space $\mathcal{K}(\mathbb{N}^\mathbb{N})_3$ to some space of countable first-order structures. The target will be the space $\mathsf{BA}_3$ of triples $(B,F,G)$, where $B$ is a countable Boolean algebra with domain $\mathbb{N}$, and $F\subseteq G$ are nested filters in $B$; the class of all such structures is clearly first-order axiomatizable in the language of Boolean algebras, together with unary predicates for $F$ and $G$, and thus Borel. The reduction is, essentially, given by Stone duality.

\begin{lemma}\label{lem:Stone_3}
	The relation $\cong_3$ on $\mathcal{K}(\mathbb{N}^\mathbb{N})_3$ is Borel reducible to the isomorphism relation $\cong$ on $\mathsf{BA}_3$.
\end{lemma}

\begin{proof}
First, define $\Psi:\mathcal{K}(\{0,1\}^{\mathbb{N}})_3\to\mathsf{BA}_3$ by $\Psi(K,K',K'')=(B,F,G)$ as follows: Let $B$ be the algebra of clopen subsets of $K$, $F$ ($G$, respectively) the filter of clopen sets whose complements are contained in $K\setminus K'$ (or $K\setminus K''$), where we identify $(B,F,G)$ with an element of $\mathsf{BA}_3$ by using an enumeration of the clopen subsets of $\{0,1\}^\mathbb{N}$. Under this identification, $\Psi$ is Borel by the arguments in \cite{MR1804507}, and is a reduction to isomorphism on $\mathsf{BA}_3$ by classical Stone duality (see \cite{MR0167440}).

It remains to produce a Borel reduction from $\cong_3$ on $\mathcal{K}(\mathbb{N}^\mathbb{N})_3$ to the similarly defined relation $\cong_3$ on $\mathcal{K}(\{0,1\}^\mathbb{N})_3$. Doing so amounts to embedding all finitely branching pruned subtrees of $\mathbb{N}^{<\infty}$ into $\{0,1\}^{<\infty}$ in a sufficiently uniform way. For $X$ countable and discrete, finitely branching pruned subtrees of $X^{<\infty}$ can be identified with compact subsets of $X^\mathbb{N}$ via their spaces of branches, $T\mapsto[T]$, and vice versa (see \cite[Ch.\ 2]{MR1321597}). Viewing the space of such trees as a Borel subset of $\{0,1\}^{X^{<\infty}}$, this identification gives a Borel bijection with $\mathcal{K}(X^\mathbb{N})$ under the Vietoris topology: If $N_{s}$ is a basic open set in $X^\mathbb{N}$, for $s\in X^{<\infty}$, then for a compact set $K\subseteq X^{\mathbb{N}}$, say $K=[T]$, $K\cap N_s\neq\varnothing$ if and only if all $t\in T$ are compatible with $s$, clearly a Borel condition.

Given a triple of finitely branching pruned subtrees $(S,S',S'')$ of $\mathbb{N}^{<\infty}$, with $S\supseteq S'\supseteq S''$, we define a map $\phi:S\to \{0,1\}^{<\infty}$ inductively as follows: Put $\phi(\varnothing)=\varnothing$. Having defined $\phi(s)$, for $s\in S$, if the immediate successors of $s$ in $S$ are $\{s^{\smallfrown}(n_1),s^{\smallfrown}(n_2),\ldots,s^{\smallfrown}(n_k)\}$, with $n_1<n_2<\cdots<n_k$, then we define $\phi(s^{\smallfrown}(n_i))=\phi(s)^{\smallfrown}t_i$, where $t_i$ is the $i$th leaf node of the complete binary tree of height $\lceil\log_2(k)\rceil$ under the lexicographic order. Let $T$, $T'$, and $T''$ be the downwards closures of the image of $S$, $S'$, and $S''$, respectively, under $\phi$. Then, $\phi$ is an order and incompatibility-preserving map which carries $S$, $S'$, and $S''$ onto dense subsets of $T$, $T'$, and $T''$, respectively. Consequently, the map $\phi^\ast:[S]\to[T]$, defined by
\[
	\phi^\ast(x)=\bigcup_{n\in\mathbb{N}}\phi(x|n)
\]
for $x\in[S]$, is a homeomorphism from $[S]$ onto $[T]$ which carries $[S']$ onto $[T']$ and $[S'']$ onto $[T'']$. It is then evident from the preceding description that the assignment $\Phi([S],[S'],[S''])=([T],[T'],[T''])$ is a Borel reduction of $\cong_3$ on $\mathcal{K}(\mathbb{N}^\mathbb{N})_3$ to the corresponding relation on $\mathcal{K}(\{0,1\}^{\mathbb{N}})_3$.  
\end{proof}

Theorem \ref{thm:surfaces} now follows from Theorem \ref{thm:surfaces_fixed_genus_orient} and Lemmas \ref{lemma:genus_Borel}, \ref{lem:partition_ctbl_structures}, and \ref{lem:Stone_3}. Moving to smooth surfaces, we have noted in Example \ref{ex:pseudogroup_comparison} that the inclusion map $\mathfrak{C}(\mathsf{C}^\infty,\mathbb{R}^2)\to\mathfrak{C}(\mathsf{Top},\mathbb{R}^2)$ is a Borel reduction of diffeomorphism $\cong_{\mathsf{C}^\infty}$ to homeomorphism $\cong_{\mathsf{Top}}$, and so the former is also classifiable by countable structures, proving:

\begin{corollary}\label{cor:smooth_surfaces}
	The diffeomorphism relation $\cong_{\mathsf{C}^\infty}$ on $\mathfrak{C}(\mathsf{C}^\infty,\mathbb{R}^2)$ is complete for countable structures.\qed
\end{corollary}

Finally, by the discussion at the end of Section \ref{sec:triangulation}, we also have the following immediate consequence:

\begin{corollary}\label{cor:equiv_polyedra_classifiable}
	The relation of combinatorial equivalence $\equiv$ on the space $\mathcal{S}_{\mathrm{poly}}$ of all $2$-dimensional polyhedra is complete for countable structures.\qed
\end{corollary}

\subsection{Disconnected surfaces}
\label{subsection:disconnected}
To close out this section, we will show that the homeomorphism and diffeomorphism problems for \emph{all} topological and smooth, respectively, $2$-manifolds, i.e., possibly \emph{disconnected} surfaces, is of the same complexity as that for connected surfaces. As the latter is an invariant subclass of the former, it suffices to show that the larger class is still classifiable by countable structures. We will focus on the topological case, as identical arguments apply in the smooth case.

Using Lemma \ref{lem:Borel_computation_of_components}, we can compute the connected components of a $2$-manifold, as an ordered sequence of open submanifolds, in a Borel fashion. This reduces the question of whether two $2$-manifolds are homeomorphic to whether their corresponding components are, \emph{up to permuting their order}. As this permutation is induced by an action of $S_\infty$, we are able to use general techniques to show that this relation is also classifiable by countable structures.

Given an analytic equivalence relation $E$ on a standard Borel space $X$, define an equivalence relation $E^\mathbb{N}$ on $X^\mathbb{N}$ by 
\[
	(x_i)_{i\in\mathbb{N}}\,E^\mathbb{N}\,(y_i)_{i\in\mathbb{N}} \text{ if and only if there is a $\sigma\in S_{\infty}$ such that $x_{\sigma(i)}\,E\,y_i$ for all $i\in\mathbb{N}$. }
\]
\begin{lemma}\label{lem:E^N_classifiable}
	If $E$ is classifiable by countable structures, then so is $E^\mathbb{N}$.
\end{lemma}

\begin{proof}
	By assumption, there is a continuous action of $S_\infty$ on a Polish space $Y$ such that if $F$ is the resulting orbit equivalence relation, then $E\leq_B F$. It follows that $E^\mathbb{N}\leq_B F^\mathbb{N}$, so it suffices to show that $F^\mathbb{N}$ is classifiable by countable structures.
	
	Consider the action of $S_\infty\times S_\infty^\mathbb{N}$ on $Y^\mathbb{N}$ given by
	\[
		(\sigma,(g_i)_{i\in\mathbb{N}})\cdot(y_i)_{i\in\mathbb{N}} = (g_i\cdot y_{\sigma(i)})_{i\in\mathbb{N}}
	\]
	for all $(\sigma,(g_i)_{i\in\mathbb{N}})\in S_\infty\times S_\infty^\mathbb{N}$. Clearly, this action is continuous and has $F^\mathbb{N}$ as its orbit equivalence relation. $S_\infty\times S_\infty^\mathbb{N}$ is isomorphic as a Polish group to $S_\infty^\mathbb{N}$, which in turn is isomorphic to a closed subgroup of $S_\infty$; identify $\mathbb{N}$ with $\mathbb{N}\times\mathbb{N}$ via some fixed bijection $\mathbb{N}\to\mathbb{N}\times\mathbb{N}$ and let the $i$th coordinate of an element $(g_i)_{i\in\mathbb{N}} \in S_\infty^\mathbb{N}$ act on the $i$th column $\{i\}\times\mathbb{N}$ in the obvious way. The orbit equivalence relation induced by a continuous action of any closed subgroup of $S_\infty$ is classifiable by countable structures (see \cite{MR1425877} and \cite[Thm.\ 2.39]{MR1725642}), hence $F^\mathbb{N}$ is as well.
\end{proof}

\begin{corollary}\label{corollary:2-manifolds}
	The homeomorphism relation $\cong_{\mathsf{Top}}$ on $\mathfrak{M}(\mathsf{Top},\mathbb{R}^2)$  is complete for countable structures. Similarly for the diffeomorphism relation $\cong_{\mathsf{C}^\infty}$ on $\mathfrak{M}(\mathsf{C}^\infty,\mathbb{R}^2)$.
\end{corollary}

\begin{proof}
	Lemma \ref{lem:Borel_computation_of_components} implies that there is a Borel map $f:\mathfrak{M}(\mathsf{Top},\mathbb{R}^2)\to\mathfrak{C}(\mathsf{Top},\mathbb{R}^2)^\mathbb{N}$ such that for each $(\mathcal{U},c)\in\mathfrak{M}(\mathsf{Top},\mathbb{R}^2)$, $f(\mathcal{U},c)=\langle (\mathcal{V}_i,c|_{\mathcal{V}_i})\mid i\in\mathbb{N}\rangle$ is a sequence of (parameters for) connected surfaces which are exactly the connected components of $M_{(\mathcal{U},c)}$. The map given therein may have repetition, i.e., different $(\mathcal{V}_i,c|_{\mathcal{V}_i})$ may represent the same component of $M_{(\mathcal{U},c)}$, but by checking one-by-one whether this is the case using Lemma \ref{lem:saturation} and dropping those which represent the same open sets as previous ones, we are able to eliminate any repetition in a Borel fashion. Thus, $f:\mathfrak{M}(\mathsf{Top},\mathbb{R}^2)\to\mathfrak{C}(\mathsf{Top},\mathbb{R}^2)^\mathbb{N}$ is a Borel reduction of $\cong_{\mathsf{Top}}$ on $\mathfrak{M}(\mathsf{Top},\mathbb{R}^2)$ to $\cong_\mathsf{Top}^\mathbb{N}$ on $\mathfrak{C}(\mathsf{Top},\mathbb{R}^2)^\mathbb{N}$. By Theorem \ref{thm:surfaces} and Lemma \ref{lem:E^N_classifiable}, the latter is classifiable by countable structures, and thus so is the former. The same argument applies in the smooth case.
\end{proof}
 
Lastly, we wish to point out a descriptive set-theoretic reason for focusing on connected manifolds in our analysis of the complexity of different classification problems. Consider the equivalence relation $E_{\mathrm{ctbl}}$, defined on $\mathbb{R}^\mathbb{N}$ by
\[
	(x_i)_{i\in\mathbb{N}}\,E_{\mathrm{ctbl}}\,(y_i)_{i\in\mathbb{N}} \quad\text{if and only if}\quad \{x_i:i\in\mathbb{N}\}=\{y_i:i\in\mathbb{N}\},
\]
where the right-hand side is considered as equality between two countable sets of reals. Notice that $E_{\mathrm{ctbl}}$ is the orbit equivalence relation of the action of $S_\infty$ on $\mathbb{R}^\mathbb{N}$ given by permuting indices. It is known that $E_{\mathrm{ctbl}}$ is not essentially countable (see \cite[Exer.\ 8.3.3]{MR2455198}). Suppose that $\mathcal{M}$ is some Borel class of manifolds which is closed under countable disjoint unions and for which the relevant (analytic) notion of isomorphism $\cong$ can be determined by showing that connected components are pairwise equivalent. We claim that if there is an uncountable Borel set $\mathcal{A}$ of pairwise inequivalent connected manifolds in $\mathcal{M}$, then $E_{\mathrm{ctbl}}$ Borel reduces to $\cong$. For such an $\mathcal{A}$, there is a Borel bijection $f:\mathbb{R}\to\mathcal{A}$. Consider the set $\mathcal{A}^\mathbb{N}$ of sequences in $\mathcal{A}$, which we may map to their disjoint unions considered as elements of $\mathcal{M}$. Together with $f$, this mapping induces a Borel reduction of $E_{\mathrm{ctbl}}$ to $\cong$ on $\mathcal{M}$. Consequently, $E_{\mathrm{ctbl}}$ is the \emph{minimum} complexity that we can expect for the associated classification problem on $\mathcal{M}$, unless we restrict to connected manifolds. As we shall see in the following sections, there are natural classification problems for classes of connected manifolds which have intermediate complexity, strictly between the concretely classifiable and $E_{\mathrm{ctbl}}$.

\section{The bireducibility of isometry and conjugacy problems}
\label{section:bireducibility}

The focus of this section is the well-known correspondence between complete connected hyperbolic $n$-manifolds and discrete subgroups of $\mathrm{Isom}(\mathbb{H}^n)$.
This is, of course, an instance of a much wider correspondence between discrete subgroups of groups acting by homeomorphism on topological spaces and their associated quotient spaces, one in which the Cartan--Hadamard Theorem centrally figures (\cite[Ch.\ II.4]{MR1744486}; see also its Thm.\ III.$\mathcal{G}$.1.13).
Our focus on the hyperbolic case is for simplicity;
the task of the present section is simply to show:
\begin{enumerate}[label=\textup{\arabic*.}]
\item that the standard conversions from such manifolds to discrete groups and back are each, within our framework, Borel, and, moreover,
\item that they are \emph{reductions} of the relevant equivalence relations, or in other words map isometric manifolds to conjugate subgroups, and vice versa.
\end{enumerate}
The second point follows from the standard constructions and theorems, which we now outline; in subsequent subsections, it will remain only to argue the first.

We precede our more formal argumentation with a brief review of the conversions in question.
Here again a number of closely related approaches are possible, but the keywords for what is probably the most intuitive conversion of manifolds to discrete groups are \emph{developing map} and \emph{holonomy}.
The basic points of this conversion are as follows (see \cite[\S 3.4--5]{MR1435975} or \cite[\S 8.4--5]{MR4221225}, for example, for more detailed and general accounts); throughout, we assume that $n\geq 2$:
\begin{itemize}
\item A hyperbolic structure (i.e., an $(\mathsf{Isom},\mathbb{H}^n)$-structure) on a connected manifold $M$ lifts to its universal cover $\tilde{M}$; recall that the latter may be viewed as the space of endpoint-preserving homotopy classes of paths in $M$ from some fixed basepoint $x\in M$.
\item If $M$ is a complete  connected hyperbolic $n$-manifold, then the \emph{developing map} defines a continuous surjection $s$ from $\tilde{M}$ onto $\mathbb{H}^n$. The idea of this map is a simple one: fix first a $z\in\mathbb{H}^n$ and extend the map $x\mapsto z$ to an isometry $e$ of a neighborhood of $x$ in $M$ with a neighborhood of $z$ in $\mathbb{H}^n$. Any ``point'' $[\gamma]\in\tilde{M}$ (i.e., endpoint-preserving homotopy class of a path $\gamma$ from $x$ to some $y\in M$) traverses a series of charts in $M$, which patch together via transition maps at overlaps to constitute a unique path extending $e\circ\gamma$ in $\mathbb{H}^n$, and we let $s([\gamma])$ equal its endpoint. The argument that the function $s$ is well-defined is facilitated by the assumption, authorized by Lemma \ref{lem:reparametrization}, that the charts and overlaps in the atlas defining $M$ are all simply-connected; it also critically leverages the rigidity of $\mathrm{Isom}(\mathbb{H}^n)$, in the sense flagged in Example \ref{ex:isometry_pseudogroups}. Note also that the focus which this construction fosters on complete connected manifolds is, from our perspective, already a natural one, by Lemmas \ref{lem:connected_Borel} and \ref{lem:complete} and the discussion in Section \ref{subsection:disconnected}.
\item The aforementioned rigidity amounts simply to the fact that any element of the pseudogroup $\mathsf{Isom}$ with connected nonempty domain uniquely extends to an element $g$ of $\mathrm{Isom}(\mathbb{H}^n)$, and hence that the association of each $\gamma$ above to a finite sequence of transition maps on connected overlaps in fact defines, via the sequential composition of their unique extensions, a map $t:\tilde{M}\to \mathrm{Isom}(\mathbb{H}^n)$. (Again, the verification that this map is well-defined is standard.) Writing $p$ for the covering map $\tilde{M}\to M$, the restriction of $t$ to $p^{-1}(x)$ then represents $\pi_1(M,x)$ as a discrete subgroup $\Gamma$ of $\mathrm{Isom}(\mathbb{H}^n)$ (this is its \emph{holonomy} representation), one with the property that the quotient structure $\mathbb{H}^n/\Gamma$ induced by the natural action of $\mathrm{Isom}(\mathbb{H}^n)$ on $\mathbb{H}^n$ is isometric to $M$.
Moreover, since $M$ is a hyperbolic manifold, and since any torsion element of $\Gamma$ would fix a point of $\mathbb{H}^n$ whose quotient neighborhoods could never then be isometric to an open subset of $\mathbb{H}^n$, we find that $\Gamma$ must be torsion-free.
\end{itemize}
The reverse conversion, of torsion-free discrete subgroups $\Gamma$ of $\mathrm{Isom}(\mathbb{H}^n)$ into hyperbolic $n$-manifolds, is even simpler; as noted, this is just the map $\Gamma\mapsto\mathbb{H}^n/\Gamma$ to the quotient by the natural action of $\Gamma$. In our framework, though, a manifold is a collection $(\mathcal{U},c)$ of charts, so we need to work a little harder: we enumerate a dense subset $Q=(q_i)_{i\in\mathbb{N}}$ of $\mathbb{H}^n$ and let each $U_i$ be a ball of maximal radius about $q_i$ such that $U_i\cap g\cdot U_i=\varnothing$ for any $g\in\Gamma\backslash\{\mathrm{id}\}$; natural transition maps $\varphi_{i,j}$ associate to this family, and we verify in Section \ref{subsection:manifolds_discrete} that this construction is both as desired and, in the appropriate sense, Borel. 

Readers for whom these conversions are unfamiliar are encouraged to consider the relation between discrete subgroups of $\mathrm{Isom}^{+}(\mathbb{R})\cong\mathbb{R}$ and the associated induced metrics on topological copies of $\mathbb{S}^1$. What this example in its transparency obscures (since $\mathrm{Isom}^{+}(\mathbb{R})$ is abelian), however, is the interconnection of the relations of conjugacy on the group side with those of isometry on the manifold side.
The conversion of a manifold $M$ to a discrete group described above involved the choice of an isometry of a neighborhood of some point $x\in M$ with an open set in $\mathbb{H}^n$. It is not difficult to see that varying these choices corresponds to conjugating the output $\Gamma$ by elements of $\mathrm{Isom}(\mathbb{H}^n)$, and it is a theorem that the isometries of hyperbolic manifolds amount to just these variations; the general form of \cite[Thm.\ 1.18]{MR1638795} (or see \cite[Cor.\ 3.5.12]{MR1435975}), for example, is the following:

\begin{theorem*}
For any complete connected hyperbolic $n$-manifold $M$, there exists a torsion-free discrete subgroup $\Gamma$ of $\mathrm{Isom}(\mathbb{H}^n)$ such that $M\cong\mathbb{H}^n/\Gamma$. Such a $\Gamma$ is unique up to conjugation by elements of $\mathrm{Isom}(\mathbb{H}^n)$, and conversely, $\mathbb{H}^n/\Gamma$ is, for any such $\Gamma$, a complete connected hyperbolic $n$-manifold.
\end{theorem*}

After a brief review of spaces of discrete subgroups, we record Borel measurable realizations of the two directions of this correspondence in Sections \ref{subsection:discrete_from_manifolds} and \ref{subsection:manifolds_discrete}.

\subsection{Spaces of discrete subgroups}
\label{subsection:spaces_discrete}

We begin by describing the spaces of subgroups to and from which our manifold parametrizations will usefully reduce.  Recall from Section \ref{subsection:spaces_of_subsets} that, given a locally compact Polish group $G$, $\mathcal{F}(G)$ denotes the Polish space of closed subsets of $G$, endowed with the Chabauty--Fell topology. The following aggregates results recorded in both \cite[I.3]{MR903850} and \cite[E.1]{MR1219310}. 

\begin{lemma}
\label{lem:D_Dtf}
Let $G$ be a locally compact Polish group.
\begin{enumerate}[label=\textup{\arabic*.}]
\item The closed subgroups of $G$ form a closed subset $\mathcal{S}(G)$ of $\mathcal{F}(G)$.
\item The set $\mathcal{D}(G)$ of discrete subgroups of $G$ is open in $\mathcal{S}(G)$.
\item The set $\mathcal{S}_{\mathrm{tf}}(G)$ of torsion-free subgroups of $G$ is closed in $\mathcal{S}(G)$.
\end{enumerate}
In particular, $\mathcal{D}(G)$, $\mathcal{S}_{\mathrm{tf}}(G)$, and $\mathcal{D}_{\mathrm{tf}}(G)=\mathcal{D}(G)\cap\mathcal{S}_{\mathrm{tf}}(G)$ each inherit a Polish subspace topology from $\mathcal{F}(G)$.
\end{lemma}
\begin{proof} All assertions except the last appear in the aforementioned references; the framing there is in terms of second countable Lie groups, but a locally compact Polish structure in each case clearly suffices. The last assertion is then immediate from the fact that the subspace topology on countable intersections of open subsets of Polish spaces is Polish.
\end{proof}

Observe that $G$ acts continuously on $\mathcal{S}(G)$ by conjugation, and the subsets $\mathcal{D}(G)$, $\mathcal{S}_{\mathrm{tf}}(G)$, and $\mathcal{D}_{\mathrm{tf}}(G)$ are all invariant under this action. 
We will write $E(G,-)$ for the resulting orbit equivalence relation on the relevant subset $(-)$ of $\mathcal{D}(G)$. In particular, note that $E(G,\mathcal{D}_{\mathrm{tf}}(G))\leq_B E(G,\mathcal{D}(G))$ via the inclusion of $\mathcal{D}_{\mathrm{tf}}(G)$ in $\mathcal{D}(G)$.
As noted, our focus for the remainder of this section will be the groups $G=\mathrm{Isom}(\mathbb{H}^n)$ for $n\geq 2$; in Sections \ref{section:isometry_for_2} and \ref{section:isometry_for_3}, we will tend to approach these groups by way of their subgroups $\mathrm{Isom}^+(\mathbb{H}^n)$ of orientation-preserving isometries, whose theory is, in oddly compatible senses, both simpler and richer (in part for its interactions with Riemannian theory; $\mathrm{Isom}^+(\mathbb{H}^n)$ is the group of conformal homeomorphisms of $\mathbb{H}^n$ \cite[\S A.4]{MR1219310}).
Such an approach is fairly standard, and natural, given how frequently results in the orientable setting generalize; in our own context, the main tool for this kind of upgrade is provided by Lemma \ref{lem:lifting_orientation_preserving} below.
Note lastly that $\mathrm{Isom}(\mathbb{H}^n)$, and hence $\mathrm{Isom}^+(\mathbb{H}^n)\leq\mathrm{Isom}(\mathbb{H}^n)$, is endowed with the compact-open topology, from which it is immediate that each group acts continuously on $\mathbb{H}^n$.

\subsection{Discrete groups from manifolds}
\label{subsection:discrete_from_manifolds}

The task of this subsection is to show that the conversion outlined in this section's introduction, of hyperbolic manifolds into discrete groups, is Borel.
More formally, for any $n\geq 1$, let $\mathfrak{C}_c^*(\mathsf{Isom},\mathbb{H}^n)$ denote the restriction of $\mathfrak{M}^{*}(\mathsf{Isom},\mathbb{H}^n)$ (see equation (\ref{eq:metric_manifolds}) of Section \ref{subsection:subclasses} and the discussion surrounding it) to parametrizations of complete connected $(\mathsf{Isom},\mathbb{H}^n)$-manifolds and $\cong_{\mathsf{Isom}(\mathbb{H}^n)}$ the isometry relation on it.
 
We will prove the following theorem.

\begin{theorem}
\label{thm:manifoldstodiscrete}
There exists a Borel function $\nu:\mathfrak{C}^*_c(\mathsf{Isom},\mathbb{H}^n)\to\mathcal{D}_{\mathrm{tf}}(\mathrm{Isom}(\mathbb{H}^n))$ witnessing that $\cong_{\mathsf{Isom}(\mathbb{H}^n)}\;\leq_B E(\mathrm{Isom}(\mathbb{H}^n),\mathcal{D}_{\mathrm{tf}}(\mathrm{Isom}(\mathbb{H}^n))$.
\end{theorem}
\begin{proof}
Observe that, by the lemmas of Section \ref{subsection:subclasses}, $\mathfrak{C}^*_c(\mathsf{Isom},\mathbb{H}^n)$ is Borel and that we may, moreover, precede $\nu$ with the reparametrizing function $r$ 

of Lemma \ref{lem:reparametrization}; we may and will, in other words, without loss of generality assume that the parametrized atlases therein all have geodesically convex charts and overlap.
Next, enumerate in ordertype $\omega$ a dense subset $Q$ of $\mathbb{H}^n$, and for any  $(\mathcal{U},c)$ let $\bar{q}_{(\mathcal{U},c)}$ be the $q\in Q\cap U_0$ of minimum index.
Any $\gamma:[0,1]\to M_{(\mathcal{U},c)}$ representing an element of $\pi_1(M_{(\mathcal{U},c)},\bar{q}_{(\mathcal{U},c)})$ is covered by (the images of) a finite sequence of $U_{i(j)}\in\mathcal{U}$ $(0\leq j\leq k)$, one in which we may take $U_{i(0)}=U_{i(k)}=U_0$.
As described in this section's introduction (wherein the identification of the quotient-image of $U_0$ with $U_0$ plays the role of $e$), $\gamma$ thereby traverses a series of transition maps $\varphi_{i(j),i(j+1)}$ $(0\leq j<k)$, each of which uniquely extends to an element $g_j$ of $\mathrm{Isom}(\mathbb{H}^n)$. The following facts are standard; see \cite[\S 8.4]{MR4221225} or \cite[I.1.4]{MR903850}:
\begin{itemize}
\item the map taking $\gamma$ to $g_0\cdot \cdots \cdot g_{k-1}\in \mathrm{Isom}(\mathbb{H}^n)$ is unaffected by endpoint-preserving homotopies of $\gamma$;
\item the image $\nu(\mathcal{U},c)$ of the induced homomorphism of $\pi_1(M_{(\mathcal{U},c)},\bar{q}_{(\mathcal{U},c)})$ into $\mathrm{Isom}(\mathbb{H}^n)$ is both discrete and torsion-free.
\end{itemize}
Our task is simply to show that the assignment $(\mathcal{U},c)\mapsto \nu(\mathcal{U},c)$ defines a Borel map from $\mathfrak{C}^{*}_c(\mathsf{Isom},\mathbb{H}^n)$ to $\mathcal{D}_{\mathrm{tf}}(\mathrm{Isom}(\mathbb{H}^n))$.
To that end, it will suffice to show that the $\nu$-preimages of the sets $\{F\in\mathcal{F}(\mathrm{Isom}(\mathbb{H}^n))\mid F\cap U\neq\varnothing\}$, where $U$ ranges over the open subsets of $\mathrm{Isom}(\mathbb{H}^n)$, are Borel, since these sets generate the Effros Borel structure on $\mathcal{F}(\mathrm{Isom}(\mathbb{H}^n))$.
Observe therefore that $\nu^{-1}[\{F\in\mathcal{F}(\mathrm{Isom}(\mathbb{H}^n))\mid F\cap U\neq\varnothing\}]$ is the set of those $(\mathcal{U},c)$ such that there exist
\begin{enumerate}[label=\textup{\roman*.}]
\item rational $0=a_0<\cdots<a_{k+1}=1$, and
\item continuous maps $\gamma_j:[a_j,a_{j+1}]\to U_{i(j)}$ for $0\leq j\leq k$, with $i(0)=i(k)=0$ and $\gamma_0(0)=\gamma_{k}(1)=\bar{q}_{(\mathcal{U},c)}$ and $g_j(\gamma_{j}(a_{j+1}))=\gamma_{j+1}(a_{j+1})$ for $0\leq j<k$, where $g_j$ is that element of $\mathrm{Isom}(\mathbb{H}^n)$ uniquely extending $\varphi_{i(j),i(j+1)}$ (and the operation of $\mathrm{Isom}(\mathbb{H}^n)$ on points of $\mathbb{H}^n$ is that induced by the action of $\mathrm{Isom}(\mathbb{H}^n)$ on $\mathbb{H}^n$),
\item such that $g_0\cdot \cdots \cdot g_{k-1} \in U$.
\end{enumerate}
On its surface, this is a $\mathbf{\Sigma}^1_1$ condition.
Note, however, that by our assumptions deriving from Lemma \ref{lem:reparametrization} we may replace (ii) with
\begin{enumerate}
\item[ii'.] maps $p_j:\{a_j,a_{j+1}\}\to Q\cap U_{i(j)}$ for $0\leq j\leq k$, with $i(0)=i(k)=0$ and $p_0(0)=p_k(1)=\bar{q}_{(\mathcal{U},c)}$, and such that $p_j(a_{j+1})$ and $p_{j+1}(a_{j+1})$ both fall in $U_{i(j),i(j+1)}$ for $0\leq j<k$.
\end{enumerate}
This is because, by the convexity of the regions in question, the functions $p_j$ will determine uniform parametrizations $\gamma_j:[a_j,a_{j+1}]\to U_{i(j)}$ of the unique geodesics from $p_j(a_j)$ to $p_j(a_{j+1})$ and from $p_{j+1}(a_{j+1})$ to $p_{j+1}(a_{j+1})$ which, concatenated, are homotopic to those of the type described in item (ii), and moreover, that any ``loop'' as in item (ii) is homotopic to one of this form. Since items (i), (ii'), and (iii) only quantify over countably many Borel conditions, then, $\nu^{-1}[\{F\in\mathcal{F}(\mathrm{Isom}^+(\mathbb{H}^n))\mid F\cap U\neq\varnothing\}]$ is Borel, as desired.
\end{proof}

Observe that the above argument will, with only superficial modifications, more generally apply to define a Borel function $\mathfrak{C}^*_c(\mathcal{G},X)\to\mathcal{D}(G)$ reducing the relation of $(\mathcal{G},X)$-isomorphism to that of $G$-conjugacy whenever the conditions of Lemma \ref{lem:reparametrization} hold (namely, that $X$ is locally geodesically convex and locally compact Polish, and $\mathcal{G}$ is a Borel subpseudogroup of $\mathsf{Isom}$).

\begin{remark}
\label{rmk:baseframes}
It's standard to topologize \emph{spaces of hyperbolic manifolds} by way of the Chabauty--Fell topology on $\mathcal{D}(\mathrm{Isom}(\mathbb{H}^n))$, together with identifications $M\mapsto\Gamma_M$ much like those constructed just above.
More precisely, both \cite{MR903850} and \cite{MR1219310} topologize \emph{the space $\mathcal{MF}^n$ of baseframed complete connected hyperbolic $n$-manifolds} via its bijection with $\mathcal{D}_{\mathrm{tf}}(\mathrm{Isom}(\mathbb{H}^n))$, then endow \emph{the space $\mathcal{MB}^n$ of basepointed complete connected hyperbolic $n$-manifolds} and \emph{the space $\mathcal{MW}^n$ of unbasepointed complete connected hyperbolic $n$-manifolds} each with the quotient topologies associating to the forgetful functors $\mathcal{MF}^n\to\mathcal{MB}^n$ and $\mathcal{MF}^n\to\mathcal{MW}^n$.
These conversions are a further reason the Chabauty--Fell topology is sometimes termed the \emph{geometric topology}: natural notions of convergence of manifold structures in more plainly geometric senses (the $(\varepsilon,r)$-relations of \cite[I.3.2]{MR903850} or the formally stronger relations of \cite[pp.\ 166--168]{MR1219310} or the \emph{smooth topology} of \cite{BLL26+} exploited in the proofs of Theorem \ref{thm:fgPSL2Rsmooth} below) turn out to coincide with those inherited, via these conversions, from the Chabauty--Fell topology on $\mathcal{D}(\mathrm{Isom}(\mathbb{H}^n))$.
Within this framework, however, the space $\mathcal{MW}^n$ of hyperbolic $n$-manifolds without distinguished basepoint fails to be Hausdorff.
A main heuristic for this fact is recorded in \cite[p.\ 68]{MR903850}; there, distinct sequences in $\mathcal{MB}^2$ which map to a single sequence in $\mathcal{MW}^2$, but whose limits do not map to a single point, are described. Against this background, two main remarks are in order.

First, observe that the quotient of $\mathfrak{C}^*_c(\mathsf{Isom},\mathbb{H}^n)$ by $\cong_{\mathsf{Isom}(\mathbb{H}^n)}$ is, as a set, naturally isomorphic to $\mathcal{MW}^n$.
Observe next that our computation, in Corollary \ref{cor:isometry_universal} below, that $\cong_{\mathsf{Isom}(\mathbb{H}^n)}$ is essentially universal countable is a significantly stronger and more precise result than the non-Hausdorffness of $\mathcal{MW}^n$; it carries, for example, the heuristic not only that the $\mathcal{MW}^n$ parametrization of hyperbolic $n$-manifolds isn't Hausdorff (or even $T_0$; see Lemma \ref{lem:Glimm}), but that no meaningful one can be.
In contrast, by the analyses of Sections \ref{section:isometry_for_2} and \ref{section:isometry_for_3}, failures of Hausdorffness for the restrictions of $\mathcal{MW}^2$ and $\mathcal{MW}^3$ to geometrically finite manifolds are of a weaker sort:
those analyses furnish each of these restrictions with a natural standard Borel structure, or, in other words, with a Borel isomorphism with a Polish space.

Second, the $\nu$ of Theorem \ref{thm:manifoldstodiscrete} is, in essence, defined by viewing $\mathfrak{C}^*_c(\mathsf{Isom},\mathbb{H}^n)$ as a space of atlases on ``covertly baseframed'' manifolds; more precisely, the choice of a dense $Q\subseteq \mathbb{H}^n$ and canonical framing of $\mathbb{H}^n$ allows us to see $\mathfrak{C}^*_c(\mathsf{Isom},\mathbb{H}^n)$ as such in a Borel way.
A similar enhancement is at work in the arguments of Section \ref{section:isometry_for_2} below.
\end{remark}

\subsection{Manifolds from discrete groups}
\label{subsection:manifolds_discrete}
In this subsection, we describe a Borel reduction reversing that of Theorem \ref{thm:manifoldstodiscrete}. An immediate corollary is that the relations of equivalence among complete connected hyperbolic $n$-manifolds and of conjugacy among discrete torsion-free subgroups of $\mathrm{Isom}(\mathbb{H}^n)$ are not just Borel bireducible; they are classwise Borel equivalent as well.
\begin{theorem}
\label{thm:discretetomanifolds} 
For any $n\geq 1$, there exists a Borel function $\sigma:\mathcal{D}_{\mathrm{tf}}(\mathrm{Isom}(\mathbb{H}^n))\to\mathfrak{C}^*_c(\mathsf{Isom},\mathbb{H}^n)$ witnessing that $E(\mathrm{Isom}(\mathbb{H}^n),\mathcal{D}_{\mathrm{tf}}(\mathrm{Isom}(\mathbb{H}^n))\leq_B\;\cong_{\mathsf{Isom}(\mathbb{H}^n)}$.
\end{theorem}
In essence, our task is to show that the conversion of a discrete $\Gamma$ to the manifold $\mathbb{H}^n/\Gamma$ factors through $\mathfrak{C}^*_c(\mathsf{Isom},\mathbb{H}^n)$ in a Borel way.
Useful for this purpose will be the following variant of \cite[\S 6.6, Lem.\ 1]{MR4221225}:
\begin{lemma}
\label{lem:Ratcliffe_variant}
For any torsion-free discrete group $\Gamma\leq\mathrm{Isom}(\mathbb{H}^n)$ and compact $K\subseteq\mathbb{H}^n$, there exists an $\ell>0$ such that $d(x,y)>\ell$ for all $x\in K$ and $y\in\{g\cdot x\mid g\in\Gamma\backslash\{\mathrm{id}\}\}\cap K$.
\end{lemma}
\begin{proof}[Proof of Theorem \ref{thm:discretetomanifolds}] Recall the dense $Q=\{q_i\mid i\in\mathbb{N}\}\subseteq \mathbb{H}^n$ of the proof of Theorem \ref{thm:manifoldstodiscrete}, fix $\Gamma\in\mathcal{D}_{\mathrm{tf}}(\mathrm{Isom}(\mathbb{H}^n))$,
and for any $x\in\mathbb{H}^n$, let $D_\Gamma(x)$ denote the open ball about $x$ of maximal radius such that $g\cdot D_\Gamma(x)\cap D_\Gamma(x)=\varnothing$ for all $g\in\Gamma\backslash\{\mathrm{id}\}$.\footnote{More natural here might be to take Dirichlet fundamental domains centered at $q_i$. This doesn't, unfortunately, simplify our argument, except when the geometric and polyhedral topologies coincide, but the authors are only aware of a proof of this relation in dimension $3$, that of \cite{MR1060898}.}
Note that $D_\Gamma(x)$ is geodesically convex.
Let $\sigma$ of such a $\Gamma$ equal the $(\mathcal{U},c)$ defined by:
\begin{itemize}
\item $U_i=D_\Gamma(q_i)$ for each $i\in\mathbb{N}$;
\item $U_{i,j}=\bigcup_{g\in\Gamma}\big(D_{\Gamma}(q_i)\cap g[D_{\Gamma}(q_j)]\big)$ for each $i,j\in\mathbb{N}$;
\item $\varphi_{i,j}:U_{i,j}\to U_{j,i}$ equals $$\bigcup_{g\in\Gamma}g^{-1}\big|_{(D_{\Gamma}(q_i)\cap g[D_{\Gamma}(q_j)])}$$
for each $i,j\in\mathbb{N}$.
\end{itemize}
The theorem will then follow from two claims: first, that $M_{(\mathcal{U},c)}$ is isometric to $\mathbb{H}^n/\Gamma$, and second, that the function $\sigma$  is Borel.

For the first item, note that as the $U_i$ all canonically embed into $\mathbb{H}^n/\Gamma$, in a manner which is moreover compatible with the transition maps $\varphi_{i,j}$, the question reduces to that of whether these $U_i$ cover $\mathbb{H}^n/\Gamma$; they will, of course, if they in fact cover $\mathbb{H}^n$.
We argue the latter by contradiction: suppose there exists an $x\in\mathbb{H}^n\backslash\bigcup_{i\in\mathbb{N}}U_i$, and fix a sequence $\bar{q}=(q_{j(i)})_{i\in\mathbb{N}}$ in $Q$ converging to $x$.
Let $K$ denote the compact unit ball centered at $x$ and take $\ell$ as in Lemma \ref{lem:Ratcliffe_variant}.
Since for any $q\in\bar{q}\cap K$ we have both that $d(q,g\cdot q)>\ell$ for all $g\cdot q\in\{g\cdot x\mid g\in\Gamma\backslash\{\mathrm{id}\}\}\cap K$ and that $d(x,s)\leq d(q,s)+d(x,q)$, and hence that $1-d(x,q)\leq d(q,s)$, for all $s\in\mathbb{H}^n\backslash K$, the radii of the sets $U_{j(i)}$ are uniformly bounded below. Since the distances $d(q_{j(i)},x)$ grow arbitrarily small, this shows that eventually $x\in U_{j(i)}$, the contradiction desired.

To see that the function $\sigma$ is Borel, recall the argument of Lemma \ref{lem:reparametrization}, Claim 1, and in particular its notation $[B(k,\ell)]$ for the collection of elements of $(\mathcal{O}(\mathbb{H}^n)\times\mathsf{Isom})^{\mathbb{N}\times\mathbb{N}}$ falling within the Borel set $B$ on their $(k,\ell)^{\mathrm{th}}$ coordinate. As before, we will argue that for arbitrary $(k,\ell)\in\mathbb{N}\times\mathbb{N}$, $\sigma^{-1}([B(k,\ell)])$ is a Borel subset of $\mathcal{D}_{\mathrm{tf}}(\mathrm{Isom}^+(\mathbb{H}^n))$ for any $B$ of the form:
\begin{enumerate}[label=\textup{\arabic*.}]
\item $\{U\in\mathcal{O}(\mathbb{H}^n)\mid U \supseteq K \}\times\mathsf{Isom}$ for some compact $K\subseteq\mathbb{H}^n$, or
\item $\mathcal{O}(\mathbb{H}^n)\times E$ for $E$ a Borel subset of $\mathsf{Isom}$.
\end{enumerate}
Note first that for any $k\in\mathbb{N}$ and $B$ as in (1) for some fixed compact $K\subseteq\mathbb{H}^n$,
$$\sigma^{-1}([B(k,k)])=\{\Gamma\mid D_\Gamma(q_k)\supseteq K\} = \bigcup_{\varepsilon > r}\sigma^{-1}([B'_\varepsilon(k,k)]),$$
where $B'_\varepsilon$ is the open ball about $q_k$ of radius $\varepsilon$ and $r = \max_{y\in K} d(q_k,y)$. By \cite[Thm.\ E.1.10]{MR1219310}, for any $s>0$ the set $\bigcup_{\varepsilon \geq s}\sigma^{-1}([B'_\varepsilon(k,k)])$ is compact; it follows that $\sigma^{-1}([B(k,k)])$ is a countable union of compact sets, and Borel in particular.

For any $k\neq \ell$ in $\mathbb{N}$ and $B$ as in (1) for some fixed compact $K\subseteq\mathbb{H}^n$, on the other hand,
$$\sigma^{-1}([B(k,\ell)])=\sigma^{-1}([B(k,k)])\cap C,$$
where $[B(k,k)]$ is as in the previous paragraph, and $C$ is as we will now describe.
The preceding paragraph may be read as arguing that the map $h:\mathcal{D}_{\mathrm{tf}}(\mathrm{Isom}^+(\mathbb{H}^n))\to\mathcal{O}(\mathbb{H}^n):\Gamma\mapsto D_\Gamma(q_k)$ is Borel, and as we will show in a moment, it follows that the map $$f:\mathcal{D}_{\mathrm{tf}}(\mathrm{Isom}(\mathbb{H}^n))\to\mathcal{O}(\mathbb{H}^n):\Gamma\mapsto \Gamma\cdot D_\Gamma(q_\ell)$$
is Borel as well.
The set $C$ is the preimage by $f$ of the set $\{U\in\mathcal{O}(\mathbb{H}^n)\mid U\supseteq K\}$, with the consequence that $\sigma^{-1}([B(k,\ell)])$ is an intersection of Borel sets, as desired.
To see that $f$ is Borel, we check that the restriction of the map $\mathcal{D}_{\mathrm{tf}}(\mathrm{Isom}(\mathbb{H}^n))\times\mathcal{O}(\mathbb{H}^n)\to\mathcal{O}(\mathbb{H}^n):(\Gamma,U)\mapsto \Gamma\cdot U$ to the range of $(\mathrm{id},h)$ is; observe that for this it will suffice to check that the set
$$Z_0=\{(\Gamma,U)\in\mathrm{ran}(\mathrm{id},h)\mid\exists g\in\Gamma\textnormal{ such that }g\cdot U\supseteq K\}$$
is Borel for any \emph{connected} compact $K\subseteq\mathbb{H}^n$. But this set is a relatively open subset of $\mathrm{ran}(\mathrm{id},h)$, for the reason that $Z_1=\{(g,U)\in\mathrm{Isom}(\mathbb{H}^n)\times\mathrm{ran}(h)\mid g\cdot U\supseteq K\}$ is relatively open in $\mathrm{Isom}(\mathbb{H}^n)\times\mathcal{O}(\mathbb{H}^n)$: if $g$ witnesses that $(\Gamma,U)\in Z_0$ then for a relatively open neighborhood $V\times W\cap\mathrm{ran}(h)\subseteq Z_1$ of $(g,U)$, the set
$$\{(\Gamma,U)\in\mathrm{ran}(\mathrm{id},h)\mid\Gamma\cap V\neq\varnothing\textnormal{ and }U\in W\}$$
is a relatively open subset of $Z_0$.

For item (2), recall from Definition \ref{defn:Borel_structure_Top} that the charts $\varphi_{k,\ell}:U_{k,\ell}\to U_{\ell,k}$ comprising the second coordinate in a manifold parameter $(\mathcal{U},c)$ are triples of the type (\emph{domain}, \emph{range}, \emph{sequence recording $\varphi_{k,\ell}$ outputs on a dense subset of its domain}). That $\sigma$-preimages of Borel subsets of its first or second coordinates are Borel we just argued.
Hence we may confine our attention to the third coordinate, which is, more formally, a sequence $(x_i)_{i\in\mathbb{N}}$ in $\mathbb{H}^n\cup\{u\}$ wherein $x_i=\varphi_{k,\ell}(q_i)$ if $q_i\in\mathrm{dom}(\varphi_{k,\ell})$ and $x_i$ otherwise equals the default variable $u$. Our task then further reduces to showing that the $\sigma$-preimage of the collection of $\mathsf{Isom}$ elements mapping a given $q_i$ to an arbitrary open $U\subseteq\mathbb{H}^n$ is Borel. But this is just $\{\Gamma\mid\Gamma\cdot q_i\cap U\neq\varnothing\}$, a relatively open subset of $\mathcal{D}_{\mathrm{tf}}(\mathrm{Isom}(\mathbb{H}^n))$ by the continuity of the action of $\mathrm{Isom}(\mathbb{H}^n)$ on $\mathbb{H}^n$, and this completes the argument.
\end{proof}
An immediate corollary of Theorems \ref{thm:manifoldstodiscrete} and \ref{thm:discretetomanifolds} is that $$\cong_{\mathsf{Isom}(\mathbb{H}^n)}\;\sim_B\,E(\mathrm{Isom}(\mathbb{H}^n),\mathcal{D}_{\mathrm{tf}}(\mathrm{Isom}(\mathbb{H}^n))$$
or, in plainer English, that the relations of isometry for complete connected hyperbolic $n$-manifolds and those of conjugacy for discrete torsion-free subgroups of $\mathrm{Isom}(\mathbb{H}^n)$ are Borel bireducible.
In fact, a stronger relation between them holds; see \cite[\S 2]{MR2833542} for further discussion of the following notion.
\begin{definition}
Equivalence relations $E$ and $F$ on standard Borel spaces $X$ and $Y$ are \emph{classwise Borel isomorphic} if there exist Borel functions $X\to Y$ and $Y\to X$ inducing inverse bijections of the sets $X/E$ and $Y/F$.
\end{definition}
The utility of this definition derives from the fact that Borel reductions do not, in general, satisfy the Schr\"{o}der--Bernstein property: \emph{a priori}, there may exist Borel lifts $X\to Y$ and $Y\to X$ of injections $X/E\to Y/F$ and $Y/F\to X/E$, respectively, without there existing Borel lifts of any bijection $X/E\to Y/F$ together with its inverse. The $\nu$ and $\sigma$ of Theorems \ref{thm:manifoldstodiscrete} and \ref{thm:discretetomanifolds}, respectively, are, on the other hand, Borel lifts of exactly the latter sort.

\begin{corollary}
\label{cor:classwise}
The equivalence relations $\cong_{\mathsf{Isom}(\mathbb{H}^n)}$ and $E(\mathrm{Isom}(\mathbb{H}^n),\mathcal{D}_{\mathrm{tf}}(\mathrm{Isom}(\mathbb{H}^n))$ of Theorems \ref{thm:manifoldstodiscrete} and \ref{thm:discretetomanifolds} are classwise Borel isomorphic.
\end{corollary}
\begin{proof}
That $\nu$ and $\sigma$ induce bijections follows from the discussion which introduced this section (along with its references); that these are inverse bijections follows from the stronger fact that $\nu\circ\sigma=\mathrm{id}$. For the latter, observe that $\sigma(\Gamma)$ is, in the sense discussed in Remark \ref{rmk:baseframes}, implicitly basepointed at $q_0$, and that $D_\Gamma(q_0)$ may, moreover, be regarded as encoding a frame at $q_0$. In fact, $\sigma$ is the realization in our setting of the map witnessing the homeomorphism $\mathcal{D}_{\mathrm{tf}}(\mathrm{Isom}(\mathbb{H}^n))\to\mathcal{MF}^n$ described in the orientation-preserving case at \cite[Prop.\ E.1.9]{MR1219310}, and $\nu$ is the realization of its left-inverse. (A subtle point does arise here, via the implicit precomposition, in the proof of Theorem \ref{thm:manifoldstodiscrete}, of $\nu$ with the reparametrizing $r$ of Lemma \ref{lem:reparametrization}: to conserve the data of $\sigma(\Gamma)$, we will require a version of $r$ taking any $(\mathcal{U},c)$ with $q_0\in U_0$ to a $(\mathcal{V},d)$ with $q_0\in V_0$; this presents no difficulty.)
\end{proof}

In closing, note that replacing the group $\mathrm{Isom}(\mathbb{H}^n)$ with its subgroup $\mathrm{Isom}^{+}(\mathbb{H}^n)$ alters essentially none of this section's arguments or results.
Rather than clutter things any further, we will occasionally reference Corollary \ref{cor:classwise} below as though it had been formulated for the latter.

\section{The conjugacy problem for discrete subgroups of Lie groups}
\label{section:conjugacy}

In this section, we will consider the conjugacy relation on discrete subgroups of more general Lie groups.
As noted in Lemma \ref{lem:D_Dtf}, the set $\mathcal{D}(G)$ of discrete subgroups of a locally compact Polish group $G$ is itself a Polish space when endowed with the Chabauty--Fell topology, and $G$ acts continuously on $\mathcal{D}(G)$ by conjugation. As above, we write $E(G,\mathcal{D}(G))$ for the resulting orbit equivalence relation, and $E(G,\mathcal{D}_{\mathrm{tf}}(G))$ for its restriction to the conjugation-invariant closed subset $\mathcal{D}_{\mathrm{tf}}(G)$ of torsion-free discrete subgroups of $G$.

We begin by noting two general facts about actions of locally compact Polish groups. First, if $G$ is compact, then \emph{any} orbit equivalence relation resulting from a continuous action of $G$ on a Polish space $X$ is concretely classifiable: each orbit is compact, being a continuous image of $G$, and can be viewed as a ``point'' in the space $\mathcal{K}(X)$ of all compact subsets of $X$, so the map which sends each $x\in X$ to its orbit is a continuous reduction to equality. In fact, when $G$ is a compact Lie group, it has only countably many conjugacy classes of closed subgroups \cite[Cor.\ 1.7.27]{MR177401}. For this reason, we confine our attention to noncompact groups.

Second, if $G$ is locally compact and acts in a Borel way on any standard Borel space, then the resulting orbit equivalence relation is essentially countable, that is, bireducible with a countable Borel equivalence relation \cite{MR1176624}. In particular, the conjugacy relation $E(G,\mathcal{D}(G))$ for such a $G$ is always essentially countable.

As mentioned in our introduction, Stuck and Zimmer \cite{MR1283875} showed that if $F_2$ is the (discrete) free group on two generators, then $E(F_2,\mathcal{D}(F_2))$ is not concretely classifiable; they then extended this to any noncompact semisimple Lie group. Thomas and Velickovic \cite{MR1700491} strengthened the result for $F_2$ by proving that $E(F_2,\mathcal{D}(F_2))$ is, in fact, a universal countable Borel equivalence relation (see \cite{MR1781575} for a simplified proof). Andretta, Camerlo, and Hjorth \cite{MR1815088} then generalized this to any countable group containing $F_2$: 

\begin{theorem}[Andretta--Camerlo--Hjorth]\label{thm:ACH}
	If $\Gamma$ is a countable group which contains a nonabelian free subgroup, then $E(\Gamma,\mathcal{D}_{\mathrm{tf}}(\Gamma))$ is a universal countable Borel equivalence relation.
\end{theorem}

Note that if $\Gamma$ contains $F_2$, then $\mathcal{D}(F_2)\subseteq\mathcal{D}_{\mathrm{tf}}(\Gamma)$ and $E(F_2,\mathcal{D}(F_2))$ is a subequivalence relation of $E(\Gamma,\mathcal{D}_{\mathrm{tf}}(\Gamma))$. However, containment between countable Borel equivalence relations does not, in general, imply Borel reducibility \cite{MR1903846}, and whether containing a universal countable Borel equivalence relation implies universality is a well-known open problem (see \cite{MR1900547} and \cite{MR2500091}).
The bulk of \cite{MR1815088} is devoted to establishing universality for specific such examples using a technical coding scheme; see also \cite{MR3651212} for related results and arguments.

Our main results in this section consist of generalizations of Theorem \ref{thm:ACH} to a variety of Lie groups. In particular, we show that if $G$ is any matrix group containing a discrete copy of $F_2$, then the conjugacy relation $E(G,\mathcal{D}_{\mathrm{tf}}(G))$ is essentially countable universal (Theorem \ref{thm:conjugacy_universal}). We then extend this to any noncompact semisimple Lie group (Corollary \ref{cor:semisimple_universal}), sharpening Stuck and Zimmer's result. Then, using the material of Section \ref{section:bireducibility}, we conclude that the isometry relation on hyperbolic $n$-manifolds is, likewise, essentially countable universal (Corollary \ref{cor:isometry_universal}),
for any $n\geq 2$. We conclude the section with some lemmas concerning restrictions of conjugacy relations to finitely generated subgroups which will be of use in Sections \ref{section:isometry_for_2} and \ref{section:isometry_for_3}.

\subsection{Matrix groups and their relatives}\label{sec:matrix_groups}

By a \emph{matrix group}, we mean a closed subgroup of $\mathrm{GL}(n,\mathbb{R})$ or $\mathrm{GL}(n,\mathbb{C})$. Our main result here is as follows:

\begin{theorem}\label{thm:conjugacy_universal}
	If $G$ is a matrix group which contains a discrete nonabelian free subgroup, then $E(G,\mathcal{D}_{\mathrm{tf}}(G))$ is essentially countable universal.
\end{theorem}

The key observation we use to prove Theorem \ref{thm:conjugacy_universal} is that, given a countable subgroup $\Gamma$ of a matrix group $G$, conjugacy of subgroups of $\Gamma$ by elements of $G$ can be reduced to conjugacy by elements of an intermediate, but still countable, group $\overline{\Gamma}$. This will enable us to leverage Theorem \ref{thm:ACH} and obtain a reduction. As a concrete example, when $G=\mathrm{SL}(n,\mathbb{R})$ and $\Gamma=\mathrm{SL}(n,\mathbb{Z})$, one can show explicitly that conjugation in $\Gamma$ by elements of $G$ can be reduced to conjugation by elements of $\overline{\Gamma}=\mathrm{SL}(n,\mathbb{Q})$; the idea is that conjugacy of two subsets corresponds to an over-determined consistent linear system, which can thus be solved in terms of the coefficients of the matrices involved. We caution that this intermediate group may fail to be discrete in $G$; it will typically be dense.

\begin{lemma}\label{lem:countable_retract}
	Let $G$ be a matrix group and $\Gamma$ a countable subgroup of $G$.  Then, there is a countable subgroup $\overline{\Gamma}$ of $G$ such that:
	\begin{enumerate}[label=\textup{\roman*.}]
		\item $\Gamma\leq\overline{\Gamma}$, and
		\item whenever $\Delta,\Delta'\subseteq\overline{\Gamma}$ are conjugate by a matrix in $G$, they are conjugate by a matrix in $\overline{\Gamma}$.
	\end{enumerate}
\end{lemma}

\begin{proof}
	Suppose that $G$ is a subgroup of $\mathrm{GL}(n,\mathbb{R})$; the complex case is identical. We construct $\overline{\Gamma}$ as follows: Begin by letting $\Gamma_0=\Gamma$. For each finitely many $A_0,\ldots,A_m,B_0,\ldots,B_m\in\Gamma_0$ such that there is some $X\in G$ for which
	\[
		XA_iX^{-1}=B_i
	\]
	for all $i\leq m$, we add one such witness $X$ to $\Gamma_0$. Let $\Gamma_1$ be subgroup of $G$ generated by $\Gamma_0$ and the countably many witnesses thus added. In particular, $\Gamma_1$ is countable. Repeat this construction to define an increasing chain of countable subgroups $\Gamma_n$ (for $n\in\mathbb{N}$) of $G$, so that whenever finitely many elements of $\Gamma_n$ are conjugate by an element of $G$, they are conjugate by an element of $\Gamma_{n+1}$. Let $\overline{\Gamma}=\bigcup_{n\in\mathbb{N}}\Gamma_n$. Then, $\overline{\Gamma}$ is a countable subgroup of $G$ satisfying (i) and having the property that if finitely many elements of $\overline{\Gamma}$ are conjugate by an element of $G$, then they are conjugate by an element of $\overline{\Gamma}$.\footnote{Logicians will recognize $\overline{\Gamma}$ as a kind of elementary submodel. In fact, any countable $\Sigma_1$-elementary submodel of $G$ containing $\Gamma$ would suffice, and can be obtained by an application of the L\"owenheim--Skolem Theorem.}
	
	Suppose that $\Delta,\Delta'$ are subsets of $\overline{\Gamma}$ which are conjugate via some $X\in G$:
\[
	X\Delta X^{-1}=\Delta'.
\]
	If $\Delta$ and $\Delta'$ are both finite, then they are conjugate via an element of $\overline{\Gamma}$, by construction. So, assume that $\Delta$ and $\Delta'$ are both infinite and enumerate them as $\Delta=\{A_m:m\in\mathbb{N}\}$ and $\Delta'=\{B_m:m\in\mathbb{N}\}$ in such a way that, for all $m\in\mathbb{N}$,
	\[
		XA_mX^{-1}=B_m,
	\]
	or equivalently,
	\[
		XA_m=B_m X.
	\]
	
	For each $m\in\mathbb{N}$, let
	\[
		\mathcal{C}_m=\{Y\in M(n,\mathbb{R}):YA_i=B_iY \text{ for all $i\leq m$}\}.
	\]
	Then, each $\mathcal{C}_m$ is a linear subspace of $M(n,\mathbb{R})$ containing $X$ and thus so is $\mathcal{C}=\bigcap_{m\in\mathbb{N}}\mathcal{C}_m$. As $M(n,\mathbb{R})$ is finite-dimensional and the $\mathcal{C}_m$'s form a decreasing chain of subspaces, there is some $m_0\in\mathbb{N}$ such that for all $m\geq m_0$, $\mathcal{C}_{m_0}=\mathcal{C}_m=\mathcal{C}$.
	
	Since $X\in \mathcal{C}_{m_0}$, we have that
	\[
		XA_iX^{-1}=B_i
	\]
	for all $i\leq m_0$, so by our construction of $\overline{\Gamma}$, there is a $Y\in\overline{\Gamma}$ such that
	\[
		YA_iY^{-1}=B_i
	\]
	for all $i\leq m_0$. Hence, $Y\in \mathcal{C}_{m_0}=\mathcal{C}$. But then,
	\[
		Y\Delta Y^{-1}=\Delta',
	\]
	as claimed.
\end{proof}

\begin{proof}[Proof of Theorem \ref{thm:conjugacy_universal}.]
	Suppose that $G$ is a matrix group and that $\Gamma$ is a discrete copy of the free group on two generators lying inside of $G$. Let $\overline{\Gamma}$ be as in Lemma \ref{lem:countable_retract} and denote by $\mathcal{D}_{\mathrm{tf}}(\overline{\Gamma})$ the space of all torsion-free subgroups of $\overline{\Gamma}$ \emph{which are discrete in $G$}. Then, $\overline{\Gamma}$ acts continuously by conjugation on $\mathcal{D}_{\mathrm{tf}}(\overline{\Gamma})$, and we can write $E_{\overline{\Gamma}}$ for the resulting orbit equivalence relation. By Lemma \ref{lem:countable_retract}(ii), the inclusion map $\mathcal{D}_{\mathrm{tf}}(\overline{\Gamma})\to\mathcal{D}_{\mathrm{tf}}(G)$ is a reduction of $E_{\overline{\Gamma}}$ to $E(G,\mathcal{D}_{\mathrm{tf}}(G))$.
	
	It remains to say why $E_{\overline{\Gamma}}$ is a universal countable Borel equivalence relation. Theorem \ref{thm:ACH} tells us that the conjugacy relation on \emph{all} torsion-free subgroups of $\overline{\Gamma}$ is universal, but we have restricted to the invariant set of those which are discrete in $G$, so we must further examine its proof. There, we see that the reduction (described on \cite[p.\ 203]{MR1815088}), from a particular known universal countable Borel equivalence relation to conjugacy, maps entirely into the space of subgroups of the arbitrary free subgroup fixed at the start of the proof. By taking that free subgroup to be the discrete $\Gamma$ living inside of our $\overline{\Gamma}$, we have that all of its further subgroups are, of course, discrete in $G$. Thus, when applied to our case, their construction shows that $E_{\overline{\Gamma}}$ is universal countable, which completes the proof.
\end{proof}

Not every Lie group is a matrix group, but many are closely related to via quotients. In general, if $G$ is a locally compact Polish group and $N$ a discrete normal subgroup, then $G/N$ is locally compact Polish and the quotient map $\pi:G\to G/N$ is continuous. If $G$ is a Lie group, then $G/N$ is also a Lie group, and $\pi$ is smooth. The quotient map induces a direct image map from subgroups of $G$ to subgroups of $G/N$, and an inverse image map going the other way. Direct images need not preserve discreteness; consider $G=\mathbb{R}$, $N=\mathbb{Z}$, and $\Gamma=t\mathbb{Z}$ for $t$ irrational. Inverse images, on the other hand, always preserve discreteness, and the resulting map is a reduction between conjugacy relations (cf.~\cite[Lem.\ 4.2]{MR2914864}).

\begin{lemma}\label{lem:quotient_reduction}
	Suppose that $G$ is a locally compact Polish group, $N$ a discrete normal subgroup of $G$, and $\pi:G\to G/N$ the quotient map. Then, the inverse image map $\pi^{-1}$ is a continuous injection $\mathcal{D}(G/N)\to\mathcal{D}(G)$ and a reduction from $E(G/N,\mathcal{D}(G/N))$ to $E(G,\mathcal{D}(G))$; likewise for the restrictions to $\mathcal{D}_{\mathrm{tf}}(G/N)$ and $\mathcal{D}_{\mathrm{tf}}(G)$.
\end{lemma}

\begin{proof}
	To see that $\pi^{-1}$ preserves discreteness, let $\Gamma\leq G/N$ be a discrete subgroup of $G/N$. Fix an open neighborhood $U$ of the identity, $eN=N$ in $G/N$, such that $\Gamma\cap U=\{eN\}$. Then,
	\[
		\pi^{-1}[\Gamma]\cap\pi^{-1}[U]=\pi^{-1}[\Gamma\cap U]=\pi^{-1}[\{eN\}]=N.
	\]	
	Let $V\subseteq\pi^{-1}[U]$ be an open neighborhood of the identity in $G$ such that $N\cap V=\{e\}$. If $g\in\pi^{-1}[\Gamma]\cap V$, then $g\in\pi^{-1}[\Gamma]\cap\pi^{-1}[U]=N$, so $g=e$. Thus, $\pi^{-1}[\Gamma]$ is discrete in $G$. That $\pi^{-1}:\mathcal{D}(G/N)\to\mathcal{D}(G)$ is injective is clear and that it is continuous follows from standard facts about the Chabauty--Fell topology.
	
	To see that $\pi^{-1}$ is a reduction, suppose that $\Delta,\Delta'\in\mathcal{D}(G/N)$ are such that $\pi(g)\Delta\pi(g)^{-1}=\Delta'$ for some $g\in G$. If $h\in\pi^{-1}[\Delta]$, then 
	\[
		\pi(ghg^{-1})=\pi(g)\pi(h)\pi(g)^{-1}\in\Delta',
	\]
	while if $h'\in\pi^{-1}[\Delta']$, then 
	\[
		\pi(g^{-1}h'g)=\pi(g)^{-1}\pi(h')\pi(g)\in\Delta,
	\] 
	showing that $g\pi^{-1}[\Delta]g^{-1}=\pi^{-1}[\Delta']$. Conversely, if $g\pi^{-1}[\Delta] g^{-1}=\pi^{-1}[\Delta']$ for some $g\in G$, then clearly $\pi(g)\Delta \pi(g)^{-1}=\Delta'$. Finally, note that $\pi^{-1}$ clearly preserves being torsion-free.
\end{proof}

Consequently, if $G$ is a Lie group which has a quotient, by a discrete normal subgroup, which is a matrix group containing a discrete nonabelian free subgroup, then $E(G,\mathcal{D}_{\mathrm{tf}}(G))$ is essentially countable universal. Examples of this sort include universal covers of matrix groups and all noncompact semisimple Lie groups.

To elaborate on the semisimple case, we will need to recall some basic facts from Lie theory, which we summarize here (we refer to \cite[Appendix B]{MR961261}; see also \cite{MR1920389}): Suppose that $G$ is a connected Lie group and $\mathfrak{g}$ its Lie algebra. Then, $G$ acts on $\mathfrak{g}$ by automorphisms via the \emph{adjoint representation} $G\to\mathrm{Ad}(G)$, where $\mathrm{Ad}(G)$ is a closed subgroup of $\mathrm{GL}(\mathfrak{g})$. In particular, $\mathrm{Ad}(G)$ is a matrix group. Since $G$ is connected, the kernel of the adjoint representation is the center $Z(G)$ of $G$ \cite[B12]{MR961261}, so $\mathrm{Ad}(G)\cong G/Z(G)$.

A Lie group $G$ is \emph{semisimple} if it is connected and has no nontrivial connected solvable normal subgroups; this implies that $Z(G)$ is discrete \cite[B38]{MR961261}. If $G$ is noncompact and semisimple, then these properties are inherited by $\mathrm{Ad}(G)$, as they can be characterized in terms of the Lie algebra (see \cite[B38, B48]{MR961261}), and so $\mathrm{Ad}(G)$ contains a discrete nonabelian free subgroup by \cite[Thm.\ 3.8]{MR961261}. Thus, $E(\mathrm{Ad}(G),\mathcal{D}_{\mathrm{tf}}(\mathrm{Ad}(G)))$ is essentially countable universal by Theorem \ref{thm:conjugacy_universal}, and hence so is $E(G,\mathcal{D}_{\mathrm{tf}}(G))$ by Lemma \ref{lem:quotient_reduction}. Thus, we have proved our strengthening of the Stuck and Zimmer result mentioned above:
\begin{corollary}\label{cor:semisimple_universal}
	If $G$ is a noncompact semisimple Lie group, then $E(G,\mathcal{D}_{\mathrm{tf}}(G))$ is essentially countable universal.\qed
\end{corollary}

We can also easily extend Theorem \ref{thm:conjugacy_universal} to quotients of matrix groups:

\begin{theorem}\label{thm:conjugacy_universal_in_quotients}
	Let $G$ be a matrix group and $N$ a discrete normal subgroup of $G$. If $G/N$ contains a discrete nonabelian free subgroup, then $E(G/N,\mathcal{D}_{\mathrm{tf}}(G/N))$ is essentially countable universal.
\end{theorem}

\begin{proof}
	It suffices to show that Lemma \ref{lem:countable_retract} applies to $G/N$, as the remainder of the proof of Theorem \ref{thm:conjugacy_universal} goes through unchanged. Suppose that $\Gamma$ is a countable subgroup of $G/N$. Since $N$ is countable, $\pi^{-1}[\Gamma]$ is a countable subgroup of $G$. Then, we can apply Lemma \ref{lem:countable_retract} to $\pi^{-1}[\Gamma]$ and $G$ to obtain a countable subgroup $\overline{\pi^{-1}[\Gamma]}$ of $G$ which contains $\pi^{-1}[\Gamma]$ and such that conjugation of subsets of $\overline{\pi^{-1}[\Gamma]}$ by elements of $G$ reduces to conjugation by elements of $\overline{\pi^{-1}[\Gamma]}$. 
	
	Let $\overline{\Gamma}=\pi[\overline{\pi^{-1}[\Gamma]}]=\overline{\pi^{-1}[\Gamma]}/N$. Then, $\overline{\Gamma}$ is a countable subgroup of $G/N$ which contains $\Gamma$. If $\Delta,\Delta'$ are subsets of $\overline{\Gamma}$ which are conjugate in $G/N$, then $\pi^{-1}[\Delta]$ and $\pi^{-1}[\Delta]$ are subsets of $\overline{\pi^{-1}[\Gamma]}$ which are conjugate in $G$, hence conjugate in $\overline{\pi^{-1}[\Gamma]}$. But then, as in the proof of Lemma \ref{lem:quotient_reduction}, $\Delta$ and $\Delta'$ must be conjugate in $\overline{\Gamma}$.
\end{proof}

We suspect that these results can be extended to all Lie groups which contain a discrete nonabelian free subgroup; for connected Lie groups, these are exactly those which are nonamenable as locally compact groups \cite[Thm.\ 3.8]{MR961261}. However, the methods used here seem specific to groups of matrices or related groups, and not applicable to the broader class of all locally compact Polish groups.

\subsection{The isometry problem for hyperbolic $n$-manifolds}

Recall that the group $\mathrm{Isom}(\mathbb{H}^n)$ of all isometries of $n$-dimensional hyperbolic space is a locally compact Polish group when endowed with the compact-open topology. In fact, $\mathrm{Isom}(\mathbb{H}^n)$ may be regarded as a matrix Lie group. The more familiar identifications of $\mathrm{Isom}^+(\mathbb{H}^2)$ and $\mathrm{Isom}^+(\mathbb{H}^3)$ with $\mathrm{PSL}(2,\mathbb{R})$ and $\mathrm{PSL}(2,\mathbb{C})$ will play a prominent role in Sections \ref{section:isometry_for_2} and \ref{section:isometry_for_3}, respectively, below. 
 
For $n\geq 1$, the \emph{Lorentz group}\footnote{This is a special case of the \emph{indefinite orthogonal group} $\mathrm{O}(m,n)$, see \cite{MR1920389}.} $\mathrm{O}(1,n)$ consists of all $(n+1)\times(n+1)$ real invertible matrices $A$ such that $AJA^t=J$, where
\[
	J=\begin{pmatrix}-1&0&0&\cdots & 0\\0&1&0&\cdots &0\\\vdots & & & &\vdots\\0&0&0&\cdots&1\end{pmatrix}.
\]
The \emph{positive Lorentz group} $\mathrm{O}^+(1,n)$ is then the closed subgroup of $\mathrm{O}(1,n)$ consisting of those matrices which preserve the positive ``time-like'' axis in the Lorentzian model of hyperbolic space \cite[Ch.~3]{MR4221225}. Then, $\mathrm{Isom}(\mathbb{H}^n)$ is topologically isomorphic to $\mathrm{O}^+(1,n)$ \cite[Thm.\ 5.2.5]{MR4221225}. Through this isomorphism, the subgroup $\mathrm{Isom}^+(\mathbb{H}^n)$ of orientation-preserving isometries is identified with the \emph{special positive Lorentz group} $\mathrm{SO}^+(1,n)$, the identity component of $\mathrm{O}^+(1,n)$ consisting of those matrices having determinant $1$.

When $n\geq 2$, $\mathrm{SO}^+(1,n)$ is a noncompact, connected, semisimple Lie group (see \cite[I.8 and I.17]{MR1920389}), which thus contains a discrete nonabelian free subgroup. Hence, the following corollary is immediate from Corollary \ref{cor:classwise} and Theorem \ref{thm:conjugacy_universal}.

\begin{corollary}\label{cor:isometry_universal}
	For any $n\geq 2$, the isometry relation $\cong_{\mathsf{Isom}}$ on $\mathfrak{C}^*_c(\mathsf{Isom},\mathbb{H}^n)$ is essentially countable universal, as is the orientation-preserving isometry relation $\cong_{\mathsf{Isom}^+}$ on $\mathfrak{C}^*_c(\mathsf{Isom}^+,\mathbb{H}^n)$.	\qed
\end{corollary}

\subsection{Conjugacy and finitely generated subgroups}
\label{subsection:weak_reduction}

While the lemmas of Section \ref{sec:matrix_groups} are concerned with transferring universality in the conjugacy relation for discrete subgroups between a group and its quotients, the lemmas below will be put to use in transferring concrete classifiability and (essential) hyperfiniteness, when restricted to finitely generated subgroups, in Sections \ref{section:isometry_for_2} and \ref{section:isometry_for_3}.
Let us first verify that the latter restriction defines a standard Borel space.

\begin{lemma}
\label{lem:Dfg}
	Let $G$ be a locally compact Polish group. The set $\mathcal{D}_{\mathrm{fg}}(G)$ of finitely generated discrete subgroups of $G$ is a Borel subset of $\mathcal{D}(G)$.
\end{lemma}

\begin{proof}
	Let $f:G^\mathbb{N}\to\mathcal{F}(G)$ be given by $f((g_n)_{n\in\mathbb{N}})=\overline{\langle g_n\rangle_{n\in\mathbb{N}}}$, the closure of the subgroup generated by the sequence $(g_n)_{n\in\mathbb{N}}$. We claim that this map is Borel. If $U\subseteq G$ is open, then $f((g_n)_{n\in\mathbb{N}})\cap U\neq\varnothing$ if and only if $\langle g_n\rangle_{n\in\mathbb{N}} \cap U\neq\varnothing$, which occurs if and only if there is a word in finitely many of the $g_n$'s which is in $U$; this is an open condition. Likewise, if $K\subseteq G$ is compact, then $f((g_n)_{n\in\mathbb{N}})\cap K=\varnothing$ if and only if there is a rational $r>0$ such that for every word in any finite number of the $g_n$'s, that word is at least distance $r$ away from $K$ (with respect to some fixed compatible metric on $G$), a Borel condition. This shows that the inverse images under $f$ of the subbasic open sets generating the Chabauty--Fell topology on $\mathcal{F}(G)$ are Borel in $G^\mathbb{N}$, and hence that $f$ is Borel. Note that $f^{-1}[\mathcal{D}(G)]$ is exactly the set of sequences in $G^\mathbb{N}$ which generate discrete groups.
	
	Let $G^{<\infty}$ be the (Borel) subset of $G^\mathbb{N}$ consisting of all sequences in $G$ which are eventually equal to the identity. Then, $\mathcal{D}_{\mathrm{fg}}(G)=f[G^{<\infty}\cap f^{-1}[\mathcal{D}(G)]]$. Since $\mathcal{D}(G)$ is Borel in $\mathcal{F}(G)$ and $f$ is countable-to-one when restricted to $G^{<\infty}\cap f^{-1}[\mathcal{D}(G)]$, it follows by Luzin--Novikov uniformization (see Section \ref{subsection:classical_DST}) that $\mathcal{D}_{\mathrm{fg}}(G)$ is Borel.
\end{proof}

Our next lemma concerns the following general notion: given Borel equivalence relations $E$ and $F$, a Borel homomorphism $f$ from $E$ to $F$ is a \emph{weak Borel reduction}\footnote{The notion of a weak Borel reduction is usually reserved for countable Borel equivalence relations, where it is equivalent to being a countable-to-one homomorphism.} if the preimage of each $F$-class under $f$ contains only countably many $E$-classes. In this case, we say that $E$ is \emph{weakly Borel reducible} to $F$.

\begin{lemma}\label{lem:weak_reduction}
	Suppose that $E$ and $F$ are essentially countable Borel equivalence relations and that $E$ is weakly Borel reducible to $F$.
	\begin{enumerate}[label=\textup{\arabic*.}]
	 \item If $F$ is concretely classifiable, then so is $E$.
	 \item If $F$ is essentially hyperfinite, then so is $E$.
	\end{enumerate}
\end{lemma}

\begin{proof}
	For the first assertion, it suffices to consider the case when $f$ is a weak Borel reduction of $E$ to $=_Y$ on some standard Borel space $Y$. Since $E$ is essentially countable, there is a countable Borel equivalence relation $E'$ which is Borel bireducible with $E$. Let $g$ be a Borel reduction from $E'$ to $E$. Then, $f\circ g$ is a weak Borel reduction of $E'$ to $=_Y$, hence by \cite[Lem.\ 2.1]{MR2563816}, $E'$ is concretely classifiable. But $E'$ is bireducible with $E$, so $E$ is concretely classifiable as well.
	
	For the second assertion, observe first that a subequivalence relation of a hyperfinite relation is hyperfinite.
	Next, \cite[Thm.\ 4.4]{MR2500091} tells us that if a countable Borel equivalence relation $E'$ weakly reduces to $E_0$, then $E'\subseteq R$ for some countable Borel equivalence relation $R\leq_B E_0$; since the latter is equivalent to the hyperfiniteness of $R$, it follows then that $E'\leq_B E_0$ as well.
Proceed now as above: fix a weak Borel reduction $f$ of an essentially countable Borel equivalence relation $E$ to $E_0$ along with a witness $g:E'\to E$ of one side of the bireducibility of a countable Borel equivalence relation $E'$ with $E$.
As before, $f\circ g$ is a weak Borel reduction of $E'$ to $E_0$, hence $E'\leq_B E_0$, and thus $E\leq_B E_0$ as well.
\end{proof}

Focusing next on $\mathcal{D}_{\mathrm{fg}}(G)$, we have the following variation on Lemma \ref{lem:quotient_reduction}.

\begin{lemma}
\label{lem:for_SL}
	Let $G$ be a locally compact Polish group, $N$ a finite normal subgroup of $G$, and $\pi:G\to G/N$ the quotient map. 
	\begin{enumerate}[label=\textup{\arabic*.}]
		\item The inverse image map $\pi^{-1}$, when restricted to $\mathcal{D}_{\mathrm{fg}}(G/N)$, is an injective Borel reduction from $E(G/N,\mathcal{D}_{\mathrm{fg}}(G/N))$ to $E(G,\mathcal{D}_{\mathrm{fg}}(G))$.
		\item The direct image map $\pi$ is a Borel homomorphism from $E(G,\mathcal{D}(G))$ to $E(G/N,\mathcal{D}(G/N))$. Moreover, when restricted to $\mathcal{D}_{\mathrm{fg}}(G)$, this map is a weak Borel reduction of $E(G,\mathcal{D}_{\mathrm{fg}}(G))$ to $E(G/N,\mathcal{D}_{\mathrm{fg}}(G/N))$. Consequently, if $E(G/N,\mathcal{D}_{\mathrm{fg}}(G/N))$ is concretely classifiable or essentially hyperfinite, respectively, then so is $E(G,\mathcal{D}_{\mathrm{fg}}(G))$.
	\end{enumerate}
\end{lemma}

\begin{proof}
	Part (1) is proved exactly like Lemma \ref{lem:quotient_reduction}, with $\pi^{-1}$ restricted to $\mathcal{D}_{\mathrm{fg}}(G/N)$. Note that if $\Gamma\leq G/N$ is generated by $\{g_0 N,\ldots,g_n N\}$, then $\pi^{-1}[\Gamma]$ is generated by $\{g_0,\ldots,g_n\}\cup N$, so this map preserves finite generation.
	
	For (2), clearly if $\Gamma\leq G$ is finitely generated, then so is $\pi[\Gamma]$. We must, however, show that the direct image preserves discreteness. Suppose that $\Gamma\leq G$ is discrete, but that $\pi[\Gamma]$ is not. Then, there is a sequence $(g_n N)_{n\in\mathbb{N}}$ in $\pi[\Gamma]$, each $g_n\in\Gamma$, converging to the identity $e N$, but such that no $g_n N=e N$, that is, no $g_n\in N$. $G/N$ is locally compact, so we may assume that all of the $g_n N$'s are in some compact neighborhood of the identity. As $N$ is finite, the quotient map is proper, so the $g_n$'s all lie in some compact set in $G$. Thus, there is a subsequence $( g_{n_\ell})_{\ell\in\mathbb{N}}$ which converges, necessarily to some $k\in N$. But $\Gamma$ is discrete, so these $g_{n_\ell}$'s must be ultimately equal to $k$, and thus $g_{n_\ell}N$ is ultimately equal to $e N$, a contradiction. This shows that the direct image map is well-defined $\mathcal{D}(G)\to\mathcal{D}(G/N)$, and it is clearly a Borel homomorphism.

	To see that the direct image map is a weak reduction when restricted to $\mathcal{D}_{\mathrm{fg}}(G)$, suppose that $\Delta\in \mathcal{D}_{\mathrm{fg}}(G)$ has generators $d_1,\ldots,d_m$ and $\Delta'\in\mathcal{D}_{\mathrm{fg}}(G)$ is such that $\pi[\Delta]$ and $\pi[\Delta']$ are conjugate in $G/N$. That is, there is some $g\in G$ such that
	\[
		g\pi[\Delta] g^{-1}N=gN\pi[\Delta]g^{-1}N=\pi[\Delta'].
	\]
	It follows that $\Delta'$ is generated by some subset of $g\{d_ik\mid k\in N,i\leq m\}g^{-1}$. 
	
	Thus, every subgroup of $G$ which is mapped to one conjugate to $\pi[\Delta]$ in $G/N$ must be conjugate in $G$ to one generated by a subset of $\{d_ik\mid k\in N,i\leq m\}$, of which there are only finitely many. Hence, $\pi$ is finite-to-one on conjugacy classes, and in particular, a weak reduction.
\end{proof}

The final lemma of this section can be used to transfer complexity results between $\mathrm{Isom}^+(\mathbb{H}^n)$ and $\mathrm{Isom}(\mathbb{H}^n)$.

\begin{lemma}
\label{lem:lifting_orientation_preserving}
	Suppose that $G$ is a locally compact Polish group, $G_0$ a closed normal subgroup of index $2$, and $i\in Z(G)$ an element of order $2$ such that $G = G_0 \sqcup i G_0$. Then:
	\begin{enumerate}[label=\textup{\arabic*.}]
		\item $E(G,\mathcal{D}(G))$ and $E(G_0,\mathcal{D}(G_0))$ are Borel bireducible.
		\item $E(G,\mathcal{D}_{\mathrm{fg}}(G))$ and $E(G_0,\mathcal{D}_{\mathrm{fg}}(G_0))$ are Borel bireducible.
	\end{enumerate}
\end{lemma}

\begin{proof}
	Suppose that $H$ and $K$ are subgroups of $G_0$. If $H$ and $K$ are conjugate by some element of $G_0$, they are conjugate in $G$. Conversely, if $H$ and $K$ are conjugate in $G$, say by some element $g$, then either $g$ is in $G_0$ or $ig$ is. In the latter case, since $i$ is central,
	\[
		(ig) H (ig)^{-1} = i g H g^{-1} i^{-1} = g H g^{-1}i i^{-1} = K.
	\]
	Thus, the inclusion map from subgroups of $G_0$ to subgroups of $G$ is a reduction of conjugacy.
	
	Next, consider the map which takes a subgroup $H$ of $G$ to its intersection with $G_0$, call it $H_0$. Note that $H=H_0\sqcup iH_0$. If $H$ and $K$ are conjugate subgroups of $G$, say $K=gHg^{-1}$ for some $g\in G$, then
	\[
		K_0\sqcup iK_0=K=g Hg^{-1}=g(H_0\sqcup iH_0)g^{-1}=gH_0g^{-1}\sqcup i(g H_0g^{-1}),
	\]
	so $K_0=gH_0g^{-1}$. Conversely, if $H_0$ and $K_0$ are conjugate by some $h\in G_0$, then
	\[
		hHh^{-1}=h(H_0\sqcup iH_0)h^{-1}=hH_0h^{-1}\sqcup i(hH_0h^{-1})=K_0\sqcup iK_0=K.
	\]
	Thus, the mapping $H\mapsto H_0$ is a reduction of conjugacy as well.
	
	Clearly, both of these mappings preserve discreteness, proving (1). The inclusion map preserves finite generation trivially, while the mapping $H\mapsto H_0$ preserves finite generation by Schreier's Lemma \cite[Lem.\ 7.56]{MR1307623}, proving (2).
\end{proof}

\section{The conjugacy problem for finitely generated Fuchsian groups}
\label{section:isometry_for_2}

In this section and its sequel, we return to hyperbolic manifolds as our primary objects of study.
A main motivation as we proceed, though, derives from the group-theoretic analyses of the previous section; this is the question \emph{what is the Borel complexity of the conjugacy relation for \emph{finitely generated} discrete subgroups of a Lie group?}
For it is among the effects of the extensions, in Theorem \ref{thm:conjugacy_universal} and Corollary \ref{cor:semisimple_universal}, of the results of \cite{MR1815088} from countable to Lie group settings to render the conjugacy relation on finitely generated discrete subgroups descriptive set-theoretically interesting in ways which it had not been before.
More precisely, although a countable group may have uncountably many discrete subgroups, only countably many of those are finitely generated, and hence the conjugacy relation on the latter collection trivially reduces to $=_{\mathbb{N}}$.
In contrast, moduli spaces $\mathcal{M}(S)$ of finite-type hyperbolic surfaces $S$ may be regarded as witnessing that both of the Lie groups $\mathrm{PSL}(2,\mathbb{R})$ and $\mathrm{PSL}(2,\mathbb{C})$ possess uncountably many conjugacy classes of finitely generated discrete subgroups (see the discussions preceding Lemma \ref{lem:uncountable} and in and just after Theorem \ref{thm:geom_fin_hyp_3_mans_smooth} below), and it is therefore natural to ask how much of the complexity attested by Theorem \ref{thm:conjugacy_universal_in_quotients} manifests on these families of subgroups alone.

The latter two groups correspond to the $n=2$ and $n=3$ instances of $\mathrm{Isom}^+(\mathbb{H}^n)$; recall that in Corollary \ref{cor:isometry_universal}, we applied our analysis of the conjugacy relation for discrete torsion-free subgroups of each to draw conclusions about the complexity of the isometry relation for the associated classes of hyperbolic $n$-manifolds.
Below, these roles will be reversed: it will tend to be classification results on the manifold side, together with the translation mechanisms of Section \ref{section:bireducibility}, which underwrite our complexity computations on the group side.
These results leverage deep and specific geometric phenomena, leaving the question of the complexity of the conjugacy relation on finitely generated discrete subgroups for other Lie or even locally compact Polish groups very generally open; see our conclusion's Question \ref{ques:fgLie}.
A related distinction of these results is their encompassing the most general of manifold finiteness conditions which one might invoke, namely the fifth of those which we now review.
\begin{definition}
\label{def:finiteness_conditions}
Fix $n\geq 2$ and let $M=\mathbb{H}^n/\Gamma$ be the quotient hyperbolic manifold associated to a torsion-free discrete subgroup $\Gamma$ of $\mathrm{Isom}(\mathbb{H}^n)$.
The \emph{convex core} of $M$ is the smallest convex submanifold of $M$ whose inclusion into $M$ is a homotopy equivalence.
\begin{enumerate}[label=\textup{\roman*.}, ref=\textup{\roman*}]
\item \label{cond:closed} $M$ is \emph{closed}, and $\Gamma$ is \emph{cocompact}, if $M$ is compact and without boundary.
\item \label{cond:finvolume} $M$ is \emph{of finite volume}, and $\Gamma$ is a \emph{lattice}, if $\mathrm{vol}(M)<\infty$.
\item \label{cond:geomfin} $M$ and $\Gamma$ are \emph{geometrically finite} if $\Gamma$ is finitely generated and the convex core of $M$ is of finite volume.\footnote{For equivalent definitions of geometric finiteness in dimension $3$, see \cite[p.\ 145]{MR3586015}; for more general discussion, see \cite[\S 2]{MR2402415}, \cite{MR1218098}, \cite[\S 12]{MR4221225}.}
\item \label{cond:toptame} $M$ is \emph{topologically finite} or \emph{tame} if $M$ is homeomorphic to the interior of a compact manifold, possibly with boundary.
\item \label{cond:algfin} $M$ is \emph{algebraically finite} if $\Gamma$ is \emph{finitely generated}.
\end{enumerate}
\end{definition}
These conditions are of (weakly) decreasing stringency; in the case of $n=2$, for example, we have $(i)\Rightarrow (ii)\Rightarrow (iii)\Leftrightarrow (iv)\Leftrightarrow (v)$, with neither of the first two implications reversible \cite[Theorem 0.8]{MR1638795}.
For each $n$, in other words, the classes of $n$-manifolds associated to these conditions are nondecreasing in size, and in this and the next section we will compute, in dimensions $2$ and $3$, respectively, the complexity of the isometry relation on each of them.
In this section, our main result is the following:

\begin{theorem}
\label{thm:fgPSL2Rsmooth}
The orientation-preserving isometry problem for any of the classes (i) through (v) of complete orientable hyperbolic $2$-manifolds listed above is Borel bireducible with the relation $=_{\mathbb{R}}$ of equality on the real numbers. 
In particular, the classification of
\begin{enumerate}[label=\textup{\arabic*.}]
\item finitely generated torsion-free discrete subgroups of $\mathrm{PSL}(2,\mathbb{R})$ up to conjugacy, or, equivalently, of
\item algebraically finite complete orientable hyperbolic $2$-manifolds up to orienta\-tion-preserving isometry
\end{enumerate}
is Borel bireducible with $=_{\mathbb{R}}$.
\end{theorem}

Via the lemmas of Section \ref{section:conjugacy}, this analysis readily extends from the finitely generated torsion-free discrete subgroups of $G=\mathrm{PSL}(2,\mathbb{R})$ to those of $G=\mathrm{SL}(2,\mathbb{R})$, and of $G=\mathrm{Isom}(\mathbb{H}^n)$, results we record in this section's conclusion.
Each strengthens the relevant instance of Stuck and Zimmer's result that \emph{for any Lie group $G$, the conjugacy problem for lattices in $G$ is concretely classifiable} \cite[Corollary 3.2]{MR1283875}, to the exact degree that condition (v) above strengthens condition (ii); see Theorem \ref{thm:complexity_3-mans} for a parallel strengthening (from (ii) to (iv)) in the $G=\mathrm{PSL}(2,\mathbb{C})$ setting below.

For now, however, for concision we let $G=\mathrm{PSL}(2,\mathbb{R})$ and recall that a \emph{Fuchsian group} is simply an element of $\mathcal{D}(G)$.
We also fix the notations $\mathcal{D}_{\mathrm{cc}}(G)$, $\mathcal{D}_{\mathrm{fv}}(G)$, $\mathcal{D}_{\mathrm{gf}}(G)$, and $\mathcal{D}_{\mathrm{af}}(G)$ for the subspaces of $\mathcal{D}_{\mathrm{tf}}(G)$ corresponding to the \emph{group} finiteness conditions (\ref{cond:closed}), (\ref{cond:finvolume}), (\ref{cond:geomfin}), and (\ref{cond:algfin}), respectively, in Definition \ref{def:finiteness_conditions} above.
In the notation of previous sections, then, Theorem \ref{thm:fgPSL2Rsmooth} then conjoins four basic claims:
\begin{enumerate}[label=(\alph*)]
\item $=_{\mathbb{R}}\:\leq_B E(G,\mathcal{D}_{\mathrm{cc}}(G))$,
\item $\mathcal{D}_{\mathrm{cc}}(G)$, $\mathcal{D}_{\mathrm{fv}}(G)$, and $\mathcal{D}_{\mathrm{gf}}(G)=\mathcal{D}_{\mathrm{af}}(G)$ form a $\subseteq$-increasing chain of conjuga\-tion-invariant Borel subsets of $\mathcal{D}(G)$,
\item $E(G,\mathcal{D}_{\mathrm{af}}(G))\leq_B\:=_{\mathbb{R}}$, and
\item items (a) through (c) together entail the naturally associated claims for isometry relations on hyperbolic manifolds.
\end{enumerate}
Items (a), (b), and (d) are essentially immediate from definitions, classical constructions, and our earlier results; in part to introduce these constructions, both for use in this section and the next one, we attend to these items first.
We will then present two proofs of the more substantial item (c).\footnote{One could even, for a third proof, deduce item (c) from Theorem \ref{thm:geom_fin_hyp_3_mans_smooth} below.}
For the first proof, we are indebted to discussions with Ferr\'{a}n Valdez, Ian Biringer, and Joan Porti; the more geometric second one, which arose in the course of these discussions, is due entirely to Biringer, and is included (in sketch form) herein with his kind permission.

Let us begin with item (d).
Much as in Corollary \ref{cor:isometry_universal}, the equivalence of Theorem \ref{thm:fgPSL2Rsmooth}'s item (1) with its item (2) follows from the orientation-preserving variant of Corollary \ref{cor:classwise}, together with the standard identification of $\mathrm{PSL}(2,\mathbb{R})$ and $\mathrm{Isom}^+(\mathbb{H}^2)$ via the homomorphism from $\mathrm{SL}(2,\mathbb{R})$ to the group $\text{M\"{o}b}(\hat{\mathbb{C}})$ of M\"{o}bius transformations of the Riemann sphere $\hat{\mathbb{C}}$ given by
\begin{equation}
\label{eq:matrixrep}
\begin{pmatrix} a & b \\ c & d \end{pmatrix}\;\mapsto\;\bigg(z\,\mapsto\,\frac{az+b}{cz+d}\bigg),
\end{equation}
with kernel $\{-I,I\}$, and with image exactly those M\"{o}bius transformations of $\hat{\mathbb{C}}$ preserving the upper half-plane on which we model $\mathbb{H}^2$.
Since this group of transformations coincides, by way of this model, with $\mathrm{Isom}^+(\mathbb{H}^2)$ in turn, this homomorphism induces the identification in question \cite[\S 1]{MR1177168}.
This in turn induces the classwise Borel equivalence in question between $E(G,\mathcal{D}_{\mathrm{af}}(G))$ and $\cong_{\mathsf{Isom}^+(\mathbb{H}^2)}$ on the restriction of $\mathcal{C}_c^*(\mathsf{Isom}^+,\mathbb{H}^2)$ to algebraically finite manifolds; similarly, clearly, for $\mathcal{D}_{\mathrm{cc}}(G)$ and $\mathcal{D}_{\mathrm{fv}}(G)$.

It's clear also that the manifold and group conditions, respectively, of Definition \ref{def:finiteness_conditions} are isometry- and conjugation-invariant; thus the following is all that remains to be shown of item (b).
\begin{lemma}
\label{lem:finiteness_conditions_Borel}
The spaces $\mathcal{D}_{\mathrm{cc}}(G)$, $\mathcal{D}_{\mathrm{fv}}(G)$, $\mathcal{D}_{\mathrm{gf}}(G)$, and $\mathcal{D}_{\mathrm{af}}(G)$ are each Borel subsets of $\mathcal{D}(G)$, as they all also are if $G$ is everywhere replaced by $\mathrm{PSL}(2,\mathbb{C})$.
\end{lemma}
\begin{proof}
The space $\mathcal{D}_{\mathrm{cc}}(G)\subseteq \mathcal{D}(G)$ is $$\sigma^{-1}\left(\mathfrak{K}(\mathsf{Isom}^+,\mathbb{H}^2)\cap\mathfrak{M}^*_c(\mathsf{Isom}^+,\mathbb{H}^2)\right)$$
where $\sigma$ is the Borel function of Theorem \ref{thm:discretetomanifolds}; this is Borel by Theorem \ref{thm:compact_Borel} and Lemmas \ref{lem:convex_Borel} and \ref{lem:complete}.
Replacing $G$ by $\mathrm{PSL}(2,\mathbb{C})$ would correspond to replacing $\mathbb{H}^2$ by $\mathbb{H}^3$, whereupon the same argument would apply. $\mathcal{D}_{\mathrm{fv}}(G)$ is Borel for either choice of $G$ by \cite[Corollary I.3.1.8]{MR903850}, as is $\mathcal{D}_{\mathrm{af}}(G)=\mathcal{D}_{\mathrm{fg}}(G)\cap\mathcal{D}_{\mathrm{tf}}(G)$ by Lemmas \ref{lem:D_Dtf} and \ref{lem:Dfg}.
This handles $\mathcal{D}_{\mathrm{gf}}(G)=\mathcal{D}_{\mathrm{af}}(G)$, but in the $\mathrm{PSL}(2,\mathbb{C})$ case, this equality no longer holds, thus we still must verify that the assignment to a $\Gamma\in \mathcal{D}_{\mathrm{af}}(\mathrm{PSL}(2,\mathbb{C}))$ of its quotient manifold's compact core is a Borel operation (whereupon the volume-computing arguments of \cite[Corollary I.3.1.4]{MR903850} will complete the proof). This verification is supplied by Lemma \ref{lem:convexcore} below; we then record this lemma's $\mathcal{D}_{\mathrm{gf}}(\mathrm{PSL}(2,\mathbb{C}))$ case as the first part of that lemma's Corollary \ref{cor:loose_ends}.
\end{proof}

The arguments of items (a) and (c) are more interesting; for these, we will require a few definitions.
We call a surface \emph{hyperbolizable} if it admits a complete hyperbolic metric.
We call it \emph{finite-type} if it's homeomorphic to a closed orientable surface with finitely many points removed; this includes, of course, the possibility that no points have been removed at all.
Note (again by \cite{MR256399}, applied either to the natural compactifications or coupled with the $n$-transitivity of manifold homeomorphism groups for all $n>0$) that there are only countably many homeomorphism classes of finite-type hyperbolizable surfaces.
Let $F$ denote some fixed collection of unique representatives of such classes and let $C$ collect the closed surfaces in $F$.
Recall that a group $\Gamma$ acts \emph{properly discontinuously} on a locally compact space $X$ if the set $\{g\in\Gamma\mid gK\cap K\neq \varnothing\}$ is finite for every compact $K\subseteq X$.
For us the primary significance of this condition is that the quotient of a topological manifold by a properly discontinuous action is a Polish space (more particularly, it is an \emph{orbifold}; see \cite[Prop.\ 13.2.1]{MR4554426}).
The following semiformal definition will suffice for our purposes (we forego, for example, defining the relevant topologies; for details, see, e.g., \cite{MR590044} or \cite{MR2850125}).
\begin{definition}
\label{def:Teich_MCG_M}
Let $S$ be a finite-type hyperbolizable surface.
Identify $S$ with the interior of a smooth compact manifold $\hat{S}$ whose boundary circles, if any, correspond to the punctures of $\hat{S}$.
Write $\mathrm{Diff}^+(S)$ for the group of orientation-preserving diffeomorphisms of $\hat{S}$ with itself which restrict to the identity on $\partial\hat{S}$ and $\mathrm{Diff}_0(S)$ for its subgroup of elements which are isotopic to the identity.
The group $\mathrm{Diff}^+(S)$ acts on the space of hyperbolic metrics on $S$ by pullback; we have thus the \emph{Teichm\"{u}ller space}
$$\mathrm{Teich}(S)=\{\text{finite-volume hyperbolic metrics on }S\}/\mathrm{Diff}_0(S),$$
\emph{mapping class group}
$$\mathrm{MCG}(S)=\mathrm{Diff}^+(S)/\mathrm{Diff}_0(S),$$
and \emph{moduli space}
$$\mathcal{M}(S)=\mathrm{Teich}(S)/\mathrm{MCG}(S)$$
of $S$.
\end{definition}
For us the essential points are then as follows:
\begin{itemize}
\item If $S$ is given by $n$ punctures of a closed genus-$g$ surface then $\mathrm{Teich}(S)$ carries a real analytic structure, of dimension $6g+2n-6$, with respect to which the natural action of $\mathrm{MCG}(S)$ is properly discontinuous \cite[2.10]{MR3586015}.
In consequence, $\mathcal{M}(S)$ is a Polish space, and it is uncountable in all cases but one, namely when $g=0$ and $n=3$.
\item In Section \ref{section:isometry_for_3}, we'll leverage a higher analogue of an alternative presentation of $\mathrm{Teich}(S)$ as the quotient by the $G$-conjugacy action of a space of discrete faithful representations $\pi_1(S)\to G$.
In this alternative presentation, the natural view is of $\mathrm{Teich}(S)$ and $\mathcal{M}(S)$ as spaces of \emph{marked} and \emph{unmarked groups}, respectively, up to conjugacy; in that of Definition \ref{def:Teich_MCG_M}, it is of $\mathrm{Teich}(S)$ and $\mathcal{M}(S)$ as spaces of \emph{marked} and \emph{unmarked hyperbolic manifolds}, respectively, up to isometry.
\end{itemize}
Here we'll privilege the latter view; by the results of Section \ref{section:bireducibility}, though, it's clear in either case that $\coprod_{S\in F}\mathcal{M}(S)$ forms a natural Polish space of complete invariants for $E(G,\mathcal{D}_{\mathrm{fv}}(G))$-equivalence classes, and even a tautological one, in the sense that the natural map
\begin{equation}
\label{eq:7_first_try}
q:\mathcal{D}_{\mathrm{fv}}(G)\,\to\coprod_{S\in F}\mathcal{M}(S)\,:\,\Gamma\,\mapsto[\mathbb{H}^2/\Gamma]
\end{equation}
not only reduces the $G$-conjugacy relation on the domain to the identity relation on the codomain, it induces a bijection of quotients.
Similarly for $\mathcal{D}_{\mathrm{cc}}(G)$ and $\coprod_{S\in C}\mathcal{M}(S)$, and from this and the first bulleted point above, item (a) is immediate.
\begin{lemma}
\label{lem:uncountable}
$=_{\mathbb{R}}\:\leq_B E(G,\mathcal{D}_{\mathrm{cc}}(G))$. \qed
\end{lemma}
To prove the (\ref{cond:closed}) and (\ref{cond:finvolume}) portions of Theorem \ref{thm:fgPSL2Rsmooth}, it now only remains only to show the map $q$ to be Borel.
For this it suffices to show for each $S\in F$ that (A) $q^{-1}(\mathcal{M}(S))$ is a Borel subset of $\mathcal{D}_{\mathrm{fv}}(G)$, and (B) the restriction of $q$ to $q^{-1}(\mathcal{M}(S))$ is a Borel measurable function.
One may easily argue (A) via Suslin's theorem, but on the pretext of a more refined calculation, we take a detour that will be useful to us in our second proof of the theorem's remainder as well.

What we will argue is that $q^{-1}(\mathcal{M}(S))$ is, for any $S\in F$, a relatively $G_\delta$ subset of $\mathcal{D}_{\mathrm{fv}}(G)$; it is, in other words, an intersection of $\mathcal{D}_{\mathrm{fv}}(G)$ with countably many open subsets of $\mathcal{D}(G)$.
We will do so via a more geometric rendering of the space $\mathcal{D}(G)$ described in \cite{2021arXiv211014401B}, which we briefly now review.
Fix a unit vector $w$ in the tangent bundle $T\mathbb{H}^2$ of $\mathbb{H}^2$.
Any $\Gamma\in\mathcal{D}_{\mathrm{tf}}(G)$ then determines a \emph{vectored oriented hyperbolic $2$-manifold} $(X,v)$ in which $X$ is the quotient of $\mathbb{H}^2$ by the natural action of $\Gamma$ and $v$ is the associated image of $w$ in $TX$. (We exclude torsion for simplicity: the equivalence we're recording is a special case of that of $\mathcal{D}(G)$ with vectored \emph{orbifolds}, as in \cite[\S 3.4]{2021arXiv211014401B}; see Remark \ref{rmk:orbifolds} below.) Quoting now from \cite[p.\ 23]{2021arXiv211014401B}:
\begin{quote}
There is a natural \emph{smooth topology} on the set of isometry classes of vectored oriented Riemannian [manifolds], where $(X_i,v_i)\to(X_\infty,v_\infty)$ if there is a sequence $r_i\to\infty$ and smooth embeddings$$\psi_i:B_{X_\infty}(p_\infty,r_i)\to X_i,$$
where $p_\infty$ is the basepoint of $v_\infty$ such that each $\psi_i$ is orientation preserving, $d\psi_i^{-1}(v_i)=v_\infty\in TX$ for all $i$, and where if $g_i$ is the Riemannian metric on $X_i$, we have $\psi^*_i g_i\to g_\infty$ in the $C^\infty$ topology on $X_\infty$.
\end{quote}
The key point for our purposes is then \cite[Proposition 3.4.3]{2021arXiv211014401B}, which might here be rendered as: \emph{any choice of $w$ as above induces a homeomorphism of the space $\mathcal{D}_{\mathrm{tf}}(G)$ with the space of vectored oriented hyperbolic $2$-manifolds endowed with the smooth topology}. We denote some fixed choice of such a homeomorphism by $h$.

Intuitively, a sequence of vectored manifolds $(X_i,v_i)$ as above converges to $(X_\infty,v_\infty)$ if their $v_i$-neighborhoods ``look more and more like'' those of $v_\infty$, in a sense quantified by the increasing sequence of real numbers $r_i$.
Fixing next a compact oriented surface $Z$ (possibly with boundary) and $r>0$, it is straightforward to see that, with respect to this topology, $$U(r,Z):=\{(X,v)\mid\text{there exists a }Y\text{ homeomorphic to }Z\text{ with }B_X(p,r)\subseteq Y\subseteq X\}$$
is open (here, as above, $p$ denotes the basepoint of $v$ and $B_X(p,r)$ denotes the radius-$r$ open ball about it in $X$).
More precisely, suppose for contradiction that a sequence $(X_i,v_i)$ in the complement of $U(r,Z)$ converges to an $(X_\infty,v_\infty)\in U(r,Z)$, as witnessed by a sequence $(r_i)$ and submanifold $Y\subseteq X_\infty$, respectively.
By compactness, $Y\subseteq B_{X_\infty}(p_\infty,r_i)$ for some $i$, and it follows that for $j\geq i$ sufficiently large, $(X_j,v_j)\in U(r,Z)$ as well, the contradiction desired.

It follows immediately that if $S\in F$ is closed, then for sufficiently large $r$, $q^{-1}(\mathcal{M}(S))=h^{-1}(U(r,S))$ and hence the former set is open in $\mathcal{D}_{\mathrm{fv}}(G)$ (cf.\ \cite[p.\ 29]{2021arXiv211014401B}).
If, on the other hand, $S$ possesses ends $e_0,\dots,e_k$, then removing disjoint open neighborhoods of each of them (with circular boundary components $c_0,\dots,c_k$, respectively) results in a compact surface  $Z\subset S$ such that the $G_\delta$ set
\begin{equation*}
\bigcap_{r\in\mathbb{N}}h^{-1}(U(r,Z))
\end{equation*}
consists precisely of those $\Gamma\in\mathcal{D}_{\mathrm{fv2}}$ for which $\mathbb{H}^2/\Gamma$ is homeomorphic to $S$.
Only part (B) of this portion of our argument --- namely, the claim that the restriction of $q$ to $q^{-1}(\mathcal{M}(S))$ is Borel measurable --- then remains.
This, though, is a special case of Theorem 4.1.1 of \cite{2021arXiv211014401B}, as discussed in Remark \ref{rmk:orbifolds} below.

Topologically, hyperbolic $2$-manifolds of the form (\ref{cond:geomfin}), (\ref{cond:toptame}), or, equivalently, (\ref{cond:algfin}), are all also of finite type, thus one might imagine some variation on (\ref{eq:7_first_try}) applying for $\mathcal{D}_{\mathrm{af}}(G)$ as well; the necessary modification, of course, would be of each $\mathcal{M}(S)$ to a moduli space of possibly \emph{infinite}-volume complete hyperbolic metrics on $S$.
Such moduli spaces exist (\cite{MR1730906} and \cite{MR2245223} both work in the requisite generality), assuredly, but involve technicalities we'd prefer to avoid, if possible.
This is indeed, fortunately, possible, via either of two prettier paths which we now describe.
As is standard, we term an infinite-volume end of a hyperbolic $S$ as above a \emph{flare}, and any finite-volume end a \emph{cusp}.
\begin{proof}[Proof 1 of Theorem \ref{thm:fgPSL2Rsmooth}]
We dispensed with items (a), (b), and (d) above; thus it remains only to prove item (c).
We will do so via a geometric operation which only usefully applies to the non-cyclic elements of $\mathcal{D}_{\mathrm{af}}(G)$.
To that end, partition the latter into its cyclic and non-cyclic portions $\mathcal{D}^{\mathrm{c}}_{\mathrm{af}}(G)$ and $\mathcal{D}^{\mathrm{nc}}_{\mathrm{af}}(G)$, respectively; it's straightforward to check that each is Borel.
We handle the cyclic case under the more general heading of \emph{elementary} groups below; more precisely, within the proof of Lemma \ref{lem:elementary}, the embedding $\mathbb{R}\to\mathbb{C}$ induces a Borel reduction of $E(G,\mathcal{D}^{\mathrm{c}}_{\mathrm{af}}(G))$ to the identity relation on $\{1,2\}\times\mathbb{R}\subseteq\{1,2,3\}\times\mathbb{C}$.

Thus to complete the proof of Theorem \ref{thm:fgPSL2Rsmooth} it will suffice by Lemma \ref{lem:weak_reduction} to show that \begin{equation*}
f:\mathcal{D}^{\mathrm{nc}}_{\mathrm{af}}(G)\,\to\coprod_{S\in F}\mathcal{M}(S)\,:\,\Gamma\,\mapsto[D(CC(M_{\sigma(\Gamma)}))]
\end{equation*}
is a weak Borel reduction of $E(G,\mathcal{D}^{\mathrm{nc}}_{\mathrm{af}}(G))$ to the identity on a Polish space.
Bracketed on the right-hand side is the double of the convex core of $M_{\sigma(\Gamma)}$, where $\sigma$ is the reduction of Theorem \ref{thm:discretetomanifolds}.
The map $f$ decomposes, in other words, as a sequence
$$\Gamma\mapsto M_{\sigma(\Gamma)}\mapsto CC(M_{\sigma(\Gamma)}) \mapsto D(CC(M_{\sigma(\Gamma)})) \mapsto \nu(D(CC(M_{\sigma(\Gamma)})))$$
taking $\mathcal{D}^{\mathrm{nc}}_{\mathrm{af}}(G)$ to $\mathcal{D}_{\mathrm{fv}}(G)$ (since the map $M_{\sigma(\Gamma)}\mapsto CC(M_{\sigma(\Gamma)}$ ``pares away'' the infinite-volume ends of $M_{\sigma(\Gamma)}$), followed by the map $q$ of line \ref{eq:7_first_try} above.
We know from Section \ref{section:bireducibility} that the first and fourth of these maps is Borel (the map $\nu$ is that of Theorem \ref{thm:manifoldstodiscrete}), and in Section \ref{subsection:the_basics} we derive from a more general analysis (Lemmas \ref{lem:limitsets} and \ref{lem:convexcore}) the Corollary \ref{cor:loose_ends} that the second map in the chain above is Borel.
To see that $f$ is Borel thus reduces to verifying that the doubling map $CC(M_{\sigma(\Gamma)}) \mapsto D(CC(M_{\sigma(\Gamma)}))$ is Borel; since this point is as intuitively clear as it is tedious to formalize, we defer it to the claim concluding the proof below.

It remains only to see that $f$ is countable-to-one on $E(G,\mathcal{D}^{\mathrm{nc}}_{\mathrm{af}}(G))$-equivalence classes.
In fact it's finite-to-one, since any two finite-type hyperbolic surfaces are isometric if and only if their convex cores are \cite[Theorem 3.4]{MR3903116}, and, up to isometry, only finitely many such convex cores $C$ double to form any fixed finite-volume hyperbolic surface $D(C)$.
To see this last assertion, recall that the isometry group of any finite-volume hyperbolic surface is finite \cite{MR80730}, and that any such $C$ is encoded by the fixed points of one of $\mathrm{Isom}(D(C))$'s elements, namely the one exchanging $C$ with its mirror image within $D(C)$.
Thus the following claim will complete our argument.
\begin{claim}
The doubling map $CC(M_{\sigma(\Gamma)}) \mapsto D(CC(M_{\sigma(\Gamma)}))$ is Borel.
\end{claim}
\begin{proof}[Proof of Claim] 
We begin by describing the domain of this map, then content ourselves with an outline of the proof.
In Lemma \ref{lem:convexcore}, the convex core of a manifold parametrized by $(\mathcal{U},c)$ is encoded, much as in Lemma \ref{lem:exhaustion}, as the complement of an open submanifold $(\mathcal{V},c|_\mathcal{V})$.
Writing $f_4$ for this operation, our domain is the image of the restriction of $f_4\circ\sigma$ to $\mathcal{D}^{\mathrm{nc}}_{\mathrm{af}}(G)$; note also that in the present ($2$-dimensional) context, each $\partial V_i\cap U_i$ consists of geodesics, since the boundary of the preimage, in the cover $\mathbb{H}^2$, of any convex core does.
By postcomposing with a reparametrization map, we may further assume that each $\partial V_i\cap U_i$ is connected, or in other words consists of a single geodesic arc $\gamma_i$.
Here, of course, $\gamma_i$ may be the empty arc; if it is not, though, then we write $\hat{\gamma}_i$ for the maximal geodesic which contains it and bisects $\mathbb{H}^2$.
Write $r_i$ for the isometry reflecting $\mathbb{H}^2$ across $\hat{\gamma}_i$.
One may then regard the heart of the matter as the fact that the map taking $(U_i,V_i, U_{i,j})$ to
\[
\mathsf{d}_i U_{i,j}:=
\begin{cases}
\varnothing & \text{if }V_i = U_i \\
U_{i,j}\backslash V_i\,\cup\,r_i[U_{i,j}\backslash V_i] & \text{if }\gamma_i\neq\varnothing\text{ and }\gamma_j\neq\varnothing \\ U_{i,j} & \text{otherwise}
\end{cases}
\]
is a Borel operation, i.e., that the preimage of any open subset of $\mathcal{O}(\mathbb{H}^2)$ is relatively Borel in the collection
$$\{(U,V,W)\in\mathcal{O}(\mathbb{H}^2)\times \mathcal{O}(\mathbb{H}^2)\mid W\subseteq U\text{ and }V\subseteq U\text{ and }\partial V\cap U\text{ is a geodesic arc}\}.$$
We leave this fact's straightforward verification to the reader.
For ease of reference, say that a $U_{i,j}$ as in the second case defining $\mathsf{d}_i U_{i,j}$ is \emph{of the second type}.

The doubling operation naturally encodes as a conversion of $((\mathcal{U},c),(\mathcal{V},c|_\mathcal{V}))$, which we may regard as indexed by $\mathbb{N}^2$, to a $(\mathcal{W},d)$ indexed by $(\mathbb{N}\times\{0,1\})^2$.
As any bijection $b:\mathbb{N}\times\{0,1\}\to\mathbb{N}$ will, without affecting Borel-measurability considerations, then convert the outcome to an honest member of $\mathfrak{M}(\mathsf{Isom}^+,\mathbb{H}^2))$, we will focus on this first conversion.
Its output's $\mathbb{N}\times\{0\}$ and $\mathbb{N}\times\{1\}$ indices roughly  correspond to the two halves of the doubled manifold; more precisely, for any $i,j\in\mathbb{N}$ and $k\in\{0,1\}$, we let $W_{(i,k),(j,k)}=\mathsf{d}_i U_{i,j}$, and we let $\varphi_{(i,k),(j,k)}=\varphi_{i,j}$ unless $U_{i,j}$ is of the second type, in which case $$\varphi_{(i,k),(j,k)}=\varphi_{i,j}\big|_{U_{i,j}\backslash V_{i,j}}\,\cup\,r_j\circ\left(\varphi_{i,j}\big|_{U_{i,j}\backslash V_{i,j}}\right)\circ r_i^{-1}.$$
Letting $\varphi_{(i,0),(i,1)}=r_i\big|_{W_{(i,0),(i,1)}}$ for each $i$ then determines all the remaining transition maps defining $(\mathcal{W},d)$, via transitivity and symmetry.

Intuitively, this construction amounts to taking two copies, indexed by $\mathbb{N}\times\{0\}$ and $\mathbb{N}\times\{1\}$, of $M_{(\mathcal{U},c)}\backslash M_{(\mathcal{V},c|_\mathcal{V})}$, and gluing them via reflections flipping the doubles $W_{(i,0),(i,1)}$ of those charts (i.e., those $U_{i,i}$ of the second type) which abut the boundary of the compact core.
As hinted, the details of the verifications that this conversion is both Borel and as desired sufficiently resemble those of our earlier work that we leave the bulk of them to the reader.
Its initial reparametrization, for example, consists in a refinement of the elements of the $\mathcal{U}$ under consideration, much as in Lemma \ref{lem:reparametrization}; to see that this can be done in a Borel way, fix a countable basis $\{B_i\mid i\in\mathbb{N}\}$ for the topology of $\mathbb{H}^2$ and note that the map
 \begin{align*}\Gamma\mapsto\{i\in\mathbb{N}\mid B_i\text{ intersects only one boundary geodesic} \\ \text{of the convex hull of }\Lambda(\Gamma)\}\end{align*} is Borel; for the definition of the limit set $\Lambda(\Gamma)$, see Section \ref{subsection:the_basics} below. 
\end{proof}
\end{proof}

\begin{remark}
\label{rmk:orbifolds}
We were convinced of Theorem \ref{thm:fgPSL2Rsmooth} long before we knew quite how we would argue it, and it was in the course of related discussions that Ferr\'{a}n Valdez directed our attention to Theorem 4.1.1 of \cite{2021arXiv211014401B}, which may be stated as follows. Let
$\mathcal{D}^o_{\mathrm{fv}}(G)=\{\Gamma\in\mathcal{D}(G)\mid\mathbb{H}^2/\Gamma\text{ has finite volume}\}$ so that, much as above, any $\Gamma\in\mathcal{D}^o_{\mathrm{fv}}(G)$ determines a vectored hyperbolic $2$-\emph{orbifold} $(X,v)$ with image $[X]$ in the associated \emph{finite-volume orbifold} moduli space $\hat{\mathcal{M}}(S)$. Denote the map $(X,v)\mapsto [X]$ by $\hat{q}$; then: \emph{for any orbifold $S$ arising in the above fashion, the restriction of $\hat{q}$ to $\hat{q}^{-1}(\hat{\mathcal{M}}(S))$ is a fiber orbibundle whose regular fibers are homeomorphic to $T^1 S$; moreover, $\hat{q}^{-1}(\hat{\mathcal{M}}(S))$ is a $6g+2(k+l)-3$ dimensional manifold, where $g$, $k$, and $l$ record the genus and numbers of cusps and cone points, respectively, of $S$.} Note the two differences from the setting of Theorem \ref{thm:fgPSL2Rsmooth}: the restriction to $\mathcal{D}_{\mathrm{fv}}(G)$ on the one hand, and an expanded attention to groups $\Gamma$ possibly possessing torsion on the other.
It seems clear upon inspection that Theorem \ref{thm:fgPSL2Rsmooth} may be extended to subsume the latter possibility as well; as noted, we have foregone this extension simply because its requirements --- definitions, parametrizations, and moduli spaces of $2$-orbifolds --- would take us too far afield.
\end{remark}

\begin{proof}[Proof 2 of Theorem \ref{thm:fgPSL2Rsmooth} (sketch)]
Again we need only to prove item (c), namely that the relation $E(G,\mathcal{D}_{\mathrm{af}}(G))$ is concretely classifiable, and we will apply to this end Glimm's theorem that an orbit equivalence relation is concretely classifiable if and only if any orbit is open in its closure (\cite[Theorem 1]{MR0136681}; more precisely, what we require is the theorem's implication (1)$\Rightarrow$(3). Note that for this implication, the theorem's requirement that the domain be locally compact (as $\mathcal{D}_{\mathrm{af}}(G)$ in general is not) is superfluous).
Here it's again valuable to consider these orbits in their materializations as families of vectored hyperbolic $2$-manifolds $(X,v)$, for fixing any one of them allows us to then view conjugations by elements of $G$ as simply relocating and reorienting the unit vector $v$ within $TX$.
It is easy to see, either directly or by the aforementioned \cite[Corollary 3.2]{MR1283875}, that if this $X$ is compact then its associated orbit is closed; the more interesting case is when $X$ possesses ends which the vectors $v$ may, in the above framing, ``exit''.
More precisely, in the interesting case $X$ is the union of a compact bounded surface $X'$, in which its genus concentrates, together with the finitely many components of $X\backslash X'$ comprising its ends; $(Y,w)$ is then in the closure of the orbit of $X$ if there exist a sequence of unit vectors $v_i\in T X$ with basepoints $p_i\in X$ along with reals $r_i\to\infty$ and maps $\psi_i:B_X(p_i,r_i)\to Y$ sending each $v_i$ to $w$ whose images converge to the metric on $Y$.
There are two nontrivial possibilities.
In the first, these $v_i$ exit a flare; here, if the $r_i$ grow sufficiently slowly, convergence is to a vectored $Y=\mathbb{H}^2$, simply, corresponding to the discrete group $\{I\}$.
In fact, this is the only exit of a flare possible, and if the $v_i$ remain within bounded distance of $X'$ then $Y$ is isometric to $X$, a point equally applying when the end in question is a cusp, which brings us to the second nontrivial possibility.
This is when the sequence $v_i$ exits a parabolic end. Assume first that these $v_i$ all point out of the cusp, and that the fixed points of the parabolic elements associated to this cusp converge to $\infty\in\hat{\mathbb{R}}$: then it is not too hard to see that the limit of the associated $\Gamma_i$ is in fact the \emph{non-discrete} group $\{z\mapsto z+t\mid t\in\mathbb{R}\}\in\mathcal{S}(G)$.
Rotations of those $v_i$ will, on a subsequence, converge to some fixed angle $\theta\in 2\pi$ with the geodesics from their basepoints straight out the cusp, and the associated $\Gamma_i$ will converge to a conjugate of $\{z\mapsto z+t\mid t\in\mathbb{R}\}$ in $\mathcal{S}(G)$ with fixed point in $\hat{\mathbb{R}}$ at an angle $\theta$ from $\infty$.
Hence the closure in $\mathcal{S}(G)$ of the orbit $O$ associated to a noncompact finite-type $(X,v)$ is the union of $O$ with, topologically, a finite number of points and circles; it follows that $O$ is open therein, and this concludes the proof.
\end{proof}
These themes of sorting and ``exiting'' ends, of mischief around cusps, and of limits of surface structures all play fundamental roles in the classification of hyperbolic $3$-manifolds, as we'll soon see.
We conclude this section with the following easy corollary of Theorem \ref{thm:fgPSL2Rsmooth}.

\begin{corollary}
\label{cor:2isometrysmooth}
The classification of complete algebraically finite hyperbolic $2$-manifolds up to isometry is Borel bireducible with the relation $=_{\mathbb{R}}$ of equality on the real numbers, as is the classification of finitely generated torsion-free discrete subgroups of $\mathrm{PGL}(2,\mathbb{R})$ or $\mathrm{SL}(2,\mathbb{R})$ up to conjugacy.
\end{corollary}
\begin{proof}
The first of these assertions records the upgrade, via Lemma \ref{lem:lifting_orientation_preserving}(2), from a complexity computation for the orientation-preserving isometry relation on the algebraically finite parameters of $\mathfrak{C}^{*}_c(\mathsf{Isom}^+,\mathbb{H}^2)$ to that on the algebraically finite parameters of $\mathfrak{C}^*_c(\mathsf{Isom},\mathbb{H}^2)$.
By Corollary \ref{cor:classwise} and the topological isomorphism $\mathrm{PGL}(2,\mathbb{R})\cong\mathrm{Isom}(\mathbb{H}^2)$, this also gives the first half of the second assertion.
For the latter's second half, apply Lemma \ref{lem:for_SL}(2) to Theorem \ref{thm:fgPSL2Rsmooth}(1).
\end{proof}

\section{The conjugacy problem for finitely generated Kleinian groups}
\label{section:isometry_for_3}
In this section, we establish, alongside a number of subsidiary propositions, the results on finitely generated Kleinian groups and hyperbolic $3$-manifolds which appear as our introduction's Theorems K and L.
Several of these results may be argued in more than one way.
Below, we will often begin by outlining the most heuristic argument, in part as a way of introducing the relevant objects, before turning to more formal, and economical, statements and proofs of subsections' main results.
We will also, in parallel with Section \ref{section:isometry_for_2}, broadly proceed according to the hierarchy of finiteness conditions described in Definition \ref{def:finiteness_conditions} therein.
\subsection{Overview}
\label{subsection:overview}
A \emph{Kleinian group} $\Gamma$ is a discrete subgroup of $\mathrm{PSL}(2,\mathbb{C})$.
If $\Gamma$ is torsion-free then the quotient of $\mathbb{H}^3$ by the natural action of $\Gamma$ is a hyperbolic $3$-manifold.
More broadly, the $n=3$ instance of Corollary \ref{cor:classwise} describes a Borel correspondence between torsion-free Kleinian groups and complete orientable hyperbolic $3$-manifolds, by way of the standard identification (again via equation (\ref{eq:matrixrep})) of $\mathrm{PSL}(2,\mathbb{C})$ with the group $\text{M\"{o}b}(\hat{\mathbb{C}})$ of all M\"{o}bius transformations of the Riemann sphere $\hat{\mathbb{C}}$.
This is because the natural identification of the latter with the boundary of the open unit ball model of hyperbolic space $\mathbb{H}^3$ in turn induces, via canonical extensions of these transformations, an isomorphism of $\text{M\"{o}b}(\hat{\mathbb{C}})$ with $\mathrm{Isom}^+(\mathbb{H}^3)$.
The rapport at work here between the \emph{geometrically unbounded} ball model of $\mathbb{H}^3$ and its \emph{topological boundary} (or \emph{sphere at infinity}) will be critical in what follows, as will the sorting of non-identity elements of $\mathrm{Isom}^+(\mathbb{H}^3)$ into the three classes \emph{parabolic}, \emph{loxodromic}, or \emph{elliptic} according to whether they fix $1$, $2$, or infinitely many points in $\mathbb{H}^3\cup\hat{\mathbb{C}}$, respectively; see \cite{MR2227047} for an efficient review of this material which aligns well with this section's main concerns.
Main book-length treatments include \cite{MR1638795, MR2553578, MR3586015} (see also \cite{MR1913879}); each abundantly attests the richness of this section's main subject, namely \emph{finitely generated Kleinian groups}, and the \emph{$\mathrm{PSL}(2,\mathbb{C})$-conjugacy relations among them}.

Our analysis will draw heavily on the manifold side of this story; let us therefore begin by reviewing its role in Section \ref{section:isometry_for_2}.
At the heart of that section were the \emph{Teichm\"{u}ller} and \emph{moduli} spaces $\mathcal{T}(S)$ and $\mathcal{M}(S)$, respectively, of a hyperbolizable surface $S$.
These are, intuitively, the spaces of \emph{marked} and \emph{unmarked} hyperbolic structures on $S$, notions we'll make more precise below.
Mediating between them is the \emph{mapping class group} $\mathrm{MCG}(S)$ of $S$; $\mathcal{M}(S)$ is, more precisely, the space of orbits of a properly discontinuous action of $\mathrm{MCG}(S)$ on $\mathcal{T}(S)$.
In the most clean and classical setting, $S$ is a closed oriented surface of finite genus $g$, punctured in some finite number $n$ of places, and $\mathcal{T}(S)$ and $\mathcal{M}(S)$ each parametrize \emph{finite-volume} hyperbolic structures on $S$; as noted, each in this case is a Polish space of dimension $6g+2n-6$.
Our interest was wider than this, though, extending, as it did, to the possibly infinite volume \emph{geometrically finite} (or, equivalently, \emph{algebraically finite}) structures on such $S$.
But the classification of such structures (weakly) reduces to the finite-volume case, and this was the crux of our main proof of Theorem \ref{thm:fgPSL2Rsmooth}.
Lastly, let us note for future reference the close relationship between $\mathrm{Out}(\pi_1(S))$ and $\mathrm{MCG}(S)$; by the Dehn--Nielsen--Baer Theorem, for example, the latter is an index-$2$ subgroup of the former if $S$ is closed (see \cite[Ch.\ 8]{MR2850125}).

Call the material of the preceding paragraph \emph{the $2$-dimensional template}; much of what follows can be understood in terms of how it both persists and breaks down in $3$ hyperbolic dimensions.
To begin with, the class of geometrically finite hyperbolic $3$-manifolds no longer coincides with that of the algebraically finite ones, and it is the former which now figures as the tractable class, the one whose elements are classifiable by points in classical Teichm\"{u}ller and moduli spaces.
Elements of the latter class, on the other hand, are in a natural sense limits of elements of the former; this is the celebrated \emph{Density Conjecture} of Bers, Sullivan, and Thurston, fully rendered a theorem in \cite{MR2821565, MR3001608}.
Heuristically, this carries the consequence that elements of the latter admit classification by \emph{limits of hyperbolic surface structures}, and this is the loose meaning of the eponymous \emph{ending lamination} in the following landmark theorem.

\begin{theoremELT}[\cite{MR2630036, MR2925381}]
A hyperbolic $3$-manifold with finitely generated fundamental group is uniquely determined by its topological type and its end invariants.
\end{theoremELT}

Variations of the theorem will appear below; the form recorded here is Brock, Canary, and Minsky's rendering of what had been known since Thurston's seminal \cite{MR0648524} as the \emph{Ending Lamination Conjecture}. (Valuable further references for this large subject include \cite{MR2044543}, \cite{MR2062319}, \cite{MR3586015}, and \cite{BowditchELT}, with the latter, in particular, supplying portions of the proof of the conjecture unrecorded in \cite{MR2630036, MR2925381}.)
For us the most immediate point is that the theorem's classifying pair (\emph{homeomorphism type}, \emph{ending invariants}) is naturally construed as an element of a Polish space, as will soon grow clear.\footnote{Note for comparison that the Ker\'{e}kj\'{a}rt\'{o}--Richards classification of surfaces in Section \ref{section:surfaces} was in terms of how discrete data (\emph{orientability}, \emph{planarity}) arrayed on elements of the uncountable Polish space $\mathcal{K}(\mathbb{N}^\mathbb{N})$ of end-configurations. In contrast, in the present section, where only finitely-ended manifolds are under consideration, the concern will be for how elements of uncountable Polish spaces of data (\emph{ends' asymptotic geometries}) array on elements of discrete spaces of end-configurations.}
Thus, much as for the main results of Sections \ref{section:surfaces}, \ref{section:bireducibility}, and \ref{section:isometry_for_2}, a major theorem of low-dimensional topology describes, from our perspective, a reduction --- one moreover which, if Borel, would appear to immediately settle the \emph{manifold framing} of this section's guiding question: \emph{What is the Borel complexity of the classification of algebraically finite hyperbolic $3$-manifolds up to isometry?}

But here we should tread carefully.
For upon closer inspection, the theorem describes a classification of \emph{marked} manifolds --- or, equivalently, on the group side, of \emph{representations}; its more careful formulation in \cite[p.\ 3]{MR2925381}, for example, is as follows (only the theorem's broader contours are presently the point; its more technical content will clarify as this section unfolds):

\begin{theoremELT}[the case of incompressible ends]
Let $\Gamma$ be a finitely generated torsion-free non-abelian group, let $\rho:\Gamma\to\mathrm{PSL}(2,\mathbb{C})$ be a discrete faithful representation, and let $N^0_\rho$ denote the hyperbolic manifold $\mathbb{H}^3/\rho(\Gamma)$ minus a standard open neighborhood of whatever cusps it may possess.
Suppose further that each of the ends of $N^0_\rho$ is incompressible.
Then the representation $\rho$ is determined up to conjugacy in $\mathrm{PSL}(2,\mathbb{C})$ by the marked homeomorphism type of the relative compact core, and the end invariants, of $N^0_\rho$.
\end{theoremELT}

The points to highlight here are, first, that the theorem's objects are discrete faithful representations of a fixed finitely generated $3$-manifold group $\Gamma$.
It will accordingly be convenient to conduct this section's complexity analyses --- whether of Kleinian groups or of hyperbolic manifolds --- in parallel with this machinery, that is, to segment our work into analyses of those groups isomorphic to some fixed such $\Gamma$, or, on the manifold side, homotopy equivalent to some fixed such $\mathbb{H}^3/\rho(\Gamma)$.
Since there are only countably many such groups (by \cite{MR380763}), and the isomorphism relation between them is plainly analytic, this practice is, from a Borel complexity perspective, an innocuous one: it will amount on either the group or manifold side to localizing attention to an element of a partition of our parameter space into countably many Borel subsets.
The details of the relevant arguments are by now sufficiently routine that we won't belabor them here or below.

Second, as noted, the theorem classifies discrete faithful representations in \allowbreak $\mathrm{PSL}(2,\mathbb{C})$, or what we may think of as \emph{marked} Kleinian groups, up to conjugacy.
Classifying \emph{unmarked} Kleinian groups isomorphic to $\Gamma$ up to conjugacy, or \emph{unmarked} hyperbolic $3$-manifolds homotopy equivalent to $\mathbb{H}^3/\rho(\Gamma)$ up to isometry, then amounts to forgetting the marking, or, more mathematically, to quotienting by a natural action of $\mathrm{Out}(\Gamma)$, much as in the $2$-dimensional template.
In these contexts, however, unlike in the $2$-dimensional setting, this action is \emph{not} in general properly discontinuous, and it is here that the story grows interesting.

It is time to be more precise. To that end, fix a finitely generated fundamental group $\Gamma$ of an orientable hyperbolic $3$-manifold and consider the following spaces:
\begin{itemize}
\item $\mathcal{D}[\Gamma]=\{\Gamma'\in\mathcal{D}(\mathrm{PSL}(2,\mathbb{C}))\mid\Gamma'\cong\Gamma\}$. By the reasoning above, $\mathcal{D}[\Gamma]$ is a Borel subset of $\mathcal{D}_{\mathrm{af}}(\mathrm{PSL}(2,\mathbb{C}))$ and, thus, itself is a standard Borel space.
\item $\mathrm{DF}(\Gamma)=\{\rho:\Gamma\to\mathrm{PSL}(2,\mathbb{C})\mid\rho\text{ is discrete and faithful}\}$, topologized as a subspace of $\mathrm{Hom}(\Gamma,\mathrm{PSL}(2,\mathbb{C}))$ endowed with the compact-open topology (it is closed therein \cite{MR427627}).
For an equivalent description, fix some $n$ generators $g_i$ of $\Gamma$ and view $\mathrm{DF}(\Gamma)$ as a subspace of the product space $\mathrm{PSL}(2,\mathbb{C})^n$ via the identification $\rho\mapsto (\rho(g_1),\dots,\rho(g_n))$.
Thus $\mathrm{DF}(\Gamma)$ may be regarded as a space of \emph{marked} Kleinian groups which are group-isomorphic to $\Gamma$.
\item $\mathrm{AH}(\Gamma)$ is the space of orbits of the conjugacy action of $\mathrm{PSL}(2,\mathbb{C})$ on $\mathrm{DF}(\Gamma)$ given by $(h\cdot\rho)(g)=h\rho(g)h^{-1}$.
\end{itemize}
Connecting these spaces are natural maps $p:\mathrm{DF}(\Gamma)\to\mathcal{D}[\Gamma]$ and $q:\mathrm{DF}(\Gamma)\to\mathrm{AH}(\Gamma)$.
\begin{lemma}
\label{lem:AH_is_Polish}
The map $p$ is Borel, the map $q$ is continuous, and if $\Gamma$ is nonabelian then $\mathrm{AH}(\Gamma)$ is Polish.
\end{lemma}
\begin{proof}
The sets $U_V=\{\Gamma'\in\mathcal{D}[\Gamma]\mid \Gamma'\cap V\neq\varnothing\}$, with $V$ ranging over the open subsets of $\mathrm{PSL}(2,\mathbb{C})$, generate the Borel $\sigma$-algebra of $\mathcal{D}[\Gamma]$. Thus to show $p$ Borel measurable, it suffices to show that the sets $p^{-1}[U_V]=\{\rho\in\mathrm{DF}(\Gamma)\mid \mathrm{im}(\rho)\cap V\neq\varnothing\}$ are Borel.
In fact they're open, for if $\rho(g)\in V$ for some word $g_{i(0)}\cdots g_{i(k)}$ in the generators of $\Gamma$ then the product of sufficiently small neighborhoods of those generators will fall within $V$ as well.

The map $q$ is continuous by definition. The more remarkable fact is that, since discrete and faithful $\mathrm{PSL}(2,\mathbb{C})$-representations of nonabelian torsion-free groups are nonelementary (see the discussion preceding Lemma \ref{lem:elementary} below for a definition), $\mathrm{AH}$ of a nonabelian $\Gamma$ forms a closed subset of the character variety, and in fact smooth complex manifold, denoted as $\mathcal{R}^0(\Gamma,\mathrm{PSL}(2,\mathbb{C}))$ in \cite{MR2553578} (see its sections 4.3, 8.1, and 8.8 for details; the key point is that the discrete and faithful representations of such a $\Gamma$ are necessarily nonradical in the sense of its p.\ 62).
Our conclusion follows.
\end{proof}

Two more objects round out the picture; since our main question revolves less around the Borel structures they possess than those they might \emph{admit}, we regard them for the moment simply as sets:
\begin{itemize}
\item $C(\Gamma)$ is the quotient of $\mathcal{D}[\Gamma]$ by $E(\mathrm{PSL}(2,\mathbb{C}),\mathcal{D}[\Gamma])$, or in other words by the natural conjugacy action of $\mathrm{PSL}(2,\mathbb{C})$.
\item $\mathcal{E}(\Gamma)$ is the collection of complete invariants (\emph{homeomorphism type}, \emph{ending invariants}) which, by the Ending Lamination Theorem, classify the points of $\mathrm{AH}(\Gamma)$.
\end{itemize}
We have then a bijection $e:\mathrm{AH}(\Gamma)\to\mathcal{E}(\Gamma)$, as well as natural set maps $u:\mathcal{D}[\Gamma]\to\mathrm{C}(\Gamma)$ and $v:\mathrm{AH}(\Gamma)\to\mathrm{C}(\Gamma)$, with the latter the quotient by the $\zeta$ of the following lemma.
\begin{lemma}
\label{lem:commuting_square}
Writing $\zeta$ for the natural action of $\mathrm{Out}(\Gamma)$ on $\mathrm{AH}(\Gamma)$, $C(\Gamma)$ is isomorphic to the set of $\zeta$-orbits in $\mathrm{AH}(\Gamma)$; moreover, if $\Gamma$ is nonabelian, then $E(\mathrm{PSL}(2,\mathbb{C}),\mathcal{D}[\Gamma])$ and the orbit equivalence relation $E_\zeta$ induced by $\zeta$ are Borel bireducible.
\end{lemma}
\begin{proof}
The natural action of $\mathrm{Aut}(\Gamma)$ on $\mathrm{DF}(\Gamma)$ descends to one on $\mathrm{AH}(\Gamma)$; $\mathrm{Inn}(\Gamma)$ falls within the kernel of the latter, and it is the induced action of $\mathrm{Out}(\Gamma)=\mathrm{Aut}(\Gamma)/\mathrm{Inn}(\Gamma)$ which we are denoting by $\zeta$.
The composition $up$ forgets first the marking, then the representative, of a conjugacy class, while $vq$ forgets them in the opposite order, but the effect is the same; more precisely, the main square in the diagram just below commutes.
For a Borel reduction of $E(\mathrm{PSL}(2,\mathbb{C}),\mathcal{D}[\Gamma])$ to $E_\zeta$, note that since $\Gamma$ is finitely generated, $\mathrm{Aut}(\Gamma)$ is countable, hence Luzin--Novikov uniformization furnishes a Borel right-inverse $s:\mathcal{D}[\Gamma]\to\mathrm{DF}(\Gamma)$ to $p$; by the aforementioned commutativity, $qs:\mathcal{D}[\Gamma]\to\mathrm{AH}(\Gamma)$ is then a reduction as desired.
For a Borel reduction of $E_\zeta$ to $E(\mathrm{PSL}(2,\mathbb{C}),\mathcal{D}[\Gamma])$, note that Lemma \ref{lem:AH_is_Polish} together with \cite[Exer.\ 18.20(iii)]{MR1321597} furnishes a Borel right-inverse $t$ to $q$; just as above, $pt:\mathrm{AH}(\Gamma)\to\mathcal{D}[\Gamma]$ is then a reduction as desired.
\end{proof}
These maps and objects array into a commutative diagram of sets as follows; here $w$ is simply $v\circ e^{-1}$, and $s$ and $t$ are as in the proof above (the existence of $t$, in particular, is premised on $\Gamma$ being non-abelian):
\begin{equation}
\label{eq:diagram}
\begin{tikzcd}[row sep=large,column sep=large]
\mathrm{DF}(\Gamma) \arrow[d, "p"] \arrow[r, "q"'] & \mathrm{AH}(\Gamma) \arrow[d, "v"'] \arrow[r, "e"] \arrow[l, "t"', dashed, bend right] & \mathcal{E}(\Gamma) \arrow[ld, "w"] \\
{\mathcal{D}[\Gamma]} \arrow[r, "u"] \arrow[u, "s", dashed, bend left] & \mathrm{C}(\Gamma)                                &                                   
\end{tikzcd}
\end{equation}
Recall now the view of $\mathrm{DF}(\Gamma)$ as a \emph{space of marked Kleinian groups isomorphic to $\Gamma$}.
Next, regard $\Gamma$ as the fundamental group of some fixed complete hyperbolic $3$-manifold $N$; $\mathrm{DF}(\Gamma)$ may then also be viewed as \emph{the space of marked complete hyperbolic $3$-manifolds $M_\rho$ which are homotopy equivalent to $N$}.
Here $M_\rho$ is what in the notation of Theorem \ref{thm:discretetomanifolds} would be denoted $M_{\sigma(\mathrm{im}(\rho))}$, but now the isomorphism $\rho:\pi_1(N)\to\mathrm{im}(\rho)$ encodes something further as well: by Whitehead's Theorem \cite[Theorem 4.5, Prop.\ 1B.9]{MR1867354}, it encodes the homotopy class of a homotopy equivalence $N\to M_\rho$, and it is this that we take as our definition of a \emph{marking} of $M_\rho$.
Note that the passage between these two perspectives requires nothing more than the machinery of Section \ref{section:bireducibility}.
In particular, it is Borel, and $\mathrm{PSL}(2,\mathbb{C})$-conjugacy on the group side continues to correspond to orientation-preserving isometry on the manifold side; formal details are left to the reader.
With the aid of a lemma from Section \ref{subsection:the_basics}, we are now in a position to prove the following.
\begin{proposition}
\label{prop:marked_smooth}
Fix a complete orientable hyperbolic $3$-manifold $N$ with finitely generated $\pi_1(N)$.
The classification of marked complete hyperbolic $3$-manifolds homotopy equivalent to $N$ up to marked orientation-preserving isometry is Borel reducible to $=_{\mathbb{R}}$; so, thus, is the classification of all marked complete hyperbolic $3$-manifolds with finitely generated fundamental group.
\end{proposition}

\begin{proof}
Since the concluding classification problem amounts to a countable sum of those of the first sort, it will suffice to prove the first claim.
By our discussion just above, our task is to show that the $\mathrm{PSL}(2,\mathbb{C})$-conjugacy relation on $\mathrm{DF}(\pi_1(N))$ Borel reduces to the equality relation on a Polish space.
By Lemma \ref{lem:AH_is_Polish}, if $\pi_1(N)$ is nonabelian then $q:\mathrm{DF}(\pi_1(N))\to\mathrm{AH}(\pi_1(N))$ is such a reduction.
If, on the other hand, $\pi_1(N)$ is abelian, then its discrete faithful images in $\mathrm{PSL}(2,\mathbb{C})$ are elementary; for a definition, see the discussion preceding Lemma \ref{lem:elementary}, whereupon the latter furnishes us with a Borel reduction of $\mathrm{PSL}(2,\mathbb{C})$-conjugacy on $\mathcal{D}[\pi_1(N)]$ to $=_{\mathbb{C}}$.
Precomposing this with the map $p:\mathrm{DF}(\pi_1(N))\to\mathcal{D}[\pi_1(N)]$ then defines a weak reduction of $\mathrm{PSL}(2,\mathbb{C})$-conjugacy on $\mathrm{DF}(\pi_1(N))$ to $=_{\mathbb{C}}$, and this, by Lemma \ref{lem:weak_reduction}, completes the argument.
\end{proof}
On display here is, arguably, a limitation of the invariant descriptive set theoretic framework: its failure to distinguish between abstract and more ``meaningful'' solutions to classification problems.
The essential content of Proposition \ref{prop:marked_smooth}, namely that marked algebraically finite hyperbolic $3$-manifolds naturally, but somewhat abstractly, correspond up to marked isometry to points in a Polish space $\mathrm{AH}(\pi_1(M))$ has been known for a long time.
But it was only with the provision, via the Ending Lamination Theorem, of more geometrically meaningful complete invariants for such $M$ that their classification problem was generally accredited as solved.

As noted in Section \ref{subsection:invariant_DST}, however, \emph{anti}-classification results derive much of their force from precisely this indiscriminancy.
And it is results of this latter sort which form this section's main contribution.
In contrast with both Proposition \ref{prop:marked_smooth} and the Ending Lamination Theorem, we will show not only that the most obvious assignment of complete invariants to fundamental classes of unmarked hyperbolic $3$-manifolds fails to take values in any Polish space, but that \emph{every} Borel assignment must fail to.
More precisely:
\begin{theorem}
\label{thm:line_bundles}
Let $S$ be a closed orientable hyperbolizable surface.
The classification up to isometry of doubly degenerate hyperbolic manifolds homeomorphic to $S\times\mathbb{R}$ is Borel equivalent to $E_0$.
In particular, there exists no Borel assignment of complete numerical invariants, up to $\mathrm{PSL}(2,\mathbb{C})$-conjugacy, to the elements of $\mathcal{D}[\pi_1(S)]$.
\end{theorem}
The theorem supplies the third item of the following.
\begin{theorem}
\label{thm:complexity_3-mans}
The Borel complexity degrees of the isometry problem for the major finiteness classes of complete hyperbolic $3$-manifolds are as follows:
\begin{itemize}
\item for finite-volume manifolds, it is $=_{\mathbb{N}}$;
\item for geometrically finite manifolds, it is $=_{\mathbb{R}}$;
\item for algebraically finite manifolds, it is at least $E_0$.
\end{itemize}
Corresponding results hold for the $\mathrm{PSL}(2,\mathbb{C})$-conjugacy relation on the corresponding classes of Kleinian groups.
\end{theorem}
Thus, intriguingly, we see both a rise and a fall in the complexity of isometry problem for classes of hyperbolic manifolds as we pass from $2$ dimensions to $3$ (compare Theorem \ref{thm:fgPSL2Rsmooth}).
We will argue Theorem \ref{thm:complexity_3-mans}'s three items in order in the following three sections, recording a number of related results along the way.

\subsection{Classifying the basics}
\label{subsection:the_basics}
To better suggest the scope of the Ending Lamination Theorem, it is useful to first consider the latter in a few restricted contexts; let us proceed according to the hierarchy of finiteness conditions recorded in Definition \ref{def:finiteness_conditions}.
Applied to item (\ref{cond:closed}) of these, for example, the Ending Lamination Theorem tells us that \emph{homeomorphism type} alone is a complete invariant for closed hyperbolic $3$-manifolds up to isometry, since the latter have no ends; as such, it recovers the dimension $3$ instance of Mostow's Rigidity Theorem \cite{MR236383}.\footnote{The theorem's more obvious implication is that \emph{marked} homeomorphism type is a complete invariant for \emph{marked} closed $3$-manifolds up to isometry, but from this the unmarked version follows.
Similar remarks apply to the finite-volume case.}
Within our framework, the corollary that \emph{the classification of closed hyperbolic $3$-manifolds up to isometry is Borel bireducible with $=_{\mathbb{N}}$} is then immediate, by \cite{MR256399}, much as it was for Corollary \ref{cor:compact_manifolds_countable}.
Turning next to item (\ref{cond:finvolume}), \emph{finite-volume} complete hyperbolic $n$-manifolds, we know from Lemma \ref{lem:finiteness_conditions_Borel} that they define a Borel subspace of $\mathfrak{C}^*_c(\mathsf{Isom}^+,\mathbb{H}^n)$, and from \cite[Cor.\ 3.2]{MR1283875} and Corollary \ref{cor:classwise} that their classification is Borel reducible to $=_{\mathbb{R}}$, as noted above.
When $n=3$, the Ending Lamination Theorem supplies something a little better, namely a \emph{countable} list of complete invariants (namely, by the proof just below, the homeomorphism types of the associated pared manifolds discussed at the end of this subsection); much more generally, Stuck--Zimmer's result admits the following refinement:

\begin{proposition}
\label{prop:nonmonotonic}
For any $n\geq 3$, the classification of finite-volume complete hyperbolic $n$-manifolds up to isometry, and hence of lattices in $\mathrm{Isom}(\mathbb{H}^n)$ up to conjugacy, is Borel bireducible with $=_{\mathbb{N}}$.
\end{proposition}
Clues to this fact are the pretty results that, viewed as a suborder of $\mathbb{R}$, the collection of finite volumes of complete hyperbolic $n$-manifolds is of ordertype $\omega^\omega$ when $n=3$ (by J\o rgensen--Thurston; see \cite{zbMATH03707297}), and of ordertype $\omega$ when $n>3$ \cite{zbMATH03367275, zbMATH01925566}, and thus that this collection, in any event, is countable.
The cleanest approach, though, is probably the following.
\begin{proof}
By Mostow--Prasad rigidity \cite{MR385005}, 
such manifolds are fully classified by their fundamental groups, and by \cite[Theorem 12.7.4]{MR4221225}, any such manifold deformation retracts onto a compact one.
Hence again by \cite{MR256399} there are only countably many such manifolds, and the argument concludes just as for Corollary \ref{cor:compact_manifolds_countable}.
\end{proof}
As noted, by the above coupled with Theorem \ref{thm:fgPSL2Rsmooth}, the classification problem for finite-volume hyperbolic $n$-manifolds \emph{decreases} (strictly, then weakly) in complexity as $n$ rises.
For the more general question of how the complexity of manifold classification problems varies with dimension, see our conclusion's Question \ref{ques:monotonic} and surrounding discussion.

We venture next onto the terrain of infinite-volume manifolds and, thereby, onto more substantial considerations of manifolds' ends.
Within the framework of the Ending Lamination Theorem, this is the point at which markings of these manifolds and of their ends assume a greater importance; so, in consequence, do the distinctions organizing diagram (\ref{eq:diagram}) of this section's introduction as well.
Recall that the preferred setting for that diagram was when $\Gamma$ was non-abelian.
By dispensing next with the class of elementary torsion-free Kleinian groups, we will free ourselves to thereafter restrict our attention, as is common, to that preferred setting.
We will review a few other objects for subsequent use along the way.

Recall the identification, discussed above, of $\hat{\mathbb{C}}$ with the sphere at infinity enclosing the Poincar\'{e} ball model of $\mathbb{H}^3$, and the attendant identification of the group of M\"{o}bius transformations of $\hat{\mathbb{C}}$ with $\mathrm{Isom}^+(\mathbb{H}^3)$.
Note next that any discrete subgroup $\Gamma$ of $\mathrm{Isom}^+(\mathbb{H}^3)$ determines a closed \emph{limit set} $\Lambda(\Gamma)\subseteq\hat{\mathbb{C}}$ of limit points of the action of $\Gamma$ on $\mathbb{H}^3$, along with a complementary \emph{ordinary set} or \emph{domain of discontinuity} $\Omega(\Gamma)=\hat{\mathbb{C}}\backslash\Lambda(\Gamma)$.
Such a $\Gamma$ is \emph{elementary} if the set $\Lambda(\Gamma)$ is finite,
or, equivalently, if $\Gamma$ contains an abelian subgroup of finite index.
Elementary Kleinian groups are of such anomalous simplicity as to regularly merit separate treatments, a tradition on display in both the first proof of Theorem \ref{thm:fgPSL2Rsmooth} and the next lemma, which will help to streamline our subsequent analyses.
Moreover, the class $\mathcal{E}$ of elementary torsion-free Kleinian groups forms a Borel subset of $\mathcal{D}_{\mathrm{af}}(\mathrm{PSL}(2,\mathbb{C}))$; to see this, observe again that the group-isomorphism relation on the latter is analytic and possesses only countably many classes, and that $\mathcal{E}$ consists of exactly $2$ of them, as the proof below will show.
\begin{lemma}
\label{lem:elementary}
The classification of elementary torsion-free Kleinian groups up to conjugacy is Borel bireducible with the relation $=_{\mathbb{R}}$ of equality on the real numbers.
\end{lemma}
\begin{proof} See \cite[pp.\ 3, 60-61]{MR3586015} and \cite[Thm.\ 4.3.1]{MR1393195} for arguments that such a group $\Gamma$ of M\"{o}bius transformations is either:
\begin{enumerate}
\item cyclic, generated by a loxodromic element conjugate to a transformation of the form $z\mapsto\lambda z$ for some unique $\lambda\in\mathbb{C}$ with $|\lambda|>1$;
\item cyclic, generated by a parabolic element conjugate to the transformation $z\mapsto z+1$;
\item a rank $2$ abelian group, conjugate to a group generated by the elements $z\mapsto z+1$ and $z\mapsto z+\lambda$ for some unique $\lambda\in\mathbb{C}$ with positive imaginary part.
\end{enumerate}
Such groups, in other words, admit complete invariants of the form (\emph{type}, $\lambda)\in\{1,2,3\}\times\mathbb{C}$ (with $\lambda=0$ when the type is $2$).
That the assignment of these invariants is, over each type, continuous is not hard to see; a type $1$ $\Gamma$'s associated $\lambda$ falls in some open $U\subseteq\mathbb{C}$, for example, if and only if its conjugate taking the repelling and attracting fixed points (with respect to a choice of generator $g\in\Gamma$) of $\Gamma$ to $0$ and $\infty$, respectively, maps $g$ to a $z\mapsto\kappa z$ with $\kappa\in U$.
This will then hold as well for all type 1 $\Gamma'$ with generators in a sufficiently small open neighborhood of $g$, and the latter is a relatively open condition; details of this and the remaining cases are left to the reader.
(To be clear, we do not claim that the assignment is continuous as a map to $\{1,2,3\}\times\mathbb{C}$: we'll note that it isn't below.
We claim only that the assignment is continuous on the preimage of each $\{i\}\times\mathbb{C}$ (as one could alternately argue for the matrix representations of these groups via the trace function), and that this establishes the lemma.)
\end{proof}
With Lemma \ref{lem:elementary} behind us, we will henceforth focus on \emph{nonelementary} torsion-free Kleinian groups; within diagram \ref{eq:diagram}, this will free us to assume that $\Gamma$ is non-abelian, whereupon Lemma \ref{lem:AH_is_Polish} will apply.
In particular, the topology which $\mathrm{AH}(\Gamma)$ therein inherits from $\mathrm{DF}(\Gamma)$ will be Polish.
Since at least \cite{MR4554426}, this topology has been known as \emph{the algebraic topology} (this is the source of the ``A'' in $\mathrm{AH}(\Gamma)$; the ``H'' is for \emph{hyperbolic}), in contradistinction to \emph{the geometric} (i.e., \emph{Chabauty--Fell}) \emph{topology} predominating in our Sections \ref{section:bireducibility} and \ref{section:isometry_for_2} above.
The natural setting for the latter topology is a set, like $\mathcal{D}[\Gamma]$, of closed subgroups of $\mathrm{PSL}(2,\mathbb{C})$, as we've seen --- but from there it pulls back along $p:\mathrm{DF}(\Gamma)\to\mathcal{D}[\Gamma]$ to define a topology on the set underlying $\mathrm{DF}(\Gamma)$ which we will denote $\tau_G$.
By a standard abuse, we will also term $\tau_G$, as well as its associated quotient topology on the set underlying $\mathrm{AH}(\Gamma)$, \emph{the geometric topology}; similarly for \emph{the algebraic topology} $\tau_A$ which we first endowed $\mathrm{DF}(\Gamma)$ with when introducing it above.

The interactions between these two topologies can be quite subtle: as we will see shortly, a sequence of groups or manifolds may converge to one thing in one topology and to another in the other.
The evolution of associated objects alongside them (elements' limit sets or domains of discontinuity, for example) is, in general, more transparent when these two limits agree.
More formally, it is more transparent when convergence is with respect to the so-called \emph{strong topology} $\tau_S$ generated by $\tau_A\cup\tau_G$ on $\mathrm{DF}(\Gamma)$.
We write $\mathrm{SH}(\Gamma)$ (``S'' is for \emph{strong}) to denote the underlying set of $\mathrm{AH}(\Gamma)$ endowed with the quotient topology induced by $\tau_S$.
The following observation will be useful below.
\begin{lemma}
\label{lem:SH_Borel}
For any finitely generated Kleinian group $\Gamma$, the Borel $\sigma$-algebras generated by $\tau_A$ and $\tau_S$ on $\mathrm{DF}(\Gamma)$ and, thus, on $\mathrm{AH}(\Gamma)$ and $\mathrm{SH}(\Gamma)$ coincide.
\end{lemma}
\begin{proof}
To prove the second assertion, clearly it suffices to prove the first one.
For this, fix a generating set $\{g_i\mid i<n\}$ of elements of $\Gamma$ and view $\mathrm{DF}(\Gamma)$ as a subspace of the product space $\mathrm{PSL}(2,\mathbb{C})^n$ in the manner of its initial definition above, so that as $i$ and $V$, respectively, range below $k$ and among the open subsets of $\mathrm{PSL}(2,\mathbb{C})$, the sets
$$U^A_{i,V}:=\{\rho:\Gamma\to\mathrm{PSL}(2,\mathbb{C})\mid \rho(g_i)\in V\}$$
define a subbasis for $\tau_A$.
Similarly, as $V$ ranges among the open subsets of $\mathrm{PSL}(2,\mathbb{C})$, the sets
$$U^G_{V}:=\{\rho:\Gamma\to\mathrm{PSL}(2,\mathbb{C})\mid \mathrm{im}(\rho)\cap V\neq\varnothing\}$$
generate the Borel $\sigma$-algebra $\mathcal{B}_G$ generated by $\tau_G$; write $\mathcal{B}_A$ and $\mathcal{B}_S$, respectively, for the algebras generated by $\tau_A$ and $\tau_S$ as well.
From this presentation it's clear that each $U^G_{V}$ is open in $\tau_A$, and hence that $\mathcal{B}_G\subseteq\mathcal{B}_A$.
Thus $\mathcal{B}_S=\mathcal{B}_A$.
\end{proof}
Let us now tie a few threads of our discussion together.
In the proof of Lemma \ref{lem:elementary}, we hinted that the type-class (1) therein isn't closed.
In fact there exist sequences of representations $\rho_i:\mathbb{Z}\to\mathrm{PSL}(2,\mathbb{C})$ with images all of type (1) which converge algebraically and geometrically to groups of type (2) and type (3), respectively.
For elegant description of one example, see \cite[Example 9.1.4]{MR4554426}; see also the treatment in \cite[\S 4.9]{MR3586015}. These are the building blocks of more complex and arresting examples (of simultaneous algebraic and geometric convergence, respectively, to finitely and infinitely generated groups, for instance (in particular, $\mathcal{D}_{\mathrm{fg}}(\mathrm{PSL}(2,\mathbb{C}))$ is not closed)), and they are emblems of the ways that loxodromic elements of Kleinian groups can ``degenerate'', in those groups' limits, into parabolic subgroups, or, at the level of the quotient manifolds, into \emph{cusps}.

We encountered such structures in Section \ref{section:isometry_for_2}; in three hyperbolic dimensions the story is as follows (see \cite[\S 3.2]{MR3586015} for further discussion):
\begin{itemize}
\item The parabolic elements of a Kleinian group $\Gamma$ organize into maximal parabolic subgroups of the stabilizers of their fixed points.
These are of rank 1 or 2 and patterned on the second or third type-class, respectively, of Lemma \ref{lem:elementary} (whose representatives therein fix the point $\infty\in\hat{\mathbb{C}}$).
Sullivan showed in \cite{MR639042} that if $\Gamma$ is finitely generated then these subgroups sort into finitely many conjugacy classes within $\Gamma$; these, in turn, correspond to the cusps of $N=\mathbb{H}^3/\Gamma$.
\item Topologically, the latter form ends (or what might be more precisely termed \emph{end neighborhoods}) of $N$ which, like their Lemma \ref{lem:elementary} prototypes' quotients, are either homeomorphic to $S^1\times\mathbb{R}\times\mathbb{R}$ or to $S^1\times S^1\times\mathbb{R}$.
\item Geometrically, cusps fall within the $\varepsilon$-\emph{thin part} of their ambient manifold $N$, where $\varepsilon$ is at most the Margulis constant for hyperbolic $3$-space: in a sense further encapsulated by horoballs, the geometry of $N$ attenuates in its cusps, ``pulling away'' from its main and thicker parts.
\end{itemize}
In consequence of this last point, hyperbolic $3$-manifolds $N$ are often studied via their \emph{decusped} portion $N^0$, resulting from replacing neighborhoods of $N$'s rank-$1$ and rank-$2$ cusps by (or, equivalently, retracting them to) boundary annuli or tori, respectively.
This is the meaning of the $N^0_\rho$ appearing in the second statement of the Ending Lamination Theorem in Section \ref{subsection:overview}.

At times, this practice in turn dictates a finer approach to diagram (\ref{eq:diagram}).
See \cite[\S 7]{MR2096234} or \cite[\S 4]{MR3001608} for succinct accounts of the deformation spaces $\mathrm{AH}(M,P)$ of \emph{pared manifolds} $(M,P)$ like $N^0_\rho$ above.
We'll forego this formalism except in the proof of Theorem \ref{thm:geom_fin_hyp_3_mans_smooth} and in Section \ref{subsection:classifying_gi_ends} below; each of the latter, though, will involve an appeal to the subdiagram of (\ref{eq:diagram}) associated to $\mathrm{AH}(M,P)$.
We defer a discussion of the details until then; note for later use, though, the following consequence of the finiteness recorded in the first bulleted item above:
\begin{lemma}
\label{lem:countable_par-sets}
Write $\mathsf{par}$ for the set of parabolic elements of $\mathrm{PSL}(2,\mathbb{C})$.
For any finitely generated $\pi_1(M)$ of a complete orientable hyperbolic $3$-manifold $M$, the set $\mathcal{P}(\pi_1(M))=\{\rho^{-1}(\mathsf{par})\mid\rho\in\mathrm{DF}(\pi_1(M))\}$ is countable.
\qed
\end{lemma}

We've now surveyed most of the basics that we need to proceed, in Sections \ref{subsection:classifying_gf_ends} and \ref{subsection:classifying_gi_ends}, respectively, to the classes of geometrically finite and algebraically finite hyperbolic $3$-manifolds.
To see that the former defines a standard Borel space, however, we must first keep a promise from Section \ref{section:isometry_for_2}: we must show that computing a manifold's convex core is a Borel operation.

To that end, recall from above the \emph{limit set} $\Lambda(\Gamma)$ and \emph{domain of discontinuity} $\Omega(\Gamma)=\hat{\mathbb{C}}\backslash\Lambda(\Gamma)$ associated to a Kleinian group $\Gamma$.
Each induces further objects of interest: the natural action of $\Gamma$ on $\Omega(\Gamma)$ is, as the latter's name suggests, properly discontinuous, with the result that $\Omega(\Gamma)/\Gamma$ is a $2$-manifold encoding essential information about the ends of $M=\mathbb{H}^3/\Gamma$, as we'll see in Section \ref{subsection:classifying_gf_ends}.
Relatedly, let $CH(\Lambda(\Gamma))$ denote the convex hull, in $\mathbb{H}^3$, of $\Lambda(\Gamma)$; the quotient manifold $CH(\Lambda(\Gamma))/\Gamma$ is then precisely the convex core $CC(M)$ of $M$.
Let us check that the derivations of these objects from $\Gamma$ are Borel.
\begin{lemma}
\label{lem:limitsets}
For any finitely generated nonelementary torsion-free Kleinian group $\Delta$, the maps
\begin{enumerate}[label=\textup{\arabic*.}]
\item $f_1:\mathcal{D}[\Delta]\to\mathcal{F}(\hat{\mathbb{C}}):\Gamma\mapsto\Lambda(\Gamma)$,
\item $f_2:\mathcal{D}[\Delta]\to\mathcal{F}(\mathbb{H}^3):\Gamma\mapsto CH(\Lambda(\Gamma))$, and
\item $f_3:\mathcal{D}[\Delta]\to\mathcal{O}(\hat{\mathbb{C}}):\Gamma\mapsto\Omega(\Gamma)$, and
\end{enumerate}
are each Borel measurable.
\end{lemma}
\begin{proof}
By \cite[Theorem 4.5.6]{MR3586015}, the map assigning limit sets to points in $\mathrm{SH}(\Delta)$ is continuous, and $f_1$ results from precomposing it with the Borel map $qs$ of diagram \ref{eq:diagram}.
Thus $f_1$ is Borel, and since complementation induces a homeomorphism $\mathcal{F}(\hat{\mathbb{C}})\to\mathcal{O}(\hat{\mathbb{C}})$, so too is $f_3$. The function $f_2$ is then the composition $hgf_1$, where $g:\mathcal{F}(\hat{\mathbb{C}})\to\mathcal{F}(\mathbb{H}^3\cup\hat{\mathbb{C}})$ maps $F$ to its convex hull in $\mathbb{H}^3\cup\hat{\mathbb{C}}$ and $h:\mathcal{F}(\mathbb{H}^3\cup\hat{\mathbb{C}})\to\mathcal{F}(\mathbb{H}^3)$ intersects the result with $\mathbb{H}^3$.
Since, plainly, $h$ is Borel, for item (2) it only remains to check that $g$ is as well.
As $U$ ranges among the open subsets of $\mathbb{H}^3\cup\hat{\mathbb{C}}$, the subbasic open sets $O_1(U)=\{H\in \mathcal{F}(\mathbb{H}^3\cup\hat{\mathbb{C}})\mid H\cap U\neq\varnothing\}$ generate the family of Borel subsets of $\mathcal{F}(\mathbb{H}^3\cup\hat{\mathbb{C}})$. Thus it suffices to show that $g^{-1}(O_1(U))$ is Borel for any such $U$. In fact, it is open: given $F\in g^{-1}(O_1(U))$ with some convex combination of $x_0,\dots,x_k\in F$ as witness, clearly there exist neighborhoods $V_0\ni x_0,\dots,V_k\ni x_k$ such that the same combination of any $y_0\in V_0,\dots,y_k\in V_k$ falls in $U$ as well, and thus that
$$F\in\bigcap_{i\leq k}\{H\in\mathcal{F}(\hat{\mathbb{C}})\mid H\cap V_i\neq\varnothing\}\subseteq g^{-1}(O_1(U)).$$
This completes the proof.
\end{proof}
It's now straightforward to conclude that the assignment of the convex core of $\mathbb{H}^3/\Gamma$ (or, equivalently, of its complementary open submanifold) to $\Gamma$ is a Borel operation.
\begin{lemma}
\label{lem:convexcore}
For any finitely generated nonelementary torsion-free Kleinian group $\Delta$, there exists a Borel map $f_4:\mathcal{D}[\Delta]\to\mathfrak{C}_c(\mathsf{Isom}^+,\mathbb{H}^3)^2$ encoding the convex core computation $\Gamma\mapsto CC(\mathbb{H}^3/\Gamma)$.
\end{lemma}
\begin{proof}
In the notation of Lemma \ref{lem:exhaustion}, $f_4(\Gamma)$ will equal $((\mathcal{U},c),(\mathcal{V},c|_\mathcal{V}))$, with $(\mathcal{U},c)$ equaling the $\sigma(\Gamma)$ of Theorem \ref{thm:discretetomanifolds}. Each $V_i$ is then $U_i\backslash CH(\Lambda(\Gamma))$; similarly for each $V_{i,j}$. That this assignment is Borel is immediate from Lemma \ref{lem:limitsets}. That these choices of $V_i$ cohere, so that $(\mathcal{V},c|_\mathcal{V})$, moreover, defines an open submanifold of $M_{(\mathcal{U},c)}$, follows from the fact that $CH(\Lambda(\Gamma))$ is $\Gamma$-invariant.
\end{proof}
The following corollary then ties up all loose ends from Section \ref{section:isometry_for_2}.
For simplicity, write $\mathfrak{C}$ for the Borel parameter space of complete orientable hyperbolic $2$-manifolds with finitely generated noncyclic fundamental group, and recall that it was just these sorts of manifolds whose convex cores we doubled in our first proof of Theorem \ref{thm:fgPSL2Rsmooth} above.
\begin{corollary}
\label{cor:loose_ends}
$\mathcal{D}_{\mathrm{gf}}(\mathrm{PSL}(2,\mathbb{C}))$ is a Borel subset of $\mathcal{D}_{\mathrm{tf}}(\mathrm{PSL}(2,\mathbb{C}))$.
In addition, there exists a Borel function $\mathfrak{C}\to\mathfrak{C}_c(\mathsf{Isom}^+,\mathbb{H}^2)^2: M\mapsto CC(M)$ computing a hyperbolic $2$-manifold's convex core.
\end{corollary}
\begin{proof}
For the first assertion, note that by Lemma \ref{lem:convexcore} coupled with the argument, for example, of \cite[Thm.\ I.3.1.4]{MR903850}, the collection of nonelementary $\Gamma$ possessing a convex core of volume at most $k$ is, for any $k>0$, a Borel set.
Thus the collection of those possessing a finite-volume convex core is Borel as well.
Since the convex cores associating to torsion-free elementary Kleinian groups are all also of finite volume (consider their limit sets), Lemma \ref{lem:elementary} then completes the argument.

For the second assertion, let $H$ denote the geodesic plane with boundary $\hat{\mathbb{R}}\subseteq\hat{\mathbb{C}}$ in the Poincar\'{e} ball model of $\mathbb{H}^3$, so that $H$ is a $\mathrm{PSL}(2,\mathbb{R})$-invariant copy of $\mathbb{H}^2$.
Aggregate the functions, as $\Delta$ ranges, of Lemma \ref{lem:convexcore} into a single function which we also, in an abuse, denote by $f_4$.
Write $f_5$ for the composition $$\mathfrak{C}\xrightarrow{\nu}\mathcal{D}^{\mathrm{nc}}_{\mathrm{af}}(\mathrm{PSL}(2,\mathbb{R}))\xrightarrow{\iota}\mathcal{D}^{\mathrm{nc}}_{\mathrm{af}}(\mathrm{PSL}(2,\mathbb{C})),$$
where $\nu$ is the function of Theorem \ref{thm:manifoldstodiscrete} and $\iota$ is the obvious embedding (the \emph{nc} superscripts are those of Proof 1 of Theorem \ref{thm:fgPSL2Rsmooth}).
Define then an $f_6$ mapping the image of $f_4\circ f_5$ to $\mathfrak{C}_c(\mathsf{Isom}^+,\mathbb{H}^2)^2$ by sending $((\mathcal{U},c),(\mathcal{V},c|_\mathcal{V}))$ to the pair given by intersecting all charts with $H$, and restricting all chart maps accordingly.
The composition $f_6\circ f_4\circ f_5$ is then the Borel function desired.
\end{proof}

\subsection{Classifying geometrically finite ends}
\label{subsection:classifying_gf_ends}
We turn in this section to the remaining manifold finiteness conditions (\ref{cond:geomfin}), (\ref{cond:toptame}), and (\ref{cond:algfin}) of \emph{geometric}, \emph{topological}, and \emph{algebraic} finiteness, respectively; let us begin with a word on their relationship.
Historically, the Ending Lamination Theorem was the second of a trio of outstanding conjectures on noncompact hyperbolic $3$-manifolds appearing in \cite{MR0648524} which were each resolved in the first decade of this century.
The first of these was the \emph{Tameness} or \emph{Marden Conjecture} that, in our working terminology, \emph{every algebraically finite hyperbolic $3$-manifold is topologically finite}. This conjecture was independently proven by \cite{2004math......5568A} and \cite{MR2188131} in 2004, and implies the structural possibilities for such manifolds' ends to be mercifully limited.
To simplify discussion, assume until otherwise indicated that $M$ is a \emph{cusp-free} algebraically finite complete orientable hyperbolic $3$-manifold with nonelementary holonomy group $\Gamma$ (for more general $M$, first pare $M$ back to a decusped manifold $M^0$ as in Section \ref{subsection:the_basics}, whereupon the discussion below will broadly apply).
In this case, in combination with \cite{MR1166330} the Tameness Theorem tells us that
\begin{itemize}
\item topologically, the ends of $M$ all possess a neighborhood of the form $S\times [0,\infty)$ for some surface $S$ (they are \emph{tame}); moreover,
\item metrically, these ends are of each of exactly one of two sorts: each is either \emph{geometrically finite} or \emph{simply degenerate} (also known as \emph{geometrically infinite}).
\end{itemize}
What distinguishes between these two sorts of ends of $M$ is their relation to the convex core $CC(M)$ of $M$: geometrically finite ends possess neighborhoods disjoint from $CC(M)$, while the geometrically infinite ends of $M$ fall entirely within $CC(M)$.
Note here the minor and defensible overloading of terminology: $M$ is a geometrically finite \emph{manifold} if and only if all of its \emph{ends} are geometrically finite, i.e., if and only if $CC(M)$ contains no geometrically infinite ends.
These sorts of manifolds are this subsection's main subject.
We'll develop some machinery in the course of a heuristic argument that their isometry relation is concretely classifiable, then we'll give a short, higher-level proof of this fact, thereby establishing the second item of Theorem \ref{thm:complexity_3-mans}.

By the first bulleted item above, associated to either geometric sort of end is a ``ray'' $[0,\infty)$ of induced geometric structures on a surface $S$.
Further differences between geometrically finite and geometrically infinite ends are how these geometric structures on $S\times \{t\}$ evolve within them as $t\to\infty$ and, relatedly, how these two types of ends are represented by the action of $\Gamma$ on the boundary of $\mathbb{H}^3$.
Recall that $\Omega(\Gamma)$ denotes the portion of that boundary on which $\Gamma$ acts properly discontinuously.
What is sometimes termed the \emph{Kleinian manifold} $\hat{M}$ associated to $\Gamma$ is the manifold with boundary $(\mathbb{H}^3\,\cup\,\Omega(\Gamma))/\Gamma = M\,\cup\,\partial\hat{M}$, and the components of this boundary bijectively correspond to the geometrically finite ends of $M$, abutting and compactifying their corresponding end ``at infinity''.
Moreover, $\Gamma$ acts by isometries on the hyperbolic Poincar\'{e} metric on $\Omega(\Gamma)$, thus this boundary $\Omega(\Gamma)/\Gamma$ naturally carries a hyperbolic structure, one which in fact is the subject of one of this field's most fundamental theorems (our formulation of it is as in \cite{MR1013665, MR2553578}).

\begin{theoremAFT}[\cite{MR167618}] If $\Gamma$ is a nonelementary finitely generated torsion-free Kleinian group then
$\Omega(\Gamma)/\Gamma$ is a finite union of finite-volume hyperbolic Riemann surfaces.
\end{theoremAFT}

It is these surfaces, viewed, as in Section \ref{section:isometry_for_2}, as points in their associated Teichm\"{u}ller spaces, which comprise the ending invariants of the geometrically finite ends of $M$ --- or, more precisely, of a \emph{marking} of $M$, for it is only by way of such markings that the Teichm\"{u}ller representations of these surfaces are well-defined.
Guided by diagram (\ref{eq:diagram}), we may be much more precise.
Suppose $\Gamma$ is the image of a discrete faithful representation $\rho$ of an abstract group $\Delta$.
Recall that a Kleinian group $\Gamma$ is geometrically finite if $M\cong\mathbb{H}^3/\Gamma$ is, and let
$$\mathrm{GF}(\Delta)=\{[\theta]\in \mathrm{AH}(\Delta)\mid \text{the image of }\theta\in \mathrm{DF}(\Delta)\text{ is geometrically finite}\}.$$
The Ending Lamination Theorem assigns complete invariants (\emph{marked homeomorphism type}, \emph{ending invariants}) to the marked manifolds associated to elements of $\mathrm{GF}(\Delta)$; as usual, to simplify statements, since (by tameness, for example) there are only countably many marked homeomorphism types of geometrically finite manifolds, we are free from a Borel complexity point of view to restrict attention to those of the same type as the marking $M_\rho$ of $M$.
We denote this subspace of $\mathrm{GF}(\Delta)$ by $\mathrm{QH}(\rho)$, and write $\mathrm{DF}(\rho)$ for its preimage by $q:\mathrm{DF}(\Delta)\to \mathrm{AH}(\Delta)$.
As $\theta$ ranges within $\mathrm{DF}(\rho)$, the homeomorphism type of $\partial\hat{M}_\theta$ is constant; by defining the Teichm\"{u}ller space of a possibly disconnected space as the product of the Teichm\"{u}ller spaces of its components, we may then paraphrase this case of the Ending Lamination Theorem\footnote{\label{footnote:incompressible} In the interests of historical accuracy, this portion of the theorem had been known long before; it's traced to ``Ahlfors, Bers, Maskit, Kra, Marden and Sullivan'' in \cite{MR3001608} (cf.\ \cite[\S 10]{Marden}), where the more general form of what we're paraphrasing appears as \emph{Theorem 4.3 (Ahlfors--Bers parametrization)}; see also \cite[p.\ 113]{MR2096234}. Our incompressibility assumption, in particular, is merely simplifying, but for completeness we recall its definition: $\hat{M}$ is \emph{boundary-incompressible} if the maps of fundamental groups induced by the inclusions of $\hat{M}$'s boundary components into $\hat{M}$ are each injective.} as follows: \emph{If $\hat{M}_\rho$ is boundary-incompressible, then the assignment $\mathrm{QH}(\rho)\to\mathrm{Teich}(\partial\hat{M}_\rho)$ outlined above is injective (and in fact bijective); in particular, it completely classifies the elements of $\mathrm{DF}(\rho)$ up to $\mathrm{PSL}(2,\mathbb{C})$-conjugacy}.
We do so to foreground its main idea, which is that the geometry of a geometrically finite $M$'s boundary surfaces completely encodes the geometry of $M$, and thus that our analysis of such manifolds' isometry relations should ultimately reduce, conceptually, to that of Section \ref{section:isometry_for_2}.

From a Borel complexity perspective, though, the italicized result above adds little to Proposition \ref{prop:marked_smooth}; moreover, our primary interest at this stage is in $\mathrm{PSL}(2,\mathbb{C})$-conjugacy relations on $\mathcal{D}_{\mathrm{gf}}(\mathrm{PSL}(2,\mathbb{C}))$, not on $\mathrm{DF}(\rho)$.
Referring again to diagram (\ref{eq:diagram}), however, observe that if we write $\mathcal{D}[\rho]$ for the $p$-image of $\mathrm{DF}(\rho)$ then since each of these sets is Borel
(by \cite[Exer.\ 18.14]{MR1321597}) and $\mathrm{PSL}(2,\mathbb{C})$-conjugacy-invariant, and $\mathcal{D}_{\mathrm{gf}}(\mathrm{PSL}(2,\mathbb{C}))$ is, as above, a countable union of such sets, our analysis of $\mathrm{PSL}(2,\mathbb{C})$-conjugacy relations on the latter readily reduces to one on any fixed such $\mathcal{D}[\rho]$; note next that for this analysis, the italicized instance of the Ending Lamination Theorem above is, in fact, quite suggestive.
For the diagram's map $s |_{\mathcal{D}[\rho]}$ may be seen as assigning its domain's elements markings; composed then with $q$, it takes $\mathcal{D}[\rho]$ into the domain of the assignment $\mathrm{QH}(\rho)\to\mathrm{Teich}(\partial\hat{M}_\rho)$ above, which we may in turn regard as a restricted instance of the map $e$ of diagram (\ref{eq:diagram}).
Thus $eqs |_{\mathcal{D}[\rho]}$ first marks the groups $\Delta$ of $\mathcal{D}[\rho]$, then classifies the resulting marked groups up to conjugacy.
Since there are only countably many ways of marking a countable group, it follows that postcomposing $eqs |_{\mathcal{D}[\rho]}$ with a map $f:\mathrm{Teich}(\partial\hat{M}_\rho)\to Y$ forgetting the data of its markings will define a weak reduction of $E(\mathrm{PSL}(2,\mathbb{C}),\mathcal{D}[\rho])$ to the identity relation on some set $Y$.
Of course, from Section \ref{section:isometry_for_2} we know such a map well, namely the quotient map $f:\mathrm{Teich}(\partial\hat{M}_\rho)\to
\mathcal{M}(\partial\hat{M}_\rho)$ induced by the natural action of $MCG(\partial\hat{M}_\rho)$ on $\mathrm{Teich}(\partial\hat{M}_\rho)$.

Thus, modulo our simplifying assumptions, $feqs$ defines a weak reduction of $E(\mathrm{PSL}(2,\mathbb{C}),\mathcal{D}[\rho])$ to the identity relation on the Polish space $\mathcal{M}(\partial\hat{M}_\rho)$, and we are close to at least the outlines of a proof that $E(\mathrm{PSL}(2,\mathbb{C}),\mathcal{D}_{\mathrm{gf}}(\mathrm{PSL}(2,\mathbb{C})))$ is concretely classifiable.
What remains to be seen is that $feqs$ is Borel; since $f$, $q$, and $s$ each are, this reduces to showing that $e$ is.
This is not in general obvious, a point we'll return to in \ref{subsection:classifying_gi_ends}; see \cite{MR1813235} more particularly for failures of maps from $\mathrm{AH}(\Gamma)$ to Polish spaces of ending invariants to be continuous.
Thus in this approach, one's obliged to prove by hand that $feqs$ is Borel.
This task is neither insurmountable --- Lemma \ref{lem:limitsets}(3) supplies a main step in the argument that $eqs$ is Borel, which suffices --- nor pleasant, particularly when we work in a generality admitting cusps, and so on.
And these are the reasons that we leave the details of this more hands-on approach to the following theorem to the interested reader, and instead record a proof by way of a few higher-level results.

\begin{theorem}
\label{thm:geom_fin_hyp_3_mans_smooth}
The classification of geometrically finite complete orientable hyperbolic $3$-manifolds up to isometry is Borel equivalent to $=_{\mathbb{R}}$, as is the classification of geometrically finite torsion-free Kleinian groups up to conjugacy.
\end{theorem}
\begin{proof}
The lower bound $=_{\mathbb{R}}\,\leq_B E(\mathrm{PSL}(2,\mathbb{C}),\mathcal{D}_{\mathrm{gf}}(\mathrm{PSL}(2,\mathbb{C})))$
follows immediately from Lemma \ref{lem:elementary}, and we record a second argument of it just after this proof's conclusion below.
For the upper bound, we'll leverage the analysis of the spaces of pared marked geometrically finite manifolds appearing in \cite[\S 7]{MR2096234}.\footnote{What we'll argue is something very close to the remark that ``[i]t is a consequence of the classical deformation theory of Kleinian groups [\dots] that $\mathrm{Out}(\pi_1(M))$ acts properly discontinuously on the interior $\mathrm{int(AH}(M))$ of $\mathrm{AH}(M)$'' \cite[p.\ 223]{MR3008919}, alongside the identification, traced in \cite[p.\ 106]{MR2096234} (a source cited in support of the remark) to \cite{MR806415}, of that interior with the geometrically finite points of $\mathrm{AH}(M)$.
Relations are a little more delicate, though, than our paraphrase suggests, the subtleties deriving entirely from the distribution of parabolics in those points (see footnote \ref{footnote:int(AH)} below for more), and this is the reason that, rather than content ourselves with a reference to that remark, we approach the proof as we do.}
We begin by modifying our notations in the direction of theirs, starting with
$$\mathrm{DF}(\Gamma,P)=\{\rho\in\mathrm{DF}(\Gamma)\mid\rho(g)\textnormal{ is parabolic for all }g\in P\}.$$
This definition makes sense for any $P\subseteq\Gamma$, but the main cases of interest are, of course, those for which $P\in\mathcal{P}(\Gamma)$ (see again Lemma \ref{lem:countable_par-sets} for notation).
Any $\mathrm{DF}(\Gamma,P)$ is plainly a Borel subset of $\mathrm{DF}(\Gamma)$; by Lemma \ref{lem:countable_par-sets}, the sets
$$\mathrm{DF}^{\star}(\Gamma,P)=\mathrm{DF}(\Gamma)\backslash\bigcup\{\mathrm{DF}(\Gamma,Q)\mid Q\in\mathcal{P}(\Gamma)\textnormal{ and }P\subsetneq Q\}$$
then are as well and, as $P$ ranges in $\mathcal{P}(\Gamma)$, together define a countable partition of $\mathrm{DF}(\Gamma)$.
For any $P\in\mathcal{P}(\Gamma)$ one then has the projection of $\mathrm{DF}(\Gamma,P)$ to the space $\mathcal{D}[\Gamma,P]$ of its elements' images $\mathrm{im}(\rho)$, as well as its quotient by $\mathrm{PSL}(2,\mathbb{C})$-conjugacy $\mathrm{AH}(\Gamma,P)$; similarly for its starred version.
Together with the natural $C^\star(\Gamma,P)\subseteq C(\Gamma)$, these starred spaces array in the following commuting subdiagram of diagram (\ref{eq:diagram}).
\begin{equation}
\label{eq:diagram2}
\begin{tikzcd}[row sep=large,column sep=large]
\mathrm{DF}^\star(\Gamma,P) \arrow[d, "p"] \arrow[r, "q"'] & \mathrm{AH}^\star(\Gamma,P) \arrow[d, "v"'] \arrow[l, "t"', dashed, bend right] \\
{\mathcal{D}^\star[\Gamma,P]} \arrow[r, "u"] \arrow[u, "s", dashed, bend left] & \mathrm{C}^\star(\Gamma,P)                                                       
\end{tikzcd}
\end{equation}
Note that for our purposes, we may, without loss of generality, assume $\Gamma$ nonelementary, whereupon $\mathrm{AH}^\star(\Gamma,P)$ is a Borel subspace of the Polish space $\mathrm{AH}(\Gamma)$.
Note also the mild notational abuse above, whereby we've identified the linking maps in diagram (\ref{eq:diagram}) with their restrictions; in the case of $v$, though, we should probably be clearer.
Write $\mathrm{Aut}(\Gamma,P)$ for the subgroup of $\mathrm{Aut}(\Gamma)$ sending $P$ to $P$ and $\mathrm{Out}(\Gamma,P)$ for $\mathrm{Aut}(\Gamma,P)/\mathrm{Inn}(\Gamma)$; the quotient map $v:\mathrm{AH}^\star(\Gamma,P)\to C^\star(\Gamma,P)$ making (\ref{eq:diagram2}) commute is then that induced by the natural action of $\mathrm{Out}(\Gamma,P)$ on $\mathrm{AH}^\star(\Gamma,P)$.
Diagram (\ref{eq:diagram2}) further restricts to its spaces' geometrically finite elements; much as above, we write $\mathcal{D}^\star_{\mathrm{gf}}(\Gamma,P)$ and $\mathrm{GF}^\star(\Gamma,P)$ for those of $\mathcal{D}^\star[\Gamma,P]$ and $\mathrm{AH}^\star(\Gamma,P)$, respectively.

It then follows immediately from \cite[Cor.\ 7.3.1]{MR2096234} that $\mathrm{Out}(\Gamma,P)$ acts properly discontinuously on $\mathrm{GF}^\star(\Gamma,P)$: the corollary describes $\mathrm{Out}(\Gamma,P)$ as an extension of a group whose action permutes components by a subgroup of $\mathrm{MCG}(\partial(\widehat{\mathbb{H}^3/\Gamma}))$ which acts properly discontinuously on the Teichm\"{u}ller parametrization of those components alluded to in our discussion preceding this theorem above.
This is all we need to conclude the proof: the spaces $\mathrm{DF}^\star(\Gamma,P)$ $(P\in\mathcal{P}(\Gamma))$ induce a countable Borel partition of $\mathcal{D}_{\mathrm{gf}}[\Gamma]$ and the $\mathrm{PSL}(2,\mathbb{C})$-conjugacy action on each summand Borel reduces to the quotient of a Borel subset of a Polish space by a properly discontinuous action by the above.
Assembling, as usual, all these countably many reductions together defines a Borel reduction of $E(\mathrm{PSL}(2,\mathbb{C}),\mathcal{D}_{\mathrm{gf}}(\mathrm{PSL}(2,\mathbb{C}))$ to an uncountable standard Borel space, and thus to $=_\mathbb{R}$.
This proves the theorem's second assertion, from which the first follows, as usual, via (the reasoning of) Lemma \ref{lem:lifting_orientation_preserving} and Corollary \ref{cor:classwise}. 
\end{proof}

Note that we've now proved all but the top item of Theorem \ref{thm:complexity_3-mans}.
Note also that in the above proof we derived our lower bound from the class of elementary Kleinian groups; a naive question is whether it holds in the nonelementary context as well.
For its connection to the material of Section \ref{subsection:classifying_gi_ends}, let us briefly argue from Bers's Simultaneous Uniformization Theorem \cite{MR111834}, a special case of the bijection $\mathrm{QH}(\rho)\to\mathrm{Teich}(\partial\hat{M}_\rho)$ discussed above, that indeed it does.
Fix a closed hyperbolizable surface $S$; in the present terminology, Bers's theorem tells us that every pair in $\mathrm{Teich}(S)^2$ forms the ending invariants of a marked $3$-manifold $M\cong\mathbb{H}^3/\Gamma\cong S\times\mathbb{R}$ for some geometrically finite $\Gamma\leq\mathrm{PSL}(2,\mathbb{C})$ (the pair encodes the marked hyperbolic structures of $\partial\hat{M}$'s two components lying at either end of $\mathbb{R}$).
Forgetting the marking then amounts to quotienting $\mathrm{Teich}(S)^2$ by the natural properly discontinuous action of $\mathrm{MCG}(S)$; since the uncountable Polish result parametrizes conjugacy classes of geometrically finite Kleinian groups, this witnesses that $=_\mathbb{R}$, indeed, is a $\leq_B$-lower bound, as desired.

\subsection{Classifying geometrically infinite ends}
\label{subsection:classifying_gi_ends}
It was Theorem \ref{thm:geom_fin_hyp_3_mans_smooth} we'd had chiefly in mind in this section's introduction when describing geometrically finite manifolds as the \emph{tractable class} of algebraically finite hyperbolic $3$-manifolds.
We now consider those manifolds parametrized by $\mathcal{D}_{\mathrm{af}}(\mathrm{PSL}(2,\mathbb{C}))\backslash\mathcal{D}_{\mathrm{gf}}(\mathrm{PSL}(2,\mathbb{C}))$, which we'd in turn described as \emph{limits} of geometrically finite manifolds.
This we'd meant in multiple senses:

\begin{itemize}
\item As noted in Section \ref{subsection:classifying_gf_ends}, $\mathrm{GF}(\Gamma)$ may be identified with the interior of the space $\mathrm{AH}(\Gamma)$.\footnote{\label{footnote:int(AH)} Let us be precise: implicit in such statements is a choice, first of all, of ambient character variety, either that denoted $\mathcal{R}^0(\Gamma)$ in \cite[\S 4.2]{MR2402415} (our default choice) or its relative version ($\mathcal{R}_{\mathrm{par}}$) in which noncyclic abelian subgroups of $\Gamma$ are only ever mapped to parabolics. Notations vary. Consider first, though, the case in which every abelian subgroup of $\Gamma$ is cyclic, so that the choice is immaterial. Then $\mathrm{int(AH}(\Gamma))$ consists of the convex cocompact points of $\mathrm{AH}(\Gamma)$; in other words, in our notation, it is exactly $\mathrm{GF}^\star(\Gamma,\varnothing)$.
For the more general case, let $P$ collect all conjugates of elements of noncyclic abelian subgroups of $\Gamma$.
Then with respect to the appropriate character variety, the interior of $\mathrm{AH}(\Gamma,P)$ is again exactly $\mathrm{GF}^\star(\Gamma,P)$; each of these cases is due to Sullivan \cite{MR806415}.}
Reciprocally, $\mathrm{AH}(\Gamma)$ is the closure of $\mathrm{GF}(\Gamma)$, as was conjectured by Bers, Sullivan, and Thurston, and ultimately proven in \cite{MR2821565, MR3001608}.
\item As we've also seen, manifolds parametrized by $\mathcal{D}_{\mathrm{gf}}(\mathrm{PSL}(2,\mathbb{C}))$ admit invariants falling in Teichm\"{u}ller spaces associated to their geometrically finite ends.
Manifolds parametrized by $\mathcal{D}_{\mathrm{af}}(\mathrm{PSL}(2,\mathbb{C}))\backslash\mathcal{D}_{\mathrm{gf}}(\mathrm{PSL}(2,\mathbb{C}))$, on the other hand, possess geometrically \emph{infinite} ends, with natural invariants falling in the \emph{boundary} of their associated \emph{electric Teichm\"{u}ller} spaces; see \cite{MR4502619}.
We'll return to these latter invariants below.
\item By the Tameness Theorem, manifolds $M =\mathbb{H}^3/\Gamma$ of the topological form $S\times\mathbb{R}$ for $S$ a finite-type hyperbolizable surface are, as $\Gamma$ ranges, the prototypes of any of these sorts of ends. 
In the simplest case, $\Gamma\cong\pi_1(S)$ is a Fuchsian subgroup of $\mathrm{PSL}(2,\mathbb{R})\leq\mathrm{PSL}(2,\mathbb{C})$ with limit set the meridian circle $\hat{\mathbb{R}}$ in $\hat{\mathbb{C}}$.
The \emph{quasi-Fuchsian} groups (of the first kind) $\Gamma$ are the quasiconformal deformations of these Fuchsian ones; their limit sets $\Lambda(\Gamma)$ are, in general, more meandering Jordan curves than $\hat{\mathbb{R}}$ in $\hat{\mathbb{C}}$.
These groups remain geometrically finite, with their associated Kleinian manifolds $(\mathbb{H}^3\,\cup\,\Omega(\Gamma))/\Gamma$ ``capped'' by two topological copies of $S$ corresponding to the two components of $\hat{C}\backslash\Lambda(\Gamma)$ (these were the groups at work in our second argument of the lower bound in Theorem \ref{thm:geom_fin_hyp_3_mans_smooth} above).
The asymptotic case, though, is that of \emph{doubly degenerate} Kleinian groups $\Gamma\cong\pi_1(S)$, whose limit sets $\Lambda(\Gamma)$ have grown to engulf $\hat{\mathbb{C}}$ entirely (so that both $\Omega(\Gamma)$ and these groups' Kleinian manifolds' boundaries are empty).
These groups are geometrically infinite, and form the setting for this subsection's main results.
\end{itemize}

Continuing with this last point, we'll term a Kleinian $\Gamma\leq\mathrm{PSL}(2,\mathbb{C})$ a \emph{surface group} if $\Gamma\cong\pi_1(S)$ for some finite-type hyperbolizable surface $S$ (see again the discussion preceding Definition \ref{def:Teich_MCG_M} for definitions).
By proceeding from surface groups to a more general analysis, this subsection reprises a pattern in the literature (the sequence \cite{MR2630036} then \cite{MR2925381}, for example; alternately, see \cite[p.\ 5]{BowditchELT}: ``understanding the particular [case of $M$ homeomorphic to $S\times\mathbb{R}$] is in some sense the key to understanding the Ending Lamination Theorem'').
In fact this section's results may be read as a further explanation for this pattern: they tell us broadly that once one's classified surface groups, it's all downhill from there.

Before proceeding, let us address an ambiguity signaled by the parenthetical ``(of the first kind)'' above; the question there was whether the peripheral elements of $\pi_1(S)$ should \emph{by definition} be forbidden from taking non-parabolic values in the isomorphism $\rho:\pi_1(S)\to\Gamma$ witnessing that $\Gamma$ is quasi-Fuchsian (if so, then the parenthetical's superfluous).
To save ourselves locutions of this sort, until otherwise indicated we'll adopt both this and the broader convention that $\Gamma$ is a surface group only if it is the image of an isomorphism $\rho:\pi_1(S)\to\Gamma$ placing the set $P$ of peripheral elements of $\pi_1(S)$ in bijection with the set of parabolic elements of $\Gamma$; moreover, wherever there's no danger of confusion, the notation $P$ should henceforth be understood to denote the above distinguished subset.
The parings of the quotient manifolds $\mathbb{H}^3/\Gamma\cong S\times\mathbb{R}$ of surface groups $\Gamma$ then possess either $0$, $1$, or $2$ geometrically infinite ends.
We addressed the first (quasi-Fuchsian) possibility in Section \ref{subsection:classifying_gf_ends}; we turn now to the latter two.
\begin{definition}
A surface group $\Gamma$ is \emph{doubly degenerate} if $\Lambda(\Gamma)=\hat{\mathbb{C}}$. For any finite-type hyperbolizable surface $S$, write $\mathcal{D}^\star_{\mathrm{dd}}[\pi_1(S)]$ for $\{\Gamma\in\mathcal{D}^\star[\pi_1(S),P]\mid\Lambda(\Gamma)=\hat{\mathbb{C}}\}$ and $\mathrm{DD}^\star(\pi_1(S))$ for $\{\rho\in\mathrm{AH}^\star(\pi_1(S),P)\mid\Lambda(\mathrm{im}(\rho))=\hat{\mathbb{C}}\}$.
A surface group $\Gamma$ is \emph{singly degenerate} if it is neither geometrically finite nor doubly degenerate; write $\mathcal{D}^\star_{\mathrm{sd}}[\pi_1(S)]\subseteq\mathcal{D}^\star[\pi_1(S),P]$ and $\mathrm{SD}^\star(\pi_1(S))\subseteq\mathrm{AH}^\star(\pi_1(S),P)$ for the naturally associated spaces of singly degenerate Kleinian groups and classes of representations, respectively.
A manifold $\mathbb{H}^3/\Gamma$ is singly or doubly degenerate if and only if $\Gamma$ is.
\end{definition}
\begin{lemma}
For any $S$ as above, $\mathcal{D}^\star_{\mathrm{dd}}[\pi_1(S)]$ is a Borel subset of $\mathcal{D}^\star[\pi_1(S)]$ and $\mathrm{DD}^\star(\pi_1(S))$ is a Borel subset of $\mathrm{AH}^\star(\pi_1(S))$; similarly for $\mathcal{D}^\star_{\mathrm{sd}}[\pi_1(S)]$ and $\mathrm{SD}^\star(\pi_1(S))$.
\end{lemma}
\begin{proof}
The first assertion follows from the proof of Theorem \ref{thm:geom_fin_hyp_3_mans_smooth} together with Lemma \ref{lem:limitsets}(1); for the second, observe that $\mathrm{DD}^\star(\pi_1(S))=(pt)^{-1}(\mathcal{D}^\star_{\mathrm{dd}}[\pi_1(S)])$ in the notation of diagram (\ref{eq:diagram2}). The third and fourth are then immediate from definitions, together with Corollary \ref{cor:loose_ends}.
\end{proof}
A main class of examples of doubly degenerate hyperbolic $3$-manifolds arises as follows.
Among the homeomorphisms $\phi$ of a finite-type hyperbolizable surface $S$ with itself are distinguished ones, termed \emph{pseudo-Anosov}, with the property of preserving (while rescaling in inverse proportions) transverse pairs of measured laminations on $S$.
We discuss the latter below; fundamental references for the former include \cite{MR3053012}, \cite{MR2850125} and especially \cite{MR4556467}, wherein it is shown that the mapping torus $M_\varphi$ of a pseudo-Anosov $\varphi$ admits a hyperbolic metric.
$M_\varphi$ is an $S$-bundle over a circle, and ``unrolling'' it along that circle defines an infinite cyclic cover $N_\varphi$ of $M_\varphi$ which is necessarily doubly degenerate.
For a gentle discussion of this example, see \cite[\S 1.2]{MR2044543}.

Our first result below, the lower bound half of Theorem \ref{thm:line_bundles}, leverages a more elaborate manifestation of this phenomenon appearing in \cite{MR4095420}.
Therein, realizations of the free group $F_n$ as Schottky subgroups of pseudo-Anosov elements of $\mathrm{MCG}(S)$ induce quasi-isometric embeddings of geodesics in $F_n$ into $\mathrm{Teich}(S)$ giving rise, in turn, to a family of doubly degenerate hyperbolic $3$-manifolds all of topological type $S\times\mathbb{R}$.
The cyclic covers described above are dense in this example; in the authors' own words (after replacing their $\Sigma$ with our $S$):
\begin{quote}
The idea is to use the Schottky group $\langle\varphi_0,\ldots,\varphi_n\rangle$ [of pseudo-Anosov elements of $\mathrm{Diff}^+(S)$] to associate to every element $\gamma\in\{0,\ldots,n\}^\mathbb{Z}$ of the shift space a pair consisting of the following elements:
\begin{enumerate}
\item a hyperbolic $3$-manifold $N_\gamma$ homeomorphic to $S\times\mathbb{R}$
\item a `coarse base point' $P_\gamma$, i.e., a subset of $N_\gamma$ with universally bounded diameter.
\end{enumerate}
Shifting a string corresponds to shifting the base point of $N_\gamma$ and convergence of $\gamma_i\in\{0,\ldots,n\}^\mathbb{Z}$ corresponds to based Gromov--Hausdorff convergence of the associated pairs $(N_\gamma,P_\gamma)$.
\end{quote}
In particular, the shift action induces isometry.
Thus when $n=1$, the construction supplies what might be regarded as a ``coarse embedding'' of the relation $E_\mathbb{Z}$ of shift-equivalence on $\{0,1\}^\mathbb{Z}$ into the $\mathrm{PSL}(2,\mathbb{C})$-conjugacy relation on $\mathcal{D}^\star_{\mathrm{dd}}[\pi_1(S)]$, and the following is what a descriptive set theorist would consequently suspect.

\begin{theorem}
\label{thm:IRSsE0}
If $S$ is a closed orientable hyperbolic surface then $$E_0\leq_B E(\mathrm{PSL}(2,\mathbb{C}),\mathcal{D}^\star_{\mathrm{dd}}[\pi_1(S)]).$$
\end{theorem}

To prove the theorem, one would like to simply extract a Borel reduction of $E_{\mathbb{Z}}$ to $E(\mathrm{PSL}(2,\mathbb{C}),\mathcal{D}^\star_{\mathrm{dd}}[\pi_1(S)])$ from the construction described above.
The senses, however, in which that construction's necessarily coarse complicate naive approaches.
What seems much easier is to couple the following (which figures as Theorem 6.4.5 in \cite{MR2455198}) with the construction to produce an $F$ with $E_0\leq_B F$, and then to show that $F$ weakly reduces to $E$, and this will be our strategy.
\begin{lemma}
\label{lem:Gao_E0_lemma}
For any Borel equivalence relation $E$ on a Polish space $X$, the following are equivalent:
\begin{itemize}
\item $E_0\leq_B E$.
\item There exists an $E$-nonatomic, $E$-ergodic probability Borel measure on $X$.
\end{itemize}
\end{lemma}

\begin{proof}[Proof of Theorem \ref{thm:IRSsE0}]
Fix an $S$ as in the statement of the theorem and write $Y$ for $\mathcal{D}^\star_{\mathrm{dd}}[\pi_1(S)]$ and $X$ for $\{0,1\}^{\mathbb{Z}}$ (more formally, we are writing $X$ for the space $\mathcal{G}(F)/F\cong\mathcal{G}_0(F)$ of \cite[12.19]{MR4095420}).
The construction described above associates to each $x\in X$ an element of $\mathrm{DD}^\star(\pi_1(S))$ which we will denote by $N_x$; by \cite[12.19]{MR4095420}, this map is both continuous and injective, hence its graph $Z_0=\{(x,N_x)\mid x\in X\}$ is a Borel subset of $X\times\mathrm{DD}^\star(\pi_1(S))$.
Recalling the notations of diagram (\ref{eq:diagram}), write $Z_1\subseteq X\times\mathrm{DF}(\pi_1(S))$ for $\{(x,\rho)\mid q(\rho)=N_x\}$ and $Z_2\subseteq X\times Y$ for $\{(x,\Gamma)\mid \Gamma=p(\rho)\text{ for some }(x,\rho)\in Z_1\}$.
The set $Z_1$ is Borel because $q$ and $Z_0$ are, thus $Z_2$ is Borel because $p$ is countable-to-one.

Now recall from Section \ref{subsection:invariant_DST} the relation $E_\mathbb{Z}$ of shift-equivalence on $X$ and define the relation $F$ on $Z_2$ by $(x,\Gamma) F (x',\Gamma')$ if and only if $x E_\mathbb{Z} x'$.
We claim that $E_0\leq_B F$.
To see this, first fix an $E_\mathbb{Z}$-nonatomic, $E_\mathbb{Z}$-ergodic probability Borel measure $\mu_X$ on $X$ (the standard product measure works).
Fix next a probability Borel measure $\mu_Y$ on $Y$ for which $\mu(Z_2)=\varepsilon>0$, where $\mu$ denotes the product measure of $\mu_X$ and $\mu_Y$.
(One way to see that such measures exist is via the Jankov--von Neumann uniformization theorem \cite[Theorem 18.1]{MR1321597}: there exists a $\sigma(\mathbf{\Sigma}^1_1)$-measurable $f:X\to Y$ with $(x,f(x))\in Z_2$ for all $x\in X$, and the pushforward $f_*\mu_X$ of $\mu_X$ by $f$ is then a $\mu_Y$ as desired.)
Next note that $\varepsilon^{-1}\mu$ is a probability Borel measure on $Z_2$ which is nonatomic and ergodic with respect to $F$ because $\mu_X$ is with respect to $E_\mathbb{Z}$.
Thus by Lemma \ref{lem:Gao_E0_lemma} there exists a Borel $h$ reducing $E_0$ to $F$.
The map $Z_2\to Y:(x,\Gamma)\mapsto\Gamma$, on the other hand, is plainly Borel.
By \cite[Prop.\ 12.19]{MR4095420} it defines a countable-to-one mapping of $F$-classes to $E(\mathrm{PSL}(2,\mathbb{C}),Y)$-classes and thus its precomposition by $h$ witnesses, via Lemma \ref{lem:weak_reduction}, that $E(\mathrm{PSL}(2,\mathbb{C}),Y)$ is not concretely classifiable.
It then follows from the Glimm--Effros Dichotomy (see Section \ref{subsection:invariant_DST}) that $E_0\leq_B E(\mathrm{PSL}(2,\mathbb{C}),Y)$.
\end{proof}

We turn next to the upper bound half of Theorem \ref{thm:line_bundles}.
The result we'll apply to that end will in fact bound the complexity of the isometry relation for a broader class of manifolds than those under consideration in Theorem \ref{thm:IRSsE0}.
But we should precede its application with some discussion, as promised, of the space $\mathcal{EL}(S)$ of ending laminations of a hyperbolic surface $S$.

A \emph{geodesic lamination} of such an $S$ is a foliation (i.e., a partition) of a closed subset of $S$ by geodesic curves.
$\mathcal{ML}(S)$ is the space of \emph{transversely measured laminations on $S$} whose support is a geodesic lamination, and it is endowed with a natural weak$^*$ topology; see \cite{MR1810534}, \cite[I.4]{MR903850} or any of the references surrounding our first statement of the Ending Lamination Theorem above.
Within $\mathcal{ML}(S)$ are distinguished \emph{filling} laminations, and their image under the measure-forgetting quotient map from $\mathcal{ML}(S)$ is what we denote the \emph{ending lamination space} $\mathcal{EL}(S)$ of $S$.
It is Polish \cite[Rmk.\ 1.6]{MR3285223}.
Its elements are what parametrize geometrically infinite structures on $(S\times [0,\infty))$-homeomorphic ends of algebraically finite hyperbolic $3$-manifolds; thus by the Ending Lamination Theorem, alongside Teichm\"{u}ller invariants for geometrically finite ends they are, as $S$ ranges, what parametrize those manifolds themselves.
Additional significance attaches to $\mathcal{EL}(S)$ via its homeomorphism with closely related spaces, namely that of minimal foliations on $S$, or the boundary at infinity of either the curve complex or electric Teichm\"{u}ller space of $S$ (see \cite{MR4502619,MR2258749}); these reinforce the heuristic that $\mathcal{EL}(S)$ describes the asymptotic behavior of sequences of curves or hyperbolic structures on $S$.

Technical constraints or parabolic elements can complicate $\mathcal{EL}(S)$'s parametrizing role as we've described it just above; see this section's conclusion for further discussion.
These elements are by definition under control, though, in the setting of $\mathrm{DD}^\star(\pi_1(S))$, and there $\mathcal{EL}(S)$ plays this role optimally; writing $\Delta$ for a product's diagonal, the following is \cite[Theorem 6.5]{MR2582104}.
\begin{theorem}
\label{thm:doubly_degenerate}
For any finite-type hyperbolizable surface $S$ the ending invariant map $e:\mathrm{DD}^\star(\pi_1(S))\to(\mathcal{EL}(S)\times\mathcal{EL}(S))\backslash\Delta$ is a homeomorphism.
\end{theorem}
Recall that $\mathrm{MCG}$ of a closed surface $S$ is an index-$2$ subgroup of $\mathrm{Out}(\pi_1(S))$.
In combination with Theorem \ref{thm:doubly_degenerate}, the following slight rephrasing of \cite[Cor.\ 1.2]{MR4252215} should then suggest the upper bound in question.
\begin{theorem}
\label{thm:PrSa}
Let $S$ be a finite-type hyperbolizable surface. The orbit equivalence relation induced by the natural action of $\mathrm{MCG}(S)$ on $\mathcal{EL}(S)$ is hyperfinite.
\end{theorem}
From this we will argue the following.
\begin{theorem}
\label{thm:E0_upper_bound}
For any finite-type hyperbolizable surface $S$, both the relations $E(\mathrm{PSL}(2,\mathbb{C}),\mathcal{D}^\star_{\mathrm{dd}}[\pi_1(S)])$ and $E(\mathrm{PSL}(2,\mathbb{C}),\mathcal{D}^\star_{\mathrm{sd}}[\pi_1(S)])$ are essentially hyperfinite.
\end{theorem}
Recall from Section \ref{subsection:invariant_DST} that the latter phrase is tantamount to ``$\leq_B E_0$''. We turn to the theorem's proof after noting its utility in proving Theorems \ref{thm:line_bundles} and \ref{thm:complexity_3-mans}.
\begin{proof}[Proof of Theorem \ref{thm:line_bundles}]
Any closed orientable hyperbolizable surface $S$ is as in the premises of both Theorems \ref{thm:IRSsE0} and \ref{thm:E0_upper_bound}, which together imply that
$$E_0\leq_B E(\mathrm{PSL}(2,\mathbb{C}),\mathcal{D}^\star_{\mathrm{dd}}[\pi_1(S)])\leq_B E_0.$$
By the reasoning of Lemma \ref{lem:lifting_orientation_preserving}, this then also holds of $E(\mathrm{Isom}(\mathbb{H}^3),\mathcal{D}^\star_{\mathrm{dd}}[\pi_1(S)])$, where we're momentarily (i.e., in only this proof and the next one) reading the second argument as the space of discrete type-preserving copies of $\pi_1(S)$ in $\mathrm{Isom}(\mathbb{H}^3)$. The translation of this result into manifold terms via Corollary \ref{cor:classwise} is by now routine; this completes the proof.
\end{proof}

\begin{proof}[Proof of Theorem \ref{thm:complexity_3-mans}]
The theorem's first two items are simply Proposition \ref{prop:nonmonotonic} and Theorem \ref{thm:geom_fin_hyp_3_mans_smooth}; its third is immediate from Theorem \ref{thm:line_bundles} together with the observation that $\mathcal{D}^\star_{\mathrm{dd}}[\pi_1(S)]$ is, in the above reading, an $\mathrm{Isom}(\mathbb{H}^3)$-conjugation invariant subset of $\mathcal{D}(\mathrm{Isom}(\mathbb{H}^3))$.
\end{proof}

\begin{proof}[Proof of Theorem \ref{thm:E0_upper_bound}]
Let us first argue the doubly degenerate case.
Fix a finite-type hyperbolizable surface $S$ and write $F$ for the equivalence relation induced by the natural action of $\mathrm{MCG}(S)$ on $\mathcal{EL}(S)$ and $E$ for the equivalence relation induced by the natural action of $\mathrm{MCG}(S)$ on $(\mathcal{EL}(S)\times\mathcal{EL}(S))\backslash\Delta$.
By Theorem \ref{thm:PrSa}, $F$ is hyperfinite, thus the product $F\times F$ is as well \cite[Prop.\ 1.3]{MR1900547}.
By the countability of $\mathrm{MCG}(S)$, the inclusion $(\mathcal{EL}(S)\times\mathcal{EL}(S))\backslash\Delta\to\mathcal{EL}(S)\times\mathcal{EL}(S)$ induces a weak reduction $E\to F\times F$; from this we conclude via Lemma \ref{lem:weak_reduction} that $E$ is hyperfinite. 

Consider next the commuting subdiagram of diagram (\ref{eq:diagram2}) anchored at opposite corners by $\mathcal{D}^\star_{\mathrm{dd}}[\pi_1(S)]$ and $\mathrm{DD}^\star(\pi_1(S))$; therein, much as in the proof of Theorem \ref{thm:geom_fin_hyp_3_mans_smooth}, $qs$ defines a Borel reduction of $E(\mathrm{PSL}(2,\mathbb{C}),\mathcal{D}^\star_{\mathrm{dd}}[\pi_1(S)])$ to the orbit equivalence relation $E'$ induced by the action of $\mathrm{Out}(\pi_1(S),P)$ on $\mathrm{DD}^\star(\pi_1(S))$.
By the Dehn--Nielsen--Baer Theorem for punctured surfaces (see \cite[Theorem 8.8]{MR2850125}), $\mathrm{MCG}(S)$ is an index-$2$ subgroup of $\mathrm{Out}(\pi_1(S),P)$, so each $E'$-class contains only finitely many classes of the $e^{-1}$-homeomorphic image of $E$.
From this it follows by \cite[Prop.\ 1.3(vii)]{MR1149121} that $E'$ is hyperfinite, and thus that $E(\mathrm{PSL}(2,\mathbb{C}),\mathcal{D}^\star_{\mathrm{dd}}[\pi_1(S)])$ is essentially hyperfinite, as claimed.

The argument for singly degenerate manifolds is almost identical; the key points here are that the ending invariant map $$e':\mathrm{SD}^\star(\pi_1(S))\to(\mathcal{EL}(S)\times\mathrm{Teich}(S))\sqcup(\mathrm{Teich}(S)\times\mathcal{EL}(S))$$ replacing the homeomorphism $e$ above is still Borel, and the $\mathrm{MCG}(S)$-action on each of the latter factors is again hyperfinite.
One may argue that $e'$ is Borel as follows (there may be better ways): argue via \cite[Theorem 8.1]{MR3001608}, for example, that its restriction to the ``slice'' associated to any fixed degenerate end-structure is Borel, and conclude that $e'$ also is via \cite[Theorem 6.6]{MR2582104} and \cite[Lem.\ 4.5.1]{MR2378491}.
\end{proof}
\begin{remark}
Upon inspection, Przytycki--Sabok in fact show in the course of proving \cite[Cor.\ 1.2]{MR4252215} that if $S$ is a finite-type hyperbolizable surface $S$ other than the thrice-punctured sphere then the equivalence relation induced by the natural action of $\mathrm{MCG}(S)$ on $\mathcal{EL}(S)$ is Borel equivalent to $E_0$ (the heart of their argument is a bireduction with a tail-equivalence relation, the $E^*_t$ of their \cite[p.\ 7]{MR4252215}).
More particularly, $E_0\leq_B F$ in the proof-notation just above (a fact not far in flavor from \cite[Theorem 2]{MR787893}), and the proof's other terms are linked closely enough that one might hope to reverse its direction to show that $E_0\leq_B E(\mathrm{PSL}(2,\mathbb{C}),\mathcal{D}^\star_{\mathrm{dd}}[\pi_1(S)])$ for this broader class of surfaces $S$.
Indeed, for this one would only need to show that $F\leq_B E$ via a reduction $\mathcal{EL}(S)\to(\mathcal{EL}(S)\times\mathcal{EL}(S))\backslash\Delta$.
The natural attempt restricts to the identity on the first coordinate, and must therefore restrict on the second coordinate to a fixed-point-free $\mathrm{MCG}(S)$-equivariant endomorphism $f:\mathcal{EL}(S)\to\mathcal{EL}(S)$, but pseudo-Anosovs' fixed points easily rule out the existence of such $f$.
The attempt for $\mathcal{D}^\star_{\mathrm{sd}}[\pi_1(S)]$ faces similar difficulties, all of which we note simply to say that the relations between dynamical phenomena like \cite{MR787893} and non-classifiability results like Theorem \ref{thm:IRSsE0} remain to be better understood.
\end{remark}
The preceding results are particularly suggestive in combination with the works \cite{MR2876139,MR3008919}.
Letting $\Gamma=\pi_1(M)$ for an algebraically finite hyperbolic $3$-manifold $M$, the central object of study therein is ``the topological quotient $$\mathrm{AI}(\Gamma)=\mathrm{AH}(\Gamma)/\mathrm{Out}(\Gamma),$$ which we may think of as the moduli space of unmarked $3$-manifolds homotopy equivalent to $M$'' (\cite[p.\ 221]{MR3008919}, with notations modified slightly to align with our own); we may, of course, equally well think of $\mathrm{AI}(\Gamma)$ as a moduli space of the $\mathrm{PSL}(2,\mathbb{C})$-conjugacy classes of $\mathcal{D}[\Gamma]$.
The summary theorem in \cite{MR3008919} is the following.
\begin{theorem}
\label{thm:Canary_Storm}
Let $M$ be a compact $3$-manifold with non-abelian fundamental group whose interior admits a complete hyperbolic metric. Then the moduli space $\mathrm{AI}(\pi_1(M))$ is $T_1$ if and only if $M$ is not an untwisted interval bundle.
\end{theorem}
In our second proof of Theorem \ref{thm:fgPSL2Rsmooth}, we invoked the $(1)\Rightarrow(3)$ portion of Glimm's \cite[Theorem 1]{MR0136681}; the following distills its $(2)\Rightarrow(3)$ portion.
\begin{lemma}
\label{lem:Glimm}
If a countable group $G$ acts Borel-measurably on a Polish space $X$ and $X/G$ is $T_0$ then the associated orbit equivalence relation is concretely classifiable.
\end{lemma}
Together with Theorem \ref{thm:IRSsE0}, this affords us the following strengthening of Theorem \ref{thm:Canary_Storm}:
\begin{corollary}
\label{cor:not_T0}
If an $M$ as in Theorem \ref{thm:Canary_Storm} isn't an untwisted interval bundle then $\mathrm{AI}(\pi_1(M))$ is Borel-isomorphic to a Polish space.
If, on the other hand, $M$ is an untwisted interval bundle over a closed orientable hyperbolic surface $S$ then for no Polish topology on $\mathrm{AH}(\pi_1(M))$ inducing its usual Borel structure is the topological quotient $\mathrm{AH}(\pi_1(M))/\mathrm{Out}(\pi_1(M))$ even $T_0$.
\end{corollary}
\begin{proof}
For any $M$ under consideration, temporarily write $E(M)$ for the orbit equivalence relation induced by the $\mathrm{Out}(\pi_1(M))$-action on $\mathrm{AH}(\pi_1(M))$.
For the first assertion, deduce from its premise, Theorem \ref{thm:Canary_Storm}, and Lemma \ref{lem:Glimm} that $E(M)$ is concretely classifiable, as witnessed by a Borel map $r$ from $\mathrm{AH}(\pi_1(M))$ to some Polish space $Y$; being countable-to-one, $\mathrm{im}(r)$ is a Borel subset of $Y$, and $r$ induces a Borel isomorphism $\mathrm{AI}(\pi_1(M))\cong\mathrm{im}(r)$. For the second assertion, observe that $\pi_1(M)=\pi_1(S)$ and for any of the topologies in question,
$$E_0\leq_B E(\mathrm{PSL}(2,\mathbb{C}),\mathcal{D}^\star_{\mathrm{dd}}[\pi_1(S)])\leq_B E(M)$$ by Theorem \ref{thm:IRSsE0} and Lemma \ref{lem:commuting_square}, respectively; then apply the contrapositive of Lemma \ref{lem:Glimm}.
\end{proof}
A second corollary bounds the Borel complexity of the isometry relation for a substantial class of hyperbolic $3$-manifolds.
\begin{corollary}
\label{cor:mostly_hyperfinite}
The classification up to isometry of algebraically finite hyperbolic $3$-manifolds whose holonomy groups $\Gamma$ either
\begin{enumerate}
\item are surface groups, or
\item aren't group-isomorphic to $\pi_1(S)$ for any surface $S$,
\end{enumerate}
is essentially hyperfinite.
\end{corollary}
\begin{proof}
In case (2), the isometry problem is concretely classifiable, by Theorem \ref{thm:Canary_Storm}, Lemma \ref{lem:Glimm}, Lemma \ref{lem:commuting_square}, and Corollary \ref{cor:classwise}.
In case (1), either $\Gamma$ is geometrically finite, whereupon by Theorem \ref{thm:geom_fin_hyp_3_mans_smooth} the isometry problem is again concretely classifiable, or $\Gamma\in\mathcal{D}^\star_{\mathrm{dd}}[\pi_1(S)]\cup\mathcal{D}^\star_{\mathrm{sd}}[\pi_1(S)]$, whereupon Theorem \ref{thm:E0_upper_bound} applies and Lemma \ref{lem:lifting_orientation_preserving} and Corollary \ref{cor:classwise} again convert our conclusion to one about the isometry problem for manifolds.
\end{proof}
It is an artifact of our restrictive definition of \emph{surface group} that Corollary \ref{cor:mostly_hyperfinite} looks more exhaustive than it really is.
For an algebraically finite hyperbolic $3$-manifold may be of a third kind as well: its holonomy group $\Gamma$ may be isomorphic to $\pi_1(S)$ for some surface $S$, but by no map placing $\Gamma$'s parabolic elements and the peripheral elements of $\pi_1(S)$ in correspondence.
Non-alignment can entail phenomena of two broad sorts --- \emph{accidental parabolics}, and \emph{compressible boundaries} --- whose analyses require extra levels of care.
These phenomena aren't unrelated to the discontinuities, already noted, of the ending invariant map (see \cite{MR1411128,MR1813235} for well-known examples, and \cite{MR2667553} for a useful survey), which complicate appeals to results like Theorem \ref{thm:PrSa}, and while Anderson--Lecuire show in \cite{MR3134412} that revision to the \emph{strong topology} of Lemma \ref{lem:SH_Borel} can rectify these discontinuities, their results apply only in incompressible settings.
We do, nevertheless, have ideas for how these analyses will go.
Faced more immediately, though, with the choice between further delaying circulation of this paper's other results and delaying a complete analysis of the complexity of $E(\mathrm{PSL}(2,\mathbb{C}),\mathcal{D}_{\mathrm{af}}(\mathrm{PSL}(2,\mathbb{C})))$, we are opting for the latter.

\section{Further questions}\label{section:questions}

We conclude with a list of some of the most salient questions which remain.
Let us stress from the outset, though, that we regard this list as unavoidably, even inherently, incomplete, for what we've found in the course of this project is that those we've consulted about it can, in general, pose even better questions within its framework than we ourselves can.
May they continue to do so.
Put differently, our aim in what follows is less to conclude the kind of dialogue that informs this work than to shape a place for more of it.

In any case, some obvious questions remain open.
The most conspicuous among them is that of the complexity of the homeomorphism relation for manifolds\footnote{Recall from Example \ref{ex:pseudogroup_comparison} our adherence to the convention that the absence of the phrase ``with boundary'' connotes ``without boundary''.} of dimension greater than $2$.
Since, as discussed in our introduction, Iannella and Weinstein \cite{IW26+} have recently settled the $3$-dimensional case, the state of the art posing of this question is as follows:
\begin{question}\label{ques:complexity_homeo}
	For any $n\geq 4$, what is the Borel complexity of the homeomorphism relation for connected topological $n$-manifolds?
\end{question}
As mentioned in Example \ref{ex:pseudogroup_comparison}, topological $3$-manifolds admit smooth structures which are unique up to diffeomorphism.
Thus, much as in our Theorem E, the Iannella--Weinstein complexity computation for $3$-manifolds up to homeomorphism readily implies a parallel one for smooth $3$-manifolds up to diffeomorphism.
In higher dimensions, however, smooth and topological structures no longer couple so neatly; thus we also ask the following:
\begin{question}\label{ques:complexity_smooth}
	For any $n\geq 4$, what is the Borel complexity of the diffeomorphism relation for connected smooth $n$-manifolds?
\end{question}
These give rise to further questions.
Recall, for example, item (2) of the Hjorth--Kechris theorem quoted in our introduction, whereby for any $n\geq 2$ the biholomorphism relation on complex $n$-manifolds is \emph{turbulent} and, consequently, \emph{not} classifiable by countable structures.
Questions \ref{ques:complexity_homeo} and \ref{ques:complexity_smooth} are partly asking whether this phenomenon manifests in the less rigid settings of topological or smooth manifolds.
As for complex manifolds, the first possibility of an affirmative answer is in $4$ real dimensions, by the results cited above.
We record the negative framing of the question since it reads a little better:
\begin{question}
\label{ques:ctbl_structures}
Is the homeomorphism relation on connected topological $n$-manifolds classifiable by countable structures for every $n\geq 1$? Similarly for the diffeomorphism relation on connected smooth $n$-manifolds.
\end{question}
A close relative of these questions is that of whether complexity weakly increases with dimension.
Recall that it does not for complete finite-volume hyperbolic manifolds, by Theorem \ref{thm:fgPSL2Rsmooth} and Proposition \ref{prop:nonmonotonic}, but that even for this class, complexity varies weakly monotonically in dimensions $n\geq 2$; thus, depending on how seriously one takes the setting of one hyperbolic dimension,\footnote{See \S 4 of the excellent \cite{MR1491098} for an argument that it should be taken seriously.} we arguably know at present of no non-complexity-monotonic class of manifolds.
\begin{question}
\label{ques:monotonic}
Write $\leq_B$ for the Borel reducibility ordering of the class of analytic equivalence relations on Polish spaces.
Is the map sending $n$ to:
\begin{itemize}
\item the homeomorphism relation on connected topological $n$-manifolds, or
\item the diffeomorphism relation on connected smooth $n$-manifolds
\end{itemize} a weakly increasing function from the positive integers to $\leq_B$?
\end{question}
Naively, of course, one expects both maps to be, in part because there exist such simple and uniform ways of converting lower-dimensional manifolds into higher-dimensional ones.
The most obvious among them, however, do not in general furnish reductions; see \cite{MR137105}, for example, for an uncountable family of non-homeomorphic contractible $3$-manifolds whose products with $\mathbb{R}$ are each homeomorphic to $\mathbb{R}^4$ (cf.\ \cite{whitehead1937, MR125573}).
In contrast, by invariance of dimension, the product with the closed interval $[0,1]$ does reduce the homeomorphism relation for connected $n$-manifolds \emph{without} boundary to the homeomorphism relation for connected $(n+1)$-manifolds \emph{with} boundary, but the problem of relating the latter to that of connected $(n+1)$-manifolds without boundary then remains.
Again the obvious ideas --- namely, taking either interiors or doubles of manifolds with boundaries --- fail in general to furnish reductions of the suitable homeomorphism relations; see  \cite{MR271948} for $3$-dimensional examples of failures of the latter.
What's lacking, in other words, are answers to the following:
\begin{question}
\label{ques:boundary}
	For any $n\geq 2$, is the homeomorphism relation for connected topological $n$-manifolds with boundary Borel reducible to that for connected topological $n$-manifolds without boundary? Similarly for smooth $n$-manifolds.
\end{question}

As noted in Example \ref{ex:pseudogroup_comparison}, the reverse reduction always holds, thus the question is whether the homeomorphism or diffeomorphism relations on $n$-manifolds with boundary and without boundary are Borel equivalent, or, conversely, whether any additional complexity derives purely from manifolds' boundaries.

Progress on these questions may well come in blocks.
A positive answer to either instance of Question \ref{ques:ctbl_structures}, for example, would completely settle the corresponding Question \ref{ques:complexity_homeo} or \ref{ques:complexity_smooth} as well as the corresponding instance of Question \ref{ques:monotonic}; the bearing of Question \ref{ques:boundary} on Question \ref{ques:monotonic} is similar.
Above, we've given topological and smooth manifolds primary and parallel attention simply because these are, classically, the most prominent classes of manifolds; note moreover that in all complexity computations thus far, these two classes have behaved identically.

Further interest attaches to the $4$-dimensional setting because it's these classes' first opportunity to behave differently; in certain restricted senses, in fact, we know that they will.
Here we're referring, of course, to Milnor's discovery of so-called ``exotic'' smooth structures on topological manifolds $M$  \cite{MR82103}.
As noted in our proof of Corollary \ref{cor:Borel_compact_with_boundary}, there are, up to diffeomorphism, only countably many such structures when $M$ is compact (and in fact just finitely many if also $\mathrm{dim}(M)\neq 4$, with exactly $28$, for example, on Milnor's initial setting of the $7$-dimensional sphere \cite{MR148075}).
On the other hand, building on work of Donaldson \cite{MR710056} and Freedman \cite{MR679066}, Taubes \cite{MR882829} showed that $\mathbb{R}^4$ posses uncountably many non-diffeomorphic smooth structures (see also \cite{MR1152230}).
That the diffeomorphism relation on smooth manifolds homeomorphic to $\mathbb{R}^4$ is of strictly greater Borel complexity than that of the homeomorphism relation on such manifolds follows trivially; the more interesting question, recorded already in \cite{MR3837073}, is just what the complexity of the former relation is.
\begin{question}
\label{ques:R4}
	What is the Borel complexity of the diffeomorphism relation on smooth structures on $\mathbb{R}^4$?
\end{question}
As noted in our introduction, Gompf and Panagiotopoulos have, at the time of writing, shown the complexity to be at least $E_0$; in addition, they've identified a topological $4$-manifold whose smooth structures, up to diffeomorphism, reduce the Friedman–-Stanley jump $=^+$ of equality \cite{GP25}.
Question \ref{ques:R4} is in some sense a subquestion of Question \ref{ques:complexity_smooth}; one might alternatively conjoin them in the question \emph{Does the Borel complexity of the diffeomorphism relation on connected smooth $4$-manifolds concentrate on the class of manifolds homeomorphic to $\mathbb{R}^4$?}
Note relatedly that our pseudogroup framework renders the first challenge in answering Question \ref{ques:R4} --- namely, parametrizing smooth structures --- particularly straightforward: if $f$ is the embedding $\mathfrak{M}(\mathsf{C}^{\infty},\mathbb{R}^4)\to\mathfrak{M}(\mathsf{Top},\mathbb{R}^4)$ induced by the forgetful pseudogroup embedding $\mathsf{C}^{\infty}\to\mathsf{Top}$, then $f^{-1}([\mathbb{R}^4]_{\cong_{\mathsf{Top}}})$ is the parameter space desired.
Observe moreover that the latter is naturally endowed with a standard Borel structure if $[\mathbb{R}^4]_{\cong_{\mathsf{Top}}}$ forms a Borel subset of $\mathfrak{M}(\mathsf{Top},\mathbb{R}^4)$, an interesting question in its own right which opens onto several more.

One might approach this last question, for example, via the homotopical or (co)homological invariants of elements of $\mathfrak{M}(\mathsf{Top},\mathbb{R}^4)$.
We have seen in Corollary \ref{cor:compact_invariants_Borel} that the computation of these invariants on the space $\mathfrak{K}(\mathsf{Top},\mathbb{R}^n)$ of all compact topological $n$-manifolds is Borel, and it appears likely a consequence of Lemma \ref{lem:Borel_computation_of_components} and Theorem \ref{thm:Borel_triangulation} that this is also the case for the space $\mathfrak{M}(\mathsf{Top},\mathbb{R}^2)$ of all topological $2$-manifolds (the same would hold for $\mathfrak{M}(\mathsf{Top},\mathbb{R}^3)$, via the Iannella--Weinstein work cited above).
However, as mentioned in the discussion surrounding Corollary \ref{cor:compact_invariants_Borel} (particularly of \cite{MR1233807}, see also \cite[33.I]{MR1321597}), their computation for broader classes of topological spaces (even just compact subsets of $\mathbb{R}^2$) is not.
Add to this the existence of non-triangulable manifolds in every dimension greater than $3$ (as surveyed in \cite{MR3526837}), and the following question is a natural one.
\begin{question}
\label{ques:Borel_computation_invariants}
For which $n$ and subspaces of $\mathfrak{M}(\mathsf{Top},\mathbb{R}^n)$ are the assignments of the classical homotopical, homological, and cohomological invariants of algebraic topology Borel measurable functions?
\end{question}
Returning to the discussion of $\mathfrak{M}(\mathsf{C}^{\infty},\mathbb{R}^4)$ above, there we saw the two main ways of altering a parameter space of manifolds; more precisely, associating to any Borel pseudogroup $\mathcal{G}$ on a model space $X$ are questions both of how $(\mathcal{G},X)$-equivalence varies on varying subdomains of $\mathfrak{M}(\mathcal{G},X)$, and as $\mathcal{G}$ is replaced by finer or coarser pseudogroups $\mathcal{H}$.
The, loosely speaking, functorial way our parametrizations interact with these sorts of variations (the way relaxing manifolds' smooth structures to topological ones manifests as an embedding $\mathfrak{M}(\mathsf{C}^{\infty},\mathbb{R}^4)\to\mathfrak{M}(\mathsf{Top},\mathbb{R}^4)$, for example) is among our framework's chief charms.
Within it, these two axes of variation entail two further families of questions.

A prospect associating to the first axis
is a stratification of classes of $(\mathcal{G},X)$-manifolds, for fixed $\mathcal{G}$ and $X$, according to the complexity they embody.
As we've seen, for example, the homeomorphism relations for compact topological $2$-manifolds and for topological $2$-manifolds, respectively, are of minimal and maximal complexity among analytic equivalence relations which possess infinitely many classes and are classifiable by countable structures.
Which intermediate complexity degrees are realized by naturally occurring intermediate classes of topological $2$-manifolds?
More precisely:
\begin{question}
\label{ques:realizability}
For which analytic equivalence relations $E$ and natural parameter spaces $\mathfrak{X}$ with $\mathfrak{K}(\mathsf{Top},\mathbb{R}^2)\subseteq \mathfrak{X}\subseteq\mathfrak{M}(\mathsf{Top},\mathbb{R}^2)$ is the relation $\cong_{\mathsf{Top}}$ restricted to $\mathfrak{X}$ Borel equivalent to $E$?
\end{question}
Similarly, of course, for dimensions $n\geq 2$.
The question's not solely of abstract interest.
In their study of the large-scale geometry of big mapping class groups, for example, Mann and Rafi explicitly cite the complexity computations of \cite{MR1804507} as motivation for restricting attention to a class of surfaces they call \emph{tame} \cite[p.\ 2242]{MR4634747}; a natural question is how much complexity they're in the process screening out.

\begin{question}
\label{ques:Mann_Rafi}
What is the Borel complexity of the homeomorphism relation for surfaces which are tame in the sense of Mann--Rafi?
\end{question}

We've also seen that the complexity of the isometry relation for classes of hyperbolic $n$-manifolds ranges much as the homeomorphism relation for topological $2$-manifolds does: by Corollary \ref{cor:isometry_universal}, Theorem \ref{thm:fgPSL2Rsmooth}, and Proposition \ref{prop:nonmonotonic}, the range of the former is from $\cong_{\mathbb{R}}$ to $E_\infty$ and $\cong_{\mathbb{N}}$ to $E_\infty$, respectively, when $n=2$ and $n>2$, with the minimum corresponding in all cases to the class of compact (or, alternately, complete finite-volume) manifolds.
One might then pose a version of Question \ref{ques:realizability} for these classes as well, however, there is a sense in which we already know an answer.
A class of equivalence relations closely associated to those induced by Polish group actions is that of the \emph{idealistic}\footnote{All orbit equivalence relations of Polish group actions are idealistic \cite{MR1176624}; it is an open question whether the converse holds up to bireducibility, see e.g., \cite{2025arXiv250608217C}.} equivalence relations (see \cite{MR1176624, MR3549382}); since this class is closed under classwise Borel isomorphism, it follows from our Corollary \ref{cor:classwise} that the isometry relations on hyperbolic manifolds all fall within this class.
Since they are, moreover, Borel equivalence relations, it follows from \cite[Lem.\ 3.7]{MR3549382} that for every $n\geq 2$ and $E\leq_B E_\infty$, there exists an $\mathfrak{X}\subseteq\mathfrak{C}^*_c(\mathsf{Isom},\mathbb{H}^n)$ such that the isometry relation on $\mathfrak{X}$ is Borel equivalent to $E$.
Note, however, that this doesn't entirely exhaust this line of inquiry: the question of how this correspondence plays out, or more particularly of which of these $\mathfrak{X}$ are already objects of interest, remains a good one, and it's in this direction that the word \emph{natural} in Question \ref{ques:realizability} is pointing above.

%

Against this background, we highlight a more definite question regarding classes of hyperbolic manifolds of intermediate complexity. As shown in Corollary \ref{cor:mostly_hyperfinite}, for a wide range of algebraically finite hyperbolic $3$-manifolds the isometry relation has Borel complexity at most $E_0$; we ask whether other possible degrees are realized. Here again we stress the word \emph{natural}, as the discussion above shows that such classes $\mathfrak{X}$ exist abstractly.

\begin{question}
	Is there a natural parameter space $\mathfrak{X}\subseteq\mathfrak{C}^*	(\mathsf{Isom},\mathbb{H}^3)$ on which the isometry relation $\cong_{\mathsf{Isom}}$ has Borel complexity strictly between $E_0$ and $E_\infty$?
\end{question}

Zooming out, observe that all of the Borel complexity degrees of manifolds in Theorems A--E and J--L of our introduction reduce to the orbit equivalence relation of some Polish group action. This leads to the rather broad question of whether orbit equivalence relations represent a global upper bound for \emph{all} classification problems for manifolds based on, say, locally compact model spaces. Note that there is a universal orbit equivalence relation, instantiated for example by the homeomorphism relation on all compact metric spaces \cite{MR3459026}.

\begin{question}
Does there exist a locally compact Polish space $X$ and Borel pseudogroup $\mathcal{G}$ on $X$ such that the relation $\cong_{\mathcal{G}}$ of $(\mathcal{G},X)$-equivalence on $\mathfrak{M}(\mathcal{G},X)$ is not Borel reducible to an orbit equivalence relation of a Polish group action?
\end{question}

Returning to the discussion of axes of variation, the second and more obvious one refers to the question of how complexity varies as we vary the pseudogroup $\mathcal{G}$ on a fixed model space $X$.
Here the questions are too numerous to list, and our caveat above that others will be able to supply even better ones than we might particularly applies.
Let us therefore merely draw attention to those $(\mathcal{G},X)$-structures that might not leap so readily to mind.
As in our Section \ref{subsection:pseudogroups}, we direct readers to the first sections of \cite[\S 3]{MR1435975} for a sampling of such structures which is suggestive in its range.
Its Example 3.1.9, \emph{foliations}, points even in the direction of structures for which the model space $X$ is not locally connected, namely the \emph{matchbox} or \emph{solenoidal} manifolds of \cite{MR1326831, MR3249058}.
To further shed the requirement that our model space $X$ be locally compact would jeapordize the underpinnings of our model space machinery (see again Section \ref{subsection:spaces_of_subsets}), and this is one of the reasons we've focused on finite-dimensional manifolds. However, by a theorem of Henderson \cite{MR250342}, every separable infinite-dimensional manifold is homeomorphic to an open subset of the separable Hilbert space $\ell^2$.
Thus the Effros--Borel structure on closed subsets of $\ell^2$ induces a standard Borel structure on its open subsets, a natural candidate parameter space for the separable infinite-dimensional manifolds on which we may consider their complexity.
 
\begin{question}
 	What is the Borel complexity of the homeomorphism relation for connected separable infinite-dimensional manifolds?
\end{question}

The next few questions concern the optimality of our results.
For example, as discussed in Section \ref{section:conjugacy}, the Andretta--Camerlo--Hjorth theorem \cite{MR1815088} tells us that whenever a discrete group contains a nonabelian free subgroup, the conjugacy relation on its space of subgroups is a universal countable Borel equivalence relation. We then saw in Theorem \ref{thm:conjugacy_universal} and its corollaries that this phenomenon also holds for the conjugacy relation on discrete subgroups in a wide variety of Lie groups. However, our arguments seem specific to ``matrix-like'' groups and the most far-reaching generalization remains open:

\begin{question}
\label{ques:iian_conjecture}
	If $G$ is a locally compact Polish group which contains a discrete nonabelian free subgroup, is the conjugacy relation $E(G,\mathcal{D}(G))$ on discrete subgroups of $G$ essentially countable universal?	
\end{question}

In the more tractable case of a connected Lie group $G$, it \emph{fails} to contain a discrete nonabelian free subgroup if and only it is a compact extension of a solvable subgroup \cite[Thm.\ 3.8]{MR961261}. Thus, we conjecture the following dichotomy for conjugacy relations on the discrete subgroups of such a group $G$:

\begin{conjecture*}
	If $G$ is a connected Lie group, then exactly one of the following holds:
	\begin{enumerate}[label=\textup{\arabic*.}]
		\item $E(G,\mathcal{D}(G))$ is essentially hyperfinite; or
		\item $E(G,\mathcal{D}(G))$ is essentially countable universal.
	\end{enumerate}	
\end{conjecture*}

Implicit in this conjecture is that the conjugacy relation on discrete subgroups of a compact-by-solvable Lie group is essentially hyperfinite. This is clearly related to the solvable case of Weiss's problem \cite{MR737417} on whether discrete amenable groups induce only hyperfinite equivalence relations. Recent progress \cite{MR4679959} has resolved this problem for a large class of solvable discrete groups and therefore may be helpful in addressing our conjecture above.

Returning to the question of optimality, recall that in Sections \ref{section:isometry_for_2} and \ref{section:isometry_for_3} we confined our attention to torsion-free subgroups solely for the sake of simplicity.
Thus we may wonder:
\begin{question}
\label{ques:Selberg}
How does the removal of the qualifier \emph{torsion-free} affect our Theorems J and K?
\end{question}

Selberg's lemma that, for example, every finitely generated Kleinian group is virtually torsion-free  \cite{MR130324}, suggests that it should not.
But the lemma is too essentially an existence result for its implications to be entirely obvious in our contexts.
As discussed in Section \ref{subsection:classifying_gi_ends}, we've also stopped short, in Theorem K, of resolving the following.
\begin{question}
\label{ques:alg_fin_3_man}
What is the Borel complexity of the isometry relation on algebraically finite hyperbolic $3$-manifolds?
\end{question}
There are, as well, multiple natural questions about the generalizations of our results from Sections \ref{section:isometry_for_2} and \ref{section:isometry_for_3} beyond $3$ hyperbolic dimensions and the Lie groups $\mathrm{PSL}(2,\mathbb{R})$ and $\mathrm{PSL}(2,\mathbb{C})$. 
\begin{question}
\label{ques:geom_alg_fin}
Fix $n\geq 4$. What is the Borel complexity of the isometry relation on the space of geometrically finite hyperbolic $n$-manifolds?
What is the Borel complexity of the isometry relation on the space of algebraically finite hyperbolic $n$-manifolds? Equivalently, what is the Borel complexity of the conjugacy relations on the spaces of geometrically finite and finitely generated discrete subgroups of $\mathrm{Isom}^+(\mathbb{H}^n)$, respectively?
\end{question}
Recall that our answers to these questions in dimensions $2$ and $3$ drew on a rich geometric literature; we lack answers in any higher dimension in part because there there's no comparable geometric theory to draw from.
As in our introduction, we note that answers to Question \ref{ques:geom_alg_fin} may help us to understand both the reasons and degrees to which this should or shouldn't remain the case.
Lastly, one may ask, still more generally:
\begin{question}
\label{ques:fgLie}
Does there exist a Lie (or even locally compact Polish) group $G$ for which the conjugacy relation on finitely generated discrete subgroups is not essentially hyperfinite?
\end{question}

\bibliography{manifolds_bib}{}
\bibliographystyle{amsalpha}

\end{document}